\documentclass{amsbook}

\usepackage{amsmath,amsfonts,amssymb,amsthm,mathabx}
\usepackage[hidelinks]{hyperref}
\usepackage{mathtools}
 \usepackage{pdfpages}
\usepackage[capitalize]{cleveref}
\usepackage{stmaryrd}
\usepackage{tikz-cd}
\usepackage{caption}

\usepackage{epigraph}  

\usepackage{multicol}        
\usepackage[bottom]{footmisc}
\usepackage{mathrsfs}

\newtheorem{theorem}{Theorem}[section]

\newtheorem{proposition}[theorem]{Proposition}
\newtheorem{lemma}[theorem]{Lemma}
\newtheorem{corollary}[theorem]{Corollary}

\newtheorem{conjecture}[theorem]{Conjecture}
\newtheorem{problem}[theorem]{Problem}

\theoremstyle{definition}
\newtheorem{definition}[theorem]{Definition}

\newtheorem{remark}[theorem]{Remark}

\newtheorem{qstn}[theorem]{Question}
\newtheorem{exmp}[theorem]{Example}

\tikzset{
math to/.tip={Glyph[glyph math command=rightarrow]},
loop/.tip={Glyph[glyph math command=looparrowleft, swap]},
loop'/.tip={Glyph[glyph math command=looparrowleft]},
 weird/.tip={Glyph[glyph math command=Rrightarrow, glyph length=1.5ex]},
  pi/.tip={Glyph[glyph math command=pi, glyph length=1.5ex, glyph axis=0pt]},
}

\DeclareMathOperator{\Di}{\mathcal D}
\DeclareMathOperator{\im}{Im}
\DeclareMathOperator{\rk}{rk}
\DeclareMathOperator{\Aut}{Aut}

\DeclareMathOperator{\End}{End}

\DeclareMathOperator{\Out}{Out}

\DeclareMathOperator{\supp}{supp}
\DeclareMathOperator{\bs}{BS}
\DeclareMathOperator{\stab}{Stab}
\DeclareMathOperator{\rr}{rr}
\DeclareMathOperator{\deck}{Deck}
\DeclareMathOperator{\cd}{cd}
\DeclareMathOperator{\cay}{Cay}
\DeclareMathOperator{\op}{op}
\DeclareMathOperator{\core}{Core}
\DeclareMathOperator{\Ore}{Ore}
\DeclareMathOperator{\ab}{ab}

\DeclareMathOperator{\Ima}{Im}
\DeclareMathOperator{\Tor}{Tor}
\DeclareMathOperator{\Ext}{Ext}

\DeclareMathOperator{\bg}{BG}
\DeclareMathOperator{\lcm}{lcm}
\DeclareMathOperator{\thick}{th}
\DeclareMathOperator{\vb}{vb}
\DeclareMathOperator{\cat}{CAT}
\DeclareMathOperator{\pr}{pr}

\newcommand{\R}{\mathbb{R}}
\newcommand{\Z}{\mathbb{Z}}
\newcommand{\N}{\mathbb{N}}
\newcommand{\C}{\mathbb{C}}
\newcommand{\Q}{\mathbb{Q}}
\newcommand{\bbS}{\mathbb{S}}

\DeclareMathOperator{\fp}{FP}
\def\immerses{\looparrowright}
\def\injects{\hookrightarrow}
\def\surjects{\twoheadrightarrow}

\newcommand{\isom}{\cong}
\newcommand{\normal}[1]{\langle\!\langle #1 \rangle\!\rangle}

\newcommand{\series}[1]{(\!( #1 )\!)}

\makeatletter
\newsavebox{\@brx}
\newcommand{\llangle}[1][]{\savebox{\@brx}{\(\m@th{#1\langle}\)}%
  \mathopen{\copy\@brx\kern-0.5\wd\@brx\usebox{\@brx}}}
\newcommand{\rrangle}[1][]{\savebox{\@brx}{\(\m@th{#1\rangle}\)}%
  \mathclose{\copy\@brx\kern-0.5\wd\@brx\usebox{\@brx}}}
\makeatother

\DeclareMathOperator{\SL}{SL}
\DeclareMathOperator{\PSL}{PSL}
\DeclareMathOperator{\GL}{GL}

\DeclareMathOperator{\FP}{FP}
\DeclareMathOperator{\BS}{BS}
\DeclareMathOperator{\gr}{\mathfrak{gr}}

\newcommand{\pres}[3]{\textnormal{#1} \langle #2 \mid #3 \rangle}

\newcommand{\cP}{\mathcal{P}}
\newcommand{\ZG}{{\Z}G}

\numberwithin{section}{chapter}

\newcounter{marcocomments}

\tikzset{
pattern size/.store in=\mcSize, 
pattern size = 5pt,
pattern thickness/.store in=\mcThickness, 
pattern thickness = 0.3pt,
pattern radius/.store in=\mcRadius, 
pattern radius = 1pt}
\makeatletter
\pgfutil@ifundefined{pgf@pattern@name@_qnfurnr9o}{
\makeatletter
\pgfdeclarepatternformonly[\mcRadius,\mcThickness,\mcSize]{_qnfurnr9o}
{\pgfpoint{-0.5*\mcSize}{-0.5*\mcSize}}
{\pgfpoint{0.5*\mcSize}{0.5*\mcSize}}
{\pgfpoint{\mcSize}{\mcSize}}
{
\pgfsetcolor{\tikz@pattern@color}
\pgfsetlinewidth{\mcThickness}
\pgfpathcircle\pgfpointorigin{\mcRadius}
\pgfusepath{stroke}
}}
\makeatother
\tikzset{every picture/.style={line width=0.75pt}} 
\usetikzlibrary{fadings}
\usetikzlibrary{patterns}
\usetikzlibrary{shadows.blur}
\usetikzlibrary{shapes}

\begin{document}

\title{{\huge The theory of one-relator groups: \\ history and recent progress} }

\author{Marco Linton\footnotemark[1]}
\address{\footnotemark[1]Instituto de Ciencias Matemáticas, Fuencarral-El Pardo, 28049, Madrid, Spain}
\email{marco.linton@icmat.es}

\author{Carl-Fredrik Nyberg-Brodda\footnotemark[2] 

\vspace{1.5cm}
\begin{tikzpicture}[x=0.75pt,y=0.75pt,yscale=0.35,xscale=0.35]
\draw  [fill={rgb, 255:red, 128; green, 3; blue, 16 }  ,fill opacity=0.49 ][line width=3.5]  (8,329) .. controls (8,151.16) and (152.16,7) .. (330,7) .. controls (507.84,7) and (652,151.16) .. (652,329) .. controls (652,506.84) and (507.84,651) .. (330,651) .. controls (152.16,651) and (8,506.84) .. (8,329) -- cycle ;
\draw  [fill={rgb, 255:red, 128; green, 3; blue, 16 }  ,fill opacity=0.49 ][line width=3.5]  (8,329) .. controls (8,157.3) and (147.19,18.1) .. (318.9,18.1) .. controls (490.6,18.1) and (629.79,157.3) .. (629.79,329) .. controls (629.79,500.7) and (490.6,639.9) .. (318.9,639.9) .. controls (147.19,639.9) and (8,500.7) .. (8,329) -- cycle ;
\draw  [fill={rgb, 255:red, 128; green, 3; blue, 16 }  ,fill opacity=0.49 ][line width=3.5]  (8,329) .. controls (8,163.43) and (142.22,29.21) .. (307.79,29.21) .. controls (473.36,29.21) and (607.59,163.43) .. (607.59,329) .. controls (607.59,494.57) and (473.36,628.79) .. (307.79,628.79) .. controls (142.22,628.79) and (8,494.57) .. (8,329) -- cycle ;
\draw  [fill={rgb, 255:red, 128; green, 3; blue, 16 }  ,fill opacity=0.49 ][line width=3.5]  (8,329) .. controls (8,169.56) and (137.25,40.31) .. (296.69,40.31) .. controls (456.13,40.31) and (585.38,169.56) .. (585.38,329) .. controls (585.38,488.44) and (456.13,617.69) .. (296.69,617.69) .. controls (137.25,617.69) and (8,488.44) .. (8,329) -- cycle ;
\draw  [fill={rgb, 255:red, 128; green, 1; blue, 15 }  ,fill opacity=1 ][line width=3.5]  (34.82,334.98) .. controls (34.82,187.8) and (154.13,68.5) .. (301.3,68.5) .. controls (448.48,68.5) and (567.78,187.8) .. (567.78,334.98) .. controls (567.78,482.15) and (448.48,601.46) .. (301.3,601.46) .. controls (154.13,601.46) and (34.82,482.15) .. (34.82,334.98) -- cycle ;
\draw  [fill={rgb, 255:red, 0; green, 69; blue, 0 }  ,fill opacity=1 ][line width=3.5]  (160.22,334.98) .. controls (160.22,222.43) and (251.46,131.2) .. (364,131.2) .. controls (476.55,131.2) and (567.78,222.43) .. (567.78,334.98) .. controls (567.78,447.52) and (476.55,538.76) .. (364,538.76) .. controls (251.46,538.76) and (160.22,447.52) .. (160.22,334.98) -- cycle ;
\draw  [fill={rgb, 255:red, 128; green, 1; blue, 15 }  ,fill opacity=1 ][line width=3.5]  (160.22,334.98) .. controls (160.22,259.23) and (221.63,197.82) .. (297.38,197.82) .. controls (373.13,197.82) and (434.54,259.23) .. (434.54,334.98) .. controls (434.54,410.73) and (373.13,472.14) .. (297.38,472.14) .. controls (221.63,472.14) and (160.22,410.73) .. (160.22,334.98) -- cycle ;
\draw  [fill={rgb, 255:red, 0; green, 69; blue, 0 }  ,fill opacity=1 ][line width=3.5]  (254.28,334.98) .. controls (254.28,285.2) and (294.63,244.84) .. (344.41,244.84) .. controls (394.19,244.84) and (434.54,285.2) .. (434.54,334.98) .. controls (434.54,384.76) and (394.19,425.11) .. (344.41,425.11) .. controls (294.63,425.11) and (254.28,384.76) .. (254.28,334.98) -- cycle ;
\draw  [fill={rgb, 255:red, 128; green, 1; blue, 15 }  ,fill opacity=1 ][line width=3.5]  (105.5,278.58) .. controls (123.1,278.58) and (138.6,287.51) .. (147.74,301.08) .. controls (141.02,309.71) and (137.01,320.55) .. (137.01,332.33) .. controls (137.01,342.83) and (140.2,352.59) .. (145.65,360.69) .. controls (136.34,372.64) and (121.82,380.32) .. (105.5,380.32) .. controls (77.41,380.32) and (54.63,357.54) .. (54.63,329.45) .. controls (54.63,301.35) and (77.41,278.58) .. (105.5,278.58) -- cycle (147.74,301.08) .. controls (157.05,289.14) and (171.57,281.46) .. (187.88,281.46) .. controls (204.82,281.46) and (219.83,289.74) .. (229.08,302.47) .. controls (222.99,310.86) and (219.4,321.18) .. (219.4,332.33) .. controls (219.4,343.49) and (222.99,353.8) .. (229.08,362.19) .. controls (219.83,374.92) and (204.82,383.2) .. (187.88,383.2) .. controls (170.29,383.2) and (154.78,374.27) .. (145.65,360.69) .. controls (152.37,352.07) and (156.37,341.23) .. (156.37,329.45) .. controls (156.37,318.95) and (153.19,309.19) .. (147.74,301.08) -- cycle (229.08,302.47) .. controls (238.32,289.74) and (253.33,281.46) .. (270.27,281.46) .. controls (286.35,281.46) and (300.69,288.93) .. (310.01,300.58) .. controls (303.05,309.28) and (298.89,320.32) .. (298.89,332.33) .. controls (298.89,344.34) and (303.05,355.38) .. (310.01,364.08) .. controls (300.69,375.74) and (286.35,383.2) .. (270.27,383.2) .. controls (253.33,383.2) and (238.32,374.92) .. (229.08,362.19) .. controls (235.16,353.8) and (238.75,343.49) .. (238.75,332.33) .. controls (238.75,321.18) and (235.16,310.86) .. (229.08,302.47) -- cycle (310.01,300.58) .. controls (319.34,288.93) and (333.68,281.46) .. (349.76,281.46) .. controls (366.7,281.46) and (381.71,289.74) .. (390.95,302.47) .. controls (384.86,310.86) and (381.27,321.18) .. (381.27,332.33) .. controls (381.27,343.49) and (384.86,353.8) .. (390.95,362.19) .. controls (381.71,374.92) and (366.7,383.2) .. (349.76,383.2) .. controls (333.68,383.2) and (319.34,375.74) .. (310.01,364.08) .. controls (316.97,355.38) and (321.14,344.34) .. (321.14,332.33) .. controls (321.14,320.32) and (316.97,309.28) .. (310.01,300.58) -- cycle (390.95,302.47) .. controls (400.2,289.74) and (415.2,281.46) .. (432.14,281.46) .. controls (449.08,281.46) and (464.09,289.74) .. (473.33,302.47) .. controls (467.24,310.86) and (463.65,321.18) .. (463.65,332.33) .. controls (463.65,343.49) and (467.24,353.8) .. (473.33,362.19) .. controls (464.09,374.92) and (449.08,383.2) .. (432.14,383.2) .. controls (415.2,383.2) and (400.2,374.92) .. (390.95,362.19) .. controls (397.04,353.8) and (400.63,343.49) .. (400.63,332.33) .. controls (400.63,321.18) and (397.04,310.86) .. (390.95,302.47) -- cycle (483.01,332.33) .. controls (483.01,343.49) and (479.42,353.8) .. (473.33,362.19) .. controls (482.58,374.92) and (497.58,383.2) .. (514.52,383.2) .. controls (542.62,383.2) and (565.39,360.43) .. (565.39,332.33) .. controls (565.39,304.24) and (542.62,281.46) .. (514.52,281.46) .. controls (497.58,281.46) and (482.58,289.74) .. (473.33,302.47) .. controls (479.42,310.86) and (483.01,321.18) .. (483.01,332.33) -- cycle ;
\draw  [pattern=_qnfurnr9o,pattern size=9pt,pattern thickness=0.75pt,pattern radius=1.5pt, pattern color={rgb, 255:red, 0; green, 0; blue, 0}][line width=3.5]  (34.82,334.98) .. controls (34.82,187.8) and (154.13,68.5) .. (301.3,68.5) .. controls (448.48,68.5) and (567.78,187.8) .. (567.78,334.98) .. controls (567.78,482.15) and (448.48,601.46) .. (301.3,601.46) .. controls (154.13,601.46) and (34.82,482.15) .. (34.82,334.98) -- cycle ;
\end{tikzpicture}\vspace{4.0cm}\vfill}

\address{\footnotemark[2]School of Mathematics, Korea Institute for Advanced Study (KIAS), Seoul 02455, Korea}
\email{cfnb@kias.re.kr}

\maketitle
\tableofcontents

\clearpage

\chapter*{Preface}


\noindent The theory of one-relator groups is now almost a century old. The authors therefore feel that a comprehensive survey of this fascinating subject is in order, and this document is an attempt at precisely such a survey. This article is divided into two chapters, reflecting the two different phases in the story of one-relator groups. The two chapters can be read independently of one another and have minimal overlap. 

The first chapter, written by the second author, covers the history of one-relator groups. The initial development of the theory of one-relator groups was parallel to that of combinatorial group theory, and the chapter begins with the intertwined history. As the theory of one-relator groups eventually grew into a subject in its own right, the chapter turns to treating problems particular to this subject. The chapter covers the history until around the year 1987, a year which marked the advent of modern geometric group theory. 

The second chapter, written by the first author, surveys recent progress on the subject. In the last few decades, new ideas coming from geometry, topology and homological algebra have decidedly changed how one-relator groups are studied, resulting in several significant developments. This chapter surveys these new developments and the tools used to prove them. Some proofs are provided and several new open problems are posed, suggesting further directions for future research.

Reading the first chapter will show the reader how many of the deep original ideas and research directions of the subject continue to influence the subject; reading the second chapter will give the reader access to the rich and modern toolkit used to study one-relator groups today. It is the hope of the authors that the combination of these two points of view will convince the reader of the depth, relevance, and perennial beauty of the theory of one-relator groups.

\

\noindent \textbf{Acknowledgements.} The first author would like to thank Sam Kim and Carl-Fredrik Nyberg-Brodda for inviting him to write a survey article (\cref{chap:survey2}) for the KIAS Springer Series in Mathematics and for their warm hospitality during his visit to Korea. He would also like to thank Carl-Fredrik Nyberg-Brodda for enriching his knowledge of the history of one-relator groups and Sam Hughes for helpful conversations on profinite invariants. The second author would like to thank Marco Linton for providing him with many insights into the theory of one-relator groups, and John Stillwell for helpful comments. The cover image is based on a design by B.\ B.\ Newman for the cover of his kangaroo leather-bound PhD thesis.

The first author acknowledges supports from the grant 202450E223 (Impulso de líneas científicas estratégicas de ICMAT). The second author is supported by the Mid-Career Researcher Program (RS-2023-00278510) through the National Research Foundation funded by the government of Korea, and by the KIAS Individual Grant MG094701 at Korea Institute for Advanced Study.

\chapter{History}

\section{Introduction}

\noindent What is the \textit{simplest} non-trivial group? This is an ambiguous question, and in attempting to specify it one is led immediately to group-theoretic measures of complexity. What does it mean for a group to be complicated? Group theory is riddled with measures for answering such questions. From the point of view of finite group theory, one might base complexity either coarsely on the cardinality of the group, or more finely on the number of distinct prime divisors of the order of the group. Many early results in finite group theory -- such as Burnside's theorem on groups of order $p^a q^b$ for primes $p, q$ -- arise from such considerations. The challenges and associated measures for finite groups are quite particular to that field, however. Instead, one might choose a broader \textit{combinatorial} perspective by treating all groups as structures satisfying a minimal set of relations (associativity and existence of inverses), together with some additional set of defining relations. This perspective, of defining groups via presentations $\pres{}{A}{R}$ by generators and relations, is certainly not always the most beneficial, and e.g.\ many questions about presentations of finite groups remain far out of reach and of dubious significance. However, we can also gain much from this combinatorial perspective. 

One of many benefits is that we have been led to \textit{free groups},  being those groups with no additional defining relations whatsoever, as a starting point for a measure of complexity of a group. Here, finite groups and infinite groups are placed on an equal footing, for the methods of combinatorial group theory deal with finite groups in the same manner as infinite groups: namely, via their presentations. Since all groups are obtainable from free groups simply by adding some set of relations, one immediately finds a natural target for describing the complexity of a group: \textit{how complicated are its defining relations?} Adding relations of a very particular form naturally leads to subjects like small cancellation theory, whereas adding very few relations leads to the subject matter of the present survey: \textit{one-relator groups}. For what could be more natural to ask than what happens if we add only a single relation to a free group? 

The study of one-relator groups is now almost a century old, and has in recent years seen a resurgence of activity. A number of old conjectures have been spectacularly resolved, and techniques from many different areas have come to dominate the area and ensure its relevance for many years to come. A survey of these developments, covering approximately the years 1990 to today (and especially the last 15 years since 2010), constitutes the second chapter of this present document. Doubtlessly, this present period will be cemented as an important part of the history of one-relator groups. The goal of this present survey is to cover the older half of the history of one-relator groups, covering approximately the years 1930--1987. These years are neither arbitrarily chosen nor perfectly delimiting. Firstly, they are chosen with care: 1930 was the year of the publication of Magnus' proof of the \textit{Freiheitssatz}, which remains one of the core results of the theory. On the other hand, 1987 was the year Gromov published his famous monograph on \textit{hyperbolic} groups \cite{Gr87}, the influence of which fundamentally changed much of combinatorial group theory (in particular overseeing its shift of focus towards \textit{geometric} group theory). Secondly, however, we do not mean to say that no articles from before 1930 are cited, nor none from after 1987. For the former, the entire first section of this survey deals with the pre-history of one-relator groups, and that thin slice of combinatorial group theory includes articles from around 1880--1920. Furthermore, there are also more recent articles which ``belong'' to the older half of the history of one-relator groups. For example, a 2000 article by McCool \cite{McCool2000} considers a very classical problem on \textit{normal roots} (see \S\ref{Subsec:Applications-of-Frei}), and deals with it in a very classical manner; thus, it naturally falls into the scope of this present survey. Some other recent developments on classical problems are similarly also referenced in passing, albeit by necessity with essentially no details of the methods of proof.

There already exist surveys on the theory of one-relator groups, and while they all feature some history, they rarely do so in a comprehensive manner, instead focussing on the select strands of the history directly relevant to the topics presented in the survey (such as the recent survey by Baumslag, Fine \& Rosenberger \cite{BFR19}), or else written from the point of view of containing the development of one-relator group theory inside the much broader development of combinatorial group theory (such as in the spectacular book by Chandler \& Magnus \cite[Chapter~II.5]{Chandler1982}). The author of this survey has also included a brief history of one-relator group theory in a related article \cite{NybergBrodda2021}, but the history presented therein is done so in a manner much too brief and lacunary to serve as a standard reference. It is the hope of the author that this present survey will fill this gap in the literature. Throughout, in line with the maxim of Edwards \cite{Edwards1981} and Stillwell \cite{Stillwell2023} to \textit{read the masters}, we have attempted to make precise references to the original works alongside pointers to further literature. In particular, it is the sincere hope of the author to leave the reader who already knows about one-relator groups with some new tidbits about their history; and to leave the reader unfamiliar with this fascinating topic with a desire to learn more. 

\

\noindent We give a brief outline of the survey:
\begin{itemize}
\item In \S\ref{Sec:2-Dehn-freegroups-presentations} (covering approx.\ 1880--1930) we give an overview of the pre-history of one-relator group theory, beginning with the advent of group presentations and leading into the work of M.\ Dehn. We emphasize his three fundamental decision problems, his work on surface groups, and the work by his students on certain one-relator knot groups. We conclude with a section on the Nielsen--Schreier Theorem and amalgamated free products. 
\item In \S\ref{Sec:3-Magnus}, we present three fundamental articles by W.\ Magnus, from 1930, 1931, and 1932, respectively. This includes a proof of the central \textit{Freiheitssatz} via the Magnus hierarchy and several of its applications, including the Conjugacy Theorem and normal roots. We conclude with Magnus' proof of the decidability of the word problem in all one-relator groups. 
\item In \S\ref{Sec:4-Lull-phase}, we outline the ``Lull Phase'' of one-relator groups, being approximately the years 1932--1960 in which a paucity of results in the subject appeared. Nevertheless, we discuss Lyndon's Identity Theorem, an important result from this period, and the parallel study of residual properties of groups, central to later sections. Finally, we discuss some plausible reasons for why this phase experienced a dearth of results.  
\item In \S\ref{Sec:5-1960s-resurgence}, beginning in 1960, we cover a period of resurgence of interest in one-relator groups, in which many new results were proved, and many central figures -- most centrally G.\ Baumslag -- would enter the playing field. We cover the Torsion Theorem of Karrass, Magnus \& Solitar and how it was prompted by the Poincaré Conjecture. Next, we discuss hopfian groups, and the remarkable non-hopfian groups constructed by Baumslag \& Solitar. Finally, we discuss the emergence of two new centres for one-relator group theory outside Germany: one in New York, and one in the USSR. The latter center gave rise to remarkable work by D.\ I.\ Moldavanskii, and we discuss his interpretation of the Magnus hierarchy. Finally, we briefly discuss (finitely-generated free)-by-cyclic one-relator groups. 
\item In \S\ref{Sec:6-OR-with-torsion}, we discuss one-relator groups with torsion, and the realization of the 1960s that the presence of torsion is a significantly simplifying factor. In this line, we begin by discussing the B.\ B.\ Newman Spelling Theorem and its many remarkable consequences, all proved in the thesis of B.\ B.\ Newman, followed by some residual and virtual properties of one-relator groups. We conclude the section with a discussion of the work by S.\ J.\ Pride following in the footsteps of B.\ B.\ Newman, including his resolution of the isomorphism problem for two-generated one-relator groups with torsion. 
\end{itemize}

\noindent The two last sections are more nebulous in their chronology, caused by the large number of distinct branches of the subject, and cover roughly the years 1960--1990.

\begin{itemize}
\item In \S\ref{Sec:7-Subgroups-of-OR-groups}, we broadly discuss subgroups of one-relator groups, beginning with a discussion of Magnus subgroups and their malnormality. We turn to the remarkable theory of the center of one-relator groups, as completed by A.\ Pietrowski, and continue by discussing abelian and solvable subgroups of one-relator groups. We discuss coherence and subgroup separability, and some further results on quotients and residual properties of one-relator groups. Finally, we end with a discussion on the lower central series and factors of one-relator groups, including parafree groups.
\item Finally, in \S\ref{Sec:8-Decisionproblems-1960-1980} we go further in-depth on some decision problems for one-relator groups, including the conjugacy problem, the isomorphism problem (both of which remain open in general), and the complexity of the word problem. We end with a brief discussion on the obstacles to extending one-relator group theory to two-relator group theory.
\end{itemize}

\clearpage

\section{Dehn, free groups, and presentations (1880--1930)}\label{Sec:2-Dehn-freegroups-presentations}

\noindent In this section, we present some of the history of the central ideas that began the development of one-relator group theory. Specifically, we will describe the groundwork laid by M.\ Dehn and students in the 1910s and 1920s, particularly focussing on the three prototypical classes of examples of one-relator groups: surface groups, torus knot groups, and the figure-eight knot (Listing's knot) group. Finally, we will also discuss the Nielsen--Schreier Theorem, stating that all subgroups of free groups are free. The methods developed by Nielsen and Schreier as part of their investigations of free groups -- Nielsen transformations and amalgamated free products, respectively -- will echo throughout this survey. We will therefore end the section with a brief overview of the theory of amalgamated free products. This section is \textbf{not} intended to serve as a complete, or even partial, history of combinatorial group theory. For such treatments, we instead refer the reader to the monographs by Chandler \& Magnus \cite{Chandler1982} or Stillwell \cite{Stillwell1993}.

\subsection{Presentations and Dehn's fundamental problems}\label{Subsec:Dyck-and-presentations}


{
\setlength{\epigraphwidth}{0.85\textwidth}
\epigraph{\textit{Dehn's work and the influence of his ideas created a synthesis of topological and group theoretical problems which may be characterized as follows: }
\begin{enumerate}
\item \textit{The theory of non-abelian groups provides effective tools for the solution of topological problems.}
\item \textit{The group theoretical results needed for the purposes of topology have characteristics which did not appear in the previously existing group theoretical literature. The new results were of an algorithmic nature} [...]
\end{enumerate}\leavevmode}{---W. Magnus, 1978 \cite[p. 135]{Magnus1978}}
}

\noindent The theory of group presentations has an intricate, deep, and long history, so we shall content ourselves with beginning the story at definite year: 1882. Group presentations had appeared at sporadic points before this. One such point, pointed out to the author by J.\ Stillwell, is a presentation of the alternating group $A_5$ appearing in an 1856 article by W.\ R.\ Hamilton \cite{Hamilton1856} on \textit{icosian calculus}. Another such sporadic point is the 1880 publication by A.\ V.\ Vasiliev \cite{Vasiliev1880}, who in his master's dissertation (see \cite{Halsted1897}) gives the same presentation for $A_5$ as Hamilton. No systematic theory of presentations is developed in either of these two articles, however. The year 1882, on the other hand, saw the publication of W.\ Dyck's \textit{Habilitationsschrift} \cite{Dyck1882b, Dyck1882}\footnote{Dyck had been supervised for his doctorate by F.\ Klein, who had in turn been influenced by Cayley's work. Thus, Dyck's article \cite{Dyck1882}, which is otherwise written in German, curiously begins with an English quote from Cayley's 1878 article \cite[p. 51]{Cayley1878}: ``A group is defined by means of the laws of combination of its symbols''.}. Therein, Dyck places the theory of group presentations on a solid footing, and gives presentations for a number of groups. However, more importantly for our story, his investigations opened up the possibility of investigating groups purely defined by means of their presentations. For example, Dyck calls a group $\pres{}{A}{\varnothing}$ with no defining relations the ``most general'' (Ger.\ \textit{allgemeinste}, \cite[p. 5]{Dyck1882}) group. Today such a group is, of course, called a \textit{free group}\footnote{The name \textit{free group} only appeared later; it is attributed to E.\ Artin by Schreier \cite[p. 163]{Sc27}.}. Dyck also describes an \textit{algorithm} -- a very early usage of this word! -- for computing all subgroups of a given finite group using presentations, and applies it to the symmetric group $S_4$ as an example (\cite[\S4.1]{Dyck1882}). Furthermore, one result (\cite[p.\ 12--13]{Dyck1882} that is central to his arguments is the following: suppose $G$ is a group given by the presentation
\begin{equation}
G = \pres{}{A}{r_1 = 1, \dots, r_k = 1}.
\end{equation}
Then the set of all elements in the free group $F_A$ on $A$ which represent the identity element in $G$ forms a normal subgroup $R$ of $F_A$, and furthermore any word $w \in R$ can be written as
\begin{equation}\label{Eq:Dyck-Prehistory}
W = \prod_{i=1}^m T_i r_{t_i}^{\pm 1} T_i^{-1} \quad \textnormal{where } T_i \in F_A, \textnormal{ and } t_i = 1, 2, \dots, n.
\end{equation}
In modern terminology, this amounts to little more than that $G$ is equal to the quotient of $F_A$ by the \textit{normal closure} of the words $r_1, \dots, r_k$. While seemingly simple, this ``normal form'' result would lead to many later developments, appearing both in the Reidemeister--Schreier process and Magnus' proof of the \textit{Freiheitssatz} (\S\ref{Subsec:Freiheitssatz}). Dyck also made further use of this result, and others, and published a sequel to his first article the very next year \cite{Dyck1883}. While this heavily features presentations, the article is primarily concerned with finite permutation groups rather than infinite groups, creating a unified and abstract approach to the former subject.\footnote{This is also emphasized by the fact that Dyck cites, in a footnote, both Hamilton's and Vasiliev's articles mentioned above \cite[p.\ 82]{Dyck1883}.} This second article can thus be seen as one of the very early articles in the subject of computational finite group theory. The reader further interested in the role Dyck's work played in the development of the abstract notion of a group is referred to Wussing's discussion of the subject \cite[p.~240--243]{Wussing1984}.

The subject of presentations of groups would lay dormant for some time, until it was picked up by M.\ Dehn (1878--1952) in the early 1900s. In 1911 \cite{De11}, he posed the following three fundamental algorithmic problems:
\begin{enumerate}
\item \textit{The word problem}; i.e.\ the problem of deciding when a given word over the generators of a group is equal to $1$ in the group;
\item \textit{The conjugacy problem}; i.e.\ the problem of deciding if two given words are conjugate or not; and
\item \textit{The isomorphism problem}; i.e.\ the problem of deciding whether two presentations define isomorphic groups or not. 
\end{enumerate}
Today, we know that all these problems are, for general finitely presented groups, algorithmically \textit{undecidable} (see \S\ref{Subsec:Undecidability}). As pointed out by Stillwell \cite[\S5]{Stillwell2012} (see also \cite[Footnote~17]{Stillwell2014}) there were some remarkable early suspicions of this, even before the formalization of undecidability appeared in the mid-1930s. For example, already in 1908, Tietze \cite[\S14, p.~80]{Tietze1908} noted that ``the question of whether two groups are isomorphic is not solvable in general''. In 1932, Reidemeister \cite[\S2.10]{Reidemeister1932} notes that ``it is in general not possible to decide whether two groups presented by generators and relations are isomorphic to each other. One also cannot decide whether such a group is free group...'' \cite[pp. 40--41]{Stillwell2014}.\footnote{However, as Stillwell notes (\textit{loc.\ cit.}), Reidemeister may have had ``some intimation of the coming wave of unsolvability'', having organized the 1930 conference where G\"odel first announced his (closely related) incompleteness theorems.} These remarkable and (at the time) unsubstantiated claims would not be proved -- or indeed formulated! -- formally until many decades later in the 1950s (see \S\ref{Subsec:Undecidability}). 

Around 1914, however, the spectre of undecidability had not yet come to haunt combinatorial group theory, and Dehn had produced the solution to the word problem in several classes of one-relator groups. The most central examples are the \textit{surface groups} and the \textit{trefoil knot group} $\pres{}{a,b}{a^2b^3 = 1}$, the word problem for which he solves using the \textit{Gruppenbild} (or \textit{Cayley graph} \cite{Cayley1878b}). We shall expand on these topics in \S\ref{Subsec:Surface-and-knot-groups} below. For now, we suffice to mention a few pieces of foreshadowing for the subsequent developments of one-relator group theory. First, in 1911 Dehn \cite[\S3, p. 144]{De11} also considered the word problem in the one-relator group $\pres{}{a,b}{a^3 b^3 = 1}$, but only provided a sketch of a solution. Dehn would also assign several of his students to work on topics adjacent to one-relator group theory around this time, two of which were H.\ Gieseking (1887--1915) and Fritz Klein (about whom almost nothing is known). We will expand on Gieseking's work below in \S\ref{Subsec:Surface-and-knot-groups}. Certainly Gieseking, and probably also Klein, died in the First World War. For the remainder of the 1910s there was little progress on one-relator group theory. 

After the war, however, the subject would again begin to stir. In 1922, in an unpublished lecture given in Leipzig\footnote{According to Reidemeister \cite{Reidemeister1932}, the lecture was an address to the Naturforscherversammlung. He also confirms Magnus' remarks in \cite[p. 141]{Ma30}, cf.\ also \cite[p. xi]{Magnus1984}, that the notes from this lecture were, unfortunately, never published.}, Dehn made a remarkable conjecture on the structure of certain subgroups of one-relator groups. The conjecture reads as follows: \textit{given a one-relator group $G = \pres{}{A}{r=1}$ in which all letters appear in $r$, and such that $r$ is not obviously conjugate to some shorter word, then the subgroup of $G$ generated by any proper subset of $A$ is free}. In other words: if one cannot spell the word $r$, then one cannot get any non-trivial relation whatsoever. This statement would become known as the \textit{Freiheitssatz} (``Freeness theorem''). The German-speaking reader will note that it is not called the \textit{Freiheitsvermutung} (``Freeness conjecture''). This is because, as we shall see in \S\ref{Sec:3-Magnus}, Dehn's student Magnus would end up writing down the first (combinatorial) proof of the result in 1930. This proof would mark one of the first milestones in one-relator group theory. Dehn himself likely had a very solid idea of a (geometric) proof of the result already in 1922\footnote{C.\ L.\ Siegel, for example, said that Dehn ``did discover a proof [of the \textit{Freiheitssatz}], however, and would on occasion tell his friends how it went'' \cite[p. 223]{Siegel1978}.}, but never published it. This was perhaps auspicious, as the timing of Magnus' eventual combinatorial proof in 1930, coinciding with the development by Reidemeister of combinatorial group theory and Schreier on amalgamated free products, could not have been better.

For a broad introduction to Dehn's life, we refer the reader to Stillwell's article \cite{Stillwell1999} and \cite{Magnus1954}. For some of Dehn's later articles, including on the famous Dehn--Nielsen Theorem, we highly recommend the reader to consult Stillwell's excellent translation and commentary \cite{Dehn1987}. Two biographies have also recently appeared: one by Peifer \cite{Peifer2015} concerning Dehn's topological and group-theoretic work, and another by Rowe \cite{Rowe2023} concerning his work on the history of mathematics. Indispensible for the broader context of Dehn's topological work is also Stillwell's \cite{Stillwell2012} history of the development of $3$-manifolds. 
  
Dehn did not ever return to one-relator group theory after his initial work on the \textit{Freiheitssatz} mentioned above, being instead drawn to other problems. We recall some of the fascinating moments in Dehn's later life as presented by Stillwell \cite{Stillwell1999}. Dehn's life would once again be affected by war: in 1938, Dehn was arrested as part of the anti-Jewish programs organized by the Nazis. He was shortly thereafter released ``because there was no more room in Frankfurt to keep prisoners under lock and key''. Eventually, with the help of W.\ Magnus, E.\ Hellinger, and C.\ Haas, was able to escape to America. Deeming the Atlantic too dangerous, the 62-year old Dehn escaped via Stockholm, Moscow, Siberia, then Japan, arriving finally in America in 1941. He eventually settled in Black Mountain College in North Carolina, a small liberal arts college. During his years in America, he was occupied with mostly non-mathematical work, writing e.g.\ a charming series of articles summarizing the early history of mathematics \cite{Dehn1943b, Dehn1943, Dehn1944, Dehn1944b}. Dehn passed away in 1952, the same year that P.\ S.\ Novikov \cite{Novikov1952} announced the general unsolvability of the first of Dehn's fundamental problems, i.e.\ the word problem in finitely presented groups. Within a few years, undecidability would permeate combinatorial group theory, heralding a new era for the subject: we discuss this period, and the large impact it likely had on one-relator group theory, in more detail in \S\ref{Subsec:Undecidability}.

\begin{remark}
When writing about the word problem, it is important to remark that at the same time as Dehn, this problem was also considered by A.\ Thue (1863--1922) in Norway. He studied the properties of the word problem in \textit{monoids}, i.e.\ structures with an associative binary operation for which an identity element exists, and solved it in some particular cases. For example, he gives some examples of one-relation monoids in which the word problem is decidable, including the monoid with three generators $a, b, c$ and the single defining relation $ab bc ab = 1$, see \cite[\S VII]{Thue1914}. This and several others monoids considered by Thue turn out to be (one-relator) groups; in the specified example, it can be shown that it is isomorphic to the free group of rank $2$. It is unlikely that Thue was in any way influenced by Dehn, since, as pointed out by Magnus \cite[p. 141]{Magnus1978}, there was only a fleeting acquaintance between Dehn and Thue. However, in the same place, Magnus remarks that ``[...] the emphasis in Thue's papers is on infinite sequences with certain properties. It is at least not obvious how his results could be applied to any word problem.'' As the above example and others like it show, Magnus' remarks seem misguided, and are not an accurate representation of Thue's 1914 work on the word problem. Indeed, Magnus' remarks seem to be more applicable either to Thue's 1906 article containing the famous infinite Thue--Morse sequence \cite{Thue1906}, or perhaps his 1910 article discussing a precursor to term rewriting systems \cite{Thue1910} (see \cite{Steinby2000} for a modern reading of this article). In any case, what is true is that Thue's work on the word problem would not have any significant impact on the word problem for groups until the resurgence of methods of \textit{string rewriting systems}, or (semi-)\textit{Thue systems}, for groups in the 1970s and 1980s; see \S\ref{Subsec:More-on-WP}. 
\end{remark}

\subsection{Surface groups and knot groups}\label{Subsec:Surface-and-knot-groups}

Some of the most fundamental examples of one-relator groups were also the first to be studied. These groups are the fundamental groups $\pi_1(\Sigma_g)$ of $2$-surfaces, and much of the history of one-relator group theory can be summarized as attempts to generalize techniques and results that hold for $\pi_1(\Sigma_g)$. Given a compact orientable $2$-surface $\Sigma_g$ of genus $g \geq 1$, reading the boundary of its fundamental polygon shows that its fundamental group has a presentation given by
\begin{equation}\label{Eq:pres-of-sigma-g}
\pi_1(\Sigma_g) = \pres{}{a_1, b_1, \dots, a_g, b_g}{[a_1, b_1][a_2,b_2] \cdots [a_g, b_g] = 1}.
\end{equation}
Similarly, given a compact \textit{non-}orientable surface $\mathcal{K}_g$ of genus $g$, a presentation of its fundamental group is given by 
\begin{equation}\label{Eq:pres-of-K-g}
\pi_1(\mathcal{K}_g) = \pres{}{a_1, \dots, a_g}{a_1^2 a_2^2 \cdots a_g^2 = 1}.
\end{equation}
The particular case of $g=2$ is the famous \textit{Klein bottle}\footnote{The Klein ``bottle'' is likely a curious case of mistranslation. The original German is ``Kleinsche Fläche'', simply meaning ``Klein surface'', and refers to the familiar surface constructed by Klein in 1882 \cite{Klein1882}. However, these words were likely misread as ``Kleinsche Flasche'', i.e.\ ``Klein bottle'', and thus this strange piece of terminology made its way into the mathematical literature. The illustration of the surface as a self-intersecting bottle is likely a back-formation.} $\mathcal{K}_2$. In this section, we will present some of Dehn's work on these groups. It bears emphasizing that the surface groups $\pi_1(\Sigma_g)$ remain today among the flagship examples of one-relator groups. 

The use of hyperbolic geometry to study topological questions appeared already in works of Poincaré in 1904 \cite{Poincare1904}, but Dehn's remarkable contribution was to merge these two topics together with combinatorial group theory. Dehn made two separate studies of the groups $\pi_1(\Sigma_g)$. In 1911 \cite{De11}, he solves the word and conjugacy problem in all such groups (and their non-orientable counterparts) by constructing the universal covering space of $\Sigma_g$ and a hefty dose of hyperbolic geometry. He also gives a brief solution to the isomorphism problem for this class, by proving that $\pi_1(\Sigma_{g_1}) \cong \pi_1(\Sigma_{g_2})$ if and only if $g_1 = g_2$ via an appeal to the abelianization, appealing to Tietze's work \cite{Tietze1908}. While this latter feat seems modest today, it is important to remember that at this point we are still in the early days of combinatorial group theory, and such arguments still had to be performed with care. 

In 1912, the very next year, Dehn \cite{De12} returned to provide a simpler combinatorial argument for the decidability of the word and conjugacy problems in surface groups. His argument was based on the following key result:

\begin{theorem}[Dehn's Spelling Theorem, \cite{De12}]
Let $G = \pi_1(\Sigma_g)$ be as in \eqref{Eq:pres-of-sigma-g}. Let $w$ be a non-trivial reduced word, in the generators of $G$, which represents the identity element of $G$. Then there exists a subword $u$ of $w$ such that $u$ is a subword of some cyclic conjugate of the defining relator of $G$ or its inverse, and such that the length of $u$ is strictly more than half the length of the defining relator.\label{Thm:Dehn-spelling-theorem}
\end{theorem}

This ``spelling'' theorem, its name being due to it providing information on the spelling of words equal to $1$ in the underlying group, gives rise to an obvious algorithm for solving the word problem: given a non-empty word $w$, to decide if it is equal to $1$ in $\pi_1(\Sigma_g)$, scan it for occurrences of all subwords $u$ of all cyclic conjugates of the defining relator and its inverse, such that $u$ has length at least half that of the defining relator. If such a word cannot be found, then $w \neq 1$ in $G$. Otherwise, replace $u$ by the shorter word representing $u^{-1}$, and repeat. Then $w = 1$ in $G$ if and only if the empty word is reached. This algorithm, known as \textit{Dehn's algorithm}, can be interpreted as a form of partially confluent string rewriting system, in which words are rewritten to shorter words by replacing certain subwords by shorter words, and which is guaranteed to be confluent to a normal form (i.e.\ converge in a finite number of steps, see \S\ref{Subsec:More-on-WP}) when applied to words representing $1$, see \cite[p. 70--71]{Book1993}. In this way, Dehn's algorithm continues to influence theoretical computer science, by leading to the notion of string rewriting systems which are confluent only on certain equivalence classes, which has seen some study \cite{Otto1987, Otto1991}.

One particularly pleasant property of working with surface groups is that their subgroups are all themselves either surface groups or free. Specifically, we have the following result: 

\begin{theorem}\label{Thm:Surface-subgroup-theorem}
Let $G$ be a surface group of genus $g$. Then any subgroup of finite index is again a surface group (of genus $\geq g$) and any subgroup of infinite index is free.
\end{theorem}

The above theorem (Theorem~\ref{Thm:Surface-subgroup-theorem}) is difficult to attribute explicitly, particularly because today it is not difficult to prove topologically. Its first published algebraic proof, using the Reidemeister--Schreier method, only appeared in 1971 by Hoare, Karrass \& Solitar \cite{Hoare1971, Hoare1972}, cf.\ also \cite[p. 306]{Magnus1969} and \cite{Burns1983}. However, we emphasize that the theorem was certainly known to Dehn, and likely known also to Poincaré. One can get explicit bounds on the genera of the surface subgroups of finite index \cite[Theorem~1]{Hoare1971}: in the orientable case, any subgroup of index $n$ in $\pi_1(\Sigma_g)$ is isomorphic to $\pi_1(\Sigma_{n(g-1)+1})$. In the non-orientable case, any subgroup of index $n$ in $\pi_1(\mathcal{K}_g)$ is either isomorphic to $\pi_1(\mathcal{K}_{n(g-2)+2})$ or to $\pi_1(\Sigma_{\frac{n}{2}(g-2) + 1})$. In particular, since $\pi_1(\Sigma_g)$ surjects every finite cyclic group, it follows that $\pi_1(\Sigma_2)$ \textit{contains $\pi_1(\Sigma_g)$ as a subgroup of finite index $g-1$ for every $g \geq 2$.} In the direction of understanding finite index subgroups of finitely generated discrete groups of motions of the hyperbolic plane, we briefly mention that it was open for quite some time whether all such groups were virtually torsion-free, a problem known as \textit{Fenchel's conjecture}. The problem admitted partial progress by Bundgaard \& Nielsen \cite{Bundgaard1951} followed by a full positive solution by Fox \cite{Fox1952} (corrected in \cite{Chau1983}; in this direction see also \cite{Mennicke1967,Mennicke1968,Lehner1970}) only in the early 1950s. 

The next early fundamental examples of one-relator groups come from \textit{knot groups}, particularly the \textit{trefoil knot group} and the \textit{figure-eight knot group} (Listing's knot group). The history of knot theory and the associated knot groups is long, and can be traced back to the very beginning of topology.\footnote{Indeed, the very word \textit{topology} was introduced by Listing already in 1848 \cite{Listing1848}, in which he among many other subjects considers the \textit{figure-eight knot}, which we will encounter below. This knot is also known (primarily in German sources) as \textit{Listing's knot} for this reason.} Early knot theorists also leave their mark on later generations of combinatorial group theorists. To name just one example, W.\ Wirtinger, one of the first knot theorists, would have much influence on K.\ Reidemeister, who would in turn have a great deal of influence on his student O.\ Schreier. We will encounter the latter two frequently throughout this part of the survey (particularly in \S\ref{Subsec:Nielsen-Schreier}). Because the history of knot theory developed alongside the theory of one-relator groups, and did so in many intricate and independent  ways, the remainder of this section will by necessity be somewhat anachronistic: we will detail those developments of knot theory pertaining to one-relator group theory taking place between c.\ 1910--1930. At the end of it, however, in preparation for our discussion of the Nielsen--Schreier Theorem in the next section (\S\ref{Subsec:Nielsen-Schreier}), we will return in time to the early 1920s. The reader wishing to skip the remainder of this section should thus have no difficulty knowing what year they will land in.

As mentioned above, we will discuss only two knot groups: the trefoil knot group and the figure eight knot group. The former group would become important to the theory of \textit{cyclically pinched} one-relator groups; and the latter group would become one of the first concrete one-relator groups studied by Magnus. We will restrict our attention to the groups as abstract groups given by generators and relations, leaving the topological considerations aside (as did many of the authors whose work will be cited in this section). 

First, on the \textit{trefoil knot group} $\pres{}{a,b}{a^2 b^3 = 1}$, as mentioned earlier, Dehn \cite[Chapter~II, \S5]{Dehn1910} had solved the word problem in this group by constructing its Cayley graph. In 1914, he was able to use the construction of the automorphism group of the trefoil knot group to prove that the left-handed trefoil knot cannot be deformed to the right-handed trefoil knot \cite{Dehn1914}; this was the beginning of the study of automorphism groups of one-relator groups, which we will encounter throughout this survey. Pursuant to this, and not long after the war, O.\ Schreier \cite{Schreier1924} became interested in this problem following a seminar by Reidemeister in Vienna \cite[p. 167, Footnote 1]{Schreier1924}], and in 1924 considered the groups $G_{m,n} = \pres{}{a,b}{a^m b^n = 1}$ for $m, n > 1$. These groups cover all torus knot groups \cite{Seifert1950}: if $\mathcal{T}_{m,n}$ is the $(m,n)$-torus knot, then $\pi_1(\R^3 \setminus \mathcal{T}_{m,n})$ is isomorphic to $G_{m,n}$. In this 1924 article, among other things, Schreier proved that the pair $\{ m,n \}$ uniquely determines $G_{m,n}$ up to isomorphism. He also determines the centre $Z(G_{m,n})$ to be cyclic, generated by $a^m = b^n$, and when $m\neq n$, he shows that the automorphism group of $G_{m,n}$ is generated by the automorphisms
\[
\theta^{\pm 1}_{W} \colon a \mapsto Wa^{\pm 1}W^{-1}, \quad b \mapsto Wb^{\pm 1} W^{-1}
\]
where $W$ is an arbitrary word, and in particular that $\operatorname{Out}(G_{m,n}) \cong \Z / 2\Z$. If $m=n$, then $\operatorname{Out}(G_{m,n}) \cong (\Z / 2\Z)^2$. He then gives a presentation for the full automorphism group of $G_{m,n}$. For $m \neq n$, he shows that
\begin{equation}\label{Eq:aut-of-torus-knot-group}
\Aut(G_{m,n}) \cong \pres{}{x,y,z}{x^2 = y^m = z^n = (xy)^2 = (xz)^2 = 1} \cong D_{m} \ast_{C_2} D_n.
\end{equation}
In modern language, this is a virtually free group, being an amalgamated free product of two finite dihedral groups. For $m = n$, there is in addition to \eqref{Eq:aut-of-torus-knot-group} the generator $q$ and relations $q^2 = [q, x] = 1$ and $qyq^{-1} = z$. Although he would not have used precisely this language, Schreier's result of fully determining the automorphism group of torus knot groups -- and other similar one-relator groups -- still stands as a gem of early combinatorial group theory (we refer the interested reader to \cite{Karrass1984a, Karrass1984b} for generalizations).

Second, returning to the pre-war period and the \textit{figure-eight knot group}, as alluded to in \S\ref{Subsec:Dyck-and-presentations}, Dehn assigned one of his students, H.\ Gieseking \cite{Gieseking1912}, the task of studying this group and other related subjects. This is the one-relator group with the presentation
\begin{equation}\label{Eq:figure-eight-knot-group}
\pres{}{x,y}{x^{-1}yxy^{-1}xy = yx^{-1}yx}.
\end{equation}
Much like Schreier above, Gieseking would first provide a solution to the word and conjugacy problem in $\pres{}{a,b}{a^\alpha b^\beta = 1}$, see \cite[Chapter 2, \S2, p. 226]{Gieseking1912}, and then investigate the one-relator group with the generators $a, b$ and defining relation $a^2 b^2 a^{-1} b^{-1} = 1$ \cite[p.\ 158]{Gieseking1912}. This latter group has an index $2$ subgroup isomorphic to the figure eight knot, generated by $a^2, ba$, and Gieseking succeeds in solving the word problem for this group and the figure eight knot group. Gieseking's 1912 thesis \cite{Gieseking1912} is typical for a student of Dehn's: topological methods appear first, and there is everywhere an abundance of geometric arguments and intutition. His thesis is today more famous for producing the first example of a hyperbolic $3$-manifold with finite volume, the \textit{Gieseking manifold}. 

The Cayley graph of the figure eight knot group was likely constructed, as Magnus \cite{Magnus1978} suggests, by a ``Fritz Klein'', a student of Dehn's who also likely died in the First World War (see \cite[p. 201]{Dehn1987}). This provides yet another solution to the word problem. Some time later, Magnus \cite{Magnus1931} would prove in 1931 that a certain set of automorphisms given by Dehn generate the full outer automorphism group of the figure eight knot group \cite[Part III]{Dehn1914}. Specifically, Magnus proved, using his ``Magnus hierarchy'' for one-relator groups (see \S\ref{Subsec:Freiheitssatz}) that the outer automorphism group is generated by the two automorphisms $j_0, j_1$ given by:
\begin{align*}
j_0 &\colon u \mapsto u^{-1}, \quad v \mapsto v^{-1} \\
j_1 &\colon u \mapsto v^{-1}u^{-1} v u, \quad v \mapsto u^{-1}v^{-1}u^{-2}v^{-1} uv.
\end{align*}
Computing the outer automorphism group of one-relator groups remains a difficult problem in general; only in 1974, for example, did Grossman \cite{Grossman1974} prove that the outer automorphism group of $\pi_1(\Sigma_g)$ is residually finite; this result extends to all ``cyclically pinched'' one-relator groups by Allenby, Kim \& Tang \cite{Allenby2001}. The figure eight knot group, an archetypical one-relator group, continued to be studied by students of Magnus, e.g.\ by A.\ Whittemore \cite{Whittemore1973}, especially via linear methods. Today we know that \textit{all} knot groups have decidable word problem \cite{Waldhausen1968}, although this algorithm is incredibly difficult to apply in practice. In many cases, especially when the knot is alternating, more practical solutions are known, see e.g.\ \cite{Weinbaum1971, Dugopolski1982}. 

Having jumped in time from the early 1910s, to the 1930s and beyond in pursuit of one-relator knot groups, we now firmly return in chronological order to the 1920s. During this time, one of the emerging themes in combinatorial group theory was the theory of the properties of free groups, and particularly their subgroups.

\subsection{The Nielsen--Schreier Theorem, amalgamated free products}\label{Subsec:Nielsen-Schreier}

\

\epigraph{\textit{Their great usefulness remained long unrecognized, except by Wilhelm Magnus.}}{---B.\ H.\ Neumann, on amalgamated free products \cite[p. 517]{Neumann1974}}

\noindent Before studying one-relator groups, one must first study zero-relator groups: that is, free groups. Many theorems about one-relator groups -- particularly the \textit{Freiheitssatz} that we will encounter so frequently -- were directly inspired by intuition arising from free groups. One of the most fundamental questions about free groups concern their subgroups, and whether any relations can arise in subgroups of free groups. This question, of whether any subgroup of a free group is free, would appear naturally in the investigations of several researchers in the early 1920s. 

One of the first to treat this question was J.\ Nielsen (1890--1959), who in many ways can be considered ``half a student of Dehn'' \cite[p. 140]{Magnus1978}. Indeed, he completed his PhD thesis under the supervision of Dehn, but had started out as a PhD student of G.\ Landsberg, who passed away in 1912. We first present one of the notable results by Nielsen, proved using topological means, and which occurred also later in the theory of one-relator groups. This was a result on \textit{bases} of free groups, particularly in the rank $2$ case. First, some terminology: a \textit{basis} of a free group is any generating set of minimal cardinality. For example, $\{ ab^{-1}, ba^{-1}b \}$ is a basis of $F_2$, the free group on $\{ a, b \}$. An element $w \in F_n$ is said to be \textit{primitive} if it is part of some basis of $F_n$. This is equivalent to being in the same $\Aut(F_n)$-orbit as one of the generators $a_i$ of $F_n$. For example, $ab^{-1}$ is primitive in $F_2$ (as the above basis shows), but neither $a^2$ nor the commutator $[a, b] = aba^{-1}b^{-1}$ are primitive, since any square resp.\ any commutator is necessarily mapped to another such element by any automorphism. We remark that determining whether an element is primitive or not is quite non-trivial, but decidable, as we shall see in \S\ref{Subsec:More-freeness-lull}. Nielsen, as part of a broader investigation, gave a necessary and sufficient condition for a pair of elements in $F_2$ to be a basis, using commutators:

\begin{proposition}[Dehn/Nielsen\footnote{We remark the curious fact that Nielsen published this article while marking his affiliation as ``im Felde'' (``in the field'', i.e.\ while serving in the military) and location as being in Constantinople (today Istanbul). Nielsen was born in Schleswig in 1890, which today is part of Denmark but then was part of the German Empire. He served in the German navy during the First World War, as a military advisor to the Ottoman government in Constantinople. Proving results about bases of free groups was, presumably, not part of his official duties in this capacity.}, 1917 \cite{Ni17}]
Let $w_1, w_2 \in F_2 = F(a,b)$. Then $\{ w_1, w_2\}$ is a basis of $F_2$ if and only if we have $[w_1, w_2] = X [a, b]^{\pm 1} X^{-1}$. \label{Prop:Dehn-Nielsen-commutator}
\end{proposition}

For example, $\{ abab^2, ab \}$ is a basis of the free group on $a, b$, and we have
\begin{align*}
[ab, abab^2] = ab \cdot abab^2 \cdot (ab)^{-1} \cdot (abab^2)^{-1} = (ab)^2 [a, b] (ab)^{-2}.
\end{align*}
Nielsen gave a proof of Proposition~\ref{Prop:Dehn-Nielsen-commutator}, but notes that the result was suggested to him by Dehn, who had a ``different proof''. The proof intended by Dehn would only appear in print in 1930, in an article by Magnus \cite{Ma30} using the theory of one-relator groups and the \textit{Freiheitssatz}; see \S\ref{Subsec:Applications-of-Frei}. We also note that Mal'cev \cite{Malcev1962} gave an independent proof of this result, seemingly unaware of either the articles of Magnus or Nielsen (cf.\ also \cite{Bezverkhnyaya1986}). Related to the above result, Nielsen proved in 1918 that there is an algorithm whether a given set of words $X \subset F_n$ of words is a basis or not \cite{Nielsen1918}. These early results indicated that some very basic problems about free groups require a substantial effort to prove, and that there is a rich combinatorial theory underlying them. Indeed, related to Proposition~\ref{Prop:Dehn-Nielsen-commutator}, we remark that the decision problem of deciding whether a given element in a free group is a commutator is quite non-trivial, and was only solved in 1962 by Wicks \cite{Wicks1962}. 

In 1921, and now a Danish citizen, Nielsen returned to the subject of free groups. There, in his first article written in Danish, he proved a remarkable theorem: \textit{every finitely generated subgroup of a free group is free.} His proof can be stated in modern terms as proving that given any finite set $X \subset F_n$ of elements of a free group, there exists a finite sequence of \textit{Nielsen transformations}, being free group automorphisms of a particularly simple form, which transform $X$ into a free basis of the subgroup it generates; thereby proving the theorem. Nielsen transformations are thus the non-commutative analogue of the familiar row reduction process for matrices. As part of his proof, Nielsen also proved that all bases of a free group have the same cardinality, and thus that the \textit{rank} of a free group is well-defined. His treatment of free groups is not group-theoretical, but rather considers them purely symbolically.\footnote{Somewhat surprisingly, it is only in \S4 of his article, and notably \textit{after} proving his theorem for free groups, that Nielsen begins discussing groups at all.} We also remark that Nielsen discusses some aspects of the word problem in the final part of his 1921 article \cite[p.~94]{Nielsen1921}. In particular, he notes that the word problem is a generalization of the membership problem in free groups, which he himself solves for finitely generated subgroups, but perspicaciously notes that for the word problem ``its general solution has until now offered insurmountable difficulties''. Throughout this final part, he also acknowledges the work of Dehn at almost every stage, showing the strong influence of the latter on the early stages of the subject. Nielsen \cite{Nielsen1955} later revisited and simplified his proof in 1955, and a good and classical account of Nielsen's method is given in \cite[Chapter~3]{Magnus1966}. 
 
Around this time, O.\ Schreier (1901--1929), a young Austrian mathematician, was beginning to work on knot theory, influenced by Reidemeister and particularly Wirtinger. He was also influenced by Artin (also in Hamburg at the time) to work on free groups, who had himself investigated the commutator subgroup of free groups \cite[p.~161]{Sc27}. Schreier was aware of Nielsen's result above, and was able to extend it to \textit{all} subgroups of free groups in 1927.\footnote{As first observed in 1930 by Levi \cite[p.~315]{Levi1930}, Schreier's proof uses the well-ordering principle. Even today, all known proofs use the axiom of choice, but it remains an open problem whether or not the Nielsen--Schreier Theorem is equivalent to the axiom of choice (in ZF). Howard \cite{Howard1985} initiated this line of investigation, which has seen some recent progress \cite{Kleppmann2015, Tachtsis2018}.} Thus, we have the following theorem today carrying both names: 

\begin{theorem}[The Nielsen--Schreier Theorem\footnote{In many ways, this theorem is one of the more remarkable properties of free groups. For example, subsemigroups of free semigroups need not be free. Indeed, there exist non-finitely presentable finitely generated subsemigroups of free semigroups.}] 
Every subgroup of a free group is itself free. \label{Thm:Nielsen-Schreier}
\end{theorem}

In fact, Schreier was able to deduce more information in the case that the subgroup has finite index \cite[p.~161]{Sc27}. Let $F_n$ be a free group of rank $n$, and let $H \leq F_n$ be a subgroup of finite index $j$ in $F_n$. Then $H$ is free of rank $k$, where
\begin{equation}\label{Eq:Schreier-index-formula}
k = 1 + (n-1) \cdot j.
\end{equation}
This result is today called the \textit{Schreier index formula}. Furthermore, Schreier proved that every finitely generated \textit{normal} subgroup of a free group has finite index \cite[p.~162]{Sc27}. 

As mentioned above, Nielsen's proof for the finitely generated case of Theorem~\ref{Thm:Nielsen-Schreier} is entirely combinatorial. Schreier's proof, on the other hand, makes use of Dehn's \textit{Gruppenbild}, i.e.\ the Cayley graph. Schreier's proof is still inherently combinatorial, and in particular contains no pictures. Nevertheless, his article attracted a great deal of attention. For example, already in 1929 Rademacher \cite{Rademacher1929} would make use of the method used by Schreier, today called the \textit{Reidemeister--Schreier} method, to determine presentations of the congruence subgroups of the modular group $\PSL_2(\Z)$. As we shall see, Magnus would also make great use of it in his investigations on one-relator groups, see \S\ref{Subsec:Freiheitssatz}.

Let us linger on this last point, and Schreier's proof of the Nielsen--Schreier Theorem. In particular, it contains a key new construction: \textit{the amalgamated free product} of two groups. Let $G_1, G_2$ be two groups such that $H$ embeds into both of $G_1$ and $G_2$. Concretely, this means that there exists some $H \leq G_1$ and an injective homomorphism $\phi \colon H \to G_2$. Then the \textit{free product} of $G_1$ and $G_2$ \textit{amalgamated in the subgroup} $H$ is the group with the presentation 
\begin{equation}\label{Eq:FPWA-def}
G_1 \ast_H G_2 := \pres{}{G_1 \cup G_2}{h = \phi(h) \: \forall h \in H}.
\end{equation}
In Schreier's article, he begins by proving that the amalgamated free product has some basic properties, such as the fact that $G_1, G_2 \leq G_1 \ast_H G_2$, that the groups $G_1, G_2$ generate the amalgam, and that $G_1 \ast_H G_2$ is the ``most general'' group with these properties; in fact, he takes such properties as definitional, and in his language proves that \textit{amalgamated free products always exist}. The idea likely dates back to Schreier's own work on the torus knot groups 
\begin{equation}\label{Eq:Torus-knot-presentation-amalgam}
\pi_1(\R^3 \setminus \mathcal{T}_{m,n}) = \pres{}{a,b}{a^m b^n = 1}
\end{equation}
discussed already in \S\ref{Subsec:Surface-and-knot-groups}; these groups can clearly be written as an amalgamated free product of two infinite cyclic groups. Schreier \cite[p.~163]{Sc27} attributes the notion of free product to Artin, and that his aim was to generalize existence results for free products also to this more generalized setting. Schreier's article is not concerned with the word problem, but he mentions in a remark without proof that the word problem in amalgams of free groups is decidable \cite[p.~163]{Sc27}; this result can be generalized to give the following standard result:

\begin{theorem}[Schreier, 1927]
Let $G_1, G_2$ be two groups, and let $H$ be a common subgroup. If the word problems for $G_1$ and $G_2$ are both decidable, and if the membership problem for $H$ in $G_1$ and in $G_2$ are both decidable; then the word problem for $G_1 \ast_H G_2$ is decidable.
\label{Thm:Membership-gives-WP-in-amalgams}
\end{theorem}

Here, the \textit{membership problem} for $H \leq G$ simply means the problem of deciding for arbitrary words $w \in G$ whether or not $w \in H$ (see \S\ref{Subsec:More-on-WP}). To the best of the author's knowledge, a proof of Theorem~\ref{Thm:Membership-gives-WP-in-amalgams}, which is essentially immediate from standard normal form result for amalgamated free products, first appeared in a 1931 article by Magnus \cite[\S4, p.~62]{Magnus1931}, his second article on one-relator groups, which we will revisit in greater detail in \S\ref{Subsec:Wordproblem-Magnus}. There, he used the above result to solve the word problem in some one-relator groups. After Magnus' spectacular application of amalgamated free products, they would come to occupy a central role in combinatorial group theory, as indicated by the quotation at the start of this section. A standard early reference is B.\ H.\ Neumann's essay \cite{Neumann1954}, and the modern reference is Serre's 1977 book \cite{Serre1977}, which places the amalgamated free product in the larger context of the fundamental group of a \textit{graph of groups}, of which one particular case is the amalgamated free product, and another is the \textit{HNN-extension}, which we shall encounter in the Magnus--Moldavanskii hierarchy of one-relator groups in \S\ref{Subsec:MM-hierarchy}. 

Several simpler proofs of the Nielsen--Schreier Theorem would appear in the literature shortly after Schreier's publication appeared. First, in 1930 Levi \cite{Levi1930} would give a short and compact proof of the theorem. Hurewicz \cite{Hurewicz1930} gave another proof in 1930, and Reidemeister \cite[\S3.9 and \S4.17]{Reidemeister1932} in his book gave another, which is very close to the modern proof (see below). Locher \cite{Locher1934} in 1934 additionally gave a conceptually very simple and graphical proof. Finally, very notably, in 1936 Baer \& Levi \cite{Baer1936} gave a topological proof of the result; indeed, they gave a proof of the more general result, today called Kurosh's Theorem \cite{Kurosh1934} after the man who proved it in 1934, that any subgroup of a free product of groups is again a free product (of a free group with conjugates of subgroups of the factors), which implies the Nielsen--Schreier Theorem. Today, the theorem is usually proved via actions on trees, as in Serre \cite{Serre1977}; Reidemeister's treatment of the theorem comes very close to precisely this modern proof already in 1932, as noted already by Stillwell \cite[Footnote 13, p.~104]{Stillwell2014}. 

Schreier would not live to see any of these developments. Like his fellow group theorists Abel and Galois before him, Schreier instead had his life cut short at the very beginning of his career\footnote{By contrast, one of Schreier's lecturers in Vienna was L.\ Vietoris (1891--2002), well-known e.g.\ for the Mayer--Vietoris sequence, who remains the oldest verified Austrian man to have ever lived.}: in 1929, just two years after publishing his proof of the Nielsen--Schreier Theorem, he died of sepsis, merely 28 years old.

\section{Magnus, the \textit{Freiheitssatz}, and the word problem (1930-1932)}\label{Sec:3-Magnus}

\noindent One name towers above all others in the history of one-relator groups: Wilhelm Magnus (1907--1990). Magnus was a PhD student of Dehn's, and had the unusual talent of being able to make significant contributions to a vast number of disparate areas of mathematics; the reader interested in these contributions is referred to the collected works \cite{Magnus1984}. Indeed, it is not possible to summarize these collected works here, nor even the full extent of Magnus' influence on one-relator groups; we hope that the many occurrences of Magnus' name throughout the remainder of the survey will serve as an indicator for this influence. In this section, we will focus on his early contributions to the subject, with three articles from 1930, 1931, and 1932, respectively, as landmarks to guide us \cite{Ma30, Magnus1931, Ma32}. 

\subsection{The \textit{Freiheitssatz}}\label{Subsec:Freiheitssatz}

{
\setlength{\epigraphwidth}{0.45\textwidth}
\epigraph{
\noindent \textit{Da sind Sie also blind gegangen!} \newline [\textit{So you did it blind!}]}
{---M.\ Dehn in 1930, upon hearing that Magnus' proof of the \textit{Freiheitssatz} was not geometric \cite[p. 139]{Magnus1978}}}

\noindent The starting point for the entirety of Magnus' work on one-relator groups comes from the \textit{Freiheitssatz}, already alluded to in \S\ref{Subsec:Dyck-and-presentations}. According to Magnus, Dehn proved this 1922 conjecture regarding the freeness of certain subgroups of one-relator groups before 1928, but ``felt his proof was not in a form suitable for publication'' \cite[p. 114]{Chandler1982}. He therefore assigned the task of its proof to his PhD student Magnus in July 1928, who proved it a couple of years later. The statement of the theorem is as follows: 

\begin{theorem}[Magnus' \textit{Freiheitssatz}, 1930 {\cite[\S1]{Ma30}}]
Let $G = \pres{}{A}{r=1}$ be a one-relator group such that the word $r$ is cyclically reduced, and such that $r$ involves every letter of $A$. Let $A_0 \subsetneq A$. Then the subgroup of $G$ generated by $A_0$ is free, and furthermore is free with $A_0$ as a basis.
\end{theorem}

In the first condition, $r$ being \textit{cyclically reduced} means that it is equal in $F_A$ to $ar' a^{-1}$ for some word $r' \in F_A$, i.e.\ that $r$ is not conjugate to a shorter word in $F_A$. This is harmless, for any relator of the form $ara^{-1}=1$ is, of course, equivalent to $r=1$. The second condition is also of little consequence, since one can simply pull out the unused letters of $A$ into a free factor, and apply the \textit{Freiheitssatz} to the one-relator group factor which now satisfies the second condition. The subgroups of a one-relator group generated by subsets of the generating set are called \textit{Magnus subgroups}. Thus, a restatement of the \textit{Freiheitssatz} is simply: if all letters of $A$ appear in $r$, then all Magnus subgroups of $\pres{}{A}{r=1}$ are free. Many more properties hold for Magnus subgroups, which we will expand on in \S\ref{Subsec:Magnus-subgroups}. One immediate consequence of the \textit{Freiheitssatz} is the following result on the ``spelling'' of a word, akin to Dehn's Spelling Theorem for surface groups $\pi_1(\Sigma_g)$ (i.e.\ Theorem~\ref{Thm:Dehn-spelling-theorem}). 

\begin{corollary}[Magnus, 1930]
Let $G = \pres{}{A}{r=1}$ be a one-relator group with $r$ cyclically reduced, and such that $r$ involves every letter of $A$. Let $w \in F_A$ be a non-trivial word such that $w =_G 1$. Then $w$ contains an occurrence of either $a$ or $a^{-1}$ for every letter $a \in A$.
\end{corollary}

Thus, the \textit{Freiheitssatz} relates the ``spelling'' of a word with some of its algebraic properties: if one is missing some letter and its inverse and thus cannot spell the defining relator, then one cannot spell \textit{any} relator. The \textit{Freiheitssatz} is hence the second of several ``spelling theorems'' (a term due to B.\ B.\ Newman) that we will encounter throughout the history of one-relator groups (cf.\ \S\ref{Subsec:BBNewman}). We remark, however, that the \textit{Freiheitssatz} provides significantly less control over the word problem than Dehn's Spelling Theorem, and does not immediately yield a solution to the word problem (although see \S\ref{Subsec:Wordproblem-Magnus}).

The proof of the \textit{Freiheitssatz} can be sketched as follows. The overall idea is to start with a one-relator group $G_1 = \pres{}{A_1}{r_1=1}$, perform a rewriting process (via the Reidemeister--Schreier process) on $r_1$, and obtain a new one-relator group $G_2 = \pres{}{A_2}{r_2=1}$. This new one-relator group will have two key properties: (1) the word $r_2$ is shorter than $r_1$, i.e.\ $|r_2| < |r_1|$; and (2) if the \textit{Freiheitssatz} holds for $G_2$, then it holds for $G_1$. By induction on the relator length, we thus arrive eventually at a free group $G_n$ for some $n$, for which the statement of the \textit{Freiheitssatz} holds by the Nielsen--Schreier theorem, and we will be done by induction. 

The way the word $r_2$ is obtained from $r_1$ is by first assuming that some letter $a$ has exponent sum zero in $r_1$. Thus we obtain a surjection from $G_1$ onto $\Z$, and the kernel of this map -- consisting of all words in $G_1$ with exponent sum zero in $a$ -- turns out to be an amalgamated free product of one-relator groups $G_{2,i}$, $i \in Z$, with Magnus subgroups being the amalgamated subgroups. The relators of the one-relator groups are all obtained by rewriting $r_1$, which lies in the kernel of the map onto $\Z$ by assumption. Thus, since the kernel is generated by conjugates of the form $b_i := a^i b a^{-i}$ for $b \in A \setminus \{ a \}$ and $i \in \Z$, we can rewrite $a^i r_1a^{-i} $ into a (necessarily shorter) word $r_{2,i}$ over the $b_i$, and the defining relation of  $G_{2,i}$ will, by the Reidemeister--Schreier process, be precisely $r_{2,i} = 1$. Analyzing the structure of the amalgam of these groups now readily yields that any relation holding in a Magnus subgroup of $G_1$ would give a relation holding in a Magnus subgroup of $G_2$; and thus (2) follows. Finally, if no letter has exponent sum zero in $r_1$, then a simple trick allows one to embed $G_1$ into a very similar one-relator group for which the exponent sum of some generator is zero, and performing the same procedure as above can be shown to yield (1) in this case too, ensuring the induction carries through. This completes the sketch of the proof.

The breakdown procedure, of passing from $G_1$ to $G_2$, from $G_2$ to $G_3$, etc.\ gives rise to the \textit{Magnus hierarchy} for the given one-relator group. This hierarchy can also be realized by means of HNN-extensions, as noticed by Moldavanskii, which we expand on in \S\ref{Subsec:MM-hierarchy}. As an example of the procedure, suppose we have the group 
\begin{equation}\label{Eq:MM-example}
G_1 = \pres{}{a,b,c}{a^2 b c a b^{-1} c a b a^2 b^{-1} = 1}.
\end{equation}
Then the defining relator of $G_1$ has exponent sum zero in $b$, so we can rewrite this word as a product of $b$-conjugates of $a$ and $c$. Indeed, if we set $a_i = b^i a b^{-i}$ and $c_i = b^i a b^{-i}$, then to obtain the new relator one simply has to remove all $b$'s, and replace each instance of an $a$ or a $c$ by the same letter subscripted by $j$, where $j$ is the number of $b$-letters occurring in the relator before that instance. This results in the new one-relator group
\begin{equation}\label{Eq:MM-example-cont}
G_{2} = \pres{}{a_0, a_1, c_0, c_1}{a_0^2 c_1 a_1 c_0 a_0 a_1^2 = 1}
\end{equation}
Magnus' method amounts to showing that $G_1$ is an amalgamated free product of infinitely many copies of $G_{2,i}$, where $G_{2,i}$ is defined by the same presentation as \eqref{Eq:MM-example-cont} but with the indices on all generators shifted by $i \in \Z$, and amalgamated over Magnus subgroups of the $G_{2,i}$. Note that $G_2$ has a shorter defining relator than $G_1$, and in fact is a free group, since $c_1$ only occurs once in the defining relator. Hence the breakdown procedure ends with $G_2$. 

There are many things that can be said about Magnus' 1930 proof of the \textit{Freiheitssatz} in \cite{Ma30}. First, the sketch above uses the language of amalgamated free products, which is very slightly anachronistic: Magnus' 1930 proof does not use this language, and reproves many of the fundamental results of such products by hand. Indeed, reading this article, one sees just how primitive of an area combinatorial group theory was at the time: no reader familiar with group theory will find themselves outpaced by Magnus' writing, nor indeed will they find any prerequisite knowledge missing (the only potential difficulties for the modern reader arise from the dated notation). But it is only \textit{slightly} anachronistic: indeed, the article ends with a footnote that many of the lemmas, particularly \cite[Lemmas~1--4]{Ma30} are essentially reproving results by Schreier on amalgamated free products, and that he will ``expand on this in future work''. This expansion would be the subject of his 1931 and 1932 articles, where the \textit{Freiheitssatz} is used to great effect.

Before elaborating further on these articles, we mention a famous anecdote about Magnus' proof of the \textit{Freiheitssatz}. As mentioned, Dehn had assigned the task of its proof to Magnus, and Dehn was a particularly geometrically-minded mathematician; thus, when Magnus told his supervisor that he had completed an entirely combinatorial and algebraic proof of the \textit{Freiheitssatz}, Dehn exclaimed the quotation presented at the start of this section. One should not take this quotation to mean that Magnus was immune to geometric intuition and arguments. The three aforementioned articles by Magnus \cite{Ma30,Magnus1931,Ma32} are indeed mostly combinatorial, but there is a relevant footnote in the second article, from 1931. There, a lemma (\cite[Lemma~1]{Magnus1931}) is proved, which states that if in $G = \pres{}{a,b}{R(a,b) = 1}$ the word problem is decidable, and if moreover deciding membership problem in $\langle a \rangle$ is decidable, then the word problem is decidable in $\overline{G} = \pres{}{t,b}{R(t^n, b) = 1}$ for all $n \geq 1$. This lemma, now vacuously true by the general decidability of the word problem in all one-relator groups proved the next year (see \S\ref{Subsec:Wordproblem-Magnus}), is proved purely combinatorially, but is accompanied by a footnote: ``This can be proved in a more immediate and natural way by constructing Dehn's \textit{Gruppenbild} [= Cayley graph] associated to $\overline{G}$'' \cite[Footnote~20, p. 63]{Magnus1931}. One cannot help but wonder why the proof was then not carried out using the \textit{Gruppenbild}! It seems that Magnus placed high value in the algebraic techniques in spite of their sometimes less immediate proofs. 

We now turn to early applications of the \textit{Freiheitssatz}, including the 1931 and 1932 articles by Magnus himself. Beyond these, we remark that the theorem was certainly also noticed by his contemporaries at an early stage. For example, a brief sketch of the proof of the \textit{Freiheitssatz} also appeared in 1932, in Reidemeister's book \cite[\S3.13]{Reidemeister1932}, and using the full power of Schreier's amalgamated free products.\footnote{The 1932 Zentralblatt review of \cite{Reidemeister1932} (\texttt{Zbl:0004.36904}) by Nielsen also mentions the ``\textit{Freiheitssatz} of Dehn and Magnus'' as a natural application of Schreier's amalgamated free products.} This book (for an English translation by Stillwell, see \cite{Stillwell2014}) was the first book written on the subject of combinatorial group theory, and Magnus' early influence can be seen already here: he is one of only four mathematicians to be acknowledged in the introduction of the book \cite{Reidemeister1932}.

\subsection{Early applications of the \textit{Freiheitssatz}.}\label{Subsec:Applications-of-Frei}

\

{\setlength{\epigraphwidth}{0.45\textwidth}
\epigraph{\textit{Even though the \textit{Freiheitssatz} seems almost trivial, it is nevertheless an extraordinarily powerful tool.}}{---W.\ Magnus, 1930  \cite[p. 157]{Ma30}.}}

\noindent Thus begins the second part of Magnus \cite{Ma30}, the article whose purpose was to prove this theorem. The first application of this tool, already appearing in his 1930 article, concerns the (normal) \textit{root problem}, and a particular case of it leading to his \textit{conjugacy theorem}. Let us begin by stating that theorem.

\begin{theorem}[{Magnus' Conjugacy Theorem, \cite[\S6]{Ma30}}] Let $r_1, r_2 \in F_A$ be two words such that the normal closures of $r_1$ and $r_2$ in $F_A$ coincide. Then $r_1$ is conjugate to $r_2^{\pm 1}$. \label{Thm:Conjugacy-Theorem}
\end{theorem}

In other words, if in $G_1 = \pres{}{A}{r_1 = 1}$ we have that $r_2 = 1$, and vice versa, then $r_1$ is conjugate to $r_2^{\pm 1}$. This is a solution to a particular case of the (normal) \textit{root problem}, which we first expand on. If $G = \pres{}{A}{r=1}$ is a one-relator group and $w \in F_A$ is a word such that $w=1$, then we say that $r$ is a \textit{normal root}\footnote{The original German is \textit{Wurzel}; we have added \textit{normal} to distinguish from other types of roots.} of $w$. For example, $a^2$ is a normal root of $a^6$, and the commutator $[a,b]$ is a normal root of every word in which the exponent sum of $a$ and $b$ are both zero. Fixing an underlying free group, the \textit{normal root problem} asks to determine all normal roots of a given word $w$. This problem is mentioned by Magnus, together with the word problem, as fundamental to one-relator groups. In general, the normal root problem remains open even today, although Magnus' solution to the word problem (\S\ref{Subsec:Wordproblem-Magnus}) shows that the roots of a given word is a recursive set. Nevertheless, one of the first applications of the \textit{Freiheitssatz} was to determine all normal roots of the commutator $[a,b]$.

\begin{proposition}[{Magnus, \cite[\S6.3]{Ma30}}]
The only (cyclically reduced) normal roots of $[a,b]$ are the primitive words together with $[a, b]^{\pm 1}$. \label{Prop:Magnus-[a,b]-only-presentaton}
\end{proposition}

As a consequence, up to renaming generators the only one-relator presentation with a cyclically reduced defining relator of the group
\begin{equation}\label{Eq:ZxZ-presentation}
\Z \times \Z = \pres{}{a,b}{[a,b]=1}
\end{equation}
is precisely that one. This is a non-trivial result, and proving it requires some work. One natural implication of this is the 1917 result by Nielsen discussed in \S\ref{Subsec:Nielsen-Schreier} (see Proposition~\ref{Prop:Dehn-Nielsen-commutator}), and a footnote by Magnus \cite[Footnote~20]{Ma30} indicates that the proof he gives is the one Dehn originally intended, unlike Nielsen's proof. 

In the same 1930 article, Magnus continued to solved the normal root problem for other words. He proceeds, by fairly elementary means, to determine the normal roots of the words $a^2 b^p, a^2 b^{2p}$, and $a^p b^{p^k}$ for arbitrary $k$ and primes $p$. There is also a small section\footnote{The section is erroneously omitted from the table of contents of \cite{Ma30}.} at the end of the first part of Magnus' 1930 article \cite[\S7.5, p. 163]{Ma30}, called ``Simple, unsolved problems''. Here, Magnus mentions the difficulties involved in finding the normal roots of all one-relator groups by giving the example of the group with the single relator $ab^6 a^{-1}b^6$. The reader familiar with one-relator group theory will immediately recognise this group: it is the \textit{Baumslag--Solitar group} $\BS(6,6)$, whose siblings $\BS(m,n)$ we will encounter again on many occasions, and particularly in \S\ref{Subsec:Baumslag-Solitar}. 

The normal root problem has not attracted particularly much attention in the literature since (indeed it does not appear in any other articles by Magnus), and we can summarize it here. A.\ Steinberg (a PhD student of Magnus') revisited the normal root problem in 1971 and 1986 \cite{Steinberg1971, Ste86}. In the former article, he studied among other things normal roots of primitive elements, and in the latter he found the normal roots of $a^k b^l$ for all primes $k, l$. Steinberg directly uses and extends Magnus' method from \cite{Ma30}, and a side-by-side reading of the two articles shows a great deal of similarity in the methods indeed. Finally, in 2000 McCool \cite{McCool2000} solved Magnus' old problem of finding the normal roots of $ab^6 a^{-1}b^6$, and, more generally, the problem of finding the normal roots of $[a, b^k]$ for $k \in \Z$. He also found all the normal roots of $a^k b^l$ indirectly, by showing that this word has only finitely many normal roots for all $k, l \in \Z$, and that these can be produced algorithmically. A problem which remains open is whether there is some element in $F_n$ but not in $[F_n, F_n]$ (the commutator subgroup of $F_n$) which has infinitely many normal roots. By contrast, every element in $[F_2, F_2]$ clearly has infinitely many normal roots. 

Let us now return to the Conjugacy Theorem (i.e.\ Theorem~\ref{Thm:Conjugacy-Theorem}). An equivalent formulation of it is: if two words $r_1, r_2$ are normal roots of one another, then $r_1$ is conjugate to $r_2^{\pm 1}$. This is a very compelling result, and relates to the isomorphism problem for one-relator groups in a natural way. Indeed, it is very natural to ask whether two one-relator groups
\[
G_1 = \pres{}{A}{r_1=1}, \quad \textnormal{and} \quad G_2 = \pres{}{A}{r_2=1}
\]
can be isomorphic \textit{without} $r_1$ being a root of $r_2$ and vice versa (i.e.\ without the normal closures $\llangle r_1 \rrangle_{F_A}$ and $\llangle r_2 \rrangle_{F_A}$ coinciding). Magnus suspected that this would not be the case, i.e.\ that this is the only way that isomorphisms of one-relator groups arise. Note that an equivalent formulation of the Conjugacy Theorem is that if $G = \pres{}{A}{r=1}$, then any automorphism of $F_A$ which defines an automorphism of $G$ must also transform $r$ into a word conjugate to $r^{\pm 1}$. Thus Magnus' suspicion above is equivalent to the following conjecture:

\begin{conjecture}[{W.\ Magnus\footnote{Possibly the first place this conjecture appears explicitly, and attributed to Magnus, in the literature is in 1959, in an article by Rapaport \cite[p. 231]{Rapaport1959}, see also \cite[p. 401]{Magnus1966}.}}] Let $G_1 = \pres{}{A}{r_1=1}$ and $G_2 = \pres{}{A}{r_2=1}$. Then $G_1 \cong G_2$ if and only if $r_1$ lies in the $\Aut(F_A)$-orbit of $r_2$. \label{Conj:Magnus-conj}
\end{conjecture}

Were this conjecture correct, it would follow that the isomorphism problem for one-relator groups would be decidable, as Whitehead \cite{Wh36} proved in 1936 that automorphic equivalence in free groups is decidable (cf.\ also \cite{Rapaport1958,Higgins1974}). However, many decades later, the conjecture would be disproved; we expand on this in \S\ref{Subsec:iso-problem}. Note also that here we easily stumble upon the question of the structure of the automorphism group $\Aut(G)$ of a given one-relator group $G$. We have seen in \S\ref{Subsec:Surface-and-knot-groups} the importance and relative difficulty of computing such groups. Clearly, given e.g.\ $\pres{}{a,b}{a^n = 1}$ with $n>1$, there are automorphisms not induced by free group automorphisms; but even in the torsion-free case (i.e.\ when the relator is not a proper power, see \S\ref{Subsec:Torsion-theorem}) there exist one-relator groups with automorphisms not arising from automorphisms of the underlying free group \cite{Rapaport1959}. Indeed, the knot group of the figure-eight knot has such automorphisms (see \S\ref{Subsec:Surface-and-knot-groups} for Magnus' work on this group. 

We mention some brief notes on extensions of the Conjugacy Theorem (i.e.\ Theorem~\ref{Thm:Conjugacy-Theorem}). By an extension of this theorem we simply mean that the free group $F_A$ in the statement of Theorem~\ref{Thm:Conjugacy-Theorem} is replaced by some other group (e.g.\ a surface group). It was extended to all small cancellation $C'(1/6)$-groups by Greendlinger \cite{Greendlinger1961} (a PhD student of Magnus'), and was subsequently extended to some related classes by Kashintsev \cite{Kashintsev1969, Kashintsev1985} and Pa\l asi\'{n}ski \cite{Pawasinski1982}. We mention also some more modern developments, as their proofs use entirely classical methods; first, it is known that for all $g \geq 1$, the surface groups $\pi_1(\Sigma_g)$ and $\pi_1(\mathcal{K}_g)$ (given by \eqref{Eq:pres-of-sigma-g} and \eqref{Eq:pres-of-K-g}) have the Magnus property, proved in the orientable case by Bogopolski \cite{Bogopolski2005} (see also \cite{Bogopolski2004}), in the non-orientable case for $g \geq 4$ by Bogopolski \& Sviridov \cite{Bogopolski2008}, and the reticent case $\pi_1(\mathcal{K}_3)$ only completed recently by Feldkamp \cite{Feldkamp2019}. Finally, Edjvet \cite{Edjvet1989} extended the result to free products of locally indicable groups. 

Finally, we remark that Magnus conjectured \cite[p.\ 24]{Ma30} in 1930 that a similar theorem ought to hold for two-relator groups. Specifically, he conjectured that if two pairs of relators $\{ r_1, r_2 \}$ and $\{ r_1', r_2' \}$ give rise to the same normal closures in $F_A$, then the two sets are connected by a set of ``Nielsen moves'', i.e.\ the closure of maps of the form $r_i \mapsto r_i r_j^{\pm 1}$ (with $i \neq j$), inversions, and conjugates. These moves are thus the natural analogues of Nielsen transformations. The conjecture remains, to the best of the authors' knowledge, open. One particular case of is of course part of the well-known \textit{Andrews--Curtis Conjecture}, which states that any presentation of the trivial group which is balanced (i.e.\ when the number of generators is the same as the number of relators) is related to a trivial presentation by precisely a sequence of such moves above. This conjecture, which is also open, was posed only in 1965 \cite{Andrews1965}, more than three decades after Magnus' article. Thus Magnus' 1930 article \cite{Ma30} contains hints not only of Baumslag--Solitar groups but also of one of the central conjectures of combinatorial group theory; It is remarkable how much can be traced back to this single article.

As a final note on the Conjugacy Theorem, it is worth noting that although many classical theorems on one-relator groups today have geometric proofs, this is not (yet?) the case for the Conjugacy Theorem, as pointed out to me by H.\ Wilton. The method of proof today remains, while streamlined, essentially the same as Magnus' original proof.

\subsection{The word problem for one-relator groups}\label{Subsec:Wordproblem-Magnus}

Although Magnus does discuss the word problem in his 1930 article \cite{Ma30} at length, he does not give a solution to it in that article. The problem was clearly on his mind at this point, however, and in his 1931 article \cite[\S4]{Magnus1931} he gives a solution to the word problem in all one-relator groups of the form 
\begin{equation}\label{Eq:abab-groups}
\pres{}{a,b}{a^{\alpha_1} b^{\beta_1} a^{\alpha_2} b^{\beta_2} = 1}.
\end{equation}
This class of one-relator groups is already rather non-trivial, and contains e.g.\ the non-linear but residually finite Dru\c{t}u--Sapir group $\pres{}{a, b}{a^2 b^2 a^{-2} b^{-1} = 1}$ studied in \cite{Drutu2005}. The method Magnus used to solve the word problem in \eqref{Eq:abab-groups} is a primitive version of a method he would revisit the very next year, when he would prove the following remarkable theorem:

\begin{theorem}[Magnus, 1932 \cite{Ma32}]
Every one-relator group has decidable word problem. \label{Thm:Decidable-word-problem}
\end{theorem}

His method of proof uses the \textit{Freiheitssatz} and the Magnus hierarchy (as in the proof of the \textit{Freiheitssatz}, see \S\ref{Subsec:Freiheitssatz}), but keeps track of some more information. We now give a brief overview of how his proof works. 

The central problem of study is the \textit{generalized word problem}\footnote{This term has also been used for the \textit{subgroup membership problem}, a much more difficult problem which remains open for one-relator groups, see \S\ref{Subsec:More-on-WP}.} or the \textit{extended identity problem}\footnote{Another name for the problem is the \textit{Magnus problem}, revisited by Boone \cite{Boone1955} along the way of proving the word problem undecidable; see also Stillwell \cite[p. 45]{Stillwell1982}.}. This problem is the following: given a one-relator group $G = \pres{}{A}{r=1}$ and a subset $A_0 \subseteq A$, decide whether or not a given word $w \in F_A$ is an element of the subgroup $\langle A_0 \rangle \leq G$. That is, the generalized word problem asks for deciding membership in Magnus subgroups of one-relator groups. Of course, if a one-relator group has decidable generalized word problem, then it also has decidable word problem; for deciding membership in the Magnus subgroup generated by the empty set is precisely the word problem.

The basic observation for adopting the Magnus hierarchy, used in the proof of the \textit{Freiheitssatz}, to also solve the generalized word problem is Theorem~\ref{Thm:Membership-gives-WP-in-amalgams}, which we here restate for simplicity: let $G = G_1 \ast_H G_2$ be an amalgamated free product of two groups with $H \leq G_1, G_2$. Then the word problem in $G$ is decidable if: 
\begin{enumerate}
\item The word problem in $G_1$ and $G_2$, respectively, is decidable; and
\item The membership problem for $H$ in $G_1$ and in $G_2$ is decidable. 
\end{enumerate}
Magnus first proves this result in \cite[\S4]{Ma30}, giving the first published proof of the result (although the result was known to Schreier, as we have seen in \S\ref{Subsec:Nielsen-Schreier}), and then proceeds as follows. Since the Magnus hierarchy decomposes a one-relator group into an iterated amalgamated free product of one-relator groups (with shorter defining relation) with amalgamated subgroups being Magnus subgroups, this gives a basis for an inductive step: one only has to show that the property of having decidable \textit{generalized word problem} is also preserved in this setting when performing the Magnus breakdown. This is not difficult, since the generators of the initial one-relator group essentially arise from generators of the one-relator group at the next layer down in the Magnus hierarchy. The base case corresponds, essentially, to that of a free group, and membership in Magnus subgroups of free groups is trivially decidable.\footnote{In fact, membership in \textit{any} finitely generated subgroup of a free group is decidable by using Nielsen's method, see \S\ref{Subsec:Nielsen-Schreier}.} This solves the generalized word problem in all one-relator groups, and hence solves the word problem in all one-relator groups. 

We make a few remarks on this solution to the word problem. First, we note that, as we shall see in \S\ref{Subsec:MM-hierarchy}, the Magnus hierarchy and the associated solution to the word problem has a natural interpretation in terms of \textit{HNN-extensions}. Furthermore, the seeming simplicity of the solution to the word problem should \textit{not} be taken as a suggestion that the word problem in one-relator groups is necessarily easy. Indeed, the time complexity of the word problem in general for one-relator groups is unknown. Furthermore, the breakdown procedure depends, in an essential way, on which generator with exponent sum zero one chooses (i.e.\ the generator mapped to the generator of $\Z$). For example, if one considers the group
\begin{equation}\label{Eq:Gray-family}
G = \pres{}{a,t}{[a, a^t] = 1}
\end{equation}
then this has a defining relator with exponent sum zero in both $a$ and $t$. If we break down the group using $a$ resp.\ $t$, we obtain the group $G_a$ resp.\ $G_t$, defined by 
\[
G_a = \pres{}{t_0, t_1, t_2}{t_1 t_2^{-1} t_1 t_0^{-1} = 1} \quad \textnormal{resp.} \quad G_t = \pres{}{a_0, a_1}{a_0 a_1 a_0^{-1} a_1^{-1} = 1}.
\]
Note that $G_a$ is a free group, whereas $G_t \cong \Z^2$. In the latter case, we must therefore proceed with the Magnus breakdown once again before terminating in a free group. One can easily generalize the above example \eqref{Eq:Gray-family} with e.g.\ the relator $[a, a^{(a^t)}] = 1$, etc., in which choosing to break down using $a$ immediately yields a free group, but breaking down using $t$ can take an arbitrary finite number of steps to reach a free group. We also here mention in this connection a result of Weinbaum \cite{We72} who proved in 1972 that in a one-relator group $\pres{}{A}{r=1}$, with $r$ cyclically reduced, no proper non-empty subword of $r$ is equal to $1$. Thus even this result, which seems straightforward, is not immediate to prove using Magnus' solution to the word problem. We will discuss the time complexity, and related questions, of the word problem in one-relator groups in \S\ref{Subsec:More-on-WP}. Before this, however, we must first cross the barren landscape of inactivity that followed Magnus' results. 

\section{The Lull Phase (1932--1960)}\label{Sec:4-Lull-phase}

Following Magnus' 1932 solution to the word problem in all one-relator groups, progress on one-relator groups slowed down. Indeed, as for significant theorems related directly to one-relator groups themselves, in the next three decades or so only one would be proved: the Identity Theorem of Lyndon. We will discuss this result in \S\ref{Subsec:LyndonIT}. However, these three decades, which we will refer to as the \textit{lull phase}, saw a number of very significant developments in combinatorial group theory as a subject, many of which would eventually end up being directly relevant to one-relator group theory. In this section, we will give an overview of this phase of the development of one-relator group theory, ending it in \S\ref{Subsec:Undecidability} by discussing some reasons for what caused progress to slow down, and subsequently resume.

\subsection{Some results on free groups}\label{Subsec:More-freeness-lull}

We begin with some developments on the theory of free groups that occurred in the 1930s. As we shall see, they would have rather an important impact on the development of one-relator group theory in the 1960s, particularly the notion of \textit{residual nilpotence}. We begin with Magnus' main result during this time: the representation of free groups in (the group of units of) a certain ring. This can be fairly easily explained. Let $\xi_1, \dots, \xi_n$ be $n$ non-commuting variables, and let $\Z[[\xi_1, \dots, \xi_n]]$ be the ring of formal power series in the non-commuting variables. Let $a_i = 1 + \xi_i$. Then each $a_i$ is, in fact, invertible, with inverse given by 
\begin{equation}\label{Eq:Magnus-inverse}
a_i^{-1} = 1 - \xi_i + \xi_i^2 - \xi_i^3 + \cdots = \sum_{j=0}^\infty (-1)^j \xi_i^j.
\end{equation}
Thus the $a_i$ generate a group; and Magnus proved that the group they generate is a free group. This has a number of immediate consequences. 

\begin{theorem}[{Magnus, 1935 \cite{Magnus1935}}]
The group generated by $1 + \xi_i$ ($1 \leq i \leq n$) is free of rank $n$, freely generated by these generators. As a consequence, free groups are residually nilpotent. \label{Thm:Magnus-rep}
\end{theorem}

Here, a group $G$ is said to be \textit{residually nilpotent} if the intersection of all terms in the lower central series of $G$ is trivial; alternatively, we will see a more general definition of residual properties in \S\ref{Subsec:Residualfiniteness}. The residual nilpotence of free groups is a direct consequence in Theorem~\ref{Thm:Magnus-rep} as follows: let $F$ be a free group. Then any element of the $k$th term of the lower central series of $F$, considered as an element of $\Z[[\xi_1, \dots, \xi_n]]$ in the above representation, must have the degree of a certain component be at least $k$ (see \cite[p.~313]{Magnus1966}). This implies, of course, that the terms of the lower central series intersect trivially. A standard reference for this representation, now often called the \textit{Magnus representation}, is \cite[Chapter~5]{Magnus1966}. See also \cite{Chen1954, Chandler1968} for some extensions to other finitely generated groups; indeed, the representation was in the first place motivated in part by a 1928 article by Shoda \cite{Shoda1928} on the automorphism group of abelian groups.

The Magnus representation would end up inspiring much research on representing free groups. While we cannot write the full history here (it shall instead appear elsewhere), one particularly famous case bears a brief mention. In the late 1930s, a group of three Soviet students working in St Petersburg, supervised by V.\ A.\ Tartakovskii, picked up the subject: these were Kh.\ A.\ Doniyakhi, D.\ I.\ Fuchs-Rabinovich, and I.\ N.\ Sanov. First, Fuchs-Rabinovich proved results about ``determinators'' of free groups, and uses these to recover Magnus' theorem above on the residual nilpotence of free groups; this appeared in a 1940 article (written in English!) \cite{FuchsRabinovich1940}\footnote{Lyndon \cite{Lyndon1953} later elaborated on this article and its connection with Magnus' work, and we primarily refer the interested reader to consult this latter source.}. Doniyakhi proved results about extending these results to free products of abelian groups \cite{Doniyakhi1940}. Fuchs-Rabinovich \cite{FuchsRabinovich1940b}, in another article, also proved results about representing  free groups in $\SL_2(\C)$, giving another proof of Magnus' result. At this time, Sanov \cite{Sanov1940} instead worked on a different problem, namely the Burnside problem for exponent $4$, which he solved affirmatively in 1940.

The three students would have rather different fates. When war broke out in 1941, Fuchs-Rabinovich was drafted into the USSR Baltic Fleet. He died of dysentery in 1942, during the Siege of Leningrad, at the age of 29. Doniyakhi was also drafted, and died in battle less than two months later in 1941, at the age of 24. Sanov survived the war.\footnote{Indeed, Sanov served as Head of Mathematics for one year at the University of Pyongyang (North Korea) in 1949.} After the war, Sanov would carry on the research his fellow students had initiated, and in 1947 published the famous article \cite{Sanov1947}. This begins by acknowledging the articles by Fuchs-Rabinovich and Doniyakhi, then presents his theorem: the (multiplicative) group generated by the two matrices
\begin{equation}
\begin{pmatrix}
1 & 2 \\ 0 & 1
\end{pmatrix} \quad \textnormal{and} \quad \begin{pmatrix}
1 & 0 \\ 2 & 1
\end{pmatrix}
\end{equation}
is a free group of rank $2$. This can be used to give yet another proof of the residual nilpotence of free groups, and would directly inspire many important results in combinatorial group theory, including the Ping-Pong Lemma \cite{Macbeath1963}.

Sanov and his friends were of course not the only researchers to be affected by the war; we have already seen in \S\ref{Subsec:Surface-and-knot-groups} how Dehn was affected. Magnus, who had been appointed as an assistant at Königsberg in 1939, was forced to resign from this position shortly thereafter because he refused to join the NSDAP (the Nazi Party). Instead, he found employment at \textit{Telefunken}, a radio company, during the war.\footnote{During this time, he would gain an interest in \textit{special functions}, which would eventually lead to him joining a project with Erd\'elyi, Oberhettinger \& Tricomi to publish their famous book \cite{Erdelyi1953} on transcendental functions. This book remains, by a wide margin, Magnus' most cited work.} This abrupt break in his usual research environment could also explain the lack of progress on one-relator group theory during this time, cf.\ \S\ref{Subsec:Undecidability}.

\

Let us now turn back to some further results on free groups pertaining to one-relator groups. The first concerns the study of \textit{primitive} elements in free groups, as discussed in \S\ref{Subsec:Nielsen-Schreier}. Whitehead, using topological methods, proved two remarkable theorems in this direction in a 1936 article \cite{Wh36}. The first states simply that there is an algorithm for deciding whether or not a given word is a primitive element or not. That is, one can decide whether a given element $w \in F_A$ can be part of a basis or not (cf.\ Nielsen's result Proposition~\ref{Prop:Dehn-Nielsen-commutator}, which gives an algorithm for deciding if a pair of words is a basis for $F_2$). Furthermore, Whitehead completely characterized when a one-relator group is free.

\begin{theorem}[{Whitehead, 1936 \cite[Theorem~4]{Wh36}}]
Let $G = \pres{}{A}{r=1}$ with $r$ non-trivial. Then $G$ is free if and only if $r$ is a primitive element of $F_A$. \label{Thm:Whitehead-theorem}
\end{theorem}

As we shall see in \S\ref{Sec:6-OR-with-torsion}, there is also a complete characterization of when a one-relator group is \textit{virtually} free (Theorem~\ref{Thm:Fischer-Karrass-Solitar}), but this would not be proved until 1972.\footnote{Indeed, the fact that a non-free virtually free one-relator group must have elements of finite order is not at all obvious, but follows from Stallings' theorem \cite{Stallings1970}.} Next, a result which can be seen as a generalization of Whitehead's result was proved by Magnus in 1939.

\begin{theorem}[Magnus, 1939 \cite{Ma39}]
Let $G$ be a group given by a presentation with $n+r$ generators and $r$ defining relations, with $n, r \geq 0$. Then:
\begin{enumerate}
\item If $G$ can be generated by $n$ elements, then $G$ is free of rank $n$. 
\item If $G^{\operatorname{ab}} \cong \Z^n$, then some set of $n$ of the generators of $G$ generate a free group. 
\item If $G$ has the same sequence of lower central factors as the free group on $n+r$ generators, then $G$ is free. 
\end{enumerate}\label{Thm:Magnus-extended-Frei}
\end{theorem}

Thus Theorem~\ref{Thm:Magnus-extended-Frei}(2) can also be seen as an extension of the \textit{Freiheitssatz} to the case of more than one defining relation. We delay the definition of \textit{lower central factors} for when we revisit Theorem~\ref{Thm:Magnus-extended-Frei}(3) in the context of \textit{parafree} groups in \S\ref{Subsec:lower-central-series}.

Finally, we remark on a set of results that relate to one-relator groups only by association. In his 1931 article, his second on one-relator groups, referenced above, Magnus \cite{Magnus1931} also proved that the commutator subgroup of the modular group $G = \pres{}{a,b}{a^2 = b^3 = 1} \cong C_2 \ast C_3$ is a free group of rank $2$. His argument is essentially a simplified version of the machinery developed to deal with one-relator groups. Later, in 1948, Nielsen \cite{Nielsen1948} extended this result to arbitrary free products of finite cyclic groups, proving using a geometric argument that that the commutator subgroup of $C_{m_1} \ast C_{m_2} \ast \cdots \ast C_{m_r}$ is free of rank 
\[
1 + m_1 m_2 \cdots m_3 \cdot \left( -1 + \sum_{i=1}^r (1 - \frac{1}{m_i})\right).
\]
Many more results could be mentioned here, as we now are heading head-first into the dawn of combinatorial group theory. We have said nothing, for example, of the automorphism groups of free groups, a topic which interested Magnus, and many others, around this time. However, these subjects increasingly stray from one-relator groups, to which we now instead firmly return. 

\subsection{Lyndon's Identity Theorem}\label{Subsec:LyndonIT}

Recall Dyck's fundamental result (see \S\ref{Subsec:Dyck-and-presentations}) that in a group $G$ with underlying free group $F$ and given by defining relations $r_1, \dots, r_n$, an element $W$ is equal to one in $G$ if and only if it can be written of the form 
\begin{equation}\label{Eq:Dyck-Lyndon}
W = \prod_{i=1}^m T_i r_{t_i}^{\pm 1} T_i^{-1} \quad \textnormal{where } T_i \in F, \textnormal{ and } t_i = 1, 2, \dots, n.
\end{equation}
One might ask how \textit{unique} such an expression \eqref{Eq:Dyck-Lyndon}. For example, in the group $\pres{}{a,b}{[a,b]=1}$, isomorphic to $\Z^2$, we obviously have $[a^2, b^2] = 1$, since $\Z^2$ is commutative. This equality is witnessed by the fact that:
\begin{equation}\label{Eq:Two-expr-for-[a2,b2]}
\begin{cases}
\: [a^2, b^2] = \left( aRa^{-1} \right) \cdot R \cdot \left( ba R (ba)^{-1} \right) \cdot \left( bR b^{-1}\right) \\
\: [a^2, b^2] = \left( aRa^{-1} \right) \cdot \left( ab R (ab)^{-1} \right) \cdot R \cdot \left( b R b^{-1} \right)
\end{cases}
\end{equation}
where we have let $R = [a, b] = aba^{-1}b^{-1}$. How different are the two expressions in \eqref{Eq:Two-expr-for-[a2,b2]}? Up to reordering the terms, we notice that they are identical, except for the conjugates of $R$ by $ab$ resp.\ $ba$. However, $ab = ba$ in $\Z^2$, and so the two expressions in \eqref{Eq:Two-expr-for-[a2,b2]} are \textit{identical up to reordering the terms and equality in $G$}. This is an instance of Lyndon's Identity Theorem, which states that this behaviour always occurs in a torsion-free one-relator group; there are no ``relations between the relations''. 

Formally, the problem of determining relations between relations asks: what are the non-trivial relations holding between the relations $r_1, \dots, r_n$ of the form 
\begin{equation}\label{Eq:Non-trivial-identity}
\prod_{i=1}^m T_i r_{t_i}^{\pm 1} T_i^{-1} = 1 \quad \textnormal{where } T_i \in F, \textnormal{ and } t_i = 1, 2, \dots, n?
\end{equation}
There are some trivial relations always holding; for example, if $r_1, r_2$ are two different relations, then $(r_2 r_1 r_2^{-1}) \cdot (r_2) \cdot (r_1^{-1}) \cdot (r_2^{-1}) = 1$, which is a product of the form \eqref{Eq:Non-trivial-identity} but not a particularly interesting one. Such relations are called \textit{Peiffer identities}. According to Chandler \& Magnus \cite[p. 118]{Chandler1982}, the idea of studying identities between relations had been described by Dehn to Magnus in the early 1930s, but this was not pursued further at the time. Instead, this would not be done until the late 1940s, when it was investigated by Peiffer (a student of Reidemeister) in 1949 \cite{Peiffer1949} and, in the same year, by Reidemeister \cite{Reidemeister1949}. The latter proved that the problem of finding all non-trivial relations \eqref{Eq:Non-trivial-identity} reduces to understanding the case when all $T_i$ are taken modulo $r_1, \dots, r_n$, and disregarding the order of the factors in the product. If we set $G = F / R$, i.e.\ $R$ is the normal closure of $r_1, \dots, r_n$, then this amounts to stating that we must understand the group $R / [R, R]$. This abelian group is naturally a $\ZG$-module via the conjugation action of $G$, and is called the \textit{relation module} of $G$ (\textit{with respect to} $r_1, \dots, r_n$, see below). Furthermore, it can be shown that the action of $G$ on $R / [R,R]$ is faithful, i.e.\ only the identity element fixes all of $R / [R,R]$, see \cite{Auslander1955}.

In 1948, Reidemeister visited the Institute for Advanced Study for a year, which had a great influence on a young postdoc in Princeton at that time: R.\ Lyndon, who had been a PhD student of S.\ Mac Lane's \cite[p.\ 3]{Appel1984}. Reidemeister, who took an early interest in one-relator groups (see e.g.\ his 1932 proof of the \textit{Freiheitssatz} in \S\ref{Subsec:Freiheitssatz}), does not appear to have been the one to introduce Lyndon to one-relator groups. Instead, Lyndon recalls that after presenting his projective resolution (see below) in a seminar at the IAS, Reidemeister came up to him and explained a natural topological interpretation of his resolution \cite[p. 25]{Ratcliffe1984a}, cf.\ also \cite[p. 651]{Ly50}. Thus in the late 1940s, combinatorial group theory was certainly on Lyndon's mind, and in a brief 1949 article \cite{Lyndon1949} he proved a result for the cohomology of groups defined by generators and relations, before finally turning his attention to one-relator groups. He was able to prove the following theorem, by using the Magnus breakdown (\S\ref{Subsec:Freiheitssatz}): 

\begin{theorem}[Lyndon's Identity Theorem, 1950 \cite{Ly50}]
Let $G = \pres{}{A}{r=1}$, with $r = w^n$ where $w$ is not a proper power and $n \geq 1$. Then $\prod_i T_i r^{\varepsilon_i} T_i^{-1} = 1$ implies that the indices can be grouped into pairs $(i,j)$ such that $e_i = -e_j$ and $T_i \equiv T_j w^{c_i}$ in $G$ for some $1 \leq c_i \leq n$. \label{Thm:Lyndon-Identity}
\end{theorem}

In other words, there are no non-trivial relations holding among the relation of a one-relator group. We now give a restatement of the theorem in more modern terminology, using the language of homological algebra. Let us make clear that at this point, group cohomology was a nascent area of research, and many of the general statements that are today a standard part of the subject were not yet formulated.\footnote{To highlight this, we remark that group cohomology had essentially only been computed for finite abelian groups at the time of Lyndon's 1947 PhD thesis \cite{Lyndon1947}. The subject of this thesis, later published as \cite{Lyndon1948}, was the cohomology of group extensions, leading to what is today known as the \textit{Lyndon--Serre-Hochschild spectral sequence}. For a lucid account of this early history of group cohomology, we refer the reader to \cite{MacLane1984}.} We begin with a restatement of Theorem~\ref{Thm:Lyndon-Identity} in the case of $n=1$, i.e.\ the torsion-free case (see \S\ref{Subsec:Torsion-theorem}):

\begin{corollary}[{Lyndon \cite[p.~663]{Ly50}}]
Let $G = \pres{}{A}{r=1}$ be a torsion-free one-relator group. Then the relation module of $G$ with respect to this presentation is a free $\ZG$-module. \label{Cor:Rel-mod}
\end{corollary}

The relation module and the second homology group of any group are, as is well-known (and as we shall see below), closely related; this was observed by Peiffer \cite{Peiffer1949b} and Miller \cite{Miller1952}. We can thus use the Identity Theorem to gain information about the cohomological dimension of a one-relator group:

\begin{corollary}[{Lyndon \cite[Corollary~11.2]{Ly50}}]
Let $G$ be a torsion-free one-relator group. Then $G$ has cohomological dimension at most $2$. 
\end{corollary}

Let us now give a restatement of the Identity Theorem in terms of an explicit resolution. Let $G = \pres{}{x_1, \dots, x_m}{r^n=1}$ be a one-relator group with $r$ not a proper power and $n \geq 1$. By attaching a $1$-cell for each generator and a single $2$-cell for the defining relator, we can always find a sequence of free $G$-modules 
\[
F_2 \xrightarrow{\partial_2} F_1 \xrightarrow{\partial_1} \Z \to 0
\]
where $F_2$ is free of rank $1$, and $F_1$ is free of rank $m$. Lyndon's Identity Theorem then states that the kernel of $\partial_2$ is generated by the element $r-1$. If $n=1$, then of course $r-1=0$ in $\ZG$, so it follows that $\partial_2$ is injective, and hence any torsion-free one-relator group $G$ has cohomological dimension $2$, and one can explicitly compute the (co)homology groups. For example, for trivial coefficients it follows almost immediately from Corollary~\ref{Cor:Rel-mod} that $H_2(G, \Z) = \Z$ if and only if $r \in [F, F]$, and otherwise $H_2(G, \Z) = 0$. Hughes \cite{Hughes1966} used the Identity Theorem to compute the cohomology of $G = F / R$ with coefficients in the relation module $R / R '$. Then, when $r = w^n$ with $n \geq 1$ and $w$ not a proper power, he proves that $H^k(G, R / R')$ is trivial for odd $k$ and cyclic of order $n$ for even $k > 2$. Furthermore, he shows that $H^2(G, R / R')$ is an extension of an easily describable group $D(G)$ by a cyclic group of order $n$. For example, in the case of $G = \pi_1(\mathcal{K}_g)$ from \eqref{Eq:pres-of-K-g}, Hughes is able to show that $H^2(\pi_1(\mathcal{K}_g), R/R') \cong \Z$. Finally, we remark that in 1984 Ratcliffe \cite{Ratcliffe1984b} gave an explicit description of the cup product $\smile\colon H^1(G, M) \times H^1(G, M) \to H^1(G, M)$ for an arbitrary $\ZG$-module $M$. 

In 1963, Cohen \& Lyndon \cite{CL63} gave a new proof of the Identity Theorem, and strengthened it by giving an explicit free basis of the relation module. Concretely, if $G =\pres{}{A}{r^n = 1}$, with $r \in F_A$ not a proper power and $n \geq 1$, then Cohen \& Lyndon construct a set $X \subset F_A$ of words such that $\{ x^{-1} r^n x \mid x \in X \}$ is a free basis of the normal closure $\langle\langle r^n \rangle\rangle$ in $F_A$. Their results, proved via the usual Magnus hierarchy but with a stronger inductive hypothesis, also included, as a byproduct, a new proof of the \textit{Freiheitssatz}. Two other simpler proofs are due to Karrass \& Solitar \cite{KarrassSolitar1972} and \cite{Burns1974}, in 1972 and 1974 respectively. 

We end with a word of caution regarding the terminology ``relation module of a one-relator group''. Indeed, relation modules depend on the presentation chosen, although if $G \cong F / R \cong F / S$, then it is known that there is a free $\ZG$-module $X$ such that $R / R' \oplus X \cong S / S' \oplus X$, see \cite{Ly62}. By combining the Identity Theorem with the above statement, we find that the relation module of a one-relator group $G \cong F / S$ is always $\ZG$-projective. However, if $S$ is not the normal closure of a single element -- even though $G$ is isomorphic to a one-relator group! -- then the relation module $S / S'$ need not be $\ZG$-free. Indeed, Dunwoody \cite{Du72} showed that in the case of the trefoil knot group $G = \pres{}{a,b}{a^2 = b^3}$ (see \S\ref{Subsec:Surface-and-knot-groups}), there is a presentation $G \cong F / S$ for which the relation module $S / S'$ cannot be generated by a single element, and indeed cannot contain $\ZG$ as a direct summand. Thus $S / S'$ is not free, but it is projective.

\subsection{Residual properties of groups}\label{Subsec:Residualfiniteness}

The notion of residual properties goes back rather far, and while we will not attempt to survey it in full here, we will mention some generalities. Let $\cP$ be a property of groups (e.g.\ being finite, free, etc.). A group $G$ is termed \textit{residually} $\cP$ if for every $1 \neq g \in G$ there exists a group $H_g$ with property $\cP$ and a homomorphism $\varphi_g \colon G \to H_g$ such that $\varphi(g) \neq 1$. The terminology \textit{residually} $\mathcal{P}$ is originally due to P.\ Hall \cite[p. 349]{Hall1954} in 1954, but the concept can be traced further back. For example, in 1940 Mal'cev \cite{Malcev1940}, who called residually finite groups \textit{finitely approximable}, proved that every finitely generated linear group is residually finite. He also proved, in connection to the \textit{Hopf problem} (the connection of which to one-relator groups we shall explore thoroughly in \S\ref{Subsec:Hopf}) that no finitely generated linear group is hopfian; his argument can be easily extended to show that in fact every finitely generated residually finite group is hopfian. 

Free groups are the prototypical examples of residually finite groups. Even before Mal'cev's result on linear groups -- and free groups are linear, using e.g.\ the representation by Fuchs-Rabinovich \cite{FuchsRabinovich1940b}, see \S\ref{Subsec:More-freeness-lull} -- it was well-known that free groups are residually finite. For example, in 1933 Levi \cite[Theorem~1]{Levi1933b} gave a proof of the residual finiteness of free groups. Indeed, he proved the stronger statement that the intersection of any decreasing series of characteristic subgroups in a free group is trivial, which shows that, in particular, free groups are residually nilpotent (cf.\ Magnus' 1934 proof in \S\ref{Subsec:More-freeness-lull}). For other early proofs of the residual finiteness of free groups, see also \cite{Iwasawa1943, Neuwirth1961}. Free groups are also known to satisfy the stronger property of being residually $p$-groups for any prime $p$, as proved by Marshall Hall \cite{Hall1950} and Takahasi \cite{Takahasi1951}. Another natural class of residually finite groups are the finitely generated nilpotent groups, as proved by Hirsch in 1946 \cite{Hirsch1946}. Thus residually nilpotent groups are also residually finite, and we will encounter many examples of residually nilpotent groups throughout this survey. Residual nilpotence, in particular, has a strong connection with the Magnus representation of free groups (\S\ref{Subsec:More-freeness-lull}) and many results on residual nilpotence of particular quotients of free groups and related topics, e.g.\ a well-known result due to Gruenberg \cite{Gruenberg1962} or Neumann \cite{Neumann1962}, can often be proved by appealing to the Magnus representation, see \cite{Bachmuth1966}. 

One of the advantages of residual finiteness is that any finitely presented residually finite group has decidable word problem. The algorithm is simple: given such a group $G = \pres{}{A}{R}$, to decide whether a given word $w \in F_A$ is equal to $1$ in $G$ or not we run two algorithms in parallel: (1) enumerate all words equal to $1$ in $G$; and (2) enumerate all homomorphisms from $G$ onto finite groups. If $w=1$, then the first algorithm detects this; if $w \neq 1$, then by residual finiteness the second algorithm will detect this. Thus $G$ has decidable word problem.\footnote{This algorithm is often attributed to McKinsey, particularly to his 1943 article \cite[Theorem~3]{McKinsey1943}, but this is not accurate, for several reasons. First, McKinsey works in a considerably more general context of ``finite reducibility'' in algebras, and as McKinsey himself points out in a footnote, ``this principle has been known for a long time [...] It was employed by L.\ Löwenheim [...] [McKinsey] also made use of it in his article \cite{McKinsey1941}'' \cite[p. 69, footnote 6]{McKinsey1943}. The reference to Löwenheim refers to Hilbert \& Ackermann \cite[p. 95]{Hilbert1938}, which in turn refers to a 1915 article by Löwenheim \cite{Lowenheim1915}, wherein the principle is admittedly well-disguised. The first to write the principle in its modern form, purely speaking of groups, was Mostowski \cite{Mostowski1966} in 1966 and, independently, V.\ Dyson \cite{Dyson1964} (see also \cite{Dyson1974}). The argument was also known to G.\ Higman and A.\ Turing, but not published (as pointed out by Boone and Schupp in their MR review of \cite{Mostowski1966}).} Since all one-relator groups have decidable word problem (\S\ref{Subsec:Applications-of-Frei}, Theorem~\ref{Thm:Decidable-word-problem}), this raises the following very natural question: is \textit{every} one-relator group residually finite? The answer is a resounding \textit{no}, and we will encounter many such counterexamples in the form of \textit{Baumslag--Solitar groups} in \S\ref{Subsec:Baumslag-Solitar}, which are of central importance due to giving rise to non-hopfian groups. Finally, we mention that Meskin \cite{Meskin1974} proved in 1974 that there exist finitely generated residually finite groups with undecidable word problem. Hence, for the above algorithm, the assumption of finite presentability is crucial.

\subsection{Why did progress halt?}\label{Subsec:Undecidability}

As we shall see in the next section (\S\ref{Sec:5-1960s-resurgence}), the area of one-relator group theory would see a resurgence of interest after the lull phase ended just before 1960. As described by Chandler \& Magnus, Baumslag proposed the following idea to explain the lull phase: ``the \textit{Freiheitssatz} and the subsequent solution of the word problem for one-relator groups by Magnus left the impression that one-relator groups are something almost as easy to deal with as free groups and that this delayed a more thorough investigation of one-relator groups for quite a while.'' \cite[p. 119]{Chandler1982}. While somewhat compelling, we would like to add two further suggestions to this, which have not appeared elsewhere in the literature. 

First, given the large number of PhD students supervised by Magnus who would end up active in the area of one-relator group theory, it also bears noticing that Magnus only had two PhD students while in Berlin (F.\ Oberhettinger and F.\ Sh\"afke) with neither working in group theory. As soon as he arrived in New York in 1950, however, he started supervising a large number of doctoral students, with the first graduating in 1954; one of the first was Abe Shenitzer\footnote{Shenitzer and Magnus would develop a life-long friendship, summarized briefly in the touching tribute \cite{Shenitzer1995}. In particular, Abe Shenitzer (1921--2022) had been imprisoned in several concentration camps during the war. After the war, he arrived in New York, and studied under Magnus. Early on in their friendship, Magnus, visibly moved, said ``You've done more for me than any person can do for another. [...] You are a Jew who was in German concentration camps, and I am a German'', cf.\ also \cite{Stillwell2022}. Magnus choosing to return to his pre-war work on one-relator groups at this point might thus have been in large part stimulated by his interaction with Shenitzer.} \cite{Shenitzer1954}, who gave a necessary and sufficient condition for a one-relator group to split as a free product. This result can thus be seen as Magnus' return to considering one-relator groups, following the turbulence of the war. Magnus would end up supervising 66 PhD students at New York University between 1950 and 1974, with some ending up as the most important names in the theory of one-relator groups and combinatorial group theory: to mention a few early names, we have D.\ Solitar (1958), M.\ Greendlinger (1960), B.\ Chandler (1962), and J.\ Birman (1968). The same idea can be proposed for B.\ H.\ Neumann, given his early contact with Magnus while in Berlin. He moved to the University of Manchester in 1948, where he supervised both J.\ Britton (1954) and G.\ Baumslag (1958), both of whom would be of central importance to the subject (in this line, cf.\ also \cite[pp. 197--198]{Chandler1982}). 

The second reason, and in our view more important, is the general undecidability of the word problem. We first give a brief summary of these developments. The word problem, having originated with Dehn and Thue as outlined in \S\ref{Subsec:Dyck-and-presentations}, gave rise to the first undecidability result in ``real'' mathematics, following relatively artificial undecidable problems on Turing machines as given by Turing \cite{Turing1936}. This first ``real'' undecidability result was the 1947 proof by Post \cite{Post1947} and, independently Markov \cite{Markov1947a,Markov1947b}, that there is a finitely presented \textit{semi}group with undecidable word problem. This would in short order lead to the undecidability of the word problem for groups. Turing \cite{Turing1950} worked on the problem, in 1950 proving the undecidability of the word problem in \textit{cancellative} semigroups\footnote{We mention here another connection to B.\ H.\ Neumann: around 1950, when he and Turing were both at the University of Manchester, the two of them had announced a resolution to the word problem for groups. However, their resolutions were contradictory: Turing announced the \textit{undecidability} of the word problem in finitely presented groups, whereas Neumann announced the general \textit{decidability} of the same. Both later retracted their claims, but Turing's work led to the undecidability for cancellative semigroups in the work \cite{Turing1950} cited above \cite[p. 169]{Collins1985}.}, a class of semigroups in which either of $xy = xz$ or $yx = zx$ implies $y=z$. This article would be used\footnote{Turing's article contains a number of gaps, somewhat amended by Boone \cite{Boone1958}, and Novikov \& Adian \cite{Novikov1958} later produced their own proof of the same result to avoid this dependency.} by Novikov \cite{Novikov1952, Novikov1955} to produce a proof of the undecidability of the word problem in groups; this was done independently by Boone \cite{Boone1957,Boone1959} and Britton \cite{Britton1958, Britton1963} around the same time. Today, it is known that there exists a 12-relator group with undecidable word problem \cite{Borisov1969}, arising from an encoding of a 3-relation semigroup with undecidable word problem due to Matiyasevich \cite{Matiyasevich1967}. 

In connection with the undecidability of the word problem, in the mid-1950s a number of further undecidability results for groups would appear, including the famous \textit{Adian--Rabin Theorem}, stating that deciding ``most'' reasonable properties (e.g.\ finiteness, triviality, being abelian or free, etc.) about finitely presented groups is undecidable \cite{Adian1955, Rabin1958, NybergBroddaAR}. Thus, around 1960 the area of combinatorial group theory had undergone something of a shock. Indeed, one might reasonably speculate that a contemporary researcher reading Magnus' 1932 article \cite{Ma32} would imagine that perhaps the same method could be extended, at great combinatorial cost and with little gain, to \textit{all} finitely presented groups, and that therefore there was nothing particularly special about one-relator groups \textit{per se}. The undecidability results of the 1950s shattered this illusion. Now, research could be directed to classes of groups where the word problem was decidable in order to see what was special about them, because something had to be special if the word problem were to be decidable. One-relator groups certainly turned out to be special indeed.

\section{The resurgence of one-relator groups (1960s)}\label{Sec:5-1960s-resurgence}

\noindent Topology gave rise to combinatorial group theory. As we have seen in \S\ref{Sec:2-Dehn-freegroups-presentations}, the early articles on one-relator groups were all clearly motivated by topology, with surface groups and knot groups being particular flagship examples. As the machinery of one-relator groups and combinatorial group theory grew increasingly powerful, it also became increasingly useful to reapply these tools back to answer topological questions. In the late 1950s, an important connection between the Poincaré Conjecture and combinatorial group theory was made by C.\ Papakyriakopoulos. While this connection would not result in a proof of the Poincaré Conjecture, it did result, directly and indirectly, in the development of one-relator group theory. Other topological questions, related directly to the old \textit{hopfian} problem, would also play a key role during this time. In this section, we will outline how these topological considerations led to a resurgence of one-relator group theory, and some of the subsequent developments in the 1960s. We will also showcase Magnus' continued influence on the area by illustrating how one of his PhD students gave rise to a new school of combinatorial group theory in the USSR.

\subsection{The Poincaré Conjecture and torsion}\label{Subsec:Torsion-theorem}

\noindent Around this time, new stirrings of topological methods appeared in combinatorial group theory. An important stirrer was C.\ D.\ Papakyriakopoulos, who rose to prominence by giving a proof of \textit{Dehn's Lemma} in 1957 \cite{Papakyriakopoulos1957, Papakyriakopoulos1957b} (the original ``proof'' of this result, given by Dehn \cite{Dehn1910} in 1910, had a gap, as pointed out by Kneser \cite[p. 260]{Kneser1929}). Soon thereafter, Papakyriakopoulos \cite{Papakyriakopoulos1958} became interested in proving the Poincaré Conjecture\footnote{The Poincaré Conjecture, which other than its brief connection to one-relator groups expounded here lies wholly outside the scope of this survey, states: \textit{a closed $3$-manifold with trivial fundamental group is homeomorphic to a $3$-sphere}. The conjecture was proved by Perelman in the early 2000s; for more historical details, see Stillwell's translation and commentary in \cite[pp.\ 1--12]{Poincare2010}.}, and in 1963 \cite{Papakyriakopoulos1963, Papakyriakopoulos1963a} he considered the following construction (with a slight correction provided by Maskit \cite{Maskit1963}). Let $\pi_1(\Sigma_g)$ be the surface group of genus $g>1$, defined by \eqref{Eq:pres-of-sigma-g}, let $w$ be any word in the commutator subgroup of $\pi_1(\Sigma_g)$, and let $P_g(w)$ be the quotient of $\pi_1(\Sigma_g)$ by the relation $[a_1, b_1w] = 1$. Then Papakyriakopoulos was able to reduce the Poincaré Conjecture to the following group-theoretic conjectures: (a) $P_g(w)$ is torsion-free; and (b) the cover of $\Sigma_g$ corresponding to the kernel of $\pi_1(\Sigma_g) \to P_g(w)$ is planar. In connection to the first of these two conjectures, Papakyriakopoulos wrote a letter to Magnus, asking whether the then-known methods of combinatorial group theory could be applied to solve (a). This letter would become the basis of a 1960 article by Karrass, Magnus \& Solitar \cite{KMS60} (D.\ Solitar was a student of Magnus', and A.\ Karrass had been a student of Solitar's). The three authors did not solve Papakyriakopolous' question in full, but provided some partial results indicating it had a positive answer. The groups $P_g(w)$ are two-relator groups, and as part of their investigations of the elements of finite order in this class, the authors of \cite{KMS60} were able to completely characterize the elements of finite order in one-relator groups:

\begin{theorem}[{Karrass, Magnus \& Solitar, 1960 \cite[Theorem~1]{KMS60}}]
Let $G = \pres{}{A}{r= 1}$ be a one-relator group with $r$ cyclically reduced. Then $G$ has elements of finite order if and only if $r$ is a proper power in $F_A$, i.e.\ $r = w^n$ with $n > 1$ and $1 \neq  w \in F_A$. \label{Thm:torsion-theorem}
\end{theorem}

The proof, of course, goes via the classical Magnus hierarchy, by induction on the relator length (Gruenberg \cite{Gruenberg1970} gave a cohomological proof of Theorem~\ref{Thm:torsion-theorem} in 1970). The theorem states, in other words, that a one-relator group has torsion precisely when it obviously has torsion, and the authors of \cite{KMS60} further characterize the elements of finite order as the obvious ones, namely those which are conjugate to a power of the relator. The authors of \cite{KMS60} proceed to provide some further insight into one-relator groups with torsion as follows: 

\begin{theorem}[{Karrass, Magnus \& Solitar \cite[Theorem~3]{KMS60}}]
Consider the two one-relator groups $G_1 = \pres{}{A}{r_1^n = 1}$ and $G_2 = \pres{}{A}{r_2^n = 1}$ where $n > 1$. If $G_1 \cong G_2$, then also $\pres{}{A}{r_1 = 1} \cong \pres{}{A}{r_2=1}$. \label{Thm:If-Rn=Sn-thenR=S}
\end{theorem}

This shows that isomorphisms for one-relator groups with torsion is, in a sense, at least as ``rigid'' as isomorphisms of torsion-free one-relator groups. Note that the converse of Theorem~\ref{Thm:If-Rn=Sn-thenR=S} fails, but this would not be proved until more than a decade later, see \S\ref{Subsec:Torsion-isomorphism-problem}. Indeed, at this point in the story, one-relator groups with torsion were still viewed with the same lens as their torsion-free counterparts; there was even some indication that the presence of torsion was a complicating factor. When we return to the topic of one-relator groups with torsion in \S\ref{Sec:6-OR-with-torsion}, we shall see that this indication was very far from the truth. 

As mentioned above, Karrass, Magnus \& Solitar were not able to completely resolve Papakyriakopoulos' conjecture (a), but provided some partial progress \cite[Theorem~5]{KMS60}. Indeed, their work can be seen as a very early form of ``relative one-relator groups'', an area that would grow in the decades to follow \cite{Ho81, Howie1984, Howie1987}. In 1964, another PhD student of Magnus', E.\ Rapaport, would provide a full affirmative solution to conjecture (a). In a beautiful 1964 article, Rapaport \cite{Rapaport1964} adapted the Magnus hierarchy to the very particular two-relator setting of conjecture (a), as suggested to her by G.\ Baumslag (a name we shall encounter frequently from this point onwards), and was able to prove conjecture (a). For further work in this direction, see \cite{Papakyriakopoulos1963b}. Ultimately, the Poincaré Conjecture would not be proved by means of Papakyriakopoulos' reduction, as in 1981, Moran \cite{Moran1981} and independently McCool \cite{McCool1981} proved that conjecture (b) is false, using a further reduction due to Swarup \cite{Swarup1979}.

Nevertheless, topological methods were to remain in combinatorial group theory. Indeed, soon after the new proof of Dehn's Lemma, J.\ Stallings, motivated in part by the old work of Kneser \cite{Kneser1929} and encouraged by Papakyriakopoulos (see \cite[p. 1]{Stallings1965}), gave a new proof of Grushko's theorem on generators of free products using purely topological methods. His proof also dropped the finite generation hypothesis appearing in Grushko's proof, reproving the more general result due to Wagner \cite{Wagner1957}. Topological methods in combinatorial group theory would bloom in the coming decades. Stallings' 1983 article \cite{St83} introducing simple graph-theoretic methods to the study of free groups marked one turning point in the history of combinatorial group theory; and similarly the ``tower method'' of Papakyriakopoulos used to proved Dehn's Lemma would become crucial in the study of locally indicable groups and one-relator groups, as pioneered by Howie \cite{Howie1983} in 1983. These methods, however, pass outside the scope of this present history.

Instead, we now return to the state of one-relator group theory in 1960, where many problems from topology motivated further problems and progress. One problem of particular importance regarded \textit{hopfian} groups, and it is here that one-relator groups would begin to reveal the full extent of their rich, versatile, and often mysterious nature. We first give a brief background on hopfian groups, and then expand on how they led to a resurgence of one-relator group theory.

\subsection{Hopfian groups}\label{Subsec:Hopf}

A group is said to be \textit{hopfian} if it is not isomorphic to any proper quotient of itself. The terminology arises from the fact that H.\ Hopf \cite{Hopf1931} in 1931 studied the fundamental groups $\pi_1(\Sigma_g)$ of genus $g>1$, which we defined in \eqref{Eq:pres-of-sigma-g} and he denoted by $\mathfrak{S}_g$. In particular, he proved by a rather complicated topological argument that $\pi_1(\Sigma_g)$ is hopfian. In this context, one is naturally lead to ask: are there finitely generated non-hopfian groups? Hopf asked this question, although not explicitly in any of his own articles; already by 1934, however, Magnus \cite[p. 276]{Magnus1935} noted that he had been informed by B.\ H.\ Neumann of Hopf asking this question. Of course, there are many non-finitely generated examples of non-hopfian groups, e.g.\ $\prod_{i \in \Z} \Z$, but the assumption of finite generation makes the problem significantly more difficult, and it would take nearly 20 years to resolve (affirmatively).

It was known already at this point that finitely generated free groups were hopfian; indeed, it is not difficult to prove using Nielsen's 1921 results (see \S\ref{Subsec:Nielsen-Schreier}). Magnus \cite{Magnus1935} explicitly gives a proof of the hopficity of free groups. Furthermore, Mal'cev \cite{Malcev1940} proved in 1940 that any finitely generated linear group is hopfian. Indeed, he proved that any such group is residually finite, a sufficient condition for hopficity as mentioned in \S\ref{Subsec:Residualfiniteness}, and Fuchs-Rabinovich \cite{FuchsRabinovich1940} moreover proved that the any finite free product of finitely generated abelian groups is hopfian. One of the first to consider the hopfian problem explicitly was Baer \cite{Baer1944} in 1944, who extended Fuchs-Rabinovich's theorem to a more general setting. He also claimed to have given a construction of a finitely generated non-hopfian group \cite[p. 272]{Baer1944}; however, this was later withdrawn due to an error \cite[Appendix, p. 741--743]{Baer1949}. Thus, for some time, the problem of finding a finitely generated non-hopfian group remained open.\footnote{We remark here that being hopfian is \textit{not} a Markov property of groups in the sense of the Adian--Rabin Theorem \cite{Adian1955, Rabin1958}, as every finitely presented group can be embedded into a hopfian group \cite{Miller1971}. However, Collins \cite{Collins1969} proved in 1969 that it is nevertheless undecidable whether a given finitely presented group is hopfian. From these results, of course proved only long after the results about to be presented, it follows that there exists an abundance of finitely presented non-hopfian groups.}

The first example of a finitely generated non-hopfian group was finally given in 1950 by B.\ H.\ Neumann \cite{Neumann1950}. This group is not finitely presented. The next year, Higman \cite{Higman1951} constructed a two-relator non-hopfian group, given by the simple presentation
\begin{equation}\label{Eq:Higman-two-relator-nonhopf}
\pres{}{a,b,c}{a^{-1}ca = b^{-1}cb = c^2}.
\end{equation}
A further finitely generated non-hopfian group was constructed in 1961 by Hall \cite{Hall1961}. In 1962, Gilbert Baumslag constructed a new example of finitely presented non-hopfian groups \cite{Baumslag1962}: take two integral Heisenberg groups, and amalgamate them over a common free abelian subgroup of rank $2$. Up to that point, the above examples were essentially the only known non-hopfian groups. But it would not be long before Baumslag would contribute a remarkable number of results in this direction. 

Gilbert Baumslag (1933--2014) first heard of Wilhelm Magnus in 1956 while attending lectures at the University of Manchester given by his PhD supervisor B.\ H.\ Neumann \cite[p. 99]{Baumslag1994}. The subject of the lectures was \textit{amalgamated free products}, which we have encountered before. Neumann pointed out that although Nielsen introduced amalgamated free products, it was Magnus who recognised their value in his work on one-relator groups. These amalgamated free products (or \textit{Schreier products}, as Baumslag calls them in e.g.\ \cite{Ba63}), would become a central part of the toolbox he used to prove a remarkable set of theorems. 

Thus, at this point, in 1962, Baumslag would be part of two remarkable results relating one-relator groups and hopfian groups. The first was a positive result. He proved that the amalgamated free product of a free group with a free abelian group, with amalgamated subgroup maximal cyclic in the free group, is residually free \cite[Theorem~1]{Ba62} (cf.\ also \cite{Dyer1968}). This general result implies many interesting corollaries. In particular, it implies that the amalgamated free product $F \ast_{u=\overline{u}} \overline{F}$ of two isomorphic free groups, with isomorphism given by $x \mapsto \overline{x}$ and $u \in F$ being any word generating its own centralizer, i.e.\ not a proper power, is residually free. Thus, many one-relator groups are residually free; and hence, by Magnus' 1935 result (see \S\ref{Subsec:More-freeness-lull}) also residually nilpotent and residually finite. One example comes surface groups, which we saw in \S\ref{Subsec:Surface-and-knot-groups}. Baumslag \cite[p. 425]{Ba62} states that Magnus had asked in a seminar whether all surface groups $\pi_1(\Sigma_g)$ are residually nilpotent. This question can be answered by using Baumslag's aforementioned theorem, by taking a commutator $u = [a_1, a_2]$, which makes 
\[
\pi_1(\Sigma_2) \cong \pres{}{a_1, a_2, \overline{a}_1, \overline{a}_1}{[a_1, a_2] = [\overline{a}_1, \overline{a}_2]} \cong F_2 \ast_\Z F_2
\]
be residually free (and hence also all $\pi_1(\Sigma_g)$ for $g > 1$ by the discussion following Theorem~\ref{Thm:Surface-subgroup-theorem}). Baumslag's results show that $\pi_1(\Sigma_g)$ is residually nilpotent and hence also hopfian, thus answering Magnus' question positively, giving many new hopfian one-relator groups.\footnote{Baumslag mentions in a footnote \cite[p.\ 425]{Ba62} that Magnus' question was also answered, independently and by a very different method, by K.\ Frederick \cite{Frederick1961, Frederick1963}, a PhD student of Magnus'. Furthermore, G.\ Higman also communicated a different proof to Baumslag.} He also proves that all $2$-generator subgroups of certain one-relator groups are free (\cite[Theorem~2]{Ba62}), comparing this result to the \textit{Freiheitssatz}. Indeed, this would be one of the first results about subgroups of one-relator groups since the \textit{Freiheitssatz}, cf.\ \S\ref{Sec:7-Subgroups-of-OR-groups}. The second contribution from 1962 would be of a negative nature, and was joint between Baumslag and D.\ Solitar (a former PhD student of Magnus'): the construction of the famous \textit{Baumslag--Solitar groups}.

\subsection{Baumslag--Solitar groups}\label{Subsec:Baumslag-Solitar}

\

{\setlength{\epigraphwidth}{0.55\textwidth}
\epigraph{\textit{These facts destroy several conjectures long and strongly held by the reviewer and others, and some published and unpublished `theorems'.}}{---B.\ H.\ Neumann, MR review of \cite{Baumslag1962c}.}

\noindent The hopficity of one-relator groups was, for some time, an overlooked question. In fact, in connection with the construction of his \textit{two}-relator group \eqref{Eq:Higman-two-relator-nonhopf}, Higman \cite[p.~61]{Higman1951} states that it can be shown, using an argument due to B.\ H.\ Neumann and H.\ Neumann, that all one-relator groups are hopfian. This is, as we shall now see, very far from the truth. In 1962, Baumslag \& Solitar \cite{Baumslag1962c} would consider very simple examples of one-relator groups: 
\begin{equation}\label{Eq:BS(m,n)}
\BS(m,n) = \pres{}{x,y}{x^{-1}y^mx=y^n},
\end{equation}
Their main result, appearing in a bulletin article, would be to prove that many of the groups $\BS(m,n)$, now called \textit{Baumslag--Solitar groups}, are not hopfian (see below for a correction of their initially incorrect statement). One of the main examples of non-hopfian groups would be the following remarkable statement:

\begin{theorem}[{Baumslag \& Solitar, 1962 \cite{Baumslag1962c}}]
The group $\BS(2,3)$ is not hopfian. 
\end{theorem}

We will not present the proof of this here, but recall that by Mal'cev's theorem (\S\ref{Subsec:Residualfiniteness}) all finitely generated residually finite groups are hopfian. Thus $\BS(2,3)$ is not residually finite. This, at least, is easy to observe directly. Indeed, let
\begin{equation}\label{Eq:BS(2,3)}
B := \BS(2,3) = \pres{}{x,y}{x^{-1}y^2x=y^3},
\end{equation}
and let $\varphi \colon B \to F$ be any homomorphism onto a finite group $F$. Let $\alpha = \varphi(x), \beta = \varphi(y)$. A simple inductive argument shows that $x^{-n} y^{2^n} x^n = y^{3^n}$ for all $n \geq 0$ in $B$, and hence $\alpha^{-n} \beta^{2^n} \alpha^n = \beta^{3^n}$ in $F$. In particular, if $n$ is the order of $\alpha$, then we deduce that the order of $\beta$ must divide $k = 3^n - 2^n$. Since $k$ is obviously coprime with both $2$ and $3$, by Bézout's identity there are integers $m_1, m_2 \in \Z$ such that $1 = m_1 k + 2m_2$, and hence $\beta = \beta^{2m_2}$ is a power of $\beta^2$. Thus, letting $\beta_1 = \varphi(a^{-1}ba)$, we have that $\beta_1$ is a power of $\alpha^{-1} \beta^2 \alpha = \beta_1^2$, and $\beta_1^2 = \beta^3$. Thus $\beta_1$ is a power of $\beta$, and so we must have $[\beta_1, \beta] = 1$. In other words, since $\varphi$ was arbitrary, the word $c = [x^{-1}yx, y]$ must vanish in all finite homomorphic images of $B$. But $c \neq 1$ in $B$, as Magnus' breakdown procedure shows that if $c=1$, then in the torus knot group $\pres{}{y_0, y_1}{y_0^2 y_1^3 = 1}$ we have $[y_0, y_1] = 1$, which is false by standard properties of amalgamated free products (or indeed since torus knots are not unknots). Thus $B$ is not residually finite. 

The above argument extends to show that if $\gcd(m,n) = 1$, and $|m|, |n| \geq 2$, then $\BS(m,n)$ is not residually finite. However, characterizing precisely when $\BS(m,n)$ is residually finite is somewhat non-trivial. Indeed, in \cite[p.~199]{Baumslag1962c}, it is claimed that if $m$ or $n$ divides the other, then $\BS(m,n)$ is residually finite. This is not correct, as one can show e.g.\ that $\BS(2,4)$ is \textit{not} residually finite. The correct characterization would only be given by Meskin in 1972, as follows: 

\begin{theorem}[{Meskin, 1972 \cite[Theorem~C]{Meskin1972}}]
A Baumslag-Solitar group $\BS(m,n)$ is residually finite if and only if $|m| = 1$ or $|n|=1$ or $|m|=|n|$.\label{Thm:BS-residual-finiteness-characterisation}
\end{theorem}

Many other complications arise in the literature of Baumslag--Solitar groups, including for stating precisely when they are hopfian. Baumslag \& Solitar themselves state that $\BS(m,n)$ is hopfian if and only if either $m$ divides $n$ or $n$ divides $m$, or if $\Pi(m) = \Pi(n)$, where $\Pi(k)$ denotes the set of prime divisors of $k$. This is not correct, as e.g.\ $\BS(2,6)$ can be shown to not be hopfian. The correct characterization is as follows.

\begin{theorem}[{Collins \& Levin \cite{Collins1983}}]A Baumslag--Solitar group $\BS(m,n)$ is hopfian if and only if it is residually finite, or if $\Pi(m) = \Pi(n)$.\label{Thm:BS-hopfian-characterisation}
\end{theorem}

Thus, for example, $\BS(12, 18)$ is hopfian, but it is not residually finite. Indeed, as stated (but not proved) by Baumslag \& Solitar \cite[Theorem~2]{Baumslag1962c}, it has a finite index subgroup which is not hopfian (and therefore also not residually finite). The earliest complete published proof of Theorem~\ref{Thm:BS-hopfian-characterisation} that we are aware of only appeared in 1992, see \cite[Corollary~1]{Andreadakis1992}. There, the authors attribute this result to Collins \& Levin \cite{Collins1983}, who do not give a full proof of the result, but prove that for all $|m|, |s| \neq 1$, if $\Pi(ms) = \Pi(m)$, then $\BS(m, ms)$ is hopfian. This is one of the more difficult parts of proving Theorem~\ref{Thm:BS-hopfian-characterisation}, and it seems very likely that Collins \& Levin were aware of, and could have written down a proof of, Theorem~\ref{Thm:BS-hopfian-characterisation}. The above characterization also leads to some curious results: for example, it is not difficult to show that in $\BS(12,18)= \pres{}{x,y}{x^{-1}y^{12}x = x^{18}}$ the group generated by $\langle x, y^6 \rangle$ is isomorphic to $\BS(2,3)$. Thus $\BS(2,3) \leq \BS(12,18)$, even though the former is non-hopfian, but the latter is hopfian because $\Pi(12) = \Pi(18)$. Thus the Baumslag--Solitar groups also serve to illustrate just how elusive a property hopficity can be from a group-theoretic point of view.\footnote{Even understanding when direct products of hopfian groups is hopfian is rather complicated, see e.g.\ \cite{Corner1965,Hirshon1969}.}

Baumslag--Solitar groups have come to play a central role in the theory of one-relator groups since their introduction (indeed, recall that some examples of them appeared already in Magnus' 1930 article, \S\ref{Subsec:Freiheitssatz}). They act as counterexamples, as we have seen, to many natural conjectures about one-relator groups, and have a rich structural theory. We now give an overview of some of this structure to finish this section. We first remark that the groups $\BS(1,m)$, with $m \in \Z$, which above are distinguished as residually finite as well as hopfian, are in fact solvable. Indeed, the derived subgroup is isomorphic to $\Z[\frac{1}{m}]$, and the groups split as a semidirect product $\Z[\frac{1}{m}] \rtimes \Z$, with obvious conjugating action. Indeed, they are linear, arising as subgroups of $\GL_2(\R)$, via the representation
\begin{equation}\label{Eq:BS(1,m)-rep}
x \mapsto  \begin{pmatrix}
m^{\frac{1}{2}} & 0 \\ 0 & m^{-\frac{1}{2}}
\end{pmatrix} \quad \textnormal{and} \quad y \mapsto \begin{pmatrix}
1 & 1 \\ 0 & 1
\end{pmatrix}
\end{equation}
In fact, a Baumslag--Solitar group is linear if and only if it is residually finite. One direction is Mal'cev's famous theorem, and the other is obtained by constructing linear representations for all groups in Theorem~\ref{Thm:BS-residual-finiteness-characterisation}. Thus, in view of the linear representation \eqref{Eq:BS(1,m)-rep}, it suffices to show that $\BS(m,m)$ is linear for all $m \in \Z$. However, it is not hard to prove that $\BS(m,m)$ has a finite index subgroup isomorphic to a direct product $F_n \times \Z$, where $F_n$ is a free group of rank $n$. Since such direct products are obviously linear, so too is $\BS(m,m)$. 

The results by Baumslag \& Solitar were quickly picked up by Rapaport, who in 1964 used them to prove certain results on free products of groups. In a rather obscure article \cite{Rapaport1964b}, she proves among other results two which are relevant to our story: first, the minimal number of defining relations of a free product of two finitely presented groups may be less than the sum of that of the two factors (later rediscovered by 
Hog, Lustig \& Metzler \cite{Hog1985}). Second, she proved a result on the \textit{deficiency} of a one-relator group. Here, the deficiency of an $n$-generator $k$-relator presentation is defined to be $n-k$, and the deficiency of a finitely presented group is the maximum, taken over all presentations of the group, such deficiency. Thus the deficiency of an $n$-generator one-relator group is at least $n-1$, and Rapaport proved that it is no larger:

\begin{theorem}[{Rapaport, 1964 \cite[Corollary~2]{Rapaport1964b}}]
Let $G = \pres{}{A}{r=1}$ be a one-relator group. Then $G$ has deficiency $|A|-1$.
\end{theorem}

Baumslag--Solitar groups continue to be a rich area of research, and although we cannot possibly mention all results regarding them, we will mention a few which are naturally connected to hopficity and the above considerations. First, the endomorphisms of $\BS(m,n)$ with $\gcd(m,n) = 1$ and $m,n \neq \pm 1$ were described by Anshel \cite{Anshel1971, Anshel1972, Anshel1973} (see \cite{Anshel1976} for corrections). Using this, he proceeded to prove that in this case, $G = \BS(m,n)$ has a fully invariant subgroup $N$ such that $G / N \cong G$, a very strong form of non-hopficity. We mention also that Hirshon \cite{Hirshon1975} has given a full description of the intersection of the subgroups of finite index in Baumslag--Solitar groups. Finally, regarding the quotients of Baumslag--Solitar groups, Moldavanskii \& Sibyakova \cite{MS95} proved (with classical techniques) that the solvable Baumslag--Solitar groups $\BS(1,k)$ are \textit{profinitely rigid} among residually finite one-relator groups, i.e.\ for any residually finite one-relator group $G$, the set of finite quotients of $G$ is the same as that of $\BS(1,k)$ if and only if $G \cong \BS(1,k)$, cf.\ also \cite{Burrow1989}. We remark finally that a natural line of investigation beyond Baumslag--Solitar groups comes from considering relations of the type $x^{-1}w^mx = w^n$, where $w$ is some arbitrary word (possibly with more generators). Hopficity in this case has been investigated e.g.\ by Collins \cite{Collins1978b}, Rosenberger \cite{Rosenberger1980}, and Collins \& Levin \cite{Collins1983}.

\subsection{The Magnus--Moldavanskii hierarchy}\label{Subsec:MM-hierarchy}

One of Magnus' most prominent students was M.\ Greendlinger (b.\ 1932), who completed his PhD in 1960.\footnote{Martin Greendlinger was born in New York in 1932. In 1957, he met Elena Ivanovna in Moscow while visiting for a conference, and they married the next year. Because Elena was unable to get a US citizenship, Martin moved to the USSR in 1960, renounced his US citizenship, and became a Soviet citizen in 1961, eventually settling in Tula. Thus Magnus' school of combinatorial group theory also made its way to the USSR.  For more details, see \cite{Vankov2022}. Occasionally, Greendlinger's name is incorrectly transliterated (via Russian) in English as \textit{Grindlinger}.} The subject matter of his thesis \cite{Greendlinger1960c} was extending Dehn's algorithm from surface groups $\pi_1(\Sigma_g)$ with $g>1$ (see \S\ref{Subsec:Surface-and-knot-groups}) to broader classes. This class of groups, which lie somewhat outside the scope of this survey, is now called \textit{small cancellation groups}, as they satisfy the ``small cancellation condition'' $C'(1/6)$ (historically, groups satisfying this condition were also called \textit{sixth groups}). This condition, which traces back to the 1940s in work by Tartakovskii \cite{Tartakovskii1949} (and independently Britton \cite{Britton1956,Britton1957}), is intuitively speaking a condition on how much the relators of a given group presentation overlap. It can be used to guarantee a solution to the word problem and the conjugacy problem, using a well-known result known as \textit{Greendlinger's lemma}, which gives rise to a similar algorithm to Dehn's in the setting of $\pi_1(\Sigma_g)$. Greendlinger would expand on his work in a series of articles throughout the 1960s \cite{Greendlinger1960, Greendlinger1960b, Greendlinger1964}, which would be generalized by Lyndon \cite{Lyndon1966,Ly72} and his student Schupp \cite{Schupp1968, Schupp1970}. Today, a classical reference for small cancellation theory is \cite[Chapter~V]{Lyndon1977}. 

Greendlinger himself did not work directly with one-relator groups, even though his results were all directly inspired by this subject. After he moved from New York to the USSR (forgoing an NSF Grant for working with B.\ H.\ Neumann in Manchester), Greendlinger would end up as one of the core members of the algebraic school of mathematics in Ivanovo State University, and then in Tula, where he began working in 1967 \cite{Bezverkhnii2017}. He would bring many of the results and methods from Magnus to the USSR, and many of Greendlinger's students would end up proving significant results on one-relator groups. One of his students in Ivanovo was D.\ I.\ Moldavanskii, who became a graduate student of Greendlinger's in 1964, but ''almost immediately declared that small cancellation theory did not appeal to him'' \cite[p.~4]{Azarov2016}. Greendlinger, very familiar with, and having access to, Western literature on the subject, gave Moldavanskii a copy of Baumslag's 1964 survey \cite{Ba64} on one-relator groups, which ends with a question on the abelian subgroups of one-relator groups. Moldavanskii began working on this question, which would turn out to be very fortuitous for the development of one-relator groups. His approach to the subject would pass (more or less explicitly) via \textit{HNN-extensions}, and would lead to a simplification of the Magnus hierarchy. This, in turn, leads to simplifications of the classical results, and for deeper results to be proved, heralding in a new era of one-relator group theory.

Before presenting Moldavanskii's method, we give a brief summary of HNN-extensions. Such extensions were first introduced by G.\ Higman, B.\ H.\ Neumann, and H.\ Neumann \cite{Higman1949} in 1949, and carries their initials. The main idea of the extensions is that if $G$ is a group with two isomorphic subgroups $H_1, H_2 \leq G$, then such an isomorphism can be extended to an inner automorphism of a larger group $G^\ast$ into which $G$ embeds. Concretely, if $G = \pres{}{A}{R}$, and $\varphi \colon H_1 \to H_2$ is an isomorphism between two subgroups of $G$, then the \textit{HNN-extension} $G^\ast_\varphi$ of $G$ with \textit{associated subgroups} $H_1$ and $H_2$ is the group with the presentation 
\begin{equation}\label{Eq:HNN-def}
G^\ast_\varphi = \pres{}{A, t}{R, \: t h_1t^{-1} = \varphi(h_2) \: (\forall h_1 \in H_1)}
\end{equation}
where $t$ is some new symbol, called the \textit{stable letter}. A basic result proved in \cite{Higman1949} is that the natural map $G \to G^\ast_\varphi$ is injective, and one original motivation for introducing these extensions was to prove that \textit{every countable group $G$ is isomorphic to a subgroup of a two-generator group $H$}. However, 10 years after its introduction, the construction would find its key use by J.\ L.\ Britton (a student of B.\ H.\ Neumann's) in his proof of the undecidability of the word problem for groups. 

Let $G^\ast_\varphi$ be the group in \eqref{Eq:HNN-def}, and let $w$ be a word written as 
\begin{equation}\label{Eq:word-in-HNN-extension}
u_0 t^{\varepsilon_1} u_1 t^{\varepsilon_2} \cdots t^{\varepsilon_n} u_n
\end{equation}
where $u_i \in F_A$ and $\varepsilon_i = \pm 1$ for all $0 \leq i \leq n$. Suppose for simplicity of formulation that $n>0$. Then Britton realized that deciding that the word problem in $G^\ast_\varphi$ can be formulated in terms of $H_1, H_2$, and $G$ in the following way: 

\begin{lemma}[{Britton\footnote{HNN-extensions, and an essentially equivalent form of Britton's Lemma, also appear independently in Novikov's proof of the undecidability of the word problem in groups, where they play a key role \cite{Novikov1952, Novikov1955}.}  
\cite[Lemma~4]{Britton1957}}]
If $w = 1$ in $G_\varphi^\ast$, then $w$ contains a subword of the form $t h_1 t^{-1}$ or $t^{-1} h_2 t$, where $h_1 \in H_1$ and $h_2 \in H_2$. 
\label{Lem:BrittonsLemma}
\end{lemma}

Thus, via this ``spelling theorem'', one can immediately get a solution to the word problem as soon as one can solve membership in $H_1, H_2 \in G$. A stronger form of Britton's Lemma, called Theorem~A by Britton \cite{Britton1963}, has a direct proof via normal forms in amalgamated free products, see \cite{Miller1968}. A direct consequence of Britton's Lemma is, of course, the aforementioned result that the map $G \to G^\ast_\varphi$ is injective. But, more generally speaking, the lemma allows \textit{encodings} into groups of ``combinatorial sequences'', particularly those associated to the instructions of a (universal) Turing machine. In finitely presented \textit{semi}groups, such encodings are rather straightforward (cf.\ \S\ref{Subsec:Undecidability}), but for groups the first known encodings, by Boone \cite{Boone1957} and Novikov \cite{Novikov1955}, were messy and full of combinatorial intricacies. The simplicity of Britton's formulation permitted the explosion of new undecidability results in the 1950s (again see \S\ref{Subsec:Undecidability}). 

This simplicity of formulation would also turn out to be applicable to one-relator group theory. In 1967, Moldavanskii \cite{Mo67} would be the first to notice this, and used some of the properties of HNN-extensions in the study of the Magnus hierarchy. First, we remark that if $G = \pres{}{A}{r=1}$ is a one-relator group in which some generator $t$, actually appearing in $r$, has exponent sum zero in $r$, then it is not hard to see that $G$ is an HNN-extension of a one-relator group $G_1$, where the associated subgroups are Magnus subgroups of $G_1$ and $t$ is the stable letter (for a proof see \cite[Theorem~1]{MS73}). Indeed, continuing the example \eqref{Eq:MM-example} and \eqref{Eq:MM-example-cont}, the stable letter is $b$, the first associated subgroup is $\langle a_0, c_0 \rangle$, the second is $\langle a_1, c_1\rangle$, and the isomorphism is induced by $a_0 \mapsto a_1, c_0 \mapsto c_1$. Furthermore, by using Nielsen transformations and Euclidean division in the obvious way, we can always assume that at least one generator of $G$ has exponent sum zero in $r$ and appears in $r$. Let us call this the \textit{Euclidean trick}. For example, in the torus knot group $\pres{}{a,b}{a^2 = b^3}$, we can apply the Nielsen transformation $a \mapsto ba$ and fixing $b$, followed by the Nielsen transformation fixing $a$ and mapping $b \mapsto a^2b$, which results in the isomorphism 
\[
\mathcal{T}_{2,3} = \pres{}{a,b}{a^2 = b^3} \cong \pres{}{a,b}{ababa^{-2}b = 1}.
\]
In this new presentation, $a$ has exponent sum zero. Hence, our group is an HNN-extension with stable letter $a$, and applying the Magnus method (from \S\ref{Subsec:Freiheitssatz}) we see that it is an HNN-extension of $\pres{}{b_0, b_1, b_2}{b_1 b_2 b_0 = 1}$, a free group. 

The Euclidean trick is sufficient for proving certain results about one-relator groups, and Moldavanskii uses it e.g.\ to give a necessary and sufficient criterion for a one-relator group to be (f.g.\ free)-by-cyclic (see \S\ref{Subsec:Free-by-cyclic} below). However, in general the Euclidean trick will not lend itself easily to inductive arguments, since it can increase the length of the relator.\footnote{Moldavanskii seemingly missed this fact at first, since he only in a footnote \cite[Proof of Theorem~3]{Mo67} remarks that the Magnus trick is needed to make a proof of his correct.} Instead, we need another trick. Suppose (for simplicity) that $G_1 = \pres{}{a,b}{r(a,b) = 1}$, where the exponent sums of both $a$ and $b$ are non-zero, say $\alpha$ resp.\ $\beta$. Then we can map $G_1$ into the group 
\[
G_1' = \pres{}{y,x,b}{r(y,b) = 1, b = x^\alpha} \cong G_1 \ast_{b = x^\alpha} \Z
\]
where the exponent sum of $x$ is zero. When rewriting the relator via the Magnus process, the relator length does indeed decrease compared to that of $G_1$, allowing the use of induction. Let us call this the \textit{Magnus trick}. Using this trick, we can then prove the following:

\begin{theorem}[Moldavanskii, 1967 \cite{Mo67}]
Every one-relator group $G_1 = \pres{}{A_1}{r_1=1}$ can be embedded in a one-relator group $G'_1$ which is the HNN-extension of a one-relator group $G_2 = \pres{}{A_2}{r_2 = 1}$, in which the associated subgroups are Magnus subgroups and $|r_2| < |r_1|$.\label{Thm:Magnus-Moldavanskii}
\end{theorem}

The finite hierarchy of one-relator groups $G_1, G_2, \dots, G_n$ that this gives rise to is called the \textit{Magnus--Moldavanskii hierarchy}. Notice that the groups of this hierarchy are the same as the original Magnus hierarchy, with the only conceptual difference being the addition of the HNN-extensions. Furthermore, Theorem~\ref{Thm:Magnus-Moldavanskii} as stated does not appear in Moldavanskii's article, but is implicit in his reasoning (e.g.\ \cite[p. 1379]{Mo67}). The main result of Moldavanskii's article is the following:

\begin{theorem}[Moldavanskii, 1967 \cite{Mo67}]
Any abelian subgroup of a one-relator group is cyclic, free abelian of rank $2$, or locally cyclic.
\label{Thm:Moldavanskii}
\end{theorem}

This \textit{almost} resolved Baumslag's 1964 question of the abelian subgroups of one-relator groups. However, it leaves unanswered the question of which locally cyclic groups can appear as subgroups of a one-relator group, and in particular whether $(\Q, +)$ can appear as a subgroup of some one-relator group. This has a negative answer, and we shall see its full resolution, given by B.\ B.\ Newman in 1968, later in \S\ref{Subsec:Abelian-subgroups}. 

Moldavanskii continued to use this new hierarchy, and other tools, to prove more results about one-relator groups and combinatorial group theory in general. For example, in 1969 he gave an algorithm for deciding whether two subgroups of a free group are conjugate \cite{Moldavanskii1969b} (cf.\ also \cite{Greendlinger1970}), and an extension of the Nielsen method to free products of groups, resulting in a new proof of Grushko's theorem \cite{Moldavanskii1969c} (for more on extensions of Nielsen's methods in the context of one-relator groups, see \S\ref{Subsec:Torsion-isomorphism-problem}). The same year, he also gave a complete proof of the following claim made by Magnus \cite[p. 297]{Ma32} in 1932: 

\begin{theorem}[Magnus / Moldavanskii \cite{Moldavanskii1969}]
Except for $\Z$ and $\BS(1,k)$ for $k \in \Z$, every one-relator group contains a free group of rank $2$ \label{Thm:F2-subgroup-theorem}
\end{theorem}

Note that the class of exceptional groups in Theorem~\ref{Thm:F2-subgroup-theorem} has a number of other characterizations, and appears as a frequent class of ``easy'' one-relator groups. They are, among other things precisely the \textit{amenable} one-relator groups, and as proved by Gildenhuys \cite{Gildenhuys1979}, the groups in Theorem~\ref{Thm:F2-subgroup-theorem} also form precisely the class of solvable groups of cohomological dimension $\leq 2$. We shall see some extensions by other students of Greendlinger of Theorem~\ref{Thm:F2-subgroup-theorem} in \S\ref{Sec:7-Subgroups-of-OR-groups}. 

For a final note on the Magnus--Moldavanskii hierarchy, we note that a modern and simple proof of its validity was given by McCool \& Schupp \cite{MS73}, which continues to be an excellent resource for understanding the elements of one-relator group theory. Indeed, therein they use the well-developed machinery of HNN-extensions to give simple proofs of the \textit{Freiheitssatz} (\S\ref{Subsec:Freiheitssatz}), the decidability of the word problem (\S\ref{Subsec:Applications-of-Frei}), and the B.\ B.\ Newman Spelling Theorem (which we will encounter later in \S\ref{Subsec:BBNewman}).

\subsection{Free-by-cyclic one-relator groups}\label{Subsec:Free-by-cyclic}

One important class of one-relator groups are those which are \textit{free-by-cyclic}. Recall that a group $G$ is said to be free-by-cyclic if there exists a short exact sequence
\begin{equation}
1 \longrightarrow F \longrightarrow G \longrightarrow C \longrightarrow 1,
\end{equation}
where $F$ is a free group and $C$ is cyclic. We do not \textit{a priori} assume that $F$ is finitely generated, nor non-trivial; similarly, we do not require $C$ to be non-trivial. However, in essentially all cases we will encounter in this brief section $F$ will be finitely generated, and $C$ will be infinite cyclic. Note that in the case that $C = \Z$, $G$ splits as a semidirect product $G = F \rtimes_\varphi \Z$, where $\Z$ acts by $F$ via conjugation by some automorphism $\varphi \in \Aut(F)$. 

Let now $G$ be a one-relator group. In 1967, Moldavanskii \cite{Mo67} proved that if $\phi \colon G \to \Z$ is any surjective homomorphism, then it is possible to find a generating set $a_1', a_2', \dots, a_n'$ of $G$, with a defining relator $r' = 1$, such that $\ker(\phi)$ is precisely the normal closure of $a_2', \dots, a_n'$. In particular, $a_1'$ has exponent sum zero in $r'$. Using this result, he proved that \cite[Corollary~2]{Mo67} the commutator subgroup of $G$ is finitely generated if and only if the following holds: $n=2$, and the normal closure of $a_2'$ is a finitely generated free group. Combining his results, we obtain the following characterization of precisely when this happens: 

\begin{theorem}[{Moldavanskii, 1967 \cite{Mo67}}]
Let $G = \pres{}{a,b}{r=1}$ be a non-abelian\footnote{In the statement of \cite[Corollary~2]{Mo67}, this assumption is inadvertently omitted, but it is there in the proof, cf.\ the statement of \cite[Theorem~1]{Mo67}.} one-relator group with the exponent sum of $a$ in $r$ equal to zero. Let $b_m$ resp.\ $b_M$ be such that $m$ resp.\ $M$ is the smallest resp.\ largest index occurring when rewriting $r$ over $b_i = a^i b a^{-i}$. Then the following are equivalent: 
\begin{enumerate}
\item $G$ is (finitely generated free)-by-cyclic; 
\item The commutator subgroup $G'$ is finitely generated; 
\item $m<M$, and $a_m$ resp.\ $a_M$ occur exactly once in the rewritten form of $r$. 
\end{enumerate}\label{Thm:Moldavanskii-Brown-Criteriton}
\end{theorem}

\begin{remark}
The equivalence (1) $\iff$ (3) is often called ``Brown's criterion'' for a one-relator group to be (finitely generated free)-by-cyclic, as it was independently rediscovered in 1987 by Brown \cite[\S4]{Br87}. 
\end{remark}

Now, (finitely generated free)-by-cyclic groups are residually finite, as proved by Mal'cev \cite[Theorem~1, p. 50]{Malcev1958} in 1958, see also \cite{Baumslag1971b} (also extended by Wong \cite{Wong1987}). Thus one consequence of Moldavanskii's results is the following: 

\begin{corollary}[{\cite[Corollary~1]{Mo67}}]
If the commutator subgroup of a one-relator group $G$ is finitely generated, then $G$ is residually finite. 
\end{corollary}

Given a one-relator group $G$ and a surjective homomorphism $\phi \colon G \to \Z$, Moldavanskii also notes that his criterion gives an effective procedure for deciding whether or not $\ker(\phi)$ is finitely generated. In particular, one corollary of his methods is the following.

\begin{corollary}[{Moldavanskii, 1967 \cite[p. 1375]{Mo67}}]
There is an algorithm for deciding if a one-relator group is (finitely generated free)-by-cyclic. 
\end{corollary}

Residual finiteness is, in general, a significantly weaker property than being free-by-cyclic, even in the class of one-relator groups. It remains an open problem whether residual finiteness is decidable for one-relator groups.

\section{One-relator groups with torsion (1960--1980)}\label{Sec:6-OR-with-torsion}

\

{ \setlength{\epigraphwidth}{0.65\textwidth}
\epigraph{\textit{It seems to be commonly believed that the presence of elements of finite order in a group with a single defining relation is a complicating rather than a simplifying factor. This note is in support of the opposite point of view.}}{---G.\ Baumslag, 1967 \cite{Ba67}}

\epigraph{[...] \textit{apart from the complications introduced by having elements of finite order, one-relator groups with torsion behave much like free groups.}}{---S.\ J.\ Pride, 1977 \cite{Pride1977}}

}

\noindent This far into the story, none of the theorems about one-relator groups have singled out one-relator groups with torsion as forming a particularly simple class compared to their torsion-free counterparts. Indeed, sometimes the opposite seems true: for example, Lyndon's Identity Theorem (Theorem~\ref{Thm:Lyndon-Identity}, \S\ref{Subsec:LyndonIT}) is significantly neater to state in the torsion-free case. Nevertheless, in the 1960s and 1970s a number of remarkable results would be proved for one-relator groups with torsion which, in full accordance with the above quotations. The crowning achievement of such theorems was the \textit{B.\ B.\ Newman Spelling Theorem}, which in modern terminology states that \textit{one-relator groups with torsion are hyperbolic}. Before stating this result, we state a number of  important conjectures made by G.\ Baumslag around this time:

\begin{conjecture}[Baumslag's Conjectures]
Let $G = \pres{}{A}{r^n = 1}$, where $n>1$, be a one-relator group with torsion. Then the following hold: 
\begin{enumerate}
\item $G$ is virtually free-by-cyclic \cite[Problem~6]{Ba86}
\item $G$ is coherent \cite[Problem~1]{Ba86}
\item $G$ is residually finite \cite[Conjecture~A]{Ba67},
\item $G$ is hopfian \cite[Conjecture~A$\flat$]{Ba67},
\end{enumerate}\label{Conj:BaumslagConjectures}
\end{conjecture}

We make some remarks on these conjectures. First, even for arbitrary groups $G$, (1) is the strongest of these properties, and implies all others: (1) $\implies$ (3), and (3) $\implies$ (4) are both classical results. The implication (1) $\implies$ (2) is due to Feighn \& Handel \cite{FH99} in 1999. Since $\BS(2,3)$ is non-hopfian, it follows that the only one of the four conjectures that can hold in the torsion-free case is (2). In fact, that is precisely the conjecture that Baumslag made in \cite[Problem~1]{Ba86}: \textit{all} one-relator groups are coherent. We note that already in 1974, Baumslag \cite{Ba71} had stated that it remained an open problem whether all one-relator groups are coherent. The above conjectures would be driving forces behind many of the developments of one-relator group theory in the 1970s until the present day. Today, we know that they all have affirmative answers, including (2) in the general case. In this section, we will present some of the partial progress made on these conjectures before their full resolution, and present some general theorems indicating why one-relator groups with torsion are significantly simpler than their torsion-free counterparts. 

\subsection{The B. B. Newman Spelling Theorem}\label{Subsec:BBNewman}

\

{ \setlength{\epigraphwidth}{0.57\textwidth}
\epigraph{\textit{I don't believe it! I don't believe it!}}{---W.\ Magnus in 1968, upon seeing B.\ B.\ Newman present his Spelling Theorem \cite[p. 119]{NybergBrodda2021}}
}

{ \setlength{\epigraphwidth}{0.5\textwidth}
\epigraph{\textit{He could not believe that an unheard-of mathematician, from some unknown university in outback Australia, could have come up with these results.}}{---B.\ B.\ Newman, remarking on Magnus' reaction above \cite[p. 120]{NybergBrodda2021}}
}

\noindent Spelling theorems, being theorems which provide information on the spelling of words equal to $1$ in a group, date back to the very beginning of one-relator group theory, the prime example being Dehn's Spelling Theorem (Theorem~\ref{Thm:Dehn-spelling-theorem}) for the fundamental groups of $2$-surfaces. Britton's Lemma (see Lemma~\ref{Lem:BrittonsLemma}), with its importance for the Magnus--Moldavanskii, is another. And in 1968, B.\ B.\ Newman, a PhD student of Baumslag's, provided perhaps the most important one. We will be brief on the biographical details of this story, since this is expanded on at length by the author in \cite{NybergBrodda2021}. We mention only that the starting point for Newman's work on the subject was Baumslag giving him a draft of the chapter on one-relator groups from the now famous book by Magnus, Karrass~\&~Solitar \cite{Magnus1966}. After some work, he was able to prove the following remarkable theorem:

\begin{theorem}[The B.\ B.\ Newman Spelling Theorem, 1968]
Let $G = \pres{}{A}{r^n = 1}$ be a one-relator group with torsion, with $n>1$, and $r$ cyclically reduced and not a proper power. Let $w \in F_A$ be a non-trivial word such that $w = 1$ in $G$. Then $w$ contains a subword $u$ such that either $u$ or $u^{-1}$ is a subword of $r^n$, and such that the length of $u$ is strictly more than $\frac{n-1}{n}$ times the length of $r^n$. 
\label{Thm:BBNewmanSpellingTheorem}
\end{theorem}

The proof of this theorem goes, as expected, by the Magnus hierarchy. However, the 1968 article in which the theorem was first announced was only a bulletin article \cite{Ne68}, and no proofs appeared therein. The proof only appeared in Newman's PhD thesis \cite{Newman1968}, which was until recently completely inaccessible: indeed, the author's (eventually successful) hunt for this thesis is the subject of \cite{NybergBrodda2021}. A simpler proof of the Spelling Theorem was given by McCool \& Schupp \cite{MS73} in 1973, using the full power of HNN-extensions (see \S\ref{Subsec:MM-hierarchy}). 

We now give some immediate consequences of the Spelling Theorem, all of which appear in B.\ B.\ Newman's thesis. First, and most obviously, it gives a linear-time solution to the word problem in all one-relator groups with torsion. This is significantly better than the time-complexity of Magnus' solution to the word problem in general one-relator groups (see \S\ref{Subsec:Wordproblem-Magnus}). Next, Newman realized that his Spelling Theorem gave an easy solution also to the conjugacy problem: 

\begin{theorem}[{B.\ B.\ Newman, 1968 \cite[Theorem~3.2.3]{Newman1968}}]
The conjugacy problem is decidable in any one-relator group with torsion. 
\end{theorem}

Newman also solved (see \cite[Theorem~3.3.1]{Newman1968}) the \textit{power problem} in one-relator groups with torsion, i.e.\ the problem of deciding whether a given element in a group is a proper power or not. This problem, which we will not discuss further, appeared in a similar setting of small cancellation groups also in work by Lipschutz \cite{Lipschutz1965, Lipschutz1968} and Reinhart \cite{Reinhart1962}. The next corollary of the Spelling Theorem was a stronger form of the \textit{Freiheitssatz}. 

\begin{theorem}[{B.\ B.\ Newman, 1968 \cite[Corollary~2.1.6]{Newman1968}}]
Let $G = \pres{}{a_1, a_2, \dots, a_k}{r^n = 1}$ with $n>1$ be a one-relator group with torsion, where $r$ is cyclically reduced involving the letters $a_1, a_2 \in A$ non-trivially. Let $\beta$ be any integer which does not divide the exponent sum of $a_1$ in $r^n$. Then $\{ a_1^\beta, a_2, \dots, a_k\}$ freely generates a free subgroup of $G$.\label{Thm:NewmanStrengthenedFrei}
\end{theorem}

Thus, for example, in $\pres{}{a,b}{(a^2b^3)^2 = 1}$, the subgroup generated by $\{ a^3, b \}$ is free on that basis. Note that the ordinary \textit{Freiheitssatz} only gives information about cyclic subgroups for this example. Thus Theorem~\ref{Thm:NewmanStrengthenedFrei} gives an abundance of free subgroups of one-relator groups with torsion. The same year, Ree \& Mendelsohn \cite{Ree1968} had proved a weaker version of this result, for the two-generator case, by using representations for some one-relator groups with torsion in $\SL(2,\C)$. Magnus \cite{Magnus1975}, however, has shown that their technique cannot yield the full strength of Theorem~\ref{Thm:NewmanStrengthenedFrei}. As a final application of the B.\ B.\ Newman Spelling Theorem, we mention the following subgroup result.

\begin{theorem}[{Newman, 1968 \cite[Corollary~2.3.4]{Newman1968}}]
The centralizer of every non-trivial element of a one-relator group with torsion is cyclic. \label{Thm:Centralizer-cyclic-in-torsion}
\end{theorem}

This theorem was also proved and extended by Karrass \& Solitar \cite[p.~956]{Karrass1971}, who elaborated on the properties of free products with a \textit{malnormal} amalgamated subgroup. B.\ B.\ Newman had made central use of some of the properties of such amalgams, particularly using the malnormality of Magnus subgroups in one-relator groups to prove his results. This property of Magnus subgroups remains also a key theme in the modern setting. We elaborate on this theme in \S\ref{Subsec:Magnus-subgroups}. 

These were only some the initial applications of the B.\ B.\ Newman Spelling Theorem, all of which appeared already in his thesis. Newman also used it to study abelian and solvable subgroups of one-relator groups with torsion; we detail this further in \S\ref{Subsec:Abelian-subgroups}. Furthermore, it would come to play a particularly prominent role in the subsequent work by S.\ J.\ Pride on one-relator groups with torsion, which we will detail more in \S\ref{Subsec:Torsion-isomorphism-problem}. We mention one of his results here, as its proof heavily relies on the Spelling Theorem. B.\ B.\ Newman and Pride once thought that \textit{every} finitely generated subgroup of a one-relator group would either be free or have torsion \cite[p. 483]{Pride1977}. That this is not true can be seen either from Theorem~\ref{Thm:Fischer-Karrass-Solitar} together with the fact that not every one-relator group with torsion is free, or else from the example \eqref{Eq:NewmanPride-example} below. Nevertheless, much can be said of the \textit{two}-generator subgroups of one-relator groups with torsion: 

\begin{theorem}[{Pride, 1977 \cite{Pride1977}}]
Every two-generator subgroup of a one-relator group with torsion is either a free product of cyclic groups, or a one-relator group with torsion. 
\end{theorem}

Further results on subgroups of one-relator groups is the subject of \S\ref{Sec:7-Subgroups-of-OR-groups}. We now conclude this present section with some brief remarks on the Spelling Theorem. First, a strengthening of the theorem was obtained by Gurevich \cite{Gurevich1972, Gurevich1973}, see also Schupp \cite{Schupp1976} for a complete proof and further strengthening. As noted by Pride \cite[p. 57]{Pride1974}, a proof of this stronger result can be obtained by the same method as the proof of Newman's result given by McCool \& Schupp \cite{MS73}. Furthermore, one of the consequences, indeed one of the reformulations, of the Spelling Theorem is that one-relator groups with torsion admit \textit{Dehn presentations}. This is known to be equivalent to Gromov hyperbolicity, and hence \textit{all one-relator groups with torsion are hyperbolic}. This places many of the above theorems in their proper context: for example, in any hyperbolic group the centralizer is virtually cyclic, yielding in particular Theorem~\ref{Thm:Centralizer-cyclic-in-torsion} for one-relator groups with torsion. However, the proof of Theorem~\ref{Thm:Centralizer-cyclic-in-torsion} predates the definition of hyperbolic groups by two decades.

\subsection{Residual and virtual properties}\label{Subsec:RF-of-Torsion}

As mentioned before, today we know that all of Baumslag's conjectures (Conjecture~\ref{Conj:BaumslagConjectures}) are true. However, there had been some partial progress on these before this resolution, including for residual finiteness. We now present some of these results. We remark first that a natural conjecture might be that all one-relator groups with torsion are residually (torsion-free) nilpotent. But this is not so, as indicated by the following example due to B.\ B.\ Newman \cite[p. 72]{Newman1968}: let 
\begin{equation}\label{Eq:NewmanPride-example}
G = \pres{}{a,b}{(a^{-1}b^{-1}ab^2)^2 = 1}
\end{equation}
Then the index $2$ subgroup of $G$ generated by $\langle a, b^2, bab^{-1} \rangle$ admits the presentation $\pres{}{x,y,z}{[y, z][y,x] = y}$, and clearly $y$ lies in the intersection of every term in the lower central series. Hence $G$ is not residually torsion-free nilpotent. One of the first theorems on residual properties of one-relator groups with torsion was also due to B.\ B.\ Newman, in his thesis:

\begin{theorem}[{\cite[Corollary~2.1.7]{Newman1968}}]
Let $G = \pres{}{a_1, a_2, \dots, a_k}{r^n = 1}$ where $n > 1$. Then $G$ is residually a two-generator one-relator group with torsion. 
\end{theorem}

The proof is not difficult, but uses the Spelling Theorem. A few years later, Fischer, Karrass \& Solitar \cite{FKS72} wrote a short but influential note giving structural information about one-relator groups with torsion. Two of their results are as follows: 

\begin{theorem}[Fischer, Karrass \& Solitar, 1972 \cite{FKS72}]
Let $G = \pres{}{A}{r^n = 1}$ with $n > 1$ be a one-relator group with torsion. Then: (1) $G$ is virtually torsion-free; and (2) $G$ is virtually free if and only if $r$ is a primitive word in $F_A$.\label{Thm:Fischer-Karrass-Solitar}
\end{theorem}

The authors also prove that the elements of finite order in $G$ generate a subgroup that is the free product of conjugates of the cyclic group generated by $r$, a result that appears implicitly already in work of Cohen \& Lyndon \cite{CL63}. Part (2) of Theorem~\ref{Thm:Fischer-Karrass-Solitar} was reproved, among other results, by Karrass, Solitar \& Pietrowski \cite{Karrass1973} the next year. Part (1) of Theorem~\ref{Thm:Fischer-Karrass-Solitar} was expanded on by Fischer \cite{Fischer1977}, who proved that (with the same notation as in Theorem~\ref{Thm:Fischer-Karrass-Solitar}) if $H$ is a subgroup of $G$ with $[G : H] = k < \infty$ and $H$ is torsion-free, then $n$ divides $k$. Furthermore, if some generator in $r$ has exponent sum $s \neq 0$, then $G$ has a torsion-free $|s|$-relator subgroup of finite index; if further $\gcd(s,n) = 1$ then $G$ has a one-relator subgroup $H$ of finite index. As a corollary, he deduced that if $G$ contains a torsion-free subgroup $H$ with $[G : H] = kn$, then $H$ is a $k$-relator group. Finally, while we cannot pass into the details of \textit{ends} of groups here, instead referring the reader to \cite{St68}, we note that another important result proved in \cite{FKS72} is that any freely indecomposable one-relator group with torsion has at most one end. Thus, in this case, it follows from Theorem~\ref{Thm:Fischer-Karrass-Solitar} that if the defining word is non-trivial and not a power of a primitive word, then the group has exactly one end. 

We now list some families of one-relator groups with torsion which were known to be residually finite before the full resolution of Baumslag's conjecture:

\begin{enumerate}
\item (Fischer, 1977 \cite{Fischer1977}) The groups $\pres{}{A \cup B}{(u v)^n = 1}$, $n > 1$, where $u \in F_A$ and $v \in F_B$ are non-trivial elements, and $A \cap B = \varnothing$.
\item (Fischer, 1977 \cite{Fischer1977}) The groups $\pres{}{a,b}{(ab^k a^{-1}y^r)^n = 1}$, where $n >1$, $n$ does not divide $r$ or $k$, and $\gcd(r + k, n) = 1$.
\item (B.\ Baumslag \& Tretkoff, 1978 \cite{BBaumslag1978}) The groups $\pres{}{a,b}{(a^{-1}b^l a b^l)^n = 1}$ where $n > 1$ and $l \in \Z$.
\item (B.\ Baumslag \& Levin, 1979 \cite{BBaumslag1979}) The groups $\pres{}{A,x}{(xux^{-1}v)^n = 1}$ where $n >1$ and $u, v \in F_A$ and either (i) $u, v$ are words over disjoint sets of generators; or (ii) $u$ and $v$ are non-trivial powers of a single element, e.g.\ $\pres{}{x,y}{(xy^kx^{-1}y^l)^n = 1}$ with $k,l \in \Z$.
\item (Allenby \& Tang, 1981 \cite{Allenby1981b,Allenby1981}) Many different sporadic classes of one-relator groups with torsion; a typical example is \cite[Theorem~4.4]{Allenby1981b}, which states that $\pres{}{a,b}{([a^\alpha, b^\beta][a^{-\gamma}, b^{-\delta}])^t = 1}$, $n >1$ and $\alpha, \beta, \gamma, \delta \in \N$.
\end{enumerate}

After these sporadic results, Egorov \cite{Egorov1981} (a student of Moldavanskii) proved the following remarkable result on positive one-relator groups with torsion:

\begin{theorem}[Egorov, 1981 \cite{Egorov1981}]
Let $G = \pres{}{A}{r^n = 1}$ with $n > 1$ and $r$ a positive word in $F_A$, i.e.\ including no inverse of any generator. Then $G$ is residually finite.\label{Thm:Egorov-positive-ORT-is-RF}
\end{theorem}

For extensions of this result, we refer the reader to Wise \cite{Wise2001}; for comparisons with positive torsion-free one-relator groups, see \S\ref{Subsec:Quotients+moreresidualproperties}. Finally, we mention a connection with another ``residual-like'' property. Tang \cite{Tang1982} proved that the groups $\pres{}{a,b}{(a^m b^n)^k = 1}$ where $k > 1$ are all \textit{conjugacy separable}, a stronger property than residual finiteness. Recall that a group $G$ is conjugacy separable if for any $g, h \in G$ which are not conjugate, there exists some finite quotient $Q$ of $G$ such that the images of $g$ and $h$ are not conjugate in $Q$. Thus conjugacy separability is to the conjugacy problem what residual finiteness is to the word problem. Further results on conjugacy separability in some classes of one-relator groups with torsion were obtained by Allenby \& Tang \cite{Allenby1986} a few years later, see also \cite{Kim1995}. Today, like residual finiteness, we know that one-relator groups with torsion are conjugacy separable, a result due to Minasyan \& Zalesski \cite{Minasyan2013}.

\subsection{The isomorphism problem for one-relator groups with torsion}\label{Subsec:Torsion-isomorphism-problem}

By deep results of Dahmani \& Guirardel \cite{DG11}, we know today that the isomorphism problem for all hyperbolic groups is decidable. As mentioned in \S\ref{Subsec:BBNewman} (Theorem~\ref{Thm:BBNewmanSpellingTheorem}), all one-relator groups with torsion are hyperbolic. Hence, the isomorphism problem is decidable for one-relator groups with torsion. However, this problem was by no means a mystery before this point, and it was known to be decidable in a large number of cases. Thus, we will now go through some of this earlier literature on the subject, focussing particularly on the remarkable results by Pride. 

First, in line with asking about isomorphisms of one-relator groups, it is natural to ask whether there is a converse to the isomorphism result from Theorem~\ref{Thm:If-Rn=Sn-thenR=S}. Such a converse would show that the isomorphism problem for one-relator groups would be somewhat more tractable, since it would show that it would be sufficient to solve it in the torsion case. However, B.\ B.\ Newman suggested to Pride that a counterexample to this converse might exist \cite[p. 71]{Pride1974}, and Pride was able to construct such a counterexample: 

\begin{theorem}[Pride, 1974 {\cite[Theorem~4.7]{Pride1974}}]
Let $n \geq 1$ and consider the groups
\[
B_n = \pres{}{a,t}{(a^{-3}t^{-1}a^2t)^n = 1} \quad \text{and} \quad C_n = \pres{}{a,t}{(a^{-1}[t,a]^2)^n = 1}.
\]
Then $B_n \cong C_n$ if and only if $n=1$. 
\end{theorem}

Thus, this theorem provided a counterexample to the converse of Theorem~\ref{Thm:If-Rn=Sn-thenR=S}. Indeed, it also provided another counterexample to an old conjecture by Magnus (Conjecture~\ref{Conj:Magnus-conj}), since it proves that there is an isomorphism
\[
B_1 = \pres{}{a,t}{t^{-1} a^2 t = a^3} \cong \pres{}{a,t}{[t,a]^2 = a} = C_1
\]
but it is readily checked that $t^{-1}a^2ta^{-3}$ and $[t,a]^2a^{-1}$ are not related by a free group automorphism (this example was also obtained by Brunner \cite{Br74}, cf.\ \cite{Br76}). 

The above negative result is, in hindsight, not surprising: the isomorphism problem for one-relator groups with torsion turns out to be significantly easier than in the torsion-free case. To state some of the positive results, we recall the notion of a \textit{$T$-system} (Ger.\ \textit{Transitivitätssystem}, i.e.\ transitivity system) of a group, following \cite{Pr75}. Let $G$ be an $n$-generated group, and let 
\[
\mathfrak{g} = (g_1, \dots, g_n), \quad \textnormal{and} \quad \mathfrak{g}' = (g'_1, \dots, g'_n)
\]
be two generating tuples for $G$. We say that $\mathfrak{g}$ and $\mathfrak{g}'$ are \textit{Nielsen equivalent} if there is some automorphism 
\[
x_i \mapsto Y_i(x_1, \dots, x_n), \quad i = 1, 2, \dots, n
\] of the free group 
$F_n$ on $x_1, \dots, x_n$ such that $g_i' = Y_i(g_1, \dots, g_n)$. Thus, for example, the tuples $(g_1, g_2)$ and $(g_2g_1^{-1}, g_1)$ are always Nielsen equivalent. We say that $\mathfrak{g}$ and $\mathfrak{g}'$ lie in the same \textit{$T$-system} if there is some automorphism $\varrho \in \Aut(G)$ such that $\mathfrak{g}'$ is Nielsen equivalent to $\varrho(\mathfrak{g})$, where
\[
\varrho(\mathfrak{g}) = (\varrho(g_1), \dots, \varrho(g_n)).
\]
This equivalence relation was defined by B.\ H.\ Neumann \& H.\ Neumann \cite{Neumann1951} in 1951. For our purposes, the importance of $T$-systems comes via Magnus' Conjugacy Theorem (Theorem~\ref{Thm:Conjugacy-Theorem}). First, continuing the above notation, let
\[
\pres{}{x_1, \dots, x_n}{r_1 = 1, \dots, r_m = 1}
\]
is a presentation of $G$ with respect to the generating tuple $\mathfrak{g}$ (i.e.\ the kernel of the map $x_i \mapsto g_i$ coincides with the normal closure of the $r_j$). Then $\mathfrak{g}'$ lies in the same $T$-system as $\mathfrak{g}$ if and only if there is some $\phi \in \Aut(F_n)$ such that 
\[
\pres{}{x_1, \dots, x_n}{\phi(r_1) = 1, \dots, \phi(r_m) = 1}
\]
is a presentation of $G$ with respect to the generating tuple $\mathfrak{g}'$. In general, $G$ may have many different $m$-relator presentations with respect to $\mathfrak{g}'$, and one of the difficulties (indeed undecidability) of the isomorphism problem in full generality arises from this issue. However, when $m=1$, i.e.\ the one-relator case, Magnus' Conjugacy Theorem (Theorem~\ref{Thm:Conjugacy-Theorem}) tells us that two generating sets lying in the same $T$-system gives very strict control over the defining relator:

\begin{proposition}
Two one-relator presentations $\pres{}{A}{r=1}$ and $\pres{}{A}{s=1}$  of a group $G$ are associated with elements from the same $T$-system in $G$ if and only if $s$ is equal in $F_A$ to the image of $r^{\pm 1}$ under some automorphism of $F_A$. 
\label{Prop:MagnusConjugacyTSystems}
\end{proposition}

In particular, if a one-relator group $G$ only has one $T$-system of generating tuples, then the isomorphism problem is decidable for the class of one-relator groups with respect to that one-relator group; in other words, it is decidable whether a given one-relator group is isomorphic to $G$ or not. Brunner's example of $\BS(2,3)$ above shows that in general a one-relator group may have infinitely many $T$-systems of generating tuples. In the case of one-relator groups with torsion, however, more can be said. 

First, for example, consider the class of groups $S_n = \pres{}{a,b}{[a,b]^n = 1}$, where $n > 1$. These groups $S_n$ are virtually surface groups $\pi_1(\Sigma_g)$, indeed it is not hard to show that $\ker(S_n \twoheadrightarrow D_{2n})$ is isomorphic to $\pi_1(\Sigma_{2n-1})$ (see e.g.\ \cite{Baumslag2008}). The fact that $S_n$ is Fuchsian was observed also by Rosenberger \cite{Rosenberger1973}, who studied these groups in 1973. He furthermore proved that the only generating sets $\{ u, v\}$ of $S_n$ are those of the form where $[u, v]$ is conjugate to $[a, b]^{\varepsilon}$, where $\varepsilon = \pm 1$ (see also Purzitsky \& Rosenberger \cite{Purzitsky1972, Purzitsky1973} and Pride \cite{Pride1973}, cf.\ \cite[p.~72]{Pride1974}). By Nielsen's result (Proposition~\ref{Prop:Dehn-Nielsen-commutator}), this implies that $\{ u, v\}$ is a basis of the underlying free group, so in particular $S_n$ has only one $T$-system of generators. In 1972, Rosenberger \cite{Rosenberger1972} also considered the ``higher genus'' $S_{n,g}$ variants of the groups $S_n$, defined by 
\begin{equation}
\pres{}{a_1, b_1, \dots, a_g, b_g}{([a_1, b_1] [a_2, b_2] \cdots [a_g, b_g])^n = 1}
\end{equation}
with $n, g \geq 1$, and proved that the isomorphism problem for one-relator groups relative to this class is decidable. For other related examples, and a slight generalization, see \cite{Rosenberger1976, Rosenberger1976b}. 

However, a few years later, in 1977 Pride \cite{Pr77} was able to clarify the situation of generating tuples of all two-generator one-relator groups with torsion. He proved that if $G = \pres{}{A}{r^n = 1}$ is a one-relator group with torsion, with $n > 1$ and $r$ not a proper power, then $G$ has only one $T$-system of generating tuples (when $n=2$) or $\frac{1}{2}\phi(n)$ such systems when $n>2$. Here $\phi$ denotes Euler's totient function. He moreover showed that these $T$-systems can be effectively found, and hence deduced the following remarkable theorem: 

\begin{theorem}[Pride, 1977 \cite{Pr77}]
The isomorphism problem is decidable for two-generator one-relator groups with torsion. \label{Thm:Pride-isomorphism-theorem}
\end{theorem}

The method of proof involves much use of the B.\ B.\ Newman Spelling Theorem and, of course, induction on the length of the Magnus--Moldavanskii hierarchy. Beyond the isomorphism problem, other structural information can also be deduced from a one-relator group having only one $T$-system of generating tuples. 

\begin{theorem}[Pride, 1977 \cite{Pr77}]
Let $G$ be a two-generator one-relator group with torsion. Then:
\begin{enumerate}
\item $G$ is a hopfian group.
\item The automorphism group of $G$ is finitely generated;
\end{enumerate}\label{Thm:Pride-2-gen-hopfian}
\end{theorem}

We remark that Rosenberger was able to extend these methods to a slightly more general setting; we do not recount all such extensions here, but mention a typical example. Let $X_1 \cap X_2 \neq \varnothing$, let $w_i \in F_i$ for $i \in \{ 1, 2\}$ be two words. Then Rosenberger \cite{Rosenberger1977} gave conditions under which the group $\pres{}{X_1, X_2}{(w_1 w_2)^n = 1}$ with $n>1$ only has one $T$-system of generating tuples, and hence the isomorphism problem for one-relator groups is decidable with respect to this class.

\section{Subgroups of one-relator groups}\label{Sec:7-Subgroups-of-OR-groups}

\noindent Subgroups of zero-relator groups, i.e.\ free groups, are themselves zero-relator groups; this is the Nielsen--Schreier Theorem, as presented in \S\ref{Subsec:Nielsen-Schreier}. Similarly, many subgroups of one-relator groups are free, as the \textit{Freiheitssatz} tells us. Thus, it seems natural to ask: what are the subgroups of one-relator groups? In particular, is every subgroup of a one-relator group a one-relator group? The answer to this latter question is emphatically \textit{no}. The subgroup structure of one-relator group carries a great richness, and although we today know that all one-relator groups are coherent, much was known already before this was proved. In this section, we will present several different flavours of such results about subgroups of one-relator groups.

We first remark that there is, in general, an abundance of subgroups in one-relator groups. Baumslag \& Miller \cite{Baumslag1986} prove that certain one-relator groups contain many non-isomorphic subgroups. Indeed, they prove, using a simple argument that as soon as a group $G$ contains a countably infinite number of non-isomorphic, freely indecomposable subgroups, then $G \ast \Z$ contains continuously many non-isomorphic subgroups. For example, one can take the solvable Baumslag--Solitar group
\[
\BS(1,2) = \pres{}{a,b}{b^{-1}ab = a^2}
\]
in which the subgroup $E_i = \langle a, b^{2^i} \rangle$ is solvable, and $E_i \not\cong E_j$ if $i \neq j$ since $H_1(E_i, \Z)$ clearly grows in $i$. Thus $\BS(1,2) \ast \Z$ contains continuously many non-isomorphic subgroups. Note also that every one-relator group can be embedded into a $2$-generator one-relator group (see e.g.\ \cite[p. 259]{Magnus1966}), and hence there are $2$-generator one-relator groups containing continuously many non-isomorphic subgroups, e.g.\ the Baumslag--Gersten group
\begin{equation}\label{Eq:Extension-of-BS(1,2)}
\operatorname{BG} = \pres{}{a,t}{ta^{-1}t^{-1}atat^{-1} = a^2}
\end{equation}
which is clearly an HNN-extension of $\BS(1,2)$. By the above results, it contains continuously many non-isomorphic subgroups. Thus it also contains countable subgroups which are not recursively presentable. Nevertheless, in general, much can be said about particular classes of subgroups of one-relator groups, particularly their \textit{finitely generated} subgroups. In this section, we will see many examples of this. 

We caution the reader that because of the often simultaneous and overlapping developments at this time, unlike previous sections, the subsections of this section (and the next) can be read more or less independently of one another, and will frequently pass from one subject to another with only some connective tissue.

\subsection{Magnus subgroups and malnormality}\label{Subsec:Magnus-subgroups}

Let $G = \pres{}{A}{r=1}$, with $r$ cyclically reduced. Then recall (from \S\ref{Subsec:Freiheitssatz}) that a \textit{Magnus subgroup} of $G$ is any subgroup generated by a proper subset of $A$. The \textit{Freiheitssatz}, being the first result about one-relator groups, was a result about these subgroups: they are free. More can be said about Magnus subgroups, however. One particular classical theme which we now expand on is that of \textit{malnormality}. For an arbitrary group $G$, a non-trivial subgroup $H \leq G$ is said to be \textit{malnormal} if $gHg^{-1} \cap H = \{ 1 \}$ for all $g \in G \setminus H$. The terminology \textit{malnormal} for this property was introduced by B.\ Baumslag in his 1965 PhD thesis \cite{BBaumslag1965}, see also \cite[p.~601]{BBaumslag1968}. The structure of a \textit{finite} group with a malnormal subgroup is rather restricted. A finite group $G$ has a (non-trivial) malnormal subgroup $H$ if and only if it is a \textit{Frobenius group}, i.e.\ it acts transitively on some finite set where no non-trivial element fixes more than one point, and some non-trivial element fixes some point. Then $H$ being malnormal is a restatement of $H$ being a \textit{Frobenius complement}, and $G$ splits as a semidirect product $G = K \rtimes H$ where $K$ is the (normal) subgroup consisting of all elements not in any conjugate of $H$ \cite{Frobenius1901}. 

In general, in infinite groups malnormal subgroups are more complicated. Indeed, even the problem of deciding whether a subgroup of a \textit{free} group is malnormal or not is non-trivial, albeit decidable \cite{Baumslag1999}. Early work on malnormal subgroups, other than that by B.\ Baumslag \cite{BBaumslag1965, BBaumslag1968}, are four other PhD theses, all of which were by students of G.\ Baumslag: the first is by Driscoll \cite{Driscoll1967}, who used it in her study of the conjugacy problem is some amalgamated free products; the second was Whittemore \cite{Whittemore1967}, who used malnormal subgroups in her study of the Frattini subgroups of amalgamated free products; and third by Taylor \cite{Taylor1963,Lewin1967}\footnote{Here, a possible source of confusion regarding surnames arises. Tekla Taylor and Jacques Lewin were the first two PhD students of G.\ Baumslag's, and they married in 1965. Tekla took her husband's surname, and the two would publish together frequently, including the article containing their famous Lewin--Lewin construction for the group ring of a torsion-free one-relator group \cite{LL78}. Nevertheless, possibly because of delays in publishing, Tekla's surname still appears as \textit{Taylor} in 1968 in her joint article with Baumslag on the centre of one-relator groups \cite{BT68}, see \S\ref{Subsec:Centre}.} in her study of groups with unique roots. The fourth, and for us the most important, was by B.\ B.\ Newman, who proved the following remarkable result in his thesis:

\begin{theorem}[{B.\ B.\ Newman, 1968 \cite[Lemma~2.3.1(I)]{Newman1968}}]
Any Magnus subgroup of a one-relator group with torsion is malnormal.\label{Thm:Newman-malnormal-thm}
\end{theorem}

This result, and a strengthening of it using the notion of \textit{strong malnormality}, was used by Newman in his solution to the conjugacy problem in all one-relator groups with torsion. Newman's result would also, many decades later, come to play a key role in Wise's recent proof of the residual finiteness of one-relator groups with torsion \cite{Wise2009}. We remark that Moldavanskii \cite{Moldavanskii1968} further proved that the intersection of Magnus subgroups generated by disjoint subsets of the generators is always trivial in a one-relator group with torsion. 

Outside the torsion case, Bagherzadeh \cite{Bagherzadeh1976, Bagherzadeh1976b} proved in 1976 that for any one-relator group $G$, if $H < G$ is a Magnus subgroup, then $H \cap g^{-1} H g$ is cyclic for $g \not\in H$. Thus any Magnus subgroup of a one-relator group is what Bagherzadeh calls \textit{cyclonormal} (a term which later appears also in work by Wise \cite{Wise2001}). It is somewhat surprising that new results are still proved about Magnus subgroups in general one-relator groups, even using mostly classical methods. Moldavanskii \cite{Moldavanskii1968} proved that if $G$ is a one-relator group $\pres{}{A}{r=1}$ with $X_1, X_2 \subseteq A$ and $X_1 \cap X_2 = \varnothing$, then $\langle X_1 \rangle \cap \langle X_2 \rangle$ is either trivial or isomorphic to $\Z$. Collins \cite{Co04} extended this result in 2004 by proving that if $X_1, X_2 \subset A$ generate two Magnus subgroups $H_1 = \langle X_1 \rangle$ and $H_2 = \langle X_2 \rangle$, then $H_1 \cap H_2$ is either free on $X_1 \cap X_2$, or else is a free product of $\Z$ with this group.

\subsection{The centre of one-relator groups}\label{Subsec:Centre}

For any group $G$, it is natural to ask about its centre $Z(G)$. This is also true of one-relator groups. Indeed, Levi, in his 1933 investigation \cite{Levi1933b} on free groups 1933, includes as an example the group $\BS(2,2) = \pres{}{x,y}{x^{-1}y^2x = y^2}$ (see \S\ref{Subsec:Baumslag-Solitar}), and notes that its centre is infinite cyclic, and generated by the element $y^2$. Generalizing such examples, in 1935 Kalashnikov \& Kurosh \cite{Kalashnikov1935} studied subgroups of amalgamated free products in which the amalgamated subgroups are all subgroups of the centres of the factors, a precursor to the tree products which we shall encounter below. Computing the centre of finitely presented groups is in general an undecidable problem\footnote{This is not a direct consequence of the Adian--Rabin theorem \cite{Adian1955,Rabin1958}, but follows from a 1959 extension of it by Baumslag, Boone \& Neumann \cite{BBN59}.}, and Remeslennikov proved that there are finitely presented groups whose centre is not finitely generated \cite{Remeslennikov1974} (answering a question of Baumslag \cite[Question~1]{Baumslag1967c}). However, it turns out that for one-relator groups, there is a complete answer. We give a brief exposition of these results in this section. 

Perhaps the first general result on the centres of one-relator groups is that the centre of a one-relator group with torsion is necessarily trivial, proved in 1960 \cite[Corollary]{KMS60}. The real impetus for the study of the centre of one-relator groups, however, came in 1964, when Murasugi \cite{Mu64} proved several beautiful theorems. Let $G = \pres{}{a_1, \dots, a_n}{R=1}$ be a one-relator group. Then Murasugi proved: first, if $n>2$, then $Z(G) = 1$.\footnote{In a sequel article, Murasugi \cite{Murasugi1965} conjectured that any group with deficiency $d \geq 2$ has trivial centre, which remains an open problem.} Furthermore, if we suppose that $n=2$, $G$ is non-abelian, and $Z(G)$ is non-trivial, then $Z(G)$ is infinite cyclic. Thus the centre of any one-relator group $G$ is either trivial, infinite cyclic, or $\Z^2$ -- and the last case happens if and only if $G \cong \Z^2$ (and there is only one cyclically reduced one-relator presentation for $\Z^2$, see \S\ref{Subsec:Applications-of-Frei}). The proofs of this pass via the \textit{Alexander polynomial} of $G$, as defined by Fox \cite{Fox1954}, and the proofs also give sufficient conditions for $Z(G)$ to be trivial. However, the question of completely characterizing when $Z(G)$ is trivial was left open. In 1968, this problem was solved by Baumslag \& Taylor \cite{BT68}, who proved that there exists an algorithm for deciding whether a one-relator group has trivial centre or not. This was the starting point for a natural line of investigation: what is the structure of one-relator groups with a non-trivial centre? 

Some partial results in this line were obtained by Meskin, Pietrowski \& Steinberg \cite{Meskin1973}. The question was fully answered by Pietrowski in 1974 (although the results already appear in his 1972 PhD thesis \cite{Pietrowski1972}) using subgroup theorems of Karrass \& Solitar. Recall that a \textit{tree product} of groups is the fundamental group of a graph of groups, where the underlying graph is a simple tree. A \textit{stem} product is a tree product in which the underlying tree is a path, introduced in \cite{KS70}. 

\begin{theorem}[Pietrowski, 1974 \cite{Pi74}]
Any one-relator group with non-trivial centre is a (finite) stem product of cyclic groups or an HNN-extension of such a group with cyclic associated subgroups.\label{Thm:Pietrowski-structure}
\end{theorem}

Associated to any finite stem product of cyclic groups is a finite sequence of pairs of integers, corresponding to the amalgamated subgroups. This sequence can, it turns out, be used to essentially uniquely determine the one-relator group up to isomorphism. As a consequence, we have the following corollary: 

\begin{corollary}[Pietrowski, 1974 \cite{Pi74}]
The isomorphism problem is decidable for one-relator groups with non-trivial centre. \label{Thm:Isoproblem-for-centre-is-decidable}
\end{corollary}

The problem of deciding precisely when a stem product of cyclic groups is a one-relator group is non-trivial and was left open by Pietrowski. Some partial results were obtained by, among others, McCool \cite{McCool1991}, and it was eventually solved in 1999 by Metaftsis \cite{Metaftsis1999}. As pointed out by Meskin, the above structural result (Theorem~\ref{Thm:Pietrowski-structure}) makes no essential use of the fact that the group is a one-relator group, but applies equally well to any two-generator infinite (f.g. free)-by-cyclic groups with a non-trivial center. This was expanded on by Karrass, Pietrowski \& Solitar \cite{Karrass1973}, where they completely characterize the structure of infinite (f.g. free)-by-cyclic groups with a non-trivial centre as a tree product of infinite cyclic groups. 

We mention another decidability result for one-relator groups with a non-trivial centre, proved by Moldavanskii \& Timofeeva in 1987. Recall that a finitely generated subgroup $H \leq G$ is said to be \textit{separable} if for every $g \in G \setminus H$, there is some subgroup $K_g \in G$ such that $H \leq K_g$ and $g \not\in K_g$. In particular, separability of the trivial subgroup is precisely residual finiteness. We say that $G$ is \textit{subgroup separable} if all its finitely generated subgroups are separable.

\begin{theorem}[Moldavanskii \& Timofeeva, 1987 \cite{Moldavanskii1987}]
All one-relator groups with non-trivial centre are subgroup separable. In particular, the subgroup membership problem is decidable in one-relator groups with non-trivial centre. \label{Thm:Centre-OR-are-subgroup-separable}
\end{theorem}

We conclude our discussion on the centre by mentioning the centre of \textit{subgroups} of one-relator groups. The subgroups of one-relator groups can be complicated in general, but their centres can be completely understood, analogously to Murasugi's theorem. 

\begin{theorem}[Mahimovski, 1971 \cite{Mahimovski1971} / Karrass, Pietrowski \& Solitar, 1974 \cite{Karrass1974}]
Let $G$ be a torsion-free one-relator group and $H < G$ be a subgroup with non-trivial centre $1 \nleq Z(H) \nleq H$. Then $Z(H)$ is infinite cyclic. \label{Thm:Mahimovski-Center-Theorem}
\end{theorem}

Related to this theorem is a question posed by Baumslag \cite{Ba71} in 1974: if $G$ is a one-relator group with at least $3$ generators, and $N \vartriangleleft G$ is such that $G / N$ is infinite, does $G / N'$ have trivial centre? Here $N'$ is the derived subgroup of $N$. Fischer \cite{Fischer1976} gave a negative answer to this question shortly thereafter. 

Just as for many other theorems about subgroups of one-relator groups, there is an important (co)homological point of view. Swan \cite[p. 156]{Swan1969} proved in 1969 that for any finitely generated non-abelian group $G$ with $\cd_{\Z}(G) < \infty$, we have $\cd_{\Z}(Z(G)) < \cd_{\Z}(Z(G))$. In particular, Swan's theorem can be used to prove Theorem~\ref{Thm:Mahimovski-Center-Theorem}, as the cohomological dimension of any torsion-free one-relator group is $\leq 2$ (see \S\ref{Subsec:LyndonIT}), and any abelian group of cohomological dimension $\leq 1$ is of course either $1$ or $\Z$. In 1976, Bieri \cite[Theorem~C]{Bieri1976} extended Swan's theorem by dropping the hypothesis of finite generation. In the same article, Bieri \cite[Theorem~D]{Bieri1976} also gave a homological generalization of many of the above results on the centre, by proving the following: let $G$ be a group of type $\FP_\infty(\Z)$ with $\cd_{\Z}(G) = n$ with $Z(G)$ free abelian of rank $n-1$. Then $G$ is a treed HNN-extension over a finite tree with all edges and all vertices $\cong Z(G)$. In particular, when $G$ is a torsion-free one-relator group with non-trivial centre, we obtain the following consequences: 
\begin{enumerate}
\item every finitely generated subgroup of $G$ is finitely presented and of type $\FP_\infty(\Z)$,
\item $[G, G] \cap Z(G) = 1$, 
\item $[G, G]$ is free.
\end{enumerate}
We refer the reader to \cite[Corollary~6.6]{Bieri1976} for more details. In this line there are also earlier results by Baumslag \& Gruenberg \cite{BaumslagGruenberg1967}, e.g.\ that for any two-generated subgroup $H \leq G$ of a one-relator group, if $Z(H)$ is cyclic and $H / Z(H)$ is periodic, then $H$ is itself cyclic.

For a final note on one-relator groups with non-trivial centre, we mention that their (outer) automorphism groups are also known. As a classical example, recall that Schreier (see \S\ref{Subsec:Surface-and-knot-groups}) determined the automorphism group of $G_{m,n} = \pres{}{a,b}{a^mb^n = 1}$, and it is easy to see that the corresponding group is virtually free (note that Pettet \cite{Pettet1995} has also proved a structural result for all finitely generated groups with virtually free automorphism group). Extending this result, Gilbert, Howie, Metaftsis \& Raptis \cite{Gilbert2000} proved that the outer automorphism group of a one-relator group $G$ with non-trivial centre falls into two cases: first, if $\rk_{\Z} G^{\operatorname{ab}} = 1$, then either $\Z / 2\Z$, or else $G \cong \pres{}{a,b}{a^m = b^m}$ for some $m \geq 2$, and in the latter case $\Out(G) \cong (\Z / 2\Z)^2$, as proved by Schreier in 1924, see \S\ref{Subsec:Surface-and-knot-groups}. If instead $\rk_{\Z} G^{\operatorname{ab}} = 2$, then either $G \cong \Z^2$, or $\Out(G) \cong D_\infty$, or $\Out(G) \cong D_\infty \times (\Z / 2\Z)$.

\subsection{Commutativity in one-relator groups}\label{Subsec:Abelian-subgroups}

As we mentioned earlier (in \S\ref{Subsec:MM-hierarchy}), in his 1964 survey article Baumslag asked which abelian groups arise as subgroups of one-relator groups. In particular, he \cite[p. 391]{Ba64} conjectured that the additive group of rationals is not a subgroup of any one-relator group. He also says ``one might however ask just which abelian groups are subgroups of a group with one defining relator''. We have already seen that Moldavanskii \cite{Mo67} gave a nearly complete answer to this question in 1967, see \S\ref{Subsec:MM-hierarchy}, but that he left unresolved the question of whether $(\Q, +)$ embeds into any one-relator group. This would be done by B.\ B.\ Newman in 1968, again in his thesis, who gave an independent proof of Moldavanskii's result (Theorem~\ref{Thm:Moldavanskii}), and clarified the situation of the locally cyclic subgroups as follows: 

\begin{theorem}[{B.\ B.\ Newman, 1968 \cite[Corollary~1.2.4]{Newman1968}}]
Let $G$ be a torsion-free one-relator group. Then no non-trivial element of $G$ has more than finitely many prime divisors. Moreover, a non-trivial element is not divisible by more than finitely many powers of a prime $p$ if $p$ is greater than the length of the relator.
\end{theorem}

As an immediate consequence of this theorem, we see that: 

\begin{corollary}[{B.\ B.\ Newman, 1968 \cite[Corollary~1.2.4]{Newman1968}}]
The additive group of rationals is not a subgroup of any one-relator group. \label{Cor:Newman-Q-not-subgroup}
\end{corollary}

This completed the classification of the abelian subgroups of one-relator groups. New proofs of these results, particularly characterizing the abelian subgroups of amalgamated free products, were given by Wright \cite{Wright1979} using Bass--Serre theory, and Nagata \cite{Nagata1980}. 

We now turn to \textit{solvable} subgroups of one-relator groups. This is a more difficult question, but again it becomes easier to answer in the case of one-relator groups with torsion. Unlike the proof of the Spelling Theorem, the proof of this result (which was also announced in \cite{Ne68}) was published in the literature, and the proof of Theorem~\ref{Thm:Newman-solvable-subgroups} is the subject of the 1973 article \cite{Newman1973}.

\begin{theorem}[{B.\ B.\ Newman, 1968 \cite[Corollary~2.3.3]{Newman1968}}]
The solvable subgroups of a one-relator group with torsion are cyclic.\label{Thm:Newman-solvable-subgroups}
\end{theorem}

Karrass \& Solitar \cite{Karrass1969} determined which one-relator groups contain abelian normal subgroups. In particular, the only one-relator group, other than $\Z^2$ itself, containing a normal subgroup isomorphic to $\Z^2$ is $\pres{}{a,b}{abab^{-1} = 1}$, the fundamental group of the Klein bottle. Another student of Greendlinger in Tula was A.\ A.\ Chebotar, who proved several remarkable theorems in the 1970s, extending Moldavanskii's work. First, in 1971 \cite{Ch71}, he extended Theorem~\ref{Thm:F2-subgroup-theorem} to all subgroups of one-relator groups, proving that if $H \leq \pres{}{A}{r=1}$ is non-abelian and such that $H$ does not contain a free subgroup of rank $2$, then either $H \cong C_2 \times C_2$ or $G \cong \BS(1,k)$ for $k \in \Z$. Thus, we also get a full classification of the solvable subgroups of one-relator groups, and a ``Tits alternative'' for one-relator groups. Next, in 1975 he extended Karrass \& Solitar's aforementioned result on abelian normal subgroups to all subgroups of one-relator groups \cite{Chebotar1975}. But perhaps the most important result proved by Chebotar was the following, which was also given a proof by Bagherzadeh around the same time:

\begin{theorem}[Bagherzadeh, 1976 \cite{Bagherzadeh1976} / Chebotar, 1978 \cite{Chebotar1978}]
Let $G = \pres{}{A}{r=1}$ be a one-relator group, and let $H_1 \times H_2 \leq G$ be a non-trivial direct product. Then one of $H_1$ and $H_2$ is isomorphic to $\Z$, and the other is locally cyclic. \label{Thm:Chebotar-direct}
\end{theorem}

In particular, the direct product $F_2 \times F_2$ of two non-abelian free groups does not embed into any one-relator group. This direct product has a range of bad properties, e.g.\ undecidable membership problem \cite{Mikhailova1958} and incoherence \cite{Stallings1977}. Today, we know that one-relator groups are coherent, giving another (significantly more difficult) proof of Chebotar's theorem. Finally, Chebotar notes that his results give an answer to Question~6 from Baumslag's 1974 survey \cite{Ba71}, published only four years earlier. This showcases the efficiency of transferring results and open problems across the Iron Curtain even during this time, and particularly the remarkable work by Greendlinger at making Tula a leading centre of combinatorial group theory.

\subsection{Coherence and subgroup separability}\label{Subsec:FP-subgroups}

The first property we will consider is \textit{coherence}. Recall that a group is said to be coherent if all of its finitely generated subgroups are finitely presented. As mentioned at the beginning of \S\ref{Sec:6-OR-with-torsion}, Baumslag conjectured that all one-relator groups are coherent, which we today know is the case. The coherence of all one-relator groups was also said to be ``conceivable'' by Bieri \cite[p. 36]{Bieri1976}. Some classical results yield coherence in small classes of one-relator groups: for example, it follows from Theorem~\ref{Thm:Surface-subgroup-theorem} that surface groups $\pi_1(\Sigma_g)$ are coherent for all $g \geq 1$. Perhaps the first large contribution to the coherence of one-relator groups came in 1976. As we have seen in \S\ref{Subsec:Centre}, one-relator groups \textit{with non-trivial centre} are in general significantly better behaved than their centreless counterparts, and general theorems easier to come by. This is true also for coherence: 

\begin{theorem}[{Bieri, 1976 \cite[Corollary~6.6]{Bieri1976}}]
One-relator groups with non-trivial centre are coherent.
\end{theorem}

Beyond this result, however, few truly general results were known regarding coherence. More was known on the structure of particular finitely presented subgroups of one-relator groups. For example, one can ask what the finitely presented \textit{normal} subgroups of one-relator groups are. In this line, we have the following theorem, which can be obtained by combining two closely related results due to Bieri and Karrass \& Solitar. 

\begin{theorem}[{Bieri  \cite[Theorem~B]{Bieri1976} / Karrass \& Solitar \cite{Karrass1978}}]
Let $G$ be a one-relator group and let $H \trianglelefteq G$ be a non-trivial finitely presented normal subgroup. If $H$ has infinite index in $G$, then $H$ is free, and $G$ is torsion-free and has two generators. Furthermore, $G$ is an infinite cyclic or infinite dihedral extension of a finitely generated free group.
\end{theorem}

In fact, Bieri proved a stronger statement: the above theorem holds for all finitely generated groups of cohomological dimension $\leq 2$, which includes all torsion-free one-relator groups (see \S\ref{Subsec:LyndonIT}). Notably, the proof of both the results of Bieri and Karrass \& Solitar use cohomological methods, including the rational Euler characteristic, cf.\ \cite{Wall1961, Chiswell1976}. 

Similarly, one can also ask about when finitely generated normal subgroups have infinite index in one-relator groups. Recall that Schreier (\S\ref{Subsec:Nielsen-Schreier}) proved that any finitely generated normal subgroup of a free group has finite index. The analogous theorem fails already for some very basic one-relator groups: for example, in the torus knot group $G = \pres{}{a,b}{a^2 = b^3}$ the commutator subgroup is finitely generated and free (cf.\ Theorem~\ref{Thm:Moldavanskii-Brown-Criteriton}), but and hence has infinite index in $G$ as the defining word $a^2b^{-3}$ is not primitive (see Theorem~\ref{Thm:Whitehead-theorem}). Note that $G \cong \Z \ast_\Z \Z$, where the amalgamated subgroup of course has finite index in both factors; this turns out to one source of obstructions to extending Schreier's result. Karrass \& Solitar \cite{Karrass1973b} proved a number of results on amalgams with no such obstructions, and we state one of the relevant corollaries of their results:

\begin{theorem}[{Karrass \& Solitar \cite[p.~23]{Karrass1973}}]
Let $G = A \ast_U B$ be an amalgamated free product where $U$ has infinite index in $A$, the subgroups of $U$ are all finitely generated, and every finitely generated normal subgroup of $A$ has finite index in $A$. Then every finitely generated normal subgroup of $G$ has finite index in $G$. 
\end{theorem}

Thus, for example, the group $G' = \pres{}{a,b,c}{a^2=b^3}$ \textit{does} have the property that every finitely generated normal subgroup of $G'$ has finite index in $G'$. 

Other than coherence, another finiteness property for groups in relation to their finitely generated subgroups is \textit{subgroup separability}, which we have mentioned above (see Theorem~\ref{Thm:Centre-OR-are-subgroup-separable}). Since not every one-relator group is residually finite, neither is every one-relator group subgroup separable. However, in general subgroup separability is a much stronger condition even for one-relator groups. Indeed, Burns, Karrass \& Solitar \cite{Burns1987} proved that the one-relator group defined by 
\begin{equation}\label{Eq:BKS-Group}
\pres{}{a,b}{[a, bab^{-1}] = 1} \cong \pres{}{a,b}{[ab, ba] = 1}
\end{equation}
is not subgroup separable. However, the group \eqref{Eq:BKS-Group} is residually finite, indeed (f.g.\ free)-by-cyclic, by Theorem~\ref{Thm:Baumslag-uv^-1-free-by-cyclic}. 

Finally, we mention the \textit{Howson} property. We say that a group $G$ has the Howson property if the intersection of any two finitely generated subgroups of $G$ is again finitely generated. Howson \cite{Howson1954} proved that free groups have the Howson property (see \cite{Imrich1977} for a simple proof), and it is known that free products of Howson groups are Howson \cite{BBaumslag1966}, also reproved by Bezverkhnii \& Rollov \cite{Bezverkhnii1974}; cf.\ also \cite{Burns1972, Kapovich1997}. It is thus natural to ask to what extent the Howson property holds in one-relator groups. Greenberg \cite[Theorem~2]{Greenberg1960} proved in 1960 that all surface groups are Howson. Moldavanskii \cite{Moldavanskii1968} proved that $\BS(1,n)$ for $n \in \Z$ have the Howson property. However, Moldavanskii \cite{Moldavanskii1968} also proved that any one-relator group with a non-trivial centre fails to have the Howson property. Kapovich \cite{Ka99} gave an example of a hyperbolic torsion-free one-relator group which is not Howson. Kapovich \cite{Kapovich1996} also proved that if $G = \pres{}{A}{r^n = 1}$ is a finitely generated one-relator group with torsion such that $n \geq 5$, then the intersection of any two $2$-generated subgroups of $G$ is again finitely generated.

\subsection{Quotients and residual properties of one-relator groups}\label{Subsec:Quotients+moreresidualproperties}

We discuss some results on quotients of one-relator groups, beginning with their finite quotients and some results on residual solvability and nilpotency. Of course, Baumslag--Solitar groups (\S\ref{Subsec:Baumslag-Solitar}) give examples of non-residually finite one-relator groups, but many other examples of residually finite one-relator groups are known.

In 1963, G.\ Baumslag \cite{Ba63} proved several fundamental results regarding residual finiteness and nilpotent groups: for example, he proved that the free product of two residually finite groups, amalgamated in a finite subgroup, is always residually finite (see also \cite{Tretkoff1973} for a simple proof of this fact). One of his results concerns one-relator groups directly, and we have already seen it in \S\ref{Subsec:Hopf}: the free product of two free groups, amalgamated in a cyclic subgroup, is residually finite \cite[Theorem~7]{Ba63} (cf.\ also \cite{Evans1973, Wehrfritz1973, Shirvani1988} in this line). Such amalgams, called \textit{cyclically pinched one-relator groups}, provide a rich source of residually finite one-relator groups (including the surface group $\pi_1(\Sigma_2)$, and hence also all $\pi_1(\Sigma_g)$ for $g>1$, see \S\ref{Subsec:Surface-and-knot-groups}). 

\begin{theorem}[Baumslag, 1963 \cite{Ba63}]
Cyclically pinched one-relator groups are residually finite. 
\end{theorem}

Strengthening this, in 1981 Allenby \cite{Allenby1981c} proved that cyclically pinched one-relator groups are \textit{potent}, a stronger form of residual finiteness (for every non-trivial element $g$ and every $n \in \N$, there is a homomorphism onto a finite group which maps $g$ onto an element of order exactly $n$). 

We turn to \textit{positive} one-relator groups and related classes of groups. A one-relator group $G$ is said to be positive if it admits some presentation $\pres{}{A}{r=1}$ where $r \in F_A$ contains no inverse letters. Note that this condition has some subtlety: for example $G = \pres{}{a,b}{bab^{-1}a = 1}$ is, in spite of its disguise, a positive one-relator group, since $G \cong \pres{}{x,y}{x^2y^2 = 1}$. Indeed, the isomorphism problem of one-relator groups with respect to the class of positive one-relator groups is still an open problem, i.e.\ it is unknown whether we can decide if a given one-relator group is positive or not. Note that e.g.\ the group $\Z^2 = \pres{}{a,b}{[a,b] = 1}$ is not positive, by Proposition~\ref{Prop:Magnus-[a,b]-only-presentaton}. Positive one-relator groups enjoy some particular properties. The first, and perhaps most striking, theorem in this line is the following:

\begin{theorem}[Baumslag, 1971 \cite{Baumslag1971}]
Positive one-relator groups are residually solvable.\label{Thm:Baumslag-positive-are-RS}
\end{theorem}

In particular, every perfect subgroup of a positive one-relator group is trivial. It follows essentially from the local indicability of one-relator groups (see below) that all \textit{finitely generated} perfect subgroups of one-relator groups are trivial; if one drops the assumption of finite generation, then the answer is negative, as evidenced by the Baumslag--Gersten group $\operatorname{BG}(1,2)$ defined in \eqref{Eq:Baumslag-Gersten-groups}. 

When the defining relator is positive \textit{and} a proper power, i.e.\ we are considering a positive one-relator group with torsion $G$, then $G$ is residually finite by Theorem~\ref{Thm:Egorov-positive-ORT-is-RF} in \S\ref{Subsec:RF-of-Torsion}. Of course, solvable groups are not in general residually finite, indeed they can have undecidable word problem, by Kharlampovich \cite{Kharlampovich1981}. Furthermore, positive one-relator groups are not residually finite in general. Indeed, any Baumslag--Solitar group $\BS(m,-n)$ with $m, n \geq 0$ admits a positive presentation, as noticed by Higman \cite[p.~166]{Baumslag1971} and, independently, Perrin \& Schupp \cite{Perrin1984}. 

One-relator groups are, in general, not residually solvable, as e.g.\ the Baumslag--Gersten group $G = \pres{}{a,b}{[a, bab^{-1}] = a}$ has $a$ lie in every term of the derived series. It is not hard to see that $G \cong \pres{}{x,y}{x^2 y x y = y x}$, and so in general one-relator groups with a defining relation of the form $uv^{-1} = 1$ for two positive words is not residually solvable. However, with some additional assumptions, more can be said also in this case. This was the subject of a 1983 article by Baumslag, some of the results of which we now present. First, recall that for a word $w \in F_A$ and a letter $a \in A$, we let $\sigma_a(w)$ be the \textit{exponent sum} of $a$ in $w$.

\begin{theorem}[{Baumslag, 1983 \cite{Baumslag1983}}]
Let $G = \pres{}{A}{u=v}$ where $u, v \in F_A$ are positive words, and where $\sigma_a(u) = \sigma_a(v)$ for all $a \in A$. Then $G$ is free-by-cyclic. \label{Thm:Baumslag-uv^-1-free-by-cyclic}
\end{theorem}

Note that one particular case of groups in the class from Theorem~\ref{Thm:Baumslag-uv^-1-free-by-cyclic} is the case of a commutator $[u, v] = 1$, where $u, v \in F_A$ are positive words. Not every one-relator group of the form $[u, v] = 1$ with general $u, v \in F_A$ is free-by-cyclic (see \cite[\S4.3]{Baumslag1983}), but it is possible to add further conditions on $u$ and $v$, other than positivity, to ensure residual finiteness.

\begin{theorem}[{Baumslag, 1983 \cite{Baumslag1983}}]
Let $G = \pres{}{A \cup B}{[u, v] = 1}$ where $A \cap B = \varnothing$ and $u \in F_A, v \in F_B$. Then some term in the lower central series of $G$ is free; and $G$ is residually finite. \label{Thm:Baumslag-commutator-theorem-RF-OR}
\end{theorem}

Thus, for example, the group $\pres{}{a,b,c}{[a^m, cb^nc^{-1}] = 1}$ is residually finite for all $m, n \in \Z$. For extensions of Theorem~\ref{Thm:Baumslag-commutator-theorem-RF-OR} to a slightly more general setting, see Loginova \cite{Loginova1999}. 

On the other hand, some one-relator groups have residual finiteness fail in a rather dramatic fashion. By basic linear algebra, every infinite one-relator group obviously surjects $\Z$, and hence also has all cyclic groups as a quotient. Baumslag \cite{Ba69} proved, remarkably, that there also exist non-cyclic one-relator groups all of whose finite quotients are cyclic. This can be seen as a very strong failure of residual finiteness. The group in question was generalized to a family of groups known as the \textit{Baumslag--Gersten groups}
\begin{equation}\label{Eq:Baumslag-Gersten-groups}
\operatorname{BG}(m,n) = \pres{}{a,b}{bab^{-1}a^m ba^{-1}b^{-1} = a^n}
\end{equation}
and the original example in \cite{Ba69} is the case of $(m,n) = (1,2)$. The groups have seen some study. Gersten \cite{Ge92} proved that the Dehn function of $\operatorname{BG}(1,2)$ grows faster than any finite tower of exponentials, later sharpened by Platonov \cite{Pl04}. Brunner \cite{Brunner1980} investigated the automorphism group of $\operatorname{BG}(m,n)$, and Kavutskii \& Moldavanskii \cite{Kavutskii1988} proved that $\operatorname{BG}(m,n)$ is residually finite if and only if $|m|=|n|$, and it is not hard to see that all of the finite quotients of $\operatorname{BG}(m,n)$ are cyclic if and only if $m = n \pm 1$. 

Recall that finitely generated nilpotent groups are residually finite (\S\ref{Subsec:Residualfiniteness}). In 1964, A. Steinberg \cite{Steinberg1963, Steinberg1964} studied free nilpotent quotients of one-relator groups. Specifically, let $G = \pres{}{A}{R = 1}$ where $R \in [F_A, F_A]$, i.e.\ the exponent sum of every generator in $R$ is zero. Steinberg then proved that there exists an algorithm for deciding whether or not $G$ has a rank $2$ free nilpotent quotient. Shapiro \& Sonn \cite{Shapiro1974} reproved his theorem in a cohomological setting and extended it to a somewhat broader class of one-relator groups, see also \cite{Sonn1974}. One of the corollaries of Steinberg's results in \cite{Steinberg1964} is that the one-relator group $\pres{}{a,b,c}{a^2b^2c^3 = 1}$ has free nilpotent quotient groups of rank $2$ and all classes; however, it does not surject the free group of rank $2$ (this follows e.g.\ by a theorem of Sch\"utzenberger \cite{Schutzenberger1959}). We shall see more of such groups in \S\ref{Subsec:lower-central-series}. Finally, we remark that although automorphism groups of one-relator groups remain in general somewhat mysterious, a very elegant and short argument due to Baumslag \cite{Baumslag1963b} shows that the automorphism group of any finitely generated residually finite group is again residually finite. Thus, automorphism groups of residually finite one-relator groups provide a rich source of residually finite groups. 

Free groups are one-relator groups, and studying the residual properties of free groups often leads to interesting questions related to one-relator groups more general. First, note that by a simple argument, $F_n$ is residually $F_2$ for all $n \geq 2$, see \cite{Peluso1966}. Peluso \cite{Peluso1966} also proved that the free group of rank $n \geq 2$ is residually $\PSL(2,p)$ for any $p > 3$. In relation to this result, we mention briefly that an old conjecture by Magnus was that for any infinite family of finite non-abelian simple groups $\mathcal{G}$, the free group $F_2$ is residually $\mathcal{G}$. This led to many significant articles, e.g.\ \cite{Katz1968, Gorcakov1970, Poss1970, Pride1972, Wiegold1977, Lubotzky1986, Wilson1991}, leading finally to a full resolution \cite{Weigel1992,Weigel1992b,Weigel1993}, cf.\ also \cite{Valette1993}. Magnus \cite[p. 309]{Magnus1969} raises the point that non-abelian free groups are residually $\pres{}{a,b}{a^n}$ for $n>1$, as proved by \cite{Katz1968}, and also ``possibly residually $S$ for a large number of other one-relator groups $S$''. This line of investigation was taken up by Pride \cite{Pride1972}, and especially in \cite[Chapter~3]{Pride1974}, where he investigates when free groups are residually one-relator groups with torsion, making heavy use of the B.\ B.\ Newman Spelling Theorem. In particular, he proves:

\begin{theorem}[{Pride, 1974 \cite[\S3.2]{Pride1974}}]
Let $G = \pres{}{x_1, x_2, \dots, x_k}{r^n = 1}$ with $n>1$ be an arbitrary one-relator group with torsion.
\begin{enumerate}
\item If $k>3$, then $F_k$ is residually $G$. 
\item If $k=2$, and $r \neq [x_1, x_2]$, then $F_2$ is residually $G$.
\item If $k=2$, and $r = [x_1, x_2]$, then $F_2$ is \textit{not} residually $G$.
\end{enumerate}
\end{theorem}
For more on the class of groups $S_n = \pres{}{a,b}{[a,b]^n = 1}$ with $n>1$, see \S\ref{Subsec:Torsion-isomorphism-problem}. 

We now turn to more general quotients of one-relator groups. There is in general no shortage of quotients of one-relator groups. Recall that a countable group $G$ is \textit{SQ-universal} if every countable group embeds in some quotient of $G$. In 1973, P.\ Neumann \cite{Ne73} conjectured that any one-relator group is either solvable (and hence either $\BS(1,n)$ or cyclic) or else is SQ-universal, which remains an open problem. In 1974, Sacerdote \& Schupp \cite{SS74} proved that every one-relator group with at least three generators is SQ-universal. B.\ Baumslag \& Pride \cite{BP78} generalized this to prove that if $G$ is a group with $n$ generators and $r$ relators, such that $n - r \geq 2$, then $G$ has a finite index subgroup which surjects $F_2$. Finally, St\"ohr proved the following general result in 1983: 

\begin{theorem}[St\"ohr, 1983 \cite{Stohr1983}]
Every one-relator group with torsion is SQ-universal.
\end{theorem}

In looking at quotients of one-relator groups, it is natural to ask the following question: given two one-relator groups which are homomorphic images of one another, are they necessarily isomorphic? This question was first posed by Moldavanskii in 1969 (see Question~3.33 in the Kourovka Notebook \cite{Khukhro2024}). The question was given a negative answer only 40 years later by Borshchev \& Moldavanskii \cite{Borshchev2006}, indicating the difficulty of even such simple-sounding problems. 

Finally, we briefly mention a very important result on quotients of one-relator groups. A group is said to be \textit{indicable} if it surjects $\Z$, and \textit{locally indicable} if all of its finitely generated subgroups also surject $\Z$. The notion of an indicable group was introduced by Higman \cite[p.~242]{Higman1940}, who proved that if $G$ is locally indicable and $R$ is an integral domain, then $RG$ has no zero divisors and no non-trivial units \cite[Theorem~12]{Higman1940}. Baumslag \cite{Ba71} conjectured that every torsion-free one-relator group is locally indicable. Brodskii \cite{Brodskii1980} announced in 1980 the theorem that this conjecture is true, i.e.\ all torsion-free one-relator groups are locally indicable. Brodskii \cite{Bro84} only published the full proofs of these facts four years later; in the meantime, the theorem was also given topological proof by Howie \cite{Ho82}. The subject of local indicability and its many fascinating links with many broader topics reaches far outside the scope of this survey, however.

\subsection{Lower central series and parafree groups}\label{Subsec:lower-central-series}

Let $G$ be any group. Recall that the \textit{lower central series} $\gamma_1(G), \gamma_2(G), \dots$ of $G$ is defined inductively by $\gamma_1(G) = G$ and $\gamma_{n+1}(G) = [G, \gamma_n(G)]$. This determines a $\Z$-graded Lie ($\Z-$)algebra $\gr(G) := \bigoplus_{n \geq 1} \gamma_n(G) / \gamma_{n+1}(G)$, where the factor $\gamma_n (G) / \gamma_{n+1}(G)$ is termed the $n$th \textit{lower central factor} of $G$. The sequence of lower central factors is sometimes called the \textit{lower central sequence} of $G$. The lower central factors of any finitely presented group are finitely generated abelian groups, and are furthermore computable from the presentation (by a commutator collection process \cite{Hall1934}, see e.g.\ \cite{Alperin1990, Baumslag1999b}) and thus form a particularly convenient set of invariants of groups. This is in stark contrast to the derived series of a group: if $G''$ denotes the second derived subgroup of $G$, then $G/G''$ need not be finitely presented, even if $G$ is a one-relator group. Indeed, as proved by Baumslag \& Strebel \cite{Baumslag1976}: if $G_{m,n,r} = \pres{}{a,t}{t^r a^m t^{-r} = a^n}$, where $mnr \neq 0$ and $\gcd(m,n)>1$ or $|r|>1$ and $mn=1$, then $H_2(G / G'', \Z)$ is not finitely generated. On the other hand, if $\gcd(m,n) = 1$ and $|m| \neq 1 \neq |n|$, then $G / G''$ is not finitely presented, even though $H_2(G / G'', \Z)$ is finitely generated. 

The structure of the Lie algebra $\gr(G)$ has seen some study in the case of one-relator groups. We mention some results briefly. First, if $G = F_n$, a free group on $n$ generators, then $\gr(F_n)$ is the free Lie algebra on $n$ generators (see \cite[Chapter~5]{Magnus1966}). 
If $G = F / R$, then there is a canonical surjection $\gr(F) \twoheadrightarrow \gr(G)$. This kernel is quite non-trivial already for some very simple examples of one-relator groups. For example, it was determined when $G = \pres{}{x,y}{x^p = 1}$ with $p$ a prime by Labute \cite{Labute1977}. In the case that the defining word is not a proper power, more can be said. Indeed, Labute \cite{Labute1970} found an explicit formula for the rank of $\gamma_n(G) / \gamma_{n+1}(G)$ when $G = \pres{}{A}{r=1}$, as long as $r$ is not an $n$th power modulo certain terms of the lower central series of the free group $F = F(A)$, and furthermore that in this case the associated Lie algebra $\gr(G)$ is itself definable by a single defining relation (as a Lie algebra), cf.\ also \cite{Labute1967}. This, for example, gives a closed formula for the rank of $\gamma_n(\pi_1(\Sigma_2))$, where $\pi_1(\Sigma_2)$ is the fundamental group of a genus $2$ surface. If $g_n$ denotes the $\Z$-rank of the free abelian group $\gamma_n(\pi_1(\Sigma_2))$, then
\begin{equation}\label{Eq:rank-for-pi2-lie}
g_n = \frac{1}{n} \sum_{d \mid n} \mu (\frac{n}{d}) \left( \sum_{0 \leq i \leq \lfloor d / 2 \rfloor} (-1)^i \frac{d}{d-i} \binom{d-i}{i} 4^{d-2i} \right)
\end{equation}
where $\mu$ denotes the Möbius function. One can use \eqref{Eq:rank-for-pi2-lie} to show that $g_n$ grows as $(2 + \sqrt{3})^n$, and more generally the exponential growth rate of the ranks in the case of $\pi_1(\Sigma_g)$ with genus $g>1$ is given by $g + \sqrt{g^2-1}$.  It also follows from the aforementioned result by Labute that $\gr(\pi_1(\Sigma_g))$ is a one-relator Lie algebra (in particular, it has decidable word problem \cite{Shirshov1962}). By contrast, it is not hard to show that $\gr(\BS(1,3))$ is not finitely presented as a Lie algebra. 

We now turn to a very important aspect of the lower central series. This goes back to the 1939 theorem by Magnus (see Theorem~\ref{Thm:Magnus-extended-Frei}(3)) that any $n$-generator group with the same sequence of lower central factors as the $n$-generator free group is itself free. This prompted H.\ Neumann to ask whether any residually nilpotent group with the same sequence of lower central factors as a free group is necessarily free. This leads to the following definition. A group $G$ is said to be \textit{parafree} if it is residually nilpotent and has the same sequence of lower central factors as some free group. Thus H.\ Neumann's conjecture simply asked: do non-free parafree groups exist? G.~Baumslag \cite{Ba64, Baumslag1965, Ba67b, Ba69b} would approach this question in a series of papers (cf.\ also \cite{BaumslagSteinberg1964}), and prove:

\begin{theorem}[G.\ Baumslag, 1964 \cite{Ba64}]
Non-free parafree groups exist. 
\end{theorem}

One such example given in \cite{Ba64}, with proof going via a careful study of amalgamated free products to be found in \cite{Baumslag1965}, is $\pres{}{a,b,c}{a^2 b^3 c^5 = 1}$. We note the reader that the terminology used in \cite{Ba64} is \textit{pseudofree}, rather than \textit{parafree}. Other examples of non-free parafree groups given in 1969 by Baumslag \cite{Ba69b} are the groups $\pres{}{a,b,c}{a = [c^i, a][c^j, b]}$ which are parafree, but not free, for all $i, j \in \Z$ such that $ij \neq 0$. Of course, the free group with which they share their sequence of lower central factors is the free group of rank $2$. This family belongs to the broader class of parafree-by-cyclic groups, studied by Wong  \cite{Wong1978,Wong1980} in his PhD thesis, supervised by G.\ Baumslag. Parafree groups share a great deal of properties with free groups. Recall (\S\ref{Subsec:Quotients+moreresidualproperties}) that any non-cyclic free group is residually free of rank two. Similarly, one can show that any non-cyclic parafree group is also residually parafree of rank two \cite[Theorem~5.2]{Ba69b}. We also have the following curious result, showing that the local behaviour of parafree groups is remarkably like that of free groups:

\begin{theorem}[{G.\ Baumslag \cite[Theorem~4.2]{Ba69b}}]
Any two-generator subgroup of a parafree group is free.\label{Thm:Parafree-2-gen-subgroups-free}
\end{theorem}

Baumslag \& Stammbach \cite{Baumslag1976b} have also given an example of a non-free parafree group all of whose countable subgroups are free. We remark that for a residually nilpotent group $G$, we have that $G$ is a parafree group if and only if $\gr(G)$ is a free Lie algebra. Using the Nielsen--Schreier theorem for free Lie algebras due to Witt \cite{Witt1937}, it is possible to give another proof of Theorem~\ref{Thm:Parafree-2-gen-subgroups-free}. This final part of the section is only intended to serve as a small sampler of parafree groups. Very much more can be written, particularly on parafree one-relator groups. Many of the techniques and results involved in the study of such groups go far beyond the classical, and is better left for a separate survey. 

We conclude this section by noting that one can go beyond residual nilpotence in the following way. Let $G$ be any group. Then its transfinite lower central series is defined as in the finite case by $\gamma_1(G) = G, \gamma_{\alpha + 1} = [G, \gamma_\alpha(G)]$, and $\gamma_\lambda (G) = \bigcap_{\alpha < \lambda} \gamma_\alpha(G)$ for limit ordinals $\lambda$. Then the nilpotency class of $G$ is the minimal $\alpha$ such that $\gamma_{1+\alpha}(G) = 1$. Thus, for example, a residually nilpotent group has nilpotency class $\leq \omega$. Baumslag \cite{Ba71} asked whether every one-relator group has nilpotency class $< \omega^2$. Indeed it was open for some time whether there were finitely presented groups of nilpotency class $> \omega$, until Levine \cite{Levine1991} gave an example of such groups in 1991. Baumslag's question was only answered in 2016, when Mikhailov \cite{Mi16} proved that $\pres{}{a,b}{ab^{-1}a^3 = b^{-2}ab}$ has nilpotency class $\omega^2$.

\section{More on decision problems}\label{Sec:8-Decisionproblems-1960-1980}

Even though the word problem was solved for one-relator groups almost a century ago, as presented in \S\ref{Subsec:Wordproblem-Magnus}, remarkably both the conjugacy problem and isomorphism problem -- the other two of Dehn's three fundamental problems -- remain open for one-relator groups. In this section, we will discuss some partial progress made on these problems, and the general approaches to the problems. Finally, we will discuss the word problem in some greater depth, including some open problems, and the possibility of extending the solution of the word problem to all \textit{two}-relator groups.

\subsection{The conjugacy problem}\label{Subsec:Conjugacyproblem}

We begin with a general remark regarding the difficulty of the conjugacy problem. Decidability of the conjugacy problem implies decidability of the word problem, and hence the conjugacy problem is for general finitely presented groups undecidable. However, P.\ S.\ Novikov \cite{Novikov1954} proved in 1955 the existence of finitely presented groups $A_{p_1, p_2}$ with undecidable conjugacy problem, the construction of which is significantly simpler than his group with undecidable word problem (see \S\ref{Subsec:Undecidability}). He left the question of the decidability of the word problem in $A_{p_1, p_2}$ open. In 1960, A.\ A.\ Fridman (a student of Novikov) proved that the groups $A_{p_1, p_2}$ have decidable word problem, thus showing that the conjugacy problem is in general strictly harder than the word problem. We note that in 1967, Bokut \cite{Bokut1967} simplified both Novikov's argument and Fridman's argument, and Collins \cite{Collins1969b} in 1969, proved that the conjugacy problem can be arbitrarily difficult -- i.e.\ of any recursively enumerable degree of unsolvability -- even in groups with decidable word problem.

More specifically to one-relator groups, recall that Magnus' solution to the word problem (\S\ref{Subsec:Wordproblem-Magnus}) in all one-relator groups depends, as clarified by Moldavanskii (\S\ref{Subsec:MM-hierarchy}), in large part on properties of HNN-extensions and amalgamated free products. Summarizing this argument, heavy appeal is made to the fact that if the word problem is decidable in two groups $G_1, G_2$, and $H \leq G_1, G_2$ is a common subgroup in which membership is decidable for both $G_1$ and $G_2$, then the amalgamated free product $G_1 \ast_H G_2$ has decidable word problem; the analogous statement also holds for HNN-extensions, \textit{mutatis mutandis}. Thus the word problem (more specifically membership in Magnus subgroups) can be passed down from a one-relator group to the next one-relator group in the Magnus--Moldavanskii hierarchy, possibly embedding the group into a larger one-relator group along the way, until a free group is reached, and the word problem is decidable in all one-relator groups by induction. In seeking to extend this line of argument to the conjugacy problem, we are faced with two problems: 
\begin{enumerate}
\item Decidability of the conjugacy problem is a much more delicate property to preserve when taking HNN-extensions or amalgamated free products; and 
\item Decidability of the conjugacy problem is not, in general, preserved by taking subgroups, even of finite index.
\end{enumerate}

We elaborate on these two points. For the first part, there is the following general undecidability result which, \textit{a priori}, shows that additional ingredients must be added to the Magnus--Moldavanskii hierarchy to adapt it to the conjugacy problem: 

\begin{theorem}[{Miller, 1971 \cite[Theorem~III.10]{Miller1971b}}]
There exists a finitely presented HNN-extension of a free group with undecidable conjugacy problem. There exists an amalgam $F_n \ast_{F_k} F_m$ of two finitely generated free groups over a finitely generated free subgroup with undecidable conjugacy problem. \label{Thm:CP-in-free-HNN-can-be-undecidable}
\end{theorem}

Note that the groups in Theorem~\ref{Thm:CP-in-free-HNN-can-be-undecidable} all have decidable word problem, since the subgroup membership problem is decidable in free groups. We refer the reader to \cite{Borovik2007, Borovik2007b} for more in-depth investigations into the complexity of the conjugacy problem in such amalgams of free groups. We also note that the conjugacy problem in HNN-extensions, and thus in some one-relator groups, has a very combinatorial flavour, and can thus often be phrased in terms of semigroup-theoretic problems. For example, for certain HNN-extensions the conjugacy problem reduces to reachability problems of vector addition systems, see \cite{Anshel1976, Anshel1976b, Anshel1978}, and sometimes further reduced to the word problem for commutative semigroups; see \cite{Anshel1976c}. Similarly, Horadam \& Farr \cite{Horadam1994} have reduced the conjugacy problem in some HNN-extensions to membership problems in inverse semigroups.  

For the second part, i.e.\ closure under taking subgroups, it was probably first recognized by Collins \cite{Collins1969c} that there exists a finitely presented group $G$ with a finitely presented subgroup $H \leq G$ such that $G$ has decidable conjugacy problem, but $H$ has undecidable conjugacy problem. In 1971, C.\ F.\ Miller \cite{Miller1971b} proved that there exist subgroups of $F_2 \times F_2$ with undecidable conjugacy problem, even though the conjugacy problem in $F_2 \times F_2$ is almost trivial. His example is based on the undecidability of the membership problem in $F_2 \times F_2$ due to K.\ A.\ Mikhailova \cite{Mikhailova1958} (a PhD student of P.\ S.\ Novikov). In 1977, Collins \& Miller \cite{Collins1977} proved that in general neither passing to index $2$ subgroups nor index $2$ extensions preserve decidability of the conjugacy problem. In this line, see also \cite{Gorjaga1975}. A natural question that arises in this context is the following:

\begin{problem}
Do subgroups of one-relator groups have decidable conjugacy problem? \label{Prob:subgroups-of-OR-CP?}
\end{problem}

We emphasize that it is possible that Problem~\ref{Prob:subgroups-of-OR-CP?} has a negative answer, even if all one-relator groups have decidable conjugacy problem. 

In spite of all the above negative results, in the setting of one-relator groups there is nevertheless a case to be made for optimism. Collins \cite{Collins1969d} proved that the conjugacy problem can be undecidable in 11-relator groups, but the problem remains open for $k$-relator groups for all $1 \leq k \leq 10$. Furthermore, as mentioned in \S\ref{Sec:2-Dehn-freegroups-presentations}, Dehn gave already in 1911 a solution for the conjugacy problem in the fundamental group $\pi_1(\Sigma_g)$ of surface groups of genus $g > 1$. In 1966, Lipschutz \cite{Lipschutz1966} generalized this to all cyclically pinched one-relator groups, i.e.\ amalgamated free products $F_n \ast_{\Z} F_m$ of free groups amalgamating a cyclic subgroup (of which surface groups are central examples), to give the following theorem:

\begin{theorem}[Lipschutz, 1966 \cite{Lipschutz1966}]
Cyclically pinched one-relator groups have decidable conjugacy problem. \label{Thm:Cyclicallypinched-CP}
\end{theorem}

Lipschutz later extended these results to other cyclic amalgams of non-free groups, see \cite{Lipschutz1969, Lipschutz1975}. For a different solution to the conjugacy problem in $\BS(m,n)$ with $\gcd(m,n) = 1$, see \cite{Anshel1974}. In 1980, Dyer \cite{Dyer1980} extended Theorem~\ref{Thm:Cyclicallypinched-CP} by proving that cyclically pinched one-relator groups are conjugacy separable. Larsen \cite{Lar77} also proved that any one-relator group with non-trivial centre has decidable conjugacy problem, using the structure theorem for one-relator groups (see \S\ref{Subsec:Centre}) with non-trivial centre. In this line, see also \cite{Hurwitz1980}. Finally, we have already mentioned that B.\ B.\ Newman proved (see \S\ref{Subsec:BBNewman}) that every one-relator group with torsion has decidable conjugacy problem. 

We remark that it might be tempting to consider the well-behaved class of residually nilpotent one-relator groups as natural targets, since finitely generated nilpotent groups have decidable conjugacy problem; indeed they are even conjugacy separable, as proved in 1965 by Blackburn \cite{Blackburn1965}. However, unlike for the word problem, it turns out that residually nilpotent groups need not have decidable conjugacy problem, as proved in 1999 by G.\ Baumslag \cite{Baumslag1999b}. Whether residually nilpotent one-relator groups have decidable conjugacy problem also seems to be an open problem.

\subsection{The isomorphism problem for torsion-free one-relator groups}\label{Subsec:iso-problem}

\

\epigraph{\textit{The isomorphism problem is perhaps the most important as well as the most intractable problem in the theory of one-relator groups.}}{---G.\ Baumslag, 1971 \cite[p. 75]{Ba71}}

\noindent We have already outlined some of the results about the isomorphism problem for one-relator groups with torsion in \S\ref{Subsec:Torsion-isomorphism-problem}. In particular, we know that in this case, the problem is decidable. By contrast, the isomorphism problem for torsion-free one-relator groups remains wide open. In this section, we will discuss some partial results, as well as some particularly difficult instances of the problem. Before beginning, we mention that by far the two largest classes of one-relator groups for which the isomorphism problem is decidable are one-relator groups with torsion (see \S\ref{Subsec:Torsion-isomorphism-problem}) and one-relator groups with non-trivial centre (see \S\ref{Subsec:Centre}). Beyond this, few truly general results are known: the quotation by G.\ Baumslag above remains as true today as it was 50 years ago. 

Nevertheless, we can say some generalities. First, recall that two one-relator groups $\pres{}{A}{r=1}$ and $\pres{}{A}{s=1}$ are said to be (\textit{Nielsen}) \textit{equivalent} if there is an automorphism of $F_A$ mapping $r$ to $s$. Of course, any two Nielsen equivalent one-relator groups are isomorphic, but whether the converse holds was for some time an open problem (Conjecture~\ref{Conj:Magnus-conj}, sometimes called the Magnus Conjecture, cf.\ e.g.\ \cite[p. 401]{Magnus1966}). We have seen that Whitehead proved the conjecture true in 1936 in the case that the one-relator group is free (Theorem~\ref{Thm:Whitehead-theorem}. However, for general one-relator groups the conjecture was disproved in 1970 by Zieschang \cite{Zi70} and independently in 1971 by McCool \& Pietrowski \cite{McCool1971}. The examples by McCool \& Pietrowski come from torus knot groups. Let $p, q \in \Z$ with $|p|, |q| > 1$, and let $T_{p,q} = \pres{}{x,y}{x^p = y^q}$. If $p = st + 1$ for some $s, t \in \N$ with $t \neq 1$, then the authors proved that $T_{p,q} \cong \pres{}{a,b}{b = (a^p b^{-s})^t}$, but that these presentations are not equivalent. Zieschang's examples also come from $T_{p,q}$. Such counterexamples produce finite classes of isomorphic but non-equivalent one-relator groups. This was extended in 1976 by Brunner \cite{Br76} (cf.\ also \cite{Brunner1980}), who proved that if for $r \geq 0$ we define
\begin{equation}
G_r = \pres{}{x,y}{y^{-1} x^{-2^r}yxy^{-1} x^{2^r} y = x^2},
\end{equation}
then $G_r \cong G_s$ for all $r, s \geq 0$, but no two such groups are Nielsen equivalent for $r \neq s$. This gives an even stronger failure of the Magnus Conjecture. Thus, it would be interesting to know whether it is decidable if a given one-relator group has infinitely many pairwise non-Nielsen-equivalent one-relator presentations. In principle this can be decidable even if the isomorphism problem is undecidable, and vice versa. 

The isomorphism problem, and the above failure of Magnus' Conjecture, is closely related to finding non-equivalent generating sets of one-relator groups, i.e.\ generating sets lying in different \textit{$T$-systems}. We discussed this subject in \S\ref{Subsec:Torsion-isomorphism-problem}, and note that there is also some literature on this subject in the torsion-free case (beyond what is cited above). For example, Zieschang's article \cite{Zi70} is primarily concerned with extending Nielsen's method of finding a basis for a subgroup in free groups to amalgamated free products. This is not possible in general, and Zieschang was able to point out several obstructions to this together with sufficient conditions on the amalgam for the method to work; such obstructions are also the source of the failure of Magnus' Conjecture. Similar conditions were given in 1978 by Peczynski \& Reiwer \cite{PeczynskiReiwer1978} for HNN-extensions. The Nielsen method in this broader context of amalgams and HNN-extensions was surveyed in \cite{Rosenberger1979}, in which it is among other things used to give a proof of the Tits alternative for $\operatorname{SL}_2(\R)$. Nielsen's method was extended to all Bass--Serre theory in \cite{Weidmann2002}; see also \cite{Fine1995} for a survey. We note that even Baumslag--Solitar groups $\BS(m,n)$ (see \S\ref{Subsec:Baumslag-Solitar}) can also be elusive from the point of view of their generating sets. In particular, the group $\BS(2,3) = \pres{}{x,y}{x^{-1}y^2x = y^3}$ requires more than one relation to define it in terms of the generators $x$ and $y^4$ (in this connection, see also \cite{Piccard1971}). Dunwoody \& Pietrowski \cite{Dunwoody1973} gave infinitely many non-equivalent generating sets of the trefoil knot group $\pres{}{a,b}{a^2 = b^3}$. These generating sets do not, however, have one-relator presentations.

One can also consider the isomorphism problem with respect to particular classes. For example, Whitehead's Theorem (Theorem~\ref{Thm:Whitehead-theorem}) implies that it is decidable whether a one-relator group is free or not; thus, the isomorphism problem is decidable for one-relator groups relative to free groups. Similarly, one can prove using classical methods that the isomorphism problem is decidable for one-relator groups relative to the class of cyclically pinched one-relator groups, i.e.\ cyclic amalgams of free groups $F_n \ast_{\Z} F_m$. This was proved for $n = m = 2$ and some constraints on the amalgam by Rosenberger \cite{Ro82}, and finally in general by Rosenberger \cite{Rosenberger1994} in 1994 (cf.\ also \cite{Zieschang1977} and \cite{Collins1978}). The strategy was to prove that such groups have only finitely many $T$-systems of generators (as in \S\ref{Subsec:Torsion-isomorphism-problem}), yielding in particular the following result:

\begin{theorem}[Rosenberger 1994, \cite{Ro82, Rosenberger1994}]
The isomorphism problem for one-relator groups is decidable relative to the class of cyclically pinched one-relator groups.
\end{theorem}

Another reason for the difficulty of the isomorphism problem in torsion-free one-relator groups comes from parafree groups (discussed in \S\ref{Subsec:lower-central-series}). In this direction, Chandler \& Magnus \cite{Chandler1982} mention a particular family of parafree one-relator groups: 
\begin{equation}\label{Eq:Gij-parafree}
G_{i,j} = \pres{}{a,b,c}{a = [c^i, a][c^j, b]} \quad (i, j \in \N)
\end{equation}
as being particularly difficult. Finding necessary and sufficient conditions on the parameters $i, j \in \Z$ for pairs of groups in the family $G_{i,j}$ to be isomorphic is particularly difficult, and for decades, only sporadic isomorphism results were proved for these groups (see e.g.\ \cite{Fine1995, Fine1995}, and \cite{BCH04} for an overview). Finally, the isomorphism problem for the family of groups $\{ G_{i,j} \mid i, j \in \N \}$ was solved in 2021 by Chen \cite{Chen2021}, using character varieties: it turns out that the group $G_{i,j}$ is uniquely determined by the pair $(i,j)$ when $i, j > 0$. In a similar vein, Meskin \cite{Meskin1968, Meskin1969} in his PhD thesis studied the isomorphism problem for the one-relator groups 
\[
\pres{}{x_1, x_2, \dots, x_t}{x_1^{m_1} x_2^{m_2} \cdots x_n^{m_t} = 1}
\]
where $t \geq 2$ and $m_1, \dots, m_t$ are arbitrary integers. These groups are parafree when $t>3$ and $\gcd(m_1, \dots, m_t) = 1$. Already \textit{within} this class of one-relator groups the isomorphism problem is difficult, but was eventually fully solved in 1975 by Meskin \cite{Meskin1975}: the integers $m_1, \dots, m_t$, up to permutation, determine the group uniquely.

In practice, there are many useful group invariants that can distinguish a given pair of torsion-free one-relator groups. The most basic invariant comes from the abelianization. The second homology group also provides some information: it is a consequence of Lyndon's Identity Theorem (\S\ref{Subsec:LyndonIT}) that if $G = \pres{}{A}{r=1}$, then $H_2(G, \Z) = \Z$ if and only if $r \in [F_A, F_A]$. In this case, when the second homology group is non-trivial, there is much more to be said. This was explored by Horadam \cite{Horadam1981b, Horadam1981} in 1981, who provides a simple and straightforward test for non-isomorphism (later expanded on in \cite{Horadam1982a, Horadam1982b}). For example, as in \cite[p. 197]{Horadam1981}, if we let $A = \{ x_1, \dots, x_4\}$, and take 
\begin{align*}
r_1 &= [x_3, x_4] [x_1, x_2] [x_1, x_4^2] [x_3, x_2] [x_2, x_4] [x_3, x_4]^3, \\
r_2 &= [x_3, x_4] [x_1, x_2] [x_1, x_4^2] [x_3, x_2] [x_2, x_4] [x_3, x_4]^2,
\end{align*}
the only difference being the final exponent, and set $G_i = \pres{}{A}{r_i=1}$ for $i=1, 2$, then her methods prove that $G_1 \not\cong G_2$. The technique can also be used to show that there exist infinitely many pairwise non-isomorphic one-relator groups of the form $\pres{}{a,b}{r = 1}$ where $r \in [F_2, F_2]$, which is not trivial to demonstrate otherwise. We note that her methods were also strengthened later by Gromadzki \cite{Gromadzki1985}.

There is more to be said on using finite quotients, and whether the set of all finite quotients of a residually finite one-relator group uniquely determines the group up to isomorphism (\textit{profinite rigidity} of one-relator groups). In this direction, we mention only \cite{MS95}, who prove that the groups $\BS(1,k)$ for $k \in \Z$ are profinitely rigid among all one-relator groups, i.e.\ if a residually finite one-relator group has the same finite images as $\BS(1,k)$ for some $k \in \Z$, then $G \cong \BS(1,k)$.

\subsection{More on the word problem}\label{Subsec:More-on-WP}

\noindent The first major result on one-relator groups was the \textit{Freiheitssatz}, and one of the first application of this result was the decidability of the word problem in all one-relator groups. It is therefore remarkable that since these results, proved almost a century ago, few general results on the word problem in one-relator groups have appeared. For example, it remains an open problem whether it can be decided in polynomial time. There also exist examples of one-relator groups with very rapidly growing Dehn functions, but nevertheless having word problem decidable in polynomial time \cite{MUW11}. In this section, we will discuss some open problems closely related to the word problem in one-relator groups.

One approach to better understand time complexity of the word problem in one-relator groups is via \textit{finite complete rewriting systems}. We give brief definitions here, referring the reader to the monographs \cite{Jantzen1988, Otto1991} for significantly more thorough treatments. Let $A$ be a finite set, $A^\ast$ the free monoid on $A$, and let $\mathcal{R} \subseteq A^\ast \times A^\ast$ be a set of pairs of words, which we call \textit{rules}. The set $\mathcal{R}$ of rules induces a relation $\xrightarrow{}_\mathcal{R}$ on $A^\ast$, where $w_1 \xrightarrow{}_\mathcal{R} w_2$ if and only if there exist words $x, y \in A^\ast$ and some rule $r = (u, v) \in \mathcal{R}$ such that $w_1 \equiv x u y$ and $w_2 \equiv x v y$. The reflexive and transitive closure of $\xrightarrow{}_\mathcal{R}$ is denoted $\xrightarrow{}_\mathcal{R}^\ast$. If there is no infinite chain $u_1 \xrightarrow{}_\mathcal{R} u_2 \xrightarrow{}_\mathcal{R} \cdots$, then we say that $\mathcal{R}$ is\textit{terminating}. For $u \in A^\ast$, if there exists some $v \in A^\ast$ such that $u \xrightarrow{}_\mathcal{R} v$, then we say that $u$ is reducible, and otherwise \textit{irreducible}. If $\mathcal{R}$ is a terminating rewriting system in which for every word $w \in A^\ast$ there exists a unique irreducible $w' \in A^\ast$ such that $w \xrightarrow{}^\ast_\mathcal{R} w'$, then we say that $\mathcal{R}$ is \textit{complete} (cf.\ \cite[Prop.~1.1.25]{Jantzen1988}). Every rewriting system $\mathcal{R}$ induces an equivalence relation $\leftrightarrow^\ast_\mathcal{R}$, being the least equivalence relation containing $\mathcal{R}$. In fact, this is a congruence, and $A^\ast / \leftrightarrow^\ast_\mathcal{R}$, which we call the monoid \textit{associated} to $\mathcal{R}$, is simply the monoid presented by $\pres{Mon}{A}{R}$. If a group $G$ is isomorphic to the monoid associated to a finite complete rewriting system, then we say that $G$ \textit{admits} a finite complete rewriting system. Of course, any group admitting a finite complete rewriting system has decidable word problem: simply rewrite any word to an irreducible one by successively applying the rules of the system.

Thus a finite complete rewriting system for a group is a particularly nice type of monoid presentation for the group, giving rise to a simple solution to the word problem. This leads to the following natural question: 

\begin{problem}[{Otto \& Kobayashi, 1997 \cite[Q3]{KO97}}]
Does every one-relator group admit a finite complete rewriting system? \label{Prob:FCRS-for-OR}
\end{problem}

If Problem~\ref{Prob:FCRS-for-OR} has a positive answer, then if the rules are concrete enough, it may be possible to gain upper bounds on the complexity of the word problem in any one-relator group. On the other hand, the property of admitting a finite complete rewriting system does not, in general, entail any upper bound (beyond recursiveness) on the complexity of the word problem \cite{Bauer1984}. Regarding a \textit{negative} answer to Problem~\ref{Prob:FCRS-for-OR}, we first note that there exists a finitely presented monoid $\pres{}{A}{R}$ which does not admit a finite complete rewriting system $\mathcal{R} \subseteq A^\ast \times A^\ast$, but which does admit a finite complete rewriting system over some other generating set \cite{Otto1984}. Thus Problem~\ref{Prob:FCRS-for-OR} is intricately linked with the problem of generating sets and alternate presentations of one-relator groups. Second, it seems difficult to even approach the question of how one might give a negative answer to Problem~\ref{Prob:FCRS-for-OR}. However, it turns out that finite complete rewriting systems are closely related to homological finiteness properties. Squier \cite{Squier1987} proved in 1987 that every monoid admitting a finite complete rewriting system has the homological finiteness property $\FP_3(\Z)$, later strengthened to $\FP_\infty(\Z)$ by Kobayashi \cite{Kobayashi1990}, cf.\ also \cite{Anick1986, Groves1990, Squier1994}. Since there are finitely presented groups with decidable word problem but which are not $\FP_3(\Z)$ \cite{Stallings1963}, there are finitely presented groups with decidable word problem but not admitting any finite complete rewriting system. Whether the same is true for one-relator groups is thus the crux of Problem~\ref{Prob:FCRS-for-OR}. Lyndon's Identity Theorem (Theorem~\ref{Thm:Lyndon-Identity}) shows that all one-relator groups are $\FP_\infty(\Z)$, showing that any negative answer to Problem~\ref{Prob:FCRS-for-OR} would have to involve some fundamentally new technique. We also remark here that there exist finitely presented groups of cohomological dimension $2$ with undecidable word problem \cite{Collins1999}, and thus not admitting a finite complete rewriting system. 

Another angle on the word problem comes from embedding one-relator groups into finitely presented simple groups: any finitely presented simple group has decidable word problem by Kuznetsov's elegant algorithm \cite{Kuznetsov1958}.\footnote{If $G = \pres{}{A}{R}$ is a finitely presented simple group and $w \in F_A$, then to decide if $w = 1$ in $G$, run the following two procedures in parallel: (1) enumerate all words equal to $1$ in $G$; (2) enumerate the elements of $G_w = \pres{}{A}{R, w=1}$. Then (1) terminates if $w = 1$, and (2) terminates if $w \neq 1$ in $G$, since $G$ being simple then implies that $G_w$ is trivial.} Recall that the classical Boone--Higman Theorem \cite{Boone1974} from 1974 states that a finitely generated group $G$ has decidable word problem if and only if there exists a simple group $S$ and a finitely presented group $K$ such that $G \leq K \leq S$. The Boone--Higman Conjecture asks whether there is furthermore some such $S$ and $K$ exist with $K = S$, i.e.\ whether every finitely generated group with decidable word problem embeds in some finitely presented simple group. This conjecture remains open, but it has recently seen remarkable progress: Belk \& Bleak \cite{Belk2023} proved in 2023 that every finitely generated hyperbolic group embeds in some finitely presented simple group. Thus, for example, by combining this with the B.\ B.\ Newman Spelling Theorem (\S\ref{Subsec:BBNewman}) we have that \textit{every one-relator group with torsion embeds into a finitely presented simple group}. Whether this can be extended to all one-relator groups seems a next natural step for the Boone--Higman Conjecture.

\begin{problem}
Can every one-relator group be embedded into a finitely presented simple group? 
\end{problem}

Extensions of the Boone--Higman Theorem to other decision problems also exist, including for the order problem \cite{Boone1975} and the conjugacy problem \cite{Sacerdote1977}, giving rise to analogous questions for one-relator groups. Dual to the above problem is the following classical problem:

\begin{problem}[{Baumslag, 1964 \cite[p. 391]{Ba64}}]
Can a one-relator group contain a non-abelian simple subgroup?
\end{problem}

Finally, we make a brief remark on membership problems, being in some ways a natural step following the word problem. The \textit{subgroup membership problem} for a group $G = \pres{}{A}{R}$ asks, on input a set of words $w_1, \dots, w_n \in F_A$ and a word $w \in F_A$, whether or not $w$ belongs to the subgroup of $G$ generated by $w_1, \dots, w_n$. Thus the word problem is simply the membership problem for the trivial subgroup. Recall that Magnus (\S\ref{Subsec:Wordproblem-Magnus}) proved that membership in Magnus subgroups of one-relator groups is decidable. However, for arbitrary finitely generated subgroups, deciding membership remains an unsolved problem in general: 

\begin{problem}\label{Prob:Subgroup-membership}
Is the subgroup membership problem decidable in every one-relator group?
\end{problem}

Any subgroup separable one-relator group has decidable subgroup membership problem. Thus a large class of one-relator groups with decidable subgroup membership problem is formed by the one-relator groups with non-trivial centre (see \S\ref{Subsec:Centre} and particularly Theorem~\ref{Thm:Centre-OR-are-subgroup-separable}). However, there are non-subgroup separable one-relator groups which nevertheless have decidable subgroup membership problem. e.g.\ the groups $H_{m,k} = \pres{}{a,b}{a^{-m}ba^m = b^k}$, with $m, n \in \N$, of Novikova \cite{Novikova1998}. Problem~\ref{Prob:Subgroup-membership} can also be extended to instances of the sub\textit{monoid} membership problem, defined analogously. Here undecidability appears rather quickly, and many one-relator groups have undecidable submonoid membership problem, e.g.\ $\pres{}{a,b}{[ab, ba] = 1}$ or $\pres{}{a,b}{[a^2, b^2] = 1}$. There are also positive one-relator groups with undecidable submonoid membership problem \cite{Foniqi2024}. However, all currently known such examples arise from embedding the right-angled Artin group $A(P_4)$, which has undecidable submonoid membership problem (see \cite{Lohrey2008}). Such subgroups obviously cannot appear as subgroups of hyperbolic groups. Thus whether there is some hyperbolic one-relator group with undecidable submonoid membership problem, and whether there is a one-relator group with torsion with undecidable submonoid membership problem, both remain open problems, being Problems~20.68 \& 20.69, respectively, in the Kourovka Notebook \cite{Khukhro2024}. 

\subsection{Two-relator groups}

\

\epigraph{\textit{So my survey ends fittingly, as it began, with the work of Wilhelm Magnus.}}{---G.\ Baumslag \cite[p. 55]{Ba86}}

\noindent This survey has told the story of one defining relation. What question could now be more natural than: what about two-relator groups? Unfortunately, here much less is known. The question of extending the methods of one-relator groups to two relators dates back to the very foundations of the former subject, being present already in an appendix to Magnus' 1930 article \cite[\S8]{Ma30}. Very little has been produced in way of results for two-relator groups \textit{qua} two-relator groups. Even the problem of classifying all finite two-relator groups is still a mystery. Certainly, any finite two-relator group must by basic linear algebra also be generated by two elements, and thus admit a balanced (i.e.\ deficiency zero) presentation. It follows that any finite two-relator group $G$ must satisfy $H_2(G, \Z) = 0$. Beyond such generalities, the class of finite two-relator groups seems, at present, far too broad to understand in its entirety, and contains groups such as $\SL(2,5) \times \SL(2,25)$, as proved in \cite{Campbell1986}, and much computational work has been directed towards finding two-generator, two-relator presentations for similar groups, see e.g.\ \cite{Campbell1986, Robertson1989}. Even the finiteness problem for two-relator groups, i.e.\ the isomorphism problem for two-relator groups with respect to the class of finite two-relator groups, is at present wide open. 

The word problem, conjugacy problem, and isomorphism problem are all open for two-relator groups. The latter two are less surprising, since they also remain open for one-relator group, and currently known methods for encoding undecidable problems in groups tend to produce a relatively large number of relators. Indeed, as for the word problem, as discussed in \S\ref{Subsec:Undecidability} the currently smallest known number of defining relators for a group with undecidable word problem is 12, as proved by Borisov \cite{Borisov1969}. His proof (simplified by Collins \cite{Collins1972}) uses the fact that the word problem in any $n$-relation monoid can be encoded into an $(n+9)$-relator group, and Matiyasevich \cite{Matiyasevich1967} gave an example of a three-relation monoid with undecidable word problem. It seems very plausible, but remains an open problem, that there is a two-relation monoid with undecidable word problem; and recent trends (see e.g.\ \cite{Gra20, Foniqi2024}) have indicated that there may even be a \textit{one}-relation monoid with undecidable word problem (see \cite{NybergBrodda2021b} for an introduction to this problem). Nevertheless, we see that Borisov's method would only yield a 11 resp.\ 10-relator group with undecidable word problem if these undecidability results would hold, respectively. Thus, if the word problem is undecidable in two-relator groups, then to prove it would require some fundamentally new way of encoding undecidability into groups. If it is decidable, then some fundamentally new way of solving word problems must be developed, perhaps by extending the Magnus hierarchy to this setting. In either case, much exciting mathematics will surely follow!

\chapter{Recent progress}
\label{chap:survey2}

\section{Introduction}

The origins of combinatorial group theory are often identified with Dehn's solutions to the word and conjugacy problem for fundamental groups of closed hyperbolic surfaces. Dehn presented two approaches, one utilising the geometry of the hyperbolic plane $\mathbb{H}^2$ \cite{De11} and another abstracting the geometric properties to purely combinatorial properties of their one-relator presentations \cite{De12}:
\[
\Sigma_g = \left\langle a_1, b_1, \ldots, a_g, b_g \, \middle\vert \, \prod_{i=1}^g[a_i, b_i]\right\rangle, \quad g\geqslant 2.
\]
Generalisations of Dehn's ideas would eventually lead to the very fruitful study of small cancellation groups and, later, hyperbolic groups. 

The theory of amalgamated free products, developed by Schreier in the decade that followed \cite{Sc27}, provided a third method of solving the word problem for surface groups. This algebraic method was picked up by Dehn's doctoral student Magnus who used it to investigate the much more general class of \emph{one-relator groups}; that is, groups admitting a presentation
\[
G = \langle S \mid w \rangle
\]
with a single defining relation. Magnus made two fundamental initial contributions to the theory by proving the Frieheitssatz in 1930 \cite{Ma30} and solving the word problem for all one-relator groups in 1932 \cite{Ma32}. These results initiated the systematic study of one-relator groups which has continued to develop and mature over the last century. After Magnus' foundational work, one-relator groups quickly became a test-piece for new conjectures in group theory. Indeed, one-relator groups exhibit a wide variety of behaviours so that few properties hold for the whole class and characterising one-relator groups with a certain property is often a nuanced problem. This has the corollary that even fundamental problems like Dehn's conjugacy and isomorphism problems are very difficult to solve for all one-relator groups and in fact remain wide open. On top of this, not a single theorem is known to hold for all two-relator groups that does not already hold for all finitely presented groups. It is therefore remarkable that one-relator groups admit a cohesive theory at all and even more so that it is sufficiently deep and developed that one can hope to approach most difficult problems.

Until the 80s, group theory remained a mostly algebraic and combinatorial subject. The focus in group theory then shifted towards more geometrically flavoured problems when Gromov introduced hyperbolic groups in \cite{Gr87}. By the very nature of their definition, one-relator groups do not admit any obvious intrinsic geometry and so this shift initially had little impact on the theory of one-relator groups. Indeed, the geometry of one-relator groups has proven to be surprisingly complicated, resisting any form of decisive classification for some time.

On the other hand, one-relator groups have turned out to be very amenable for study via topological and homological means. The recent developments in $\cat(0)$ cube complexes, 2-complexes with non-positive immersions and $L^2$-invariants have all had a significant impact on one-relator groups, resulting in new insights and resolutions of old conjectures that had seen little progress for many years. This includes, but is not limited to, the residual finiteness conjecture (posed by Baumslag in \cite{Ba67} and solved by Wise in \cite{Wi21}), the coherence conjecture (posed by Baumslag in \cite{Ba74} and solved by Jaikin-Zapirain and the first author in \cite{JZL23}) and the virtually free-by-cyclic conjecture (posed by Baumslag in \cite{Ba86} and solved by Kielak and the first author in \cite{KL24}). 

The purpose of this chapter is to survey the current state of the art in the theory of one-relator groups with a particular emphasis on homological, topological and geometric aspects of one-relator groups, topics which, as mentioned above, have seen considerable advancement in the last decade. Some of the content of this survey is based on work of the first author, but the majority is based on recent work of Andrei Jaikin-Zapirain, Dawid Kielak, Lars Louder, Henry Wilton, Daniel Wise and many others, without whom this article would not warrant writing.

We aim to complement other pre-existing treatments of one-relator groups which are more classical and combinatorial in nature, see \cite{Ba64,Ba74,Ba86,LS01,MKS04,BFR19} for example. We include proofs of a few selected results, but mostly content ourselves with sketch proofs and ideas. Throughout, we also present open problems and directions for future research. As such, we hope this document can serve both as an invitation to the theory of one-relator groups and as a point of reference for anyone wishing to contribute and expand on the topics presented here.

Many results and questions we discuss here in reality apply (or can be generalised) to broader classes of groups which include:
\begin{itemize}
\item One-relator products of groups (possibly with extra conditions on the factors),
\item Fundamental groups of 2-complexes with non-positive immersions,
\item Groups of cohomological dimension two with vanishing second $L^2$-Betti number.
\end{itemize}
Indeed, often results for these classes of groups arise from direct generalisations of arguments used to prove the same result for the simplest non-trivial case in the corresponding class, one-relator groups. For the sake of simplicity, in this article we mostly only state the versions of results for one-relator groups, pointing the reader towards references for further information.

\subsection{A brief description of the content}

In \cref{sec:one-relator_splittings} we begin by discussing Magnus' main tool for studying one-relator groups, the Magnus hierarchy. We discuss subsequent improvements to the hierarchy and cover a topological version using two-complexes in detail, as developed by the first author in \cite{Lin22}. The treatment includes several examples and a somewhat complete proof of the Freiheitssatz.

In \cref{sec:splittings} we cover what is known about splittings of one-relator groups. We begin with an in depth analysis of the action of a one-relator group on the Bass--Serre tree associated with the splittings arising from its hierarchy. We then provide a sufficient condition for when a one-relator group is free-by-cyclic which we conjecture to also be necessary. We then consider the more specific case of when a one-relator group is \{finitely generated free\}-by-cyclic, which was completely determined by Moldavanski\u{\i} and later by Brown. Finally, we cover the work of Friedl--Tillmann and Henneke--Kielak on polytopes and splitting complexities for one-relator groups, finishing with some open problems on JSJ and amalgam decompositions.

In \cref{sec:cohomology} we cover (co)homological topics and properties of the group ring of one-relator groups. We discuss Lyndon's identity theorem and the implications this has for the (co)homology of a one-relator group. We then cover recent results on groups with cyclic relation module and connections with the relation gap and relation lifting problems. We also cover $L^2$-invariants, the resolution of the strong Atiyah conjecture for one-relator groups by Jaikin-Zapirain--L\'{o}pez-\'{A}lvarez and embeddings of group rings of torsion-free one-relator groups into division rings. Finally, we sketch a proof of the fact that the rational group ring of a one-relator group is coherent.

In \cref{sec:subgroups} we cover properties of subgroups of one-relator groups. We begin by covering the non-positive immersions property as introduced by Wise. We then cover work of Louder--Wilton who characterise when a one-relator group has the stronger property of negative immersions. This then leads us to connections between the $k$-freeness property and primitivity rank of elements of the free group. Combining these results with some homological properties, we then sketch a proof of the fact that all one-relator groups are coherent, covering also a strengthening of this theorem for the subclass of 2-free one-relator groups. Finally, we cover what is known about when a one-relator group is virtually free-by-cyclic, Mel'nikov's conjecture, surface subgroups and subgroups of infinite index.

In \cref{sec:geometry} we cover geometric properties of one-relator groups. We discuss Gersten's conjecture on the hyperbolicity of one-relator groups, discuss a method of proving a one-relator group is hyperbolic which leads to a reduction of Gersten's conjecture to the case of primitive extension groups and mention results that are known to hold for hyperbolic one-relator groups. Then, we cover $\cat(0)$ and cubulated one-relator groups, demonstrating further the difficulty of geometric characterisations of one-relator groups with some pathological examples due to Gardam--Woodhouse. Finally, we discuss Wise's quasi-convex hierarchies, Wise's solution of the residual finiteness conjecture and an extension to the class of 2-free one-relator groups.

In \cref{sec:residual} we cover residual properties such as residual solvability, residual nilpotence, parafree groups and profinite completions.

In \cref{sec:algorithms} we cover the work that has been carried out on the complexity of the word problem for one-relator groups, the conjugacy problem, various membership problems, the normal root problem, generating sets of one-relator groups and the isomorphism problem.

\section{One-relator splittings and the hierarchy}
\label{sec:one-relator_splittings}

\subsection{The Magnus--Moldavanski\u{\i}--Masters hierarchy}

If $S$ is a set, we shall write $S^{\pm1}$ to denote $S\sqcup S^{-1}$, the union of $S$ with the set of its formal inverses and by $F(S)$ the free group freely generated by the set $S$. A \emph{one-relator group} is a group of the form 
\[
G = F/\normal{w}
\]
where $F$ is a free group and where $\normal{w}$ denotes the normal closure of the single element $w\in F$. We will always be assuming that $w\neq 1$. The most general tool used to study one relator groups is known as the \emph{Magnus hierarchy}, developed by Magnus in his thesis \cite{Ma30} in order to prove the Freiheitssatz.

\begin{theorem}[Freiheitssatz]
\label{original_frei}
Let $F$ be a free group over a generating set $S$ and let $w\in F$ be a cyclically reduced element. If $T\subset S$ is a subset such that $w\notin \langle T\rangle$, then the natural map
\[
F(T) \to G = F/\normal{w}
\]
is injective.
\end{theorem}

A free subgroup of a one-relator group arising from \cref{original_frei} is called a \emph{Magnus subgroup}. The reader is directed to \cref{Sec:3-Magnus} for more on the Freiheitssatz.

The idea of the original proof of \cref{original_frei} is to find a suitable epimorphism $G\to \Z$ in which the kernel can be decomposed as an infinite amalgamated free product of one-relator groups with strictly shorter relator length and then use the explicit structure of this decomposition, together with induction, to prove that the subset of generators we chose freely generate a free subgroup. There is a caveat with this sketch; such an epimorphism exists only when a generator appears with exponent sum zero in the relator, and this cannot always be the case. To get around this, Magnus embeds the one-relator group into another suitably chosen one-relator group in which some generator does appear with exponent sum zero. Then some technical arguments allows one to show that if the Freiheitssatz holds for the larger one-relator group, then it holds for the original one-relator group. Several proofs of \cref{original_frei} have appeared that do not (explicitly) make use of Magnus' hierarchy. See, for instance, \cite{Ly72} and \cite{LW22}.

The ideas in Magnus' proof are extremely powerful as they subsequently allowed Magnus to prove several strong results on the structure of one-relator groups, such as the decidability of the word problem \cite{Ma32}. As is often the case, in order to make an inductive argument work, one needs to prove something stronger than the desired statement. Indeed, in order to solve the word problem in HNN-extensions or amalgamated free products, one needs both decidability of the word problem in the base groups as well as decidability of the membership problem in the associated subgroups. Thus, the actual statement that Magnus proved inductively in order to solve the word problem for one-relator groups was that the membership problem for any Magnus subgroup is decidable. Most other results on one-relator groups follow a similar structure, usually one proves that the result is true `relative' (in an appropriate sense) to Magnus subgroups by induction on the length of a hierarchy, highlighting the importance of Magnus subgroups.

Over the many years that have passed since Magnus' first paper, there have been several improvements to Magnus' original version of the hierarchy, the most well known being due to Moldavanski\u{\i} \cite{Mo67} which utilises HNN-extensions. Explicitly, Moldavanski\u{\i} noticed that the amalgamated product decompositions of the kernel of some epimorphism $G\to \Z$ that arose in Magnus' work could be bundled together to obtain a HNN-decomposition $G\isom H*_{\psi}$, with the original epimorphism being recoverd by quotienting by the normal closure of $H$. In particular, if $G$ is a one-relator group, then Moldavanski\u{\i} proved that there is a diagram of one-relator groups
\[
\begin{tikzcd}
G = G_0 \arrow[r, hook]		   & G'_0 \\
G_1	 \arrow[ur, hook] \arrow[r, hook] & G'_1 \\
\cdots \arrow[ur, hook] \arrow[r, hook] & \cdots \\
G_N \arrow[ur, hook]			& 
\end{tikzcd}
\]
where $G_N$ is a free product of a free group and a (possibly trivial) finite cyclic group and where $G_{i-1}' \isom G_i*_{\psi_i}$ where $\psi_i$ identifies a pair of Magnus subgroups of $G_i$. All embeddings and splittings can also be described very explicitly in terms of the presentations. We shall not do this here, we direct the reader to \cite{LS01} for more details. We remark that HNN-extensions had been introduced almost two decades after \cite{Ma32}. A variation of this version of the hierarchy was also proven by Mihalik--Tschantz \cite{MT92} who proved that all the embeddings above can be replaced with embeddings as (proper) free factors.

Expressing the Magnus hierarchy with HNN-extensions simplifies many of Magnus' original arguments; indeed, McCool--Schupp then used Moldavanski\u{\i}'s improvement to reprove many of Magnus' original results on one-relator groups in a short paper \cite{MS73}.

Another version of the hierarchy that is worth mentioning is due to Masters. In an unpublished prepreint, Masters \cite{Ma06} showed that one can do away with the embedding step in the hierarchy and obtained the following theorem.

\begin{theorem}[The Algebraic Hierarchy]
\label{algebraic_hierarchy}
Let $F$ be a finitely free generated group, $w\in F$ cyclically reduced and consider the one-relator group $G = F/\normal{w}$. There is sequence of finitely generated subgroups $F_N\leqslant \ldots \leqslant F_1\leqslant F_0 = F$ containing $w$ and isomorphisms $\psi_i\colon A_i\to B_i$ between Magnus subgroups of $G_i = F_i/\normal{w_i}$ for each $1\leqslant i\leqslant N$, such that $F_N\isom \Z$ and
\[
G_{i-1} \isom G_i*_{\psi_i}
\]
for each $0< i\leqslant N$.
\end{theorem}

The splittings that arise in this version are the same as those from the Magnus--Moldavanski\u{\i} hierarchy. However, in order to find them, Masters modifies the original presentation via Nielsen moves until a presentation with relator having exponent sum zero on some generator appears. This idea was already present in previous articles, the issue in using it to obtain a hierarchy had previously been that there was no argument to overcome the fact that the relator might get longer under this modification. The novelty of Masters proof is to show that, regardless, continuing to split one eventually reaches a free product of a free group with a finite cyclic group.

Some classical results follow directly from this version of the Magnus hierarchy, combined with general theory of HNN-extensions. However, the splittings are not given so explicitly in terms of the original presentation of the one-relator group making it more difficult to apply than the Magnus--Moldavanski\u{\i} hierarchy.

Yet another version of the hierarchy is provided by the first author in \cite{Lin22}. This hierarchy is phrased and proven in terms of 2-complexes. It both eliminates the embedding step from the Magnus--Moldavanski\u{\i} hierarchy and expresses the HNN-splittings very explicitly. Given that in this article we shall mostly be studying one-relator groups via 2-complexes, we shall cover this hierarchy in much more detail now. Moreover, this version implies \cref{algebraic_hierarchy}.

\subsection{The topological hierarchy and some classical results}
\label{sec:hierarchy}

Before reaching the topological hierarchy statement, let us first set up some notation and definitions. 

A \emph{graph} is a 1-dimensional CW-complex. A combinatorial map of graphs is one that sends 0-cells to 0-cells and open 1-cells homeomorphically to open 1-cells. An \emph{immersion} of graphs is a combinatorial map that is locally injective, which we denote by $\immerses$. If $\lambda\colon S^1\immerses\Lambda$ is an immersion of a cycle, denote by \emph{$\deg(\lambda)$} the maximal degree of a covering map $S^1\immerses S^1$ that $\lambda$ factors through. Note that $\deg(\lambda) = 1$ if and only if $\lambda_*(1)$ is not a proper power in $\pi_1(\Lambda)$, where here $1\in \Z = \pi_1(S^1)$ is a generator. If $\lambda = \sqcup_i\lambda_i$ where $\lambda_i \colon S^1\immerses \Lambda$, then write $\deg(\lambda) = \sum_i\deg(\lambda_i)$. Here we are implicitly putting a combinatorial structure on $S^1$. When $\lambda$ is an immersion of cycles $\bbS = \bigsqcup S^1$, we will write $|\lambda|$ to denote the number of 1-cells in $\bbS$.

A \emph{2-complex} will be used throughout to mean a two-dimensional CW-complex $X$ in which all attaching maps of 2-cells are combinatorial immersions. We will usually write $X = (\Lambda, \lambda)$ where $\Lambda = X^{(1)}$ is the 1-skeleton and where $\lambda\colon \bbS = \bigsqcup S^1\immerses \Lambda$ is the attaching map for the 2-cells in $X$.

Our maps between 2-complexes will always be assumed to be \emph{combinatorial}, that is, they will send open $n$-cells homeomorphically to $n$-cells. An \emph{immersion} of 2-complexes is a locally injective combinatorial map which we shall usually use the arrow $\immerses$ to denote. If $Z\immerses X$ is an immersion of 2-complexes, it will be useful to bear in mind that there is an induced diagram of immersions
\[
\begin{tikzcd}
\mathbb{S}_{\Gamma} \arrow[r, loop->] \arrow[d, loop->] & \mathbb{S}_{\Lambda} \arrow[d, loop->] \\
\Gamma \arrow[r, loop->]                        & \Lambda                       
\end{tikzcd}
\]
where the map $\mathbb{S}_{\Gamma}\immerses\mathbb{S}_{\Lambda}$ is a homeomorphism on each component.

A natural 2-complex to associate to a one-relator group is its presentation complex. That is, the 2-complex with a single 0-cell, a 1-cell for each generator and a single 2-cell whose attaching map spells out the relator. However, for the statement and proofs of results in this article it will be more useful to consider arbitrary 2-complexes with a single 2-cell.

\begin{definition}
A connected 2-complex $X = (\Lambda, \lambda)$ is a \emph{one-relator complex} if $\lambda$ is an immersion of a single cycle.
\end{definition}

In direct analogy with the definition of Magnus subgroups, we define Magnus subgraphs.

\begin{definition}
If $X = (\Lambda, \lambda)$ is a one-relator complex, say a connected subgraph $\Gamma\subset \Lambda$ is a \emph{Magnus subgraph} if $\Gamma$ does not support $\lambda$.
\end{definition}

The aim of the next sections is to sketch a proof of topological versions of Magnus' original results. The Freiheitssatz has an obvious translation.

\begin{theorem}[Topological Freiheitssatz]
\label{topological_Freiheitssatz}
If $X = (\Lambda, \lambda)$ is a one-relator complex and $\Gamma\subset \Lambda$ is a Magnus subgraph, then the map on $\pi_1$ induced by inclusion
\[
\pi_1(\Gamma) \to \pi_1(X)
\]
is injective.
\end{theorem}

The Magnus hierarchy and variations do not have an immediate useful translation to the setting of 2-complexes. However, in \cite{Lin22}, the following theorem is proven. The statement is heavily influenced by Howie's towers \cite{Ho81}.

\begin{theorem}[The Topological Hierarchy]
\label{topological_hierarchy}
If $X = (\Lambda, \lambda)$ is a finite one-relator complex, then there exists a sequence of immersions
\[
X_N \immerses \ldots \immerses X_1\immerses X_0 = X.
\]
such that
\begin{enumerate}
\item For all $0\leqslant i\leqslant N$, $X_i$ is a finite one-relator complex.
\item $\pi_1(X_N)\isom\Z/n\Z$ where $n = \deg(\lambda)$.
\item For all $1\leqslant i\leqslant N$, we have 
\[
\pi_1(X_{i-1}) \isom \pi_1(X_{i})*_{\psi_{i}}
\]
where $\psi_{i}$ is induced by an isomorphism between Magnus subgraphs of $X_i$.
\end{enumerate}
\end{theorem}

\begin{remark}
One could drop the assumption that $X$ be finite in \cref{topological_hierarchy} at the expense of removing the word `finite' from the first conclusion and modifying the second conclusion to `$\pi_1(X_N) \isom F*\Z/n\Z$ for some free group $F$ and where $n = \deg(\lambda)$'. However, if $X$ is a one-relator complex, all the interesting information about $\pi_1(X)$ can be derived by only looking at the smallest one-relator subcomplex, which is necessarily finite.
\end{remark}

We now note some classical results which were proved using Magnus original hierarchy for one-relator groups. The proofs can all be simplified somewhat by using \cref{topological_hierarchy} and the explicit description of the splittings provided by \cref{one-relator_splitting}.

The first result is known as the \emph{Magnus property}. A group $G$ has the Magnus property if for any two elements $w, v\in G$, we have $\normal{w} = \normal{v}$ if and only if $w$ is conjugate to $v$ or $v^{-1}$. Magnus proved in \cite{Ma30} that free groups have the Magnus property. Below is the restatement in terms of one-relator complexes.

\begin{theorem}[Magnus property]
\label{Magnus_property}
Let $X = (\Lambda, \lambda_1)$ and $Y = (\Lambda, \lambda_2)$ be two one-relator complexes such that the identity map on $\pi_1(\Lambda)$ induces an isomorphism $\pi_1(X)\to \pi_1(Y)$. Then the identity map on $\Lambda$ extends to an isomorphism of one-relator complexes $X\to Y$.
\end{theorem}

Cockcroft characterised when a one-relator complex is aspherical in \cite{Co54}. His proof made use of Lyndon's identity theorem \cite{Ly50}, which we shall cover in \cref{sec:homology}.

\begin{theorem}[Asphericity]
\label{aspherical}
If $X = (\Lambda, \lambda)$ is a one-relator complex, then $X$ is aspherical if and only if $\deg(\lambda) = 1$.
\end{theorem}

Karrass--Magnus--Solitar characterised which elements in a one-relator group have finite order in \cite{KMS60}. The criterion is very easily checked by inspecting the relator.

\begin{theorem}[Torsion]
\label{finite_order}
If $X = (\Lambda, \lambda)$ is a one-relator complex, then $\pi_1(X)$ has torsion if and only if $\deg(\lambda)\geqslant 2$. Moreover, every element of finite order in $\pi_1(X)$ is conjugate to a root of $\lambda_*(1)$.
\end{theorem}

\cref{finite_order} follows from \cref{topological_hierarchy} by noting that an element of finite order acting on a tree (the Bass--Serre tree of the splittings) must have a fixed point.

\cref{finite_order} can be stated in terms of presentations as follows. Let $F$ be a free group, $w\in F$ not a proper power and $n\geqslant 1$ an integer. Then $G = F/\normal{w^n}$ has torsion if and only if $n\geqslant 2$. Moreover, all elements of finite order in $G$ are conjugate into the subgroup $\langle w\rangle$.

To conclude this section, we state a result due to Howie \cite{Ho82} and, independently, Brodski\u{\i} \cite{Bro84}. Recall that a group is \emph{locally indicable} if every non-trivial finitely generated subgroup admits an epimorphism to $\Z$. The proof of this theorem is not carried out using the Magnus hierarchy, rather using the method of towers.

\begin{theorem}[Local indicability]
\label{locally_indicable}
If $X = (\Lambda, \lambda)$ is a one-relator complex with $\deg(\lambda) = 1$, then $\pi_1(X)$ is locally indicable.
\end{theorem}

\subsection{One-relator splittings}

Before diving into the details of \cref{topological_hierarchy}, we now explain further where exactly these splittings come from. Firstly, we need a notion of complexity. This notion is related to the relator length as used by Magnus, and is also closely related to a notion of complexity used by Howie in \cite{Ho82}. We are following \cite{Lin22} in this section and those that follow.

\begin{definition}
If $\lambda\colon S^1\immerses \Lambda$ is an immersed cycle, denote by
\[
c(\lambda) = \frac{|\lambda|}{\deg(\lambda)} - \left|\Ima(\lambda)^{(0)}\right|
\]
the \emph{complexity of $\lambda$}. If $X = (\Lambda, \lambda)$ is a one-relator complex, the \emph{complexity of $X$} is the tuple
\[
c(X) = (c(\lambda), -\chi(X))
\]
which is given the dictionary order.
\end{definition}

\begin{remark}
\label{rem:base_case}
If $X = (\Lambda, \lambda)$ is a one-relator complex, then $c(\lambda) = 0$ if and only if $\lambda$ factors through an embedded cycle $S^1\injects \Lambda$. In particular, if $c(\lambda) = 0$, then $\pi_1(X) \isom F*\Z/n\Z$ where $F$ is a free group and $n = \deg(\lambda)$. Moreover, $c(X)$ reaches its minimal value $(0, 0)$ if and only if $\pi_1(X)\isom \Z/n\Z$.
\end{remark}

The following theorem explicitly identifies the splittings that arise in \cref{topological_hierarchy}. The statement is rather technical, but we will show with some examples how exactly one can apply it to compute hierarchies. Recall that a $\Z$-cover is a covering space $\rho\colon Y \to X$ such that $\deck(\rho)\isom \Z$. We shall be using the notion of graphs of spaces, as introduced by Scott--Wall \cite{SW79}.

\begin{theorem}[One-relator splittings]
\label{one-relator_splitting}
Let $X = (\Lambda, \lambda)$ be a finite one-relator complex and let $\rho\colon Y\to X$ be a connected $\Z$-cover with $t\in \deck(\rho)$ a generator. 
\begin{itemize}
\item (Existence of a one-relator $\Z$-domain) There exists a finite connected subcomplex $Z\subset Y$, called a \emph{$\Z$-domain}, with the following properties:
\begin{enumerate}
\item $Z$ is one-relator and $Z^{(1)}$ supports the attaching map of precisely one 2-cell in $Y$.
\item $Y = \bigcup_{i\in \Z}t^i(Z)$.
\item $Z_0 = t^{-1}(Z)\cap Z$ and $Z_1 = t(Z_0) = Z\cap t(Z)$ are Magnus subgraphs of $Z$.
\item For any $i>0$, we have $Z\cap t^i(Z)\subset Z\cap t^{i-1}(Z)$.
\end{enumerate}
\item (Graph of spaces/groups decomposition) If $Z\subset Y$ is any one-relator $\Z$-domain as above, then, denoting by $\iota\colon Z_0\injects Z$ the inclusion and by
\[
\mathcal{X} = Z \sqcup \left(Z_0\times[-1, 1]\right)/\{ \iota(z)\sim (z, -1),\, (t\circ\iota)(z) \sim (z, 1),\, z\in Z_0\},
\]
the map
\[
h\colon \mathcal{X}\to X
\]
given by $\rho$ when restricted to $Z$ or $Z_0\times \{i\}$ for all $i\in (-1, 1)$, is a homotopy equivalence. This homotopy equivalence induces an isomorphism:
\[
\pi_1(X) \isom \pi_1(Z)*_{\psi}
\]
where $\psi\colon \pi_1(Z_0)\to \pi_1(Z_1)$ is induced by $t\mid_{Z_0}\colon Z_0\to Z_1$.
\item (Decreasing complexity) There exists some one-relator $\Z$-domain $Z\subset Y$ as above with
\[
c(Z)<c(X).
\]
\end{itemize}
\end{theorem}

The first bullet point is \cite[Proposition 4.9]{Lin22}, the second is \cite[Proposition 4.3]{Lin22} and the third is \cite[Proposition 4.11]{Lin22}.

\begin{definition}
If $X$ is a one-relator complex, we will call a splitting $\pi_1(X)\isom \pi_1(Z)*_{\psi}$ as arising from \cref{one-relator_splitting} a \emph{one-relator splitting} of $\pi_1(X)$.
\end{definition}

\begin{definition}
If $X$ is a one-relator complex, a \emph{one-relator tower} for $X$ is a sequence of immersions of one-relator complexes
\[
X_N\immerses \ldots\immerses X_1\immerses X_0 = X
\]
such that for each $i$, $X_{i+1}$ is obtained from $X_i$ as in \cref{one-relator_splitting}.
\end{definition}

If $X$ is a 2-complex, the $\Z$-covers of $X$ are in bijection with non-trivial homomorphisms $\pi_1(X)\to \Z$, or rather, with non-zero elements of $H^1(X, \Z)$. The connected $\Z$-covers are in bijection with epimorphisms $\pi_1(X)\surjects\Z$, or rather, with basis elements of $H^1(X, \Z)$. We will make this bijection more explicit in \cref{sec:character_sphere}. When $X$ is a one-relator complex, such epimorphisms exist if and only if $\pi_1(X)$ is not isomorphic to a finite cyclic group. In particular, at least one epimorphism exists precisely when $c(X)>(0, 0)$. \cref{topological_hierarchy} now follows by \cref{one-relator_splitting} and induction.

\subsection{Explicit examples of the topological hierarchy}

Our first example will be the example which was the inspiration for \cref{one-relator_splitting}. The prototypical one-relator complex is a closed surface. So let $X$ be a closed surface, let $C\subset X$ be a non-separating simple closed curve and consider the epimorphism $\pi_1(X) \to H_1(X, \Z) \to \Z\cdot [C]$, where here $[C]$ denotes the homology class of $C$ in $H_1(X, \Z)$. We can always modify the 2-complex structure of $X$ so that $C$ actually lies in the 1-skeleton of $X$. Then if $\rho\colon Y\to X$ is the $\Z$-cover associated with this homomorphism, one can take $Z\subset Y$ from \cref{one-relator_splitting} to be the closure of $\rho^{-1}(X - C)$. Here $Z$ is a surface with two boundary components corresponding to two distinct lifts of $C$. If $t\in \deck(\rho)$ is a generator, then $Z_0 = t^{-1}(Z)\cap Z$ and $Z_1 = Z\cap t(Z)$ are precisely these two lifts. The space $\mathcal{X}$ from \cref{one-relator_splitting} is then obtained from $Z$ by attaching a cylinder connecting these two boundary components and expresses $\pi_1(X)$ as a HNN-extension over an infinite cyclic group. \cref{fig:genus_2} illustrates this example explicitly in the case $X$ is a genus two surface and $C$ is a simple closed curve going around one of the genus.

\begin{figure}
\centering
\includegraphics[scale = 0.5]{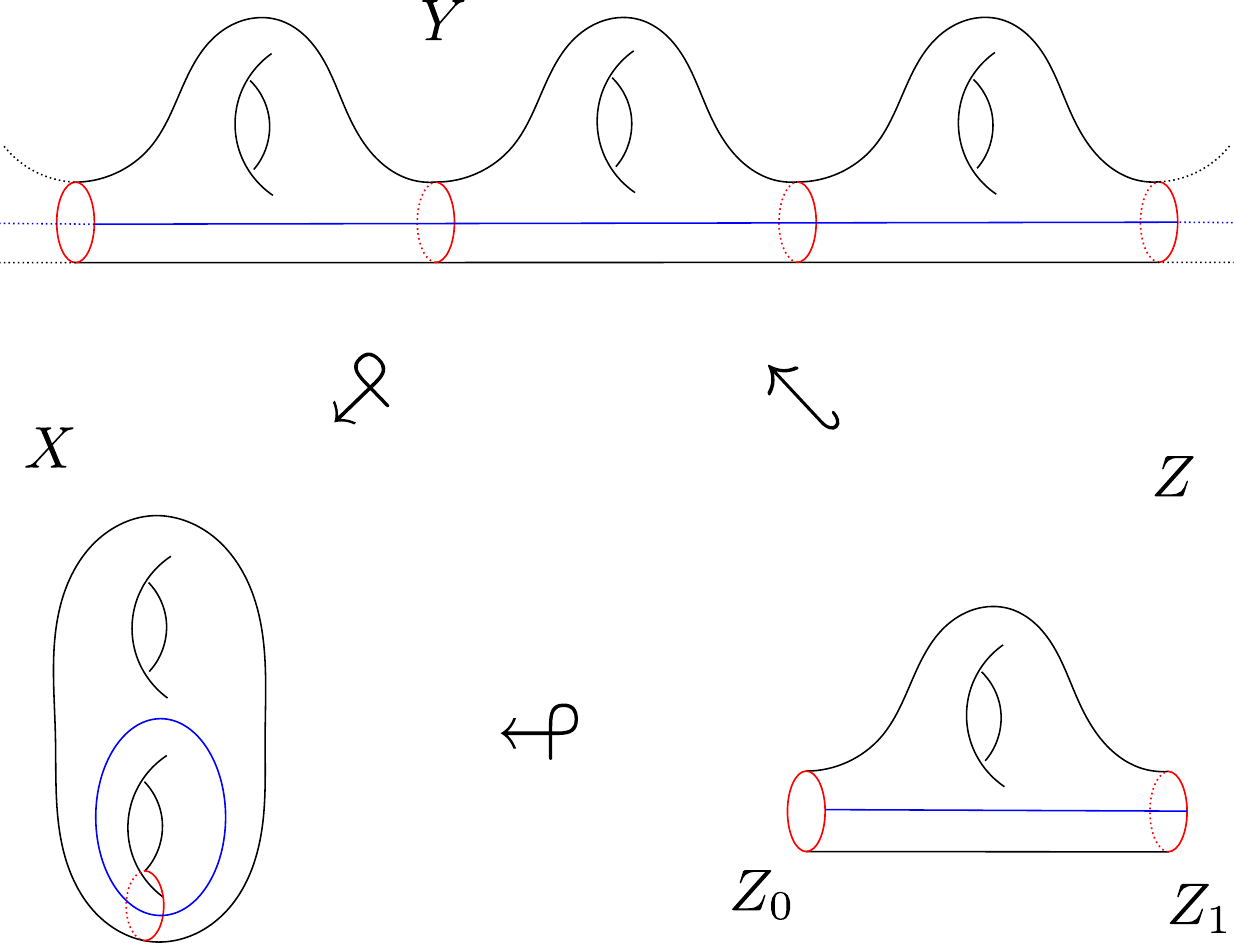}
\caption{A simple illustration of \cref{one-relator_splitting} in the case of a genus two surface.}
\label{fig:genus_2}
\end{figure}

Our next example will be the presentation complex $X$ of the following one-relator group
\[
G = \langle a, b \mid a^2 = b^3\rangle.
\]
A simple computation shows that $H_1(G, \Z) \isom \Z$ and so there is only one epimorphism $\phi\colon G\to \Z$ (up to change of sign) given by
\[
\phi(a) = 3, \quad \phi(b) = 2.
\]
Denote by $\rho\colon Y\immerses X$ the associated $\Z$-cover. Choose some lift $v_0\in Y$ of the 0-cell in $X$ and lifts $a_0, b_0\subset Y$ of the two 1-cells in $X$ such that they lead out of $v_0$. Denote by $v_i = t^i(v)$ and $a_i = t^i(a_0)$, $b_i = t^i(b_0)$ for each $i\in \Z$. Reading out the relator $a^2b^{-3}$, we see that for each $i\in \Z$ there is a 2-cell in $Y$ whose attaching map is given by the path
\[
a_ia_{i+3}b_{i+4}^{-1}b_{i+2}^{-1}b_{i}^{-1}.
\]
Hence, we see that there is a single attaching map (the one corresponding to $i = 0$) of a 2-cell of $Y$ supported in the subgraph
\[
\Gamma = \left(\bigcup_{i=0}^6v_i\right)\cup \left(\bigcup_{i=0}^3a_i\right)\cup\left(\bigcup_{i=0}^4b_i\right)
\]
and no attaching map of a 2-cell is supported in $t^{-1}(\Gamma)\cap \Gamma$ or $\Gamma\cap t(\Gamma)$. Letting $Z\subset Y$ be the one-relator subcomplex with 1-skeleton $\Gamma$, we see that $Z$ satisfies all the hypotheses of \cref{one-relator_splitting}. This is depicted in \cref{fig:hierarchy2}.

\begin{figure}
\centering
\includegraphics[scale = 0.5]{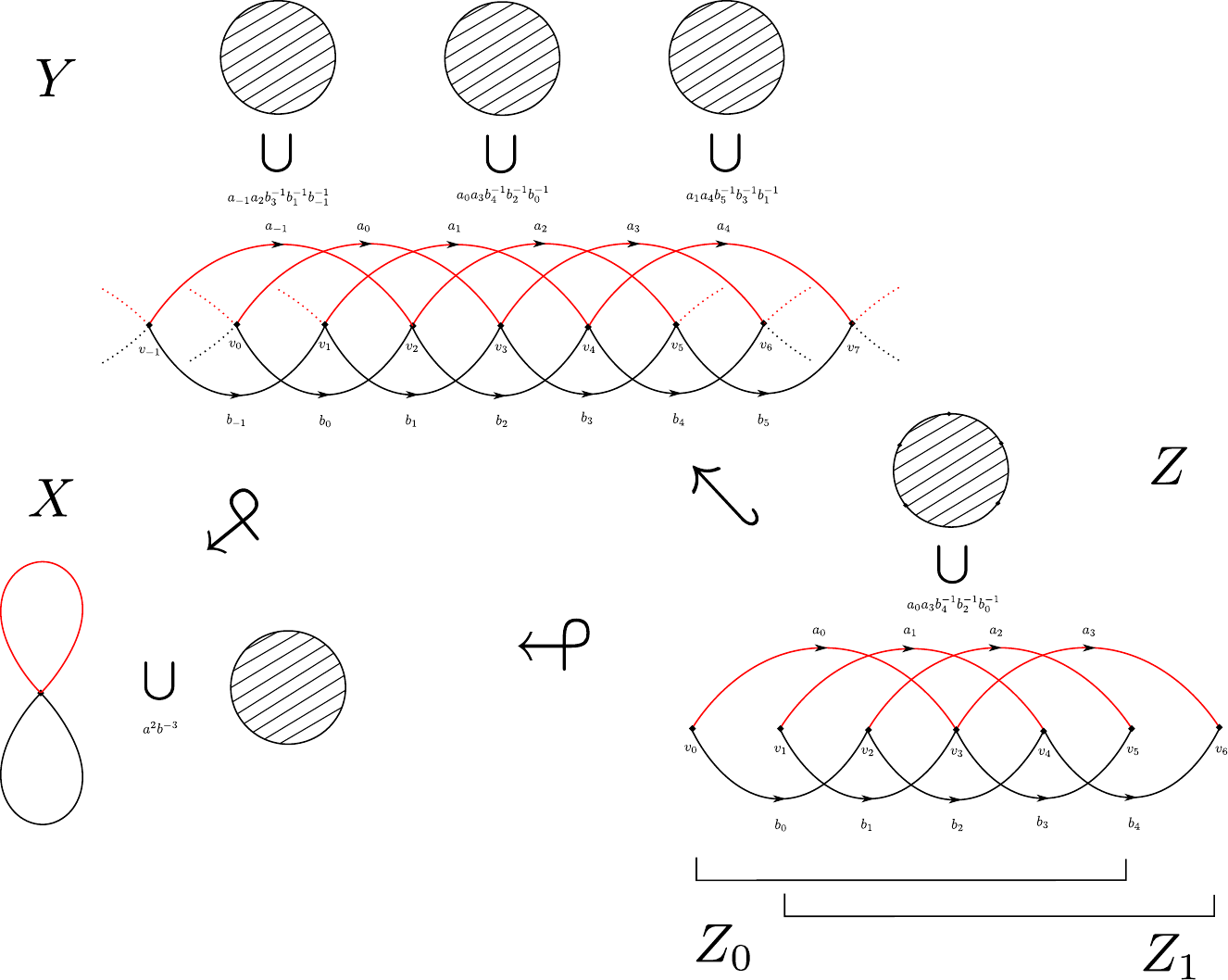}
\caption{A one-relator $\Z$-domain.}
\label{fig:hierarchy2}
\end{figure}

In order to obtain the one-relator splitting explicitly we need to choose generators for $\pi_1(Z)$. The elements
\[
c_1 = a_0b_1^{-1}a_1b_2^{-1}b_0^{-1}, \quad c_2 = b_0a_2b_3^{-1}a_0^{-1}, \quad c_3 = a_0a_3b_4^{-1}b_2^{-1}b_0^{-1}
\]
are seen to generate $\pi_1(Z, v_0)$, with the first two generating $\pi_1(Z_0, v_0)$ and the last being trivial. They come from choosing the spanning tree $T\subset \Gamma$ such that $\Gamma - T$ contains the 1-cells $a_1, a_2, a_3$. Denoting by
\[
d = a_0b_1^{-1},
\]
we have
\begin{align*}
\pi_1(X) \isom&\, \pi_1(Z)*_{\psi}\\
		\isom&\, \langle c_1, c_2, c_3 \mid c_3 = 1\rangle*_{\psi}\\
		\isom&\, \langle c_1, c_2, c_3, d \mid c_3 = 1, dc_1d^{-1} = c_2c_3, dc_2d^{-1} = c_3c_1^{-1}\rangle\\
		\isom&\, \langle c_1, c_2, d \mid dc_1d^{-1} = c_2, dc_2d^{-1} = c_1^{-1}\rangle
\end{align*}
where the last isomorphism is obtained by applying a Tietze transformation removing the redundant generator $c_3$ and relator $c_3 = 1$.

The last presentation above expresses $\pi_1(X)$ as an $F_2$-by-cyclic group with monodromy an order four automorphism. One can see the fact that $\pi_1(X)$ is free-by-cyclic geometrically: the 2-cell in $Z$ traverses the bottom edges and the top edges each precisely once and so $Z$ collapses to $Z_0 = t^{-1}(Z)\cap Z$ and to $Z\cap t(Z)$. Thus, the one-relator splitting expresses $\pi_1(X)$ as a HNN-extension where the identifying isomorphism is defined on all of $\pi_1(Z)$ - the fundamental group of a graph.

We remark here that our presentation for $\pi_1(X)$ depended on our chosen spanning tree, however it is not hard to see that we also could have obtained the following presentation:
\[
\pi_1(X) \isom \langle c_1, c_2, c_3, d \mid c_3c_1, d^{-1}c_1d = c_2, d^{-1}c_2d = c_3\rangle
\]
which is perhaps more what those familiar with the Magnus--Moldavanski\u{\i} method might expect.

\begin{remark}
If $G = F(S)/\normal{w}$ is a one-relator group in which some generator $t\in S$ appears with exponent sum zero, there is an epimorphism $\phi\colon G\to \Z$ defined by sending each generator different to $t$ to zero and by sending $t$ to one. Taking $X$ to be the presentation complex for $G$, the $\Z$-cover associated with $\phi$ has 1-skeleton a biinfinite line consisting of the lifts of 1-cells corresponding to $t$ and at each 0-cell there is a loop for every generator in $S - \{t\}$. The one-relator splitting arising from a $\Z$-domain of this cover as in \cref{one-relator_splitting} is precisely the type of HNN-splitting that Moldavanski\u{\i} finds in \cite{Mo67}.
\end{remark}

\subsection{Sketch of the proof of \cref{one-relator_splitting}}

We now sketch out the main ideas of the proof of \cref{one-relator_splitting}. There are essentially three parts to the proof with each part corresponding to the three points in \cref{one-relator_splitting}.

\subsubsection{Part one: HNN-decompositions from $\Z$-domains} The first step involves showing that certain kinds of subcomplexes of $\Z$-covers give rise to HNN-splittings. 

\begin{definition}
Let $\rho\colon Y\to X$ be a $\Z$-cover of 2-complexes and let $t\colon Y\to Y$ be a generator of $\deck(\rho) = \Z$. A subcomplex $Z\subset Y$ is a \emph{$\Z$-domain} if the following holds:
\begin{enumerate}
\item\label{itm:prop1} The $\Z$-translates of $Z$ cover all of $Y$. That is, $Y = \bigcup_{i\in \Z}t^i(Z)$.
\item\label{itm:prop2} $Z\cap t(Z)$ is connected and non-empty.
\item\label{itm:prop3} For each $i>0$, we have that $Z\cap t^i(Z)\subset Z\cap t^{i-1}(Z)$.
\end{enumerate}
\end{definition}

The following result is the motivation behind the definition of $\Z$-domains. 

\begin{proposition}
\label{splitting}
Let $\rho\colon Y\to X$ be a $\Z$-cover of 2-complexes, let $t\colon Y\to Y$ be a generator of $\deck(\rho)$, let $Z\subset Y$ be a $\Z$-domain and denote by $Z_{0} = t^{-1}(Z)\cap Z$ and $\iota\colon Z_0\injects Z$ the inclusion. If
\[
\mathcal{X} = Z \sqcup \left(Z_0\times[-1, 1]\right)/\{ \iota(z)\sim (z, -1),\, (t\circ\iota)(z) \sim (z, 1),\, z\in Z_0\},
\]
then the map
\[
h\colon \mathcal{X}\to X
\]
given by $\rho$ when restricted to $Z$ or $Z_0\times \{i\}$ for all $i\in (-1, 1)$, is a homotopy equivalence.

Moreover, if $\iota_*$ and $(t\circ\iota)_*$ are injective, then
\[
\pi_1(X) \isom \pi_1(Z)*_{\psi}
\]
where 
\[
\psi = (\iota\circ t\circ\iota^{-1})_*\colon \iota_*(\pi_1(Z_0))\to (t\circ \iota)_*(\pi_1(Z_0)).
\]
\end{proposition}

\begin{proof}
The map $h$ is surjective by property (\ref{itm:prop1}) of $\Z$-domains. Moreover, property \ref{itm:prop2} says that $Z_0$ is connected and non-empty. Finally, property (\ref{itm:prop3}) implies that the fibers $h^{-1}(x)$ are all connected (possibly infinite) lines. In particular, this implies that $h$ is a homotopy equivalence. See \cite[Proposition 3.2]{Lin22} for further details and a more general result.

If the homomorphisms $\iota_*, (t\circ\iota)_*$ are injective, then $\mathcal{X}$ is a graph of spaces with a single vertex space $Z$ and a single edge space $Z_0\times[-1, 1]$, attached to $Z$ via $\iota$ and $t\circ\iota$. Hence, we have $\pi_1(\mathcal{X}) \isom \pi_1(Z)*_{\psi}$. Since $h$ is a homotopy equivalence, we also have $\pi_1(X) \isom \pi_1(Z)*_{\psi}$ as claimed.
\end{proof}

\subsubsection{Part two: existence of one-relator $\Z$-domains} The second step involves proving that $\Z$-domains always exist for $\Z$-covers of one-relator complexes and that they may additionally be chosen so that they only support a single attaching map of a 2-cell. Then the Freiheitssatz combined with \cref{splitting} will allow us to find the desired splitting.

\begin{proposition}{\cite[Lemma 4.7]{Lin22}}
\label{Z-domain_graph}
If $\rho\colon Y\to \Lambda$ is a connected $\Z$-cover of 1-complexes, then there exists a $\Z$-domain $Z\subset Y$. Moreover, if $Z$ is minimal under inclusion of $\Z$-domains, then $Z\cap t(Z)$ is a tree and, if $\Lambda$ is finite, $Z$ is finite.
\end{proposition}

From \cref{Z-domain_graph} it is not hard to prove the following.

\begin{corollary}
\label{Z-domain}
If $X = (\Lambda, \lambda)$ is a one-relator complex and $\rho\colon Y\to X$ is a $\Z$-cover, then there exists a $\Z$-domain $Z\subset Y$, finite if $X$ is finite, which supports precisely one lift of $\lambda$.
\end{corollary}

\begin{proof}
Let $\Gamma'\subset Y^{(1)}$ be a $\Z$-domain for the $\Z$-cover $Y^{(1)}\to X^{(1)}$. Let $k\geqslant 0$ be the smallest integer such that
\[
\Gamma = \bigcup_{i=0}^kt^k(\Gamma')
\]
supports some lift of $\lambda$. Let $Z\subset Y$ be the one-relator subcomplex obtained from $\Gamma$ by attaching a 2-cell along a lift of $\lambda$. Since $\Gamma'$ was a $\Z$-domain for $Y^{(1)}\to X^{(1)}$, the three properties (\ref{itm:prop1}), (\ref{itm:prop2}) and (\ref{itm:prop3}) required of a $\Z$-domain holds for $Z$. Since $\Gamma'$ is finite when $\Lambda$ is, $Z$ is finite when $X$ is. Furthermore, if $Z$ does not contain precisely one lift of an attaching map of some 2-cell in $X$, this would imply that $t^{-1}(Z)\cap Z$ supports a lift of $\lambda$, contradicting our choice of $\Gamma$. Thus, $Z$ supports precisely one lift of $\lambda$.
\end{proof}

It will sometimes be useful to have a refinement of \cref{Z-domain} specifically for one-relator complexes. This statement is \cite[Proposition 4.4]{Lin24b}.

\begin{proposition}
\label{Z-domain_one-relator}
Let $X = (\Lambda, \lambda)$ be a one-relator complex with $c(\lambda)>0$ and let $e\subset X$ be a 1-cell traversed by $\lambda$. There exists a $\Z$-cover $\rho\colon Y\to X$ and a one-relator $\Z$-domain $(\Lambda_Z, \lambda_Z) = Z\subset Y$ with $c(\lambda_Z)<c(\lambda)$ such that $\rho^{-1}(\Lambda - e)\subset Z$ and the following holds. If $\tilde{e}\subset Z$ is a lift of $e$ to $Z$ and $m, M\in \Z$ are, respectively, the minimal and maximal integers such that $t^m(\tilde{e})\subset Z$ and $t^M(\tilde{e})\subset Z$, then $\lambda_Z$ traverses both $t^m(\tilde{e})$ and $t^M(\tilde{e})$.
\end{proposition}

For an arbitrary $\Z$-cover of a one-relator complex $\rho\colon Y\to X$, by \cref{Z-domain} we know that there is \emph{some} 1-cell $e\subset X$ and a one-relator $\Z$-domain $Z\subset Y$ such that the attaching map for the 2-cell in $Z$ traverses the minimal lift of $e$ and the maximal lift. The advantage of \cref{Z-domain_one-relator} is that it allows us to first choose a 1-cell in $X$ and then procure a $\Z$-cover and a $\Z$-domain with these properties. This version is particularly useful when generalising results on one-relator groups to the setting of one-relator products of locally indicable groups.

\begin{remark}
One should compare \cref{Z-domain_one-relator} with the usual method of finding a splitting of a one-relator group over a one-relator subgroup, as described, for example, in \cite{LS01}.
\end{remark}

\begin{remark}
With almost the same proof, \cref{Z-domain} can easily be generalised to arbitrary finite 2-complexes $X$. Then using \cref{splitting} one can prove that for any finitely presented group $G$ and any epimorphism $\phi\colon G\to \Z$, there is an element $t\in G$ such that $\phi(t) = 1$, a finitely generated subgroup $H\leqslant \ker(\phi)$ and two finitely generated subgroups $A = H^{t^{-1}}\cap H$, $B = H\cap H^t$ of $H$ such that
\[
G \isom H*_{\psi}
\]
where $\psi\colon A\to B$ is the isomorphism induced by conjugation by $t$. In fact, one only needs $G$ to have type $\fp_2(R)$ for some ring $R$. This is proven by Bieri--Strebel in \cite[Theorem A]{BS78}. We will state a much more general, but much less explicit, result in \cref{sec:coherence}.
\end{remark}

\subsubsection{Part three: decrease of complexity} The final step involves showing that the complexity of the one-relator $\Z$-domain from \cref{Z-domain} goes down.

\begin{lemma}
\label{decrease_complexity}
If $Z = (\Lambda_Z, \lambda_Z)$ and $X = (\Lambda_X, \lambda_X)$ are one-relator complexes and $\sigma\colon Z\immerses X$ is an immersion, then
\[
c(\lambda_Z)\leqslant c(\lambda_X)
\]
with equality precisely when $\sigma\mid\Ima(\lambda_Z)$ is injective.
\end{lemma}

\begin{proof}
If $\lambda_X$ is the attaching map for $X$ and $\lambda_Z$ is the attaching map for $Z$, we have that
\[
\frac{|\lambda_Z|}{\deg(\lambda_X)} = \frac{|\lambda_Z|}{\deg(\lambda_Z)}.
\]
Moreover, we have
\[
\left|\Ima(\lambda_X)^{(0)}\right|\leqslant \left|\Ima(\lambda_Z)^{(0)}\right|
\]
with equality precisely when $\sigma\mid \Ima(\lambda_Z)^{(0)}$ is a bijection. Since $\sigma$ is an immersion, this would force $\sigma$ to be injective on all of $\Ima(\lambda_Z)$.
\end{proof}

\begin{corollary}
\label{Z-domain_decrease_complexity}
Let $X = (\Lambda, \lambda)$ be a one-relator complex and let $\rho\colon Y\immerses X$ be a $\Z$-cover. If $Z\subset Y$ is a one-relator $\Z$-domain, minimal under inclusion amongst one-relator $\Z$-domains, then
\[
c(Z)<c(X).
\]
\end{corollary}

\begin{proof}
Let $\lambda_Z$ denote the attaching map of the 2-cell in $Z$. The cover $\rho$ sends $\Ima(\lambda_Z)$ surjectively to $\Ima(\lambda)$. If $c(\lambda_Z)\geqslant c(\lambda)$, then $c(\lambda_Z) = c(\lambda)$ by \cref{decrease_complexity} and $\Ima(\lambda_Z)$ maps isomorphically to $\Ima(\lambda)$. By collapsing each 2-cell to a point in $X$, we obtain an induced commutative diagram
\[
\begin{tikzcd}
Z \arrow[r, hook] \arrow[d, two heads] & Y \arrow[r, "\rho"] \arrow[d, two heads] & X \arrow[d, two heads] \\
Z' \arrow[r, hook]                     & Y' \arrow[r, "\rho'"]                    & X'                    
\end{tikzcd}
\]
where vertical maps are collapse maps. In particular, $Z', Y', X'$ are all 1-complexes. Since $Z' - t^{-1}(Z')\cap Z'$ contains precisely one lift of each cell in $X'$, we have $-\chi(Z') = -\chi(X') -\chi(t^{-1}(Z')\cap Z')$. Hence, by \cref{Z-domain_graph}, since $Z$ was minimal, it follows that $Z'\cap t(Z')$ is a tree and so $-\chi(Z') = -\chi(X') - 1$. This implies that $-\chi(Z)<\chi(X)$.
\end{proof}

\subsubsection{Completing the proof: the Freiheitssatz}

Combining \cref{splitting}, \cref{Z-domain} and \cref{Z-domain_decrease_complexity}, we have proven \cref{one-relator_splitting} under the assumption that the Freiheitssatz holds. This is because in order to apply \cref{splitting}, we needed $\pi_1(Z_0), \pi_1(Z_1)\to \pi_1(Z)$ to be injective.

In order to prove the Freiheitssatz for $X = (\Lambda, \lambda)$, we assume it holds for all one-relator complexes $Z$ with $c(Z)<c(X)$ (the base case $c(X) = 0$ is clear). Let $e\subset \Lambda$ be a 1-cell traversed by $\lambda$ and denote by $\Gamma = \Lambda - e$. By \cref{Z-domain_one-relator}, there is a $\Z$-cover $\rho\colon Y\to \Z$ and a one-relator $\Z$-domain $Z\subset Y$ with $c(Z)<c(X)$ and such that $Z\cap t(Z)$ does not support any attaching map of a 2-cell and such that $\rho^{-1}(\Gamma)\subset Z$. Thus, the homomorphisms $\pi_1(Z_0), \pi_1(Z_1)\to \pi_1(Z)$ are both injective by the inductive hypothesis and we may apply \cref{splitting} to conclude that
\[
\pi_1(X) \isom \pi_1(Z)*_{\psi}.
\]
In particular, $\pi_1(Z)$ injects into $\pi_1(X)$. This implies that two paths in $X$ that lift to $Z$ are path homotopic if and only if they are path homotopic in $Z$. Since $\rho^{-1}(\Gamma)\subset Z$, any pair of paths in $\Gamma$ lift to $Z$. Since each component of $\rho^{-1}(\Gamma)$ is a Magnus subgraph of $Z$, by induction it follows that any pair of paths in $\rho^{-1}(\Gamma)$ are path homotopic in $Z$ if and only if they are path homotopic in $\rho^{-1}(\Gamma)$. Thus, any pair of paths in $\Gamma$ that are path homotopic in $X$ are path homotopic in $\Gamma$. Thus, the map
\[
\pi_1(\Gamma) \to \pi_1(X)
\]
induced by inclusion is injective. We have proven the Freiheitssatz and hence, \cref{one-relator_splitting}.

\subsection{The character sphere and describing all one-relator splittings}
\label{sec:character_sphere}

If $G$ is a group, the \emph{character sphere} $S(G)$ is the following space
\[
S(G) = \left(H^1(G, \R) - \{0\}\right)/\R_{>0}
\]
of equivalence classes of homomorphisms $\phi\colon G\to \R$, where $\phi\sim\phi'$ if $\phi = r\cdot \phi'$ for some $r\in \R_{>0}$.

Let $S$ be a finite set of cardinality $n$, let $F = F(S)$ be a free group, let $w\in F$ be a cyclically reduced word and let $G = F/\normal{w}$ be a one-relator group. Denote by $\Gamma\subset \R^{n} = H^1(F, \R)$ the Cayley graph for $F_{\ab} = \Z^{n}$ over the generating set given by the image of $S$ in $F_{\ab}$. Let $\lambda_w\colon I\to \Gamma$ denote the path at the origin spelling out the word $w$ and denote by $\Lambda_w\subset \R^{n}$ its image. Let $p_w$ denote the endpoint of $\lambda_w$. We have $p_w = 0$ if and only if $w\in [F, F]$. The path $\lambda_w$ is called the \emph{trace} of $w$. See \cref{fig:abelian_cover} for an example.

There is a natural correspondence:
\[
S(G) \leftrightarrow \left\{\text{Directed lines $L\subset \R^{n}$ through the origin with $p_w\in L^{\bot}$}\right\}
\]
where a line $L_{\phi}\subset \R^{n}$ is associated with a homomorphism $\phi\colon G\to \R$ (unique in the $S(G)$ equivalence class) via the composition of the map $G\to H_1(G, \R) = \R^n$ with the orthogonal projection of $\R^{n}$ onto the line $L_{\phi}$. More explicitly, if $p_{\phi}\in L_{\phi}$ is the point at (positive) unit distance from the origin, then each element in $r\cdot p_{\phi} + L_{\phi}^{\bot}$ gets mapped to $r\in \R$.

\begin{figure}
\centering
\includegraphics[scale = 0.5]{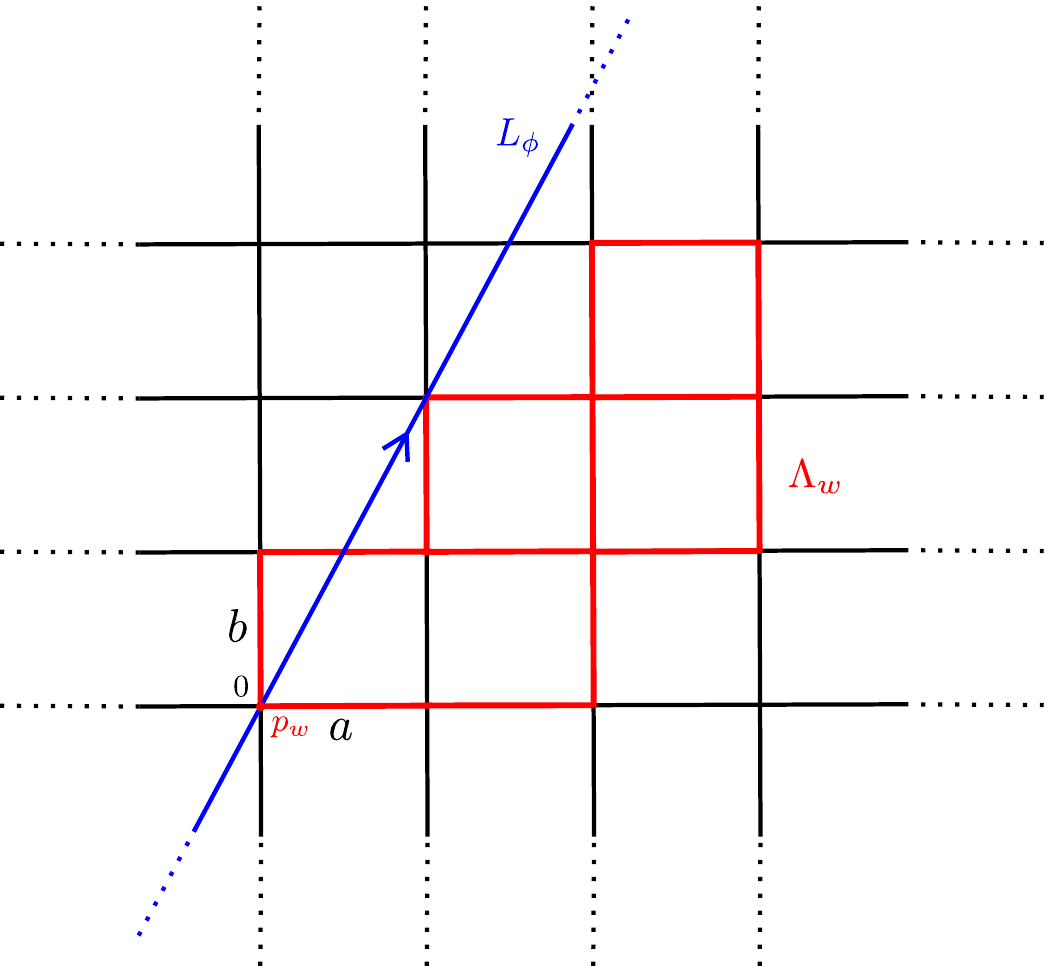}
\caption{The trace of the word $a^2b^3ab^{-2}a^{-1}bab^{-1}a^{-2}bab^{-1}a^{-2}b^{-1}$.}
\label{fig:abelian_cover}
\end{figure}

Now let $X = (\Lambda, \lambda)$ be the presentation complex for $G = F(S)/\normal{w}$ and let $\phi\in H^1(G, \R)$ be a homomorphism $G\to \R$. Let $L_{\phi}\subset \R^{n}$ be the corresponding line. The 1-skeleton of the abelian cover $Y\to X$ associated with $\phi$ is the quotient of $\Gamma$ defined by identifying all $\Z^{n}$ orbits of points that lie in the same plane $r\cdot p_{\phi} + L_{\phi}^{\bot}$ for each $r\in \R$. Note that $Y^{(1)}$ has a natural map $Y^{(1)}\to \R$ given by sending the points in $r\cdot p_{\phi} + L_{\phi}^{\bot}$ to $r\in \R$. The path $\lambda_w\colon I \to \Gamma$ projects to a loop in $Y^{(1)}$. Denote by $\lambda_{\phi}\colon S^1\immerses Y^{(1)}$ the cycle obtained from $I \xrightarrow{\lambda_w} \Gamma\xrightarrow{} Y^{(1)}$ by identifying the endpoints of $I$. The abelian cover $Y$ is then obtained from $Y^{(1)}$ by attaching cycles along the $G$-orbit of $\lambda_{\phi}$. Denote also by $\Lambda_{\phi}\subset Y^{(1)}$ the image of $\Lambda_w$ so that $(\Lambda_{\phi}, \lambda_{\phi}) = Z_{\phi}\subset Y$ is a one-relator subcomplex. This construction can also be generalised for arbitrary one-relator complexes by contracting and expanding a spanning tree, we leave the details to the reader. Call $Z_{\phi}$ \emph{the one-relator complex associated with $\phi$}. Note that $Z_{\phi} \isom Z_{\psi}$ for all $\psi\in \pm[\phi]$.

\begin{exmp}
\label{Z_example}
Let us illustrate how to obtain the one-relator complex $Z_{\phi}$ associated with $\phi$ through an example. In \cref{fig:abelian_cover} we have depicted the trace of the word $w = a^2b^3ab^{-2}a^{-1}bab^{-1}a^{-2}bab^{-1}a^{-2}b^{-1}$. Since $w\in [F(a, b), F(a, b)]$ we have $p_w = 0$ and so any directed line through the origin gives us a homomorphism $\phi\colon G\to \R$ where $G = F(a, b)/\normal{w}$. Let $L_{\phi}\subset \R^2$ be the directed line passing through the origin and through the point $(1, 2)$. Then the induced surjective map $\Lambda_w\to Z^{(1)}_{\phi}$ identifies two pairs of vertices, the pairs $(0, 1), (2, 0)$ and $(1, 2), (3, 1)$, and does not identify any edges. See \cref{fig:projection} for a depiction of $Z_{\phi}^{(1)}$. Then $Z_{\phi}$ is obtained from $Z_{\phi}^{(1)}$ by attaching a 2-cell along the loop $I \to \Lambda_w\to Z_{\phi}^{(1)}$.

Note that by adding the missing $b$-edge and the missing $a$-edge to $Z_{\phi}$ (that is, the $b$-edge connecting $\frac{1}{\sqrt{5}}$ with $\frac{3}{\sqrt{5}}$ and the $a$-edge connecting $\frac{7}{\sqrt{5}}$ with $\frac{8}{\sqrt{5}}$), then we obtain a $\Z$-domain for the $\Z$-cover of the presentation complex induced by $\phi$.
\end{exmp}

\begin{figure}
\centering
\includegraphics[scale = 0.65]{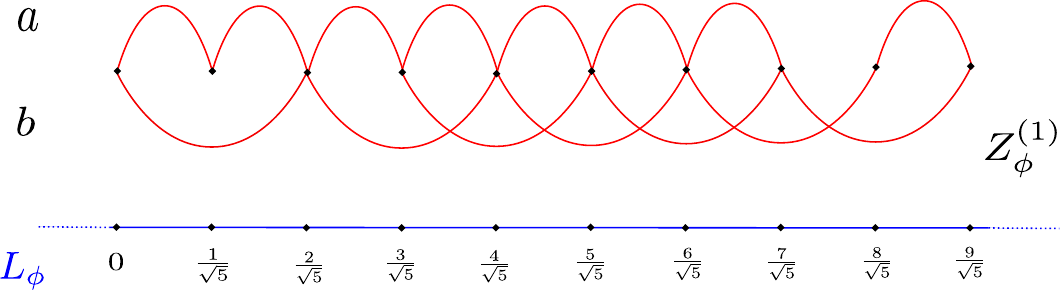}
\caption{The 1-skeleton of $Z_{\phi}$, obtained by projecting $\Lambda_w$ onto $L_{\phi}$.}
\label{fig:projection}
\end{figure}

We have the following which follows from the definition of one-relator splittings and the fact that every 1-cell in $\Lambda_{\phi}$ is traversed by the attaching map $\lambda_{\phi}$.

\begin{lemma}
\label{base_groups}
Let $X$ be a one-relator complex, let $\pi_1(X) \isom \pi_1(Z)*_{\psi}$ be a one-relator splitting with $\phi\colon \pi_1(X) \to \pi_1(X)/\normal{\pi_1(Z)} = \Z$ the associated epimorphism. Then, after possibly replacing $Z$ by a translate, the one-relator complex $Z_{\phi}$ associated with $\phi$ is a subcomplex of $Z$. In particular, $\pi_1(Z)$ splits as a free product of a free group and $\pi_1(Z_{\phi})$.
\end{lemma}

In this way we may see all base groups of all one-relator splittings of $\pi_1(X)$ by looking at appropriate quotients of $\Lambda_w, \lambda_w$. In fact, there is always a maximal such quotient, in the sense that this quotient one-relator complex quotients onto all other $Z_{\phi}$. Indeed, we may always choose a homomorphism $\phi\colon G\to \R$ which factors through an epimorphism $G\to \Z$ such that two points $p, q\in \Lambda_w\subset \Gamma$ in the same $\Z^n$-orbit are identified under the quotient map $\Lambda_w\to \Lambda_{\phi}$ if and only if $p + \R\cdot p_w = q + \R\cdot p_w$. Let $\Lambda_{\max}$ denote this maximal quotient and let $\lambda_{\max}\colon S^1\immerses \Lambda_{\max}$ be the induced cycle. Then by construction, the one-relator complex $Z_{\max} = (\Lambda_{\max}, \lambda_{\max})$ admits a unique immersion $Z_{\max}\immerses Z_{\phi'}$ for all epimorphisms $\phi'\colon G\to \Z$. We have proven the following lemma in the case $X$ has a single 0-cell, the more general case is left as an easy exercise to the reader.

\begin{lemma}
If $X$ is a one-relator complex with $\pi_1(X)$ infinite, then there is a unique immersion of a one-relator complex $Z_{\max}\immerses X$ such that $Z_{\max}\immerses X$ factors through all immersions of $\Z$-domains of all $\Z$-covers of $X$, with at least one $\Z$-domain containing $Z_{\max}$ as a subcomplex.
\end{lemma}

Another way to see $Z_{\max}$ is as a minimal one-relator subcomplex of the maximal torsion-free abelian cover of $X$. This is explained in more detail in \cite{Lin22_thesis}.

As an illustration of how to obtain $Z_{\max}$, in \cref{Z_example} we may take the line $L_{\phi}$ to pass through the origin and $(1, 3)$ instead of $(1, 2)$. In this way, no vertices are identified under the projection map $\Lambda_w \to Z_{\phi}$.

\section{Splittings of one-relator groups}
\label{sec:splittings}

\subsection{Magnus subgroups and the Bass--Serre theory of one-relator splittings}

In this subsection, we discuss some properties of the action of a one-relator group on the Bass--Serre tree associated with a one-relator splitting. In order to understand the action, we need to understand the interaction between the edge groups and their conjugates. Thus, in the first part of this subsection we shall look at intersections of Magnus subgroups.

\subsubsection{Intersections of Magnus subgroups}

Two results due to Collins are particularly important for this topic. The first involves exceptional intersections. A pair of Magnus subgraphs $A, B\subset X$ of a one-relator complex with $A\cap B$ are said to have \emph{exceptional intersection} if $\pi_1(A)\cap \pi_1(B)\neq \pi_1(A\cap B)$. In other words, if the intersection is not the obvious one. Collins characterised exceptional intersections in \cite{Co04}. Howie also characterised exceptional intersections for one-relator products and presented an algorithm to compute intersections of Magnus subgroups in \cite{Ho05}.

\begin{theorem}
\label{Collins_intersections_1}
Let $X = (\Lambda, \lambda)$ be a one-relator complex and let $A, B\subset X$ be Magnus subgraphs such that $A\cap B$ is connected. One of the following holds:
\begin{itemize}
\item $\pi_1(A)\cap \pi_1(B) = \pi_1(A\cap B)$.
\item $\pi_1(A)\cap \pi_1(B) = \pi_1(A\cap B)*\langle c\rangle$ for some $c\in \pi_1(X)$.
\end{itemize}
Moreover, if $\deg(\lambda)\geqslant 2$, then only the first situation occurs.
\end{theorem}

Howie's result \cite{Ho05} provides further information: if $\pi_1(A)\cap \pi_1(B)$ have an exceptional intersection, then there is an immersion of a one-relator complex $Z\immerses X$ which expresses this intersection in an explicit way. This extra information was exploited heavily in \cite{Lin24}.

The second result involves intersections between conjugates of Magnus subgroups and was proved by Collins in \cite{Co08}. These turn out to be much more controlled.

\begin{theorem}
\label{Collins_intersections_2}
Let $X = (\Lambda, \lambda)$ be a one-relator complex and let $A, B\subset X$ be Magnus subgraphs such that $A\cap B$ is connected. If $g\in \pi_1(X)$, one of the following holds:
\begin{itemize}
\item $g\in \pi_1(A)\cdot \pi_1(B)$, in which case $\pi_1(A)^g\cap \pi_1(B) = (\pi_1(A)\cap\pi_1(B))^b$ for some $b\in \pi_1(B)$.
\item $\pi_1(A)^g\cap \pi_1(B) = 1$.
\item $\pi_1(A)^g\cap \pi_1(B) = \langle c\rangle$ for some $c\in \pi_1(X)$.
\end{itemize}
Moreover, if $\deg(\lambda)\geqslant 2$, then only the first two situations occur.
\end{theorem}

The \emph{rank} $\rk(G)$ of a group $G$ is the smallest cardinality of a generating set for $G$. The \emph{reduced rank} is then defined as:
\[
\rr(G) = \min{\{0, \rk(G)-1\}}.
\]
A much stronger, but less precise result on intersections of Magnus subgroups of one-relator groups was proven by Jaikin-Zapirain and the first author in \cite{JZL23}.

\begin{theorem}
\label{inert}
If $G = F/\normal{w}$ is a one-relator group and $A\leqslant G$ is a Magnus subgroup, then for any finitely generated subgroup $H\leqslant G$, we have
\[
\sum_{HgA}\rr(A\cap H^g)\leqslant \rr(H).
\]
\end{theorem}

Dicks--Ventura introduced the notion of being inert in \cite{DV96}: a subgroup $A\leqslant G$ is \emph{inert} if for all finitely generated subgroups $H\leqslant G$ we have $\rk(A\cap H)\leqslant \rk(H)$. \cref{inert} says that Magnus subgroups of one-relator groups are \emph{strongly inert}.

\subsubsection{The action on the Bass--Serre tree}

Recall that if $G \isom H*_{\psi}$ is a HNN-extension, where $\psi\colon A\to B$ is the identifying isomorphism and $t\in G$ is the stable letter, then the \emph{Bass--Serre tree} associated with the splitting is the tree $T$ with vertices and edges:
\begin{align*}
V(T) &= \bigsqcup_{gH\in G/H}gH\\
E(T) &= \bigsqcup_{gA\in G/A} gA
\end{align*}
and adjacency given by:
\begin{align*}
o(gA) &= gH\\
t(gA) &= gtH.
\end{align*}
A priori it is not obvious why this is a tree: the fundamental theorem of Bass--Serre theory states that it is, see Serre's monograph \cite{Se80} for details. The Bass--Serre tree has a natural left $G$ action with a single orbit of vertices and a single orbit of edges.

If $gH\in V(T)$ is a vertex at distance $n$ from the vertex $H\in V(T)$, then we have
\[
gH = h_1t^{\epsilon_1}h_1\ldots h_nt^{\epsilon_n}H
\]
for some sequence of elements $h_i\in H$ and $\epsilon_i = \pm1$. 

If $g_1H, g_2H\in V(T)$, denote by $S_{g_1H, g_2H}\subset T$ the unique geodesic segment connecting $g_1H$ with $g_2H$. Then we have
\begin{align*}
\stab\left(S_{g_1H, g_2H}\right) &= \stab(g_1H)\cap \stab(g_2H)\\
							&= H^{g_1^{-1}}\cap H^{g_2^{-1}}\\
					&= (H\cap H^{g_2^{-1}g_1})^{g_1^{-1}}\\
					&= \stab\left(S_{H, g_1^{-1}g_2H}\right)^{g_1^{-1}}
\end{align*}
where $\stab(-)$ denotes the pointwise stabiliser. Note that the length of $S_{g_1H, g_2H}$ is precisely the minimal number of $t$-letters which appear in the reduced form of $g_1^{-1}g_2$. Call this the \emph{$t$-length}. More precisely, the $t$-length of an element $g\in H*_{\psi}$, denoted by $|g|_t$, is the smallest number of $t^{\pm1}$-letters that appear in a word over $H$ and $t^{\pm1}$ that equals $g$ in $H*_{\psi}$. 

If $n\geqslant 0$ is any integer, denote by
\[
W_n = \{ g\in G \mid |g|_t = n\}.
\]
By definition, we have $W_0 = H$. More generally, we have:
\[
d(H, gH) = n \iff g\in W_n.
\]

Using Collin's results, in \cite{Lin22} the following result is shown.

\begin{theorem}
If $G\isom H*_{\psi}$ is a one-relator splitting, $\psi\colon A\to B$ is the identifying isomorphism, $t\in G$ is the stable letter and $n\geqslant 1$ is an integer, then
\[
\sum_{g\in HW_nH} \rr(H\cap H^g)\leqslant \rr(A) + \rr(B).
\]
If $A\cap B$ has no exceptional intersection, then for some $k\geqslant 1$, we have
\[
\rr(H\cap H^g) = 0,
\]
for all $g\in HW_nH, n\geqslant k$.
\end{theorem}

If $k\geqslant 0$ is an integer, an action of a group $G$ on a tree $T$ is \emph{$k$-acylindrical} if for every segment $S\subset T$ of length at least $k$, we have $\stab(S) = 1$. Say the action is \emph{acylindrical} if it is $k$-acylindrical for some $k$. The following theorem was proven by Wise for one-relator groups with torsion in \cite{Wi21}. The full statement is one of the main theorems of \cite{Lin22}.

\begin{theorem}
\label{acylindrical}
Let $G$ be a one-relator group and let $G\isom H*_{\psi}$ be a one-relator splitting. If $G$ has torsion or is 2-free, then
\[
\sum_{HgH}\rk(H\cap H^g)<\infty.
\]
In particular, the action of $G$ on the Bass--Serre tree is acylindrical.
\end{theorem}

Recall that a group $G$ is \emph{$k$-free} if all $k$-generated subgroups of $G$ are free. We will revisit 2-free one-relator groups in detail in \cref{sec:primitivity_rank} and again in \cref{sec:coherence_uni}. Importantly, the property of being 2-free is, in a precise sense, generic and there is a straightforward algorithm which decides whether a one-relator group is 2-free. Both of these results can be found in \cite{LW22}.

\subsection{Free co-abelian subgroups}
\label{sec:coabelian}

A \emph{co-abelian} subgroup of a group $G$ is a normal subgroup $N\triangleleft G$ such that $G/N$ is torsion-free abelian. If $G = F/\normal{w}$ is a one-relator group, the aim of this section is to provide a simple criterion for when a co-abelian subgroup of $G$ is a free group. In particular, if $G = H*_{\psi}$ is a one-relator splitting, we have a natural epimorphism 
\[
G \to G/\normal{H} = \Z
\]
and we want to determine when the co-abelian subgroup $\normal{H}$ is a free group.

Let $X$ be the presentation complex for $G$. Any co-abelian subgroup of $G$ can be obtained as $\ker(\phi)$ for some $\phi\in H^1(G, \R) \isom H^1(X, \R)$. Let $\phi\in H_1(X, \R)$ and let $Y\to X$ be the induced abelian cover. Let $\mu\colon Y^{(1)}\to \R$ be the map given by projection to $L_{\phi}$, as described in \cref{sec:character_sphere}. Since $\ker(\phi) = \pi_1(Y)$, we want to understand the structure of $Y$. For this we describe two conditions that depend only on the homotopy type of $Z_{\phi}\subset Y$.

Recall that a subcomplex of a combinatorial complex is \emph{full} if every cell whose attaching map is supported in the subcomplex, also lies in the subcomplex. If $J\subset \R$ is a connected subset, denote by $Y_{J}$ the full subcomplex of $Y$ containing all vertices that map within $J\subset \R$ via $\mu$. Let $a_{\min}, a_{\max}\in \R$ be the minimal and maximal values that 0-cells in $Z_{\phi}$ map to. Consider the following two conditions:
\begin{enumerate}
\item\label{itm:condition1}(\emph{$\phi$ is freely descending}) There is a 1-cell 
\[
e_L\subset Z_{\phi} - \left(Z_{\phi}\cap Y_{(a_{\min}, a_{\max}]}\right)
\]
such that, $Z_{\phi}$ is homotopic to 
\[
Z_{\phi}^{(1)} - e_L
\]
relative to $Z_{\phi}\cap Y_{(a_{\min}, a_{\max}]}$.
\item\label{itm:condition2}(\emph{$\phi$ is freely ascending}) There is a 1-cell 
\[
e_U\subset Z_{\phi} - \left(Z_{\phi}\cap Y_{[a_{\min}, a_{\max})}\right)
\]
such that, $Z_{\phi}$ is homotopic to 
\[
Z_{\phi}^{(1)} - e_U
\]
relative to $Z_{\phi}\cap Y_{[a_{\min}, a_{\max})}$.
\end{enumerate}

\begin{remark}
A homomorphism $\phi\in H^1(X, \R)$ is freely descending if, for instance, some minimal edge in $Z_{\phi}$ is a free edge in the sense that it is traversed precisely once by the attaching map of the 2-cell in $Z_{\phi}$. By symmetry, the same can be said about when $\phi$ is freely descending with minimal replaced with maximal.

Sapir--\v{S}pakulov\'{a} show in \cite{SS10} that a random one-relator group on three or more generators has some epimorphism $\phi\colon G\to \Z$ such that $Z_{\phi}$ has a topmost or bottommost edge that is free.
\end{remark}

\begin{theorem}
\label{free_kernel}
Let $X$ be a one-relator complex with a single 0-cell. If $\phi\in H^1(X, \R)$ is freely descending and ascending then $\ker(\phi)$ is a free group.
\end{theorem}

\begin{exmp}
\label{exmp:free_face}
Let us return to \cref{Z_example}. In this example, both the topmost $b$-edge and the bottommost $b$-edge in $Z_{\phi}$ are traversed precisely once by the attaching map of the 2-cell. Hence, we may collapse the 2-cell through these free edges to obtain a homotopy equivalence as required by the definitions of $\phi$ being freely descending and ascending. Hence, the homomorphism $\phi$ from \cref{Z_example} is both freely descending and ascending so that $\ker(\phi)$ is a free group by \cref{free_kernel}. In fact, $\ker(\phi)$ is a finitely generated free group.
\end{exmp}

\begin{exmp}
For another type of example, consider the one-relator group
\[
G = \langle a, b \mid a^2babca^3a^2bababc^{-1}\rangle.
\]
Let $\phi\colon G\to \Z$ be the epimorphism defined by sending $c$ to $1$ and $a, b$ to $0$. See \cref{fig:free_kernel} for the one-relator complex $Z_{\phi}$. We see that $Z_{\phi}$ does not have any free faces and so we cannot conclude that $\ker(\phi)$ is free as in \cref{exmp:free_face}. However, since $a_0^2b_0a_0b_0$ is a primitive element of $F(a_0, b_0)$ (we can see this by applying the automorphism $a_0 \to a_0b_0^{-2}$, $b_0\to a_0^{-1}b_0$), we see that $Z_{\phi}$ is homotopy equivalent to $Z_{\phi}^{(1)} - a_0$ relative to $Z_{\phi}^{(1)} - (v_0\cup a_0\cup b_0\cup c_0) = Z_{\phi}\cap Y_{(0, 1]}$. Hence, $\phi$ is freely descending. Since $a_1^2b_1a_1b_1a_1b_1$ is a primitive element of $F(a_1, b_1)$, we may also conclude that $\phi$ is freely ascending. Thus, $\ker(\phi)$ is free by \cref{free_kernel}.
\end{exmp}

\begin{figure}
\centering
\includegraphics[scale = 0.9]{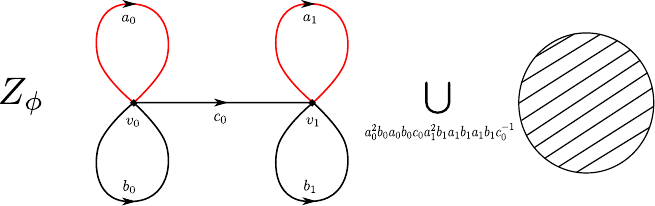}
\caption{An example of a one-relator complex $Z_{\phi}$ with $\phi$ freely descending and ascending.}
\label{fig:free_kernel}
\end{figure}

\begin{proposition}
\label{freely_descending}
Let $X$ be a one-relator complex with a single 0-cell and suppose that $\phi\in H^1(X, \R)$ factors through $\Z$ and is freely descending. 
Let $t$ be the generator of $\deck(Y\to X)\isom \Z$ that translates in the positive direction. Let $\Delta_L\subset Y^{(1)}$ be the subgraph consisting of the vertex mapping to $a_{\min}$ and all outgoing edges and their endpoints except for $e_L$. Then $Y$ is homotopy equivalent to
\[
\left(\bigcup_{i< 0}t^i(\Delta_L)\right)\cup Y_{[a_{\min}, +\infty)}
\]
relative to $Y_{[a_{\min}, +\infty)}$.
\end{proposition}

\begin{proof}
Since $\phi$ factors through $\Z$, after multiplying by a suitable real number, we may assume that $\phi$ maps surjectively to $\Z\subset \R$. After possibly replacing $Z_{\phi}$ with a translate, we may assume that $a_{\min} = 0$. If $k< 0$ is an integer, denote by
\[
Y_k = t^k(\Delta_L)\cup Y_{[k+1, +\infty)}.
\]
Since we have
\[
t^k(Z_{\phi})\cap Y_{[k+1, +\infty)} = t^k\left(Z_{\phi}\cap Y_{[a_{\min}+1, a_{\max}]}\right) = t^k\left(Z_{\phi}\cap Y_{(a_{\min}, a_{\max}]}\right)
\]
and since $Y_{[k, +\infty)} - Y_{[k+1, +\infty)}$ contains no 2-cell, it follows from the definition of $\phi$ being freely descending that $Y_{[k, +\infty)}$ is homotopy equivalent to $t^k(\Delta_L)\cup Y_{[k+1, +\infty)}$ relative to $Y_{[k+1, +\infty)}$. Hence, for each $k<0$ there is a map $f_k\colon Y_{[k, +\infty)}\to Y_k$ that is a homotopy equivalence relative to $Y_{[k+1, +\infty)}$. Now define a map 
\[
Y\to \left(\bigcup_{i< 0}t^i(\Delta_L)\right)\cup Y_{[a_{\min}, +\infty)}
\]
as the identity on $Y_{[a_{\min}, +\infty)}$ and as $f_k$ when restricted to $Y_{[k, +\infty)} - Y_{[k+1, +\infty)}$ for all $k< 0$. It can be checked that this is a homotopy equivalence relative to $Y_{[a_{\min}, +\infty)}$, using the same idea for the homotopy inverse of each $f_k$.
\end{proof}

A symmetric proof yields the ascending version of \cref{freely_descending}.

\begin{proposition}
\label{freely_ascending}
Let $X$ be a one-relator complex with a single 0-cell and suppose that $\phi\in H^1(X, \R)$ factors through $\Z$ and is freely ascending. 
Let $t$ be the generator of $\deck(Y\to X)\isom \Z$ that translates in the positive direction. Let $\Delta_U\subset Y^{(1)}$ be the subgraph consisting of the vertex mapping to $a_{\max}$ and all outgoing edges and their endpoints except for $e_U$. Then $Y$ is homotopy equivalent to
\[
\left(\bigcup_{i> 0}t^i(\Delta_U)\right)\cup Y_{(-\infty, a_{\max}]}
\]
relative to $Y_{(-\infty, a_{\max}]}$.
\end{proposition}

\begin{proposition}
\label{freely_descending_ascending}
Let $X$ be a one-relator complex with a single 0-cell and suppose that $\phi\in H^1(X, \R)$ factors through $\Z$ and is freely ascending. 
Let $t$ be the generator of $\deck(Y\to X)\isom \Z$ that translates in the positive direction. Let $\Delta_L, \Delta_U$ be the subgraphs from \cref{freely_descending,freely_ascending}. Then there is a homotopy equivalence
\[
Y \to \left(\bigcup_{i\leqslant 0}t^i(\Delta_L)\right)\cup Y_{(a_{\min}, a_{\max}]}\cup \left(\bigcup_{i>0}t^i(\Delta_U)\right)
\]
fixing $Y_{(a_{\min}, a_{\max}]}$. In particular, $Y$ is homotopy equivalent to a graph.
\end{proposition}

\begin{proof}
Combining \cref{freely_descending} and \cref{freely_ascending}, we see that $Y$ is homotopy equivalent to
\[
Y \to \left(\bigcup_{i< 0}t^i(\Delta_L)\right)\cup Y_{[a_{\min}, a_{\max}]}\cup \left(\bigcup_{i>0}t^i(\Delta_U)\right).
\]
Since $Y_{[a_{\min}, a_{\max}]}$ is homotopy equivalent to $\Delta_L \cup Y_{(a_{\min}, a_{\max}]}$, the result follows.
\end{proof}

\begin{proof}[Proof of \cref{free_kernel}]
The image of $\pi_1(X)$ under $\phi$ is isomorphic to $\Z^n$ for some $n$. Since $Z_{\phi}$ is finite, we may find an appropriate epimorphism $\Z^n\to \Z$ such that, if $\phi'$ denotes the composition of $\phi$ with this map, then there is an isomorphism $Z_{\phi}\to Z_{\phi'}$ respecting the order on the vertices. This may be done for instance by choosing a line $L\subset \R^{m}$ (where $m = \rk(\pi_1(X^{(1)}))$) with rational slope sufficiently close to the line associated with $\phi$. In particular, $\phi'$ is freely descending and ascending because $\phi$ is and so $\ker(\phi')$ is free by \cref{freely_descending_ascending}. Since $\ker(\phi)\leqslant\ker(\phi')$, we have that $\ker(\phi)$ is free as claimed.
\end{proof}

\begin{conjecture}
\label{conjecture_free_kernel}
If $X$ is a one-relator complex with a single 0-cell and if $\phi\in H^1(X, \R)$, then $\ker(\phi)$ is free if and only if $\phi$ is freely ascending and descending.
\end{conjecture}

In the next section we discuss Brown's criterion which states that Conjecture \ref{conjecture_free_kernel} is true when $\chi(X) = 0$. In other words, when $\pi_1(X)$ is a two-generator one-relator group.

\cref{freely_descending,freely_ascending} actually imply the following weakening of \cref{free_kernel}. Recall that a HNN-extension $H*_{\psi}$ is \emph{ascending} if the domain of $\psi$ is all of $H$, \emph{descending} if the target of $\psi$ is all of $H$.

\begin{theorem}
\label{ascending}
Let $X$ be a one-relator complex with a single 0-cell and let $\phi\in H^1(X, \R)$. If $\phi$ is freely descending or freely ascending, then $\ker(\phi)$ is locally free. If $\phi$ also factors through $\Z$, then $\pi_1(X)$ splits as a descending or ascending HNN-extension of a free group $\pi_1(X) \isom F*_{\psi}$ such that $\normal{F} = \ker(\phi)$.
\end{theorem}

\begin{remark}
\label{ascending_remark}
A finitely generated ascending HNN-extension of a free group embeds into an ascending HNN-extension of a finitely generated free group by work of Chong--Wise \cite{CW24}. Since an ascending HNN-extension of a finitely generated free group is residually finite by work of Borisov--Sapir \cite{BS05}, \cref{ascending} also yields a criterion for residual finiteness of a one-relator group.
\end{remark}

We may also state a variation of Conjecture \ref{conjecture_free_kernel}.

\begin{conjecture}
\label{conjecture_locally_free_kernel}
If $X$ is a one-relator complex with a single 0-cell and if $\phi\in H^1(X, \R)$, then $\ker(\phi)$ is locally free if and only if $\phi$ is freely ascending or descending.
\end{conjecture}

Baumslag asked in \cite{Ba74} whether a (one-relator) group with locally free derived subgroup is residually finite. Answering Conjecture \ref{conjecture_locally_free_kernel} would solve this problem by \cref{ascending_remark}.

\subsection{Algebraic fibring, Brown's criterion and the Friedl--Tillmann polytope}
\label{sec:brown}

An \emph{algebraic fibring} of a group $G$ is an epimorphism $\phi\colon G\to \Z$ with $\ker(\phi)$ finitely generated. One can think of this as an algebraic analog of a fibration of a manifold over a circle. Algebraic fibrations of a group $G$ are encoded in its BNS-invariant, see Bieri--Neumann--Strebel \cite{BNS87}. We shall not define what the BNS-invariant of a group $G$ is, but we remark that it is an open subset of the character sphere $\Sigma^1(G)\subset S(G)$ and that $\ker(\phi)$ is finitely generated if and only if $[\phi], -[\phi]\in \Sigma^1(G)$. The reader is directed to \cite{St13} for further details.

Brown computed the BNS-invariant for all one-relator groups in \cite{Br87} and thus characterised when a one-relator group admits an algebraic fibring. Brown's theorem, known as \emph{Brown's criterion}, completely determines which homomorphisms $G\to \Z$ of a one-relator group have finitely generated kernel in terms of the relator. Below we state Brown's criterion in the language of \cref{sec:character_sphere}.

\begin{theorem}[Brown's criterion]
\label{brown}
Let $G = F(S)/\normal{w}$ be a one-relator group with $w$ cyclically reduced, let $X$ be its presentation complex and let $\phi\colon G\to \Z$ be an epimorphism. Then:
\begin{enumerate}
\item $[\phi]\in \Sigma^1(G)$ if and only if $|S| = 2$ and a bottommost edge in $Z_{\phi} = (\Lambda_{\phi}, \lambda_{\phi})$ is traversed once by $\lambda_{\phi}$.
\item $-[\phi]\in \Sigma^1(G)$ if and only if $|S| = 2$ and a topmost edge in $Z_{\phi} = (\Lambda_{\phi}, \lambda_{\phi})$ is traversed once by $\lambda_{\phi}$.
\end{enumerate}
In particular, $\ker(\phi)$ is finitely generated if and only if $|S| = 2$ and a topmost and a bottommost edge in $Z_{\phi}$ is traversed once by $\lambda_{\phi}$.
\end{theorem}

\begin{remark}
Dunfield--Thurston investigated algebraic fibring of one-relator groups from a probabilistic perspective in \cite{DT06} using Brown's criterion. Specifically they showed that the probability that a random two-generator one-relator group algebraically fibres lies somewhere between $0.0006$ and $0.975$. They also remark that computer experiments indicate the probability may be limiting towards $0.94$.
\end{remark}

As mentioned in \cref{sec:coabelian}, Brown's criterion is a converse to \cref{free_kernel} for two-generator one-relator groups. Let us provide a sketch proof of the last statement of \cref{brown}. This part of the statement was proved first by Moldavanski\u{\i} \cite{Mo67}.

\begin{proof}[Sketch proof]
Let us first assume that a topmost and bottommost edge in $Z_{\phi}$ is traversed once by $\lambda_{\phi}$, then by \cref{freely_descending_ascending}, we have
\[
\ker(\phi) = \pi_1\left(\left(\bigcup_{i\leqslant 0}t^i(\Delta_L)\right)\cup Y_{(a_{\min}, a_{\max}]}\cup \left(\bigcup_{i>0}t^i(\Delta_U)\right)\right)
\]
For each integer $k\geqslant1$ , denote by
\[
Y_k = \left(\bigcup_{-k\leqslant i\leqslant 0}t^i(\Delta_L)\right)\cup Y_{(a_{\min}, a_{\max}]}\cup \left(\bigcup_{0<i\leqslant k}t^i(\Delta_U)\right)
\]
By definition of $\Delta_L$ and $\Delta_U$, $Y_{k+1} - Y_k$ contains a single vertex and $|S| - 1$ many edges. Thus, $\chi(Y_{k+1}) = \chi(Y_k) - |S| + 2$. This implies that $\ker(\phi)$ is finitely generated if and only if $|S| = 2$. Note that $|S| = 1$ cannot hold since otherwise there would be no epimorphism to $\Z$. We have established one direction of the claim.

Now suppose that $\ker(\phi)$ is finitely generated. By \cref{one-relator_splitting}, if $\rho\colon Y\to X$ denotes the cover induced by $\phi$ and $t\in \deck(\rho)$ is a generator, there is a one-relator $\Z$-domain $Z\subset Y$ such that $\pi_1(X)\isom \pi_1(Z)*_{\psi}$ where $\psi$ is induced by an isomorphism between two Magnus subgraphs $Z_0, Z_1\subset Z$, where $Z_0 = t^{-1}(Z)\cap Z$ and $Z_1 = Z\cap t(Z)$. Importantly, we have that $Z_0, Z_1$ both have one more vertex and $|S|$ many more edges than $Z$. We have
\begin{align}
\label{amalgam}
\ker(\phi) = \pi_1(Y) = \ldots \underset{\pi_1(Z_0)}{*} \pi_1(Z)\underset{\pi_1(t(Z_0))}{*}\pi_1(t(Z))\underset{\pi_1(t^2(Z_0))}{*}\ldots
\end{align}
Hence, if $\pi_1(Y)$ is finitely generated, then for some $k\geqslant 0$ we must have
\[
\pi_1(Y) = \langle \pi_1(t^{-k}(Z)), \ldots, \pi_1(Z), \pi_1(t(Z)), \ldots, \pi_1(t^k(Z))\rangle.
\]
By the amalgam decomposition (\ref{amalgam}), we have that $\pi_1(t^{k+1}(Z))\leqslant \pi_1(t^k(Z))$ which implies that $\pi_1(t^{k+1}(Z)) = \pi_1(t^{k+1}(Z_0))$. Arguing in the same way for $\pi_1(t^{-k-1}(Z))$, we see that $\pi_1(t^{-k-1}(Z)) = \pi_1(t^{-k-1}(Z_1))$ also. Hence, we have
\begin{align}
\label{eqn:equality_1}
\pi_1(Z) &= \pi_1(Z_0)\\
\label{eqn:equality_2}		&= \pi_1(Z_1)
\end{align}
We shall now need a lemma. It is the topological translation of \cite[Lemma 1]{Mo67}.

\begin{lemma}
\label{Magnus_subgroup_isomorphism}
Let $X = (\Lambda, \lambda)$ be a one-relator complex and let $M\subset X$ be a Magnus subgraph. The inclusion $\pi_1(M)\injects \pi_1(X)$ is an isomorphism if and only if there is a 1-cell $e\subset \Lambda - M$ traversed by $\lambda$ precisely once and $\Lambda - e$ deformation retracts onto $M$.
\end{lemma}

In order to complete the proof, we may use (\ref{eqn:equality_1}) and (\ref{eqn:equality_2}) combined with \cref{Magnus_subgroup_isomorphism} to conclude that there are 1-cells $e_L\subset Z - Z_1$ and $e_U\subset Z - Z_0$ that are both traversed precisely once by $\lambda_{\phi}$. By the first part of the proof this also implies that $|S| = 2$ and so the proof is complete.
\end{proof}

\subsubsection{The Friedl--Tillmann polytope}

In this section we describe the Friedl--Tillmann of a two-generator one-relator group, as introduced in \cite{FT20}. This is a marked polytope in $H_1(G, \R) = \R^2$ which encodes the BNS-invariant of a one-relator group via its marking. It can be constructed directly from a presentation and it is (up to translation) independent of the choice of one-relator presentation.

Let us fix a free group $F = F(a, b)$ and a cyclically reduced word $w$ contained in the commutator subgroup $[F(a, b), F(a, b)]$. Recall from \cref{sec:character_sphere} the definition of $\Lambda_w$. Let $\mathcal{C}_w\subset \R^2$ be the convex hull of $\Lambda_w$. Mark all vertices in $\mathcal{C}_w$ that are traversed precisely once by $\lambda_w$, the trace of $w$. The \emph{Friedl--Tillmann polytope} $\mathcal{M}_w\subset H_1(G, \R) = \R^2$ for $F/\normal{w}$ is then defined as the convex hull of the midpoints of the squares contained in $\mathcal{C}_w$. A vertex of the polytope $\mathcal{M}_w$ is marked if all the vertices of the corresponding square in $\mathcal{C}_w$ are marked. See \cref{fig:polytope} for an illustration of the marked polytope $\mathcal{M}_w$ for the one-relator group from \cref{Z_example}.

\begin{figure}
\centering
\includegraphics[scale = 0.32]{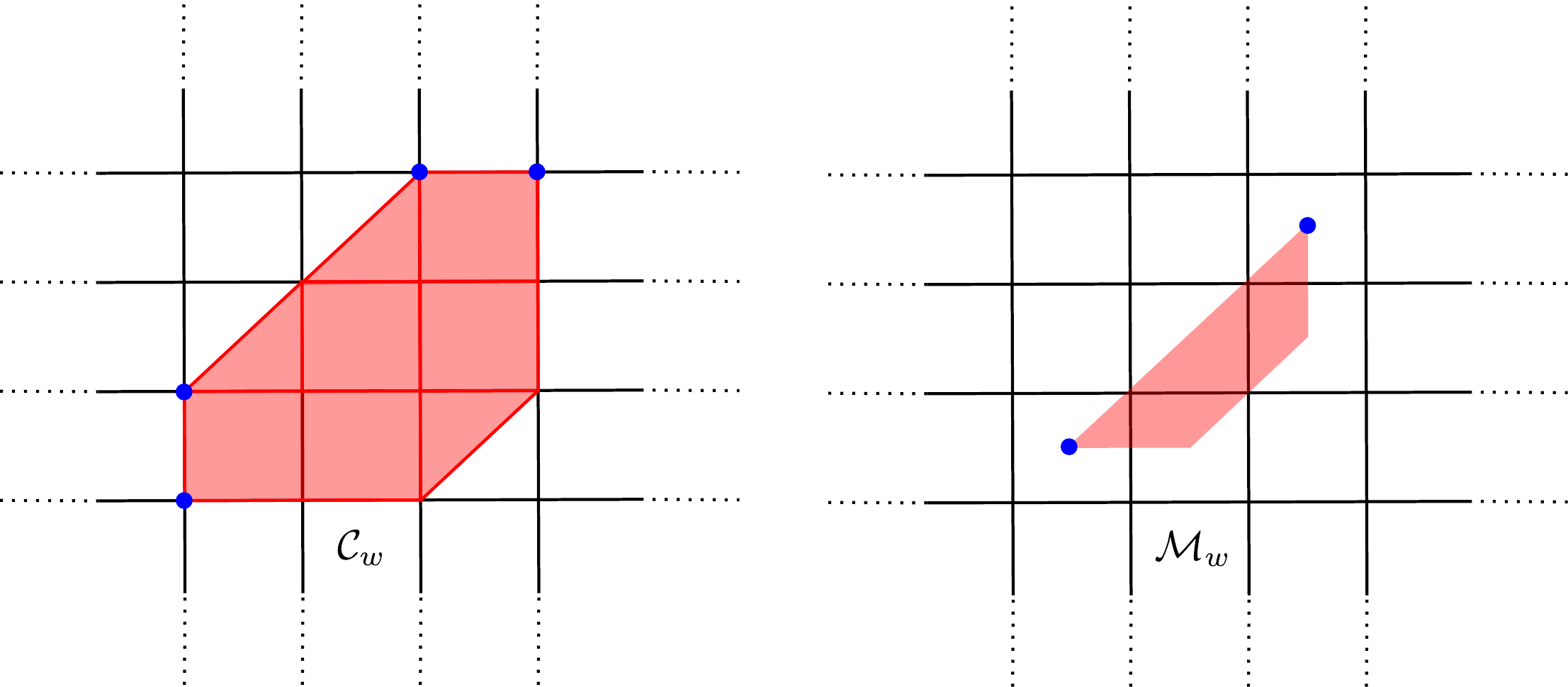}
\caption{Friedl--Tillmann's marked polytope $\mathcal{M}_w$ for the one-relator group with relator $w = a^2b^3ab^{-2}a^{-1}bab^{-1}a^{-2}bab^{-1}a^{-2}b^{-1}$. The marked vertices are in blue.}
\label{fig:polytope}
\end{figure}

An element $\phi\in H^1(G, \R)$ \emph{pairs maximally} with a vertex $v$ of $\mathcal{M}_w$ if $\phi(v)>\phi(w)$ for all other vertices $w\neq v$ in $\mathcal{M}_w$. Friedl--Tillmann prove the following result in \cite{FT20}, relating their polytope to the BNS-invariant of a one-relator group.

\begin{theorem}
Let $F = F(a, b)$, let $w\in [F, F]$ be a cyclically reduced word and let $G = F/\normal{w}$. Then $\phi\in H^1(G, \R)$ lies in the BNS-invariant $\Sigma^1(G)$ if and only if $\phi$ pairs maximally with a marked vertex of $\mathcal{M}_w$.
\end{theorem}

In the same article, Friedl--Tillmann conjecture that $\mathcal{P}_w$ is an invariant of the group $F/\normal{w}$, up to translation. They prove that it is the case for a large class of one-relator groups, but their conjecture was only confirmed in full by Henneke--Kielak in \cite[Theorem 5.12]{HK20}.

\begin{theorem}
Let $F = F(a, b)$, let $w\in [F, F]$ be a cyclically reduced word and let $G = F/\normal{w}$. Then the marked polytope $\mathcal{M}_w$ is an invariant of $G$, up to translation.
\end{theorem}

See \cref{sec:splitting_complexity} for a connection between the thickness of the Friedl--Tillmann polytope of a one-relator group and its splitting complexity.

\subsection{Splitting complexity, JSJ-decompositions and amalgamated free products}
\label{sec:splitting_complexity}

If $G$ is a group and $\phi\colon G\to \Z$ is an epimorphism, a \emph{splitting of $(G, \phi)$} is an HNN-extension decomposition $G \isom H*_{\psi}$ such that $\normal{H} = \ker(\phi)$. Recall that if $\psi$ identifies $A\leqslant H$ with $B\leqslant H$, then $A$ and $B$ are the \emph{associated groups}. The \emph{splitting complexity} of $(G, \phi)$ is
\[
c(G, \phi) = \min{\{\rk(A) \mid (G, \phi) \text{ splits with associated group $A$}\}}.
\]
The \emph{free splitting complexity} of $(G, \phi)$ is
\[
c_f(G, \phi) = \min{\{\rk(A) \mid (G, \phi) \text{ splits with free associated group $A$}\}}.
\]
If $G$ is a finitely presented group (or more generally, of type $\fp_2(\Z)$) and $\phi\colon G\to \Z$ is an epimorphism, then Bieri--Strebel proved \cite[Theorem A]{BS78} that $c(G, \phi)<\infty$. Although $c_f(G, \phi)$ in general could be infinite, when $G$ is a finitely generated one-relator group it is finite for all $\phi$ by \cref{one-relator_splitting}.

In \cite{FT20}, Friedl--Tillmann related the splitting complexity of many pairs $(F/\normal{w}, \phi)$ with the \emph{thickness} of the polytope $\mathcal{M}_w$ with respect to $\phi$:
\[
\thick_{\phi}(\mathcal{M}_w) = \max{\{\phi(p) - \phi(q) \mid p, q\in \mathcal{M}_w\}}
\]
where here $\phi$ is defined on $\mathcal{M}_w$ via its natural extension to a homomorphism $H_1(G, \R) = \R^2\to \R$. Henneke--Kielak then generalised their results to all two-generator one-relator groups in \cite[Theorem 6.4]{HK20}.

\begin{theorem}
\label{splitting_complexity}
Let $F = F(a, b)$, let $w\in [F, F]$ be a cyclically reduced word that is not a proper power and let $G = F/\normal{w}$. For any epimorphism $\phi\colon G\to \Z$, we have
\[
c(G, \phi) = c_f(G, \phi) = \thick_{\phi}(\mathcal{M}_w) + 1.
\]
\end{theorem}

\begin{remark}
\label{remark_Z_splitting}
The upper bound $c_f(G, \phi)\leqslant \thick_{\phi}(\mathcal{M}_w)$ from \cref{splitting_complexity} is \cite[Proposition 7.3]{FT20}. The proof shows that there is a one-relator splitting for $(G, \phi)$ with associated Magnus subgroups of rank precisely $\thick_{\phi}(\mathcal{M}_w) + 1$. This can also be proved directly using \cref{one-relator_splitting} and the description of one-relator splittings from \cref{sec:character_sphere}. 
\end{remark}

\begin{qstn}
Is there an analogue to \cref{splitting_complexity} valid for all one-relator groups?
\end{qstn}

Let $G = F/\normal{w}$ be a one-relator group as in \cref{splitting_complexity}. One immediate consequence of \cref{splitting_complexity} is that $G$ splits as a HNN-extension over $\Z$ if and only if $\thick_{\phi}(G) = 0$ for some $\phi$, or, equivalently, if and only if the polytope $\mathcal{M}_{w}$ is a line. In light of \cref{remark_Z_splitting}, $G$ then splits as a HNN-extension over an infinite cyclic group if and only if $G$ admits a one-relator splitting with associated Magnus subgroups isomorphic to $\Z$.

Gardam--Kielak--Logan applied this to two-generator one-relator hyperbolic groups in \cite{GKL24} to compute the JSJ-decomposition of $G$. Logan was able to understand the case of two-generator one-relator groups with torsion in even more detail and worked out a complete characterisation of all outer automorphism groups in \cite{Lo16}. The reader is directed towards Guirardel--Levitt \cite{GL17} for more information on JSJ-decompositions.

\begin{problem}
Characterise the JSJ-decompositions of one-relator groups.
\end{problem}

Although a lot is known about HNN-extension decompositions, as evidenced above and by \cref{sec:one-relator_splittings}, not much is known about amalgamated free product decompositions of one-relator groups. Call an amalgamated free product $G = A*_CB$ \emph{proper} if $A, B, C$ are finitely generated, $C$ is a proper subgroup of $A$ and $B$, and $C$ has index at least three in $A$ or $B$. Wall asked the following question in \cite[Question F2]{Wa79}.

\begin{qstn}[Wall]
Which one-relator groups split as proper amalgamated free products?
\end{qstn}

It is not too hard to show that a proper amalgamated free product contains a free subgroup of rank two. Thus, $\bs(1, n)$ is an infinite family of one-relator groups that cannot admit any proper amalgamated free product decomposition. Nevertheless, many one-relator groups do andmit such decompositions. Baumslag--Shalen showed in \cite{BS90} that every group $G$ of deficiency at least two admits a \emph{proper} amalgamated free product decomposition. That is, $G = A*_CB$ with $A, B, C$ finitely generated, $C$ a proper subgroup of $A$ and $B$ and with $C$ of index greater than two in $A$ or $B$. In particular, this applies to one-relator groups generated by at least three elements. Such decompositions were also shown to hold for all non-cyclic one-relator groups with torsion by Benyash-Krivets \cite{BK98}. Wall's question does not yet have an answer for torsion-free two-generator one-relator groups.

\begin{theorem}
A one-relator group $G = F/\normal{w}$ splits as a proper amalgamated free product unless $\rk(F) = 2$ and $G$ is torsion-free.
\end{theorem}

Beyond existence of amalgamated free product decompositions of one-relator groups, it would also be interesting to know what structure such amalgamated free products can have. Fine--Peluso proposed several conjectures in \cite{FP99} which became known as the Amalgam Conjectures. We state the weakest possible such conjecture.

\begin{conjecture}[The Amalgam Conjecture]
\label{q: BS}
If $G$ is a one-relator group with at least three generators, then $G\isom A*_CB$ with $C$ a free group and $A, B$ either one-relator groups or free groups.
\end{conjecture}

Note that in general, if $A*_CB$ is finitely presented and $A, B, C$ are finitely generated, it is possible that $A, B, C$ are not finitely presented, see \cite[Section 6]{BS90} for an example. In case $G = A*_CB$ is a one-relator group, we know that $A, B, C$ are finitely presented since $G$ is coherent, see \cref{sec:coherence}.

\section{Group algebra properties}
\label{sec:cohomology}

Recall that if $G$ is a group and $R$ is a ring, then the \emph{group algebra} $RG$ (or $R[G]$) is the $R$-algebra with free $R$-basis in correspondence with the elements of $G$ and with multiplication of basis elements given by the multiplication from $G$. In other words, it is the ring with underlying set
\[
RG = \left\{ \sum_{g\in G}r_gg \, \middle\vert\, r_g\in R \text{ and } r_g = 0 \text{ for all but finitely many $g\in G$}\right\}.
\]
and ring operations defined in the natural way. In this section we will be studying the properties of group rings of one-relator groups.

\subsection{(Co)homology and the relation module}
\label{sec:homology}

Let $F$ be a free group and let $N\triangleleft F$ be a normal subgroup. Then $F$ acts by conjugation on the abelian group 
\[
N_{\ab} = N/[N, N].
\]
This action extends linearly to a (left) action of $\Z F$ on $N_{\ab}$. Since $N$ acts trivially on its abelianisation, this action descends to an action of $\Z G$, making $N_{\ab}$ a (left) $\Z G$-module called the \emph{relation module}. Another way to view the relation module of a group presentation $G = F(S)/N$ is as the $\Z G$-module of cycles $N_{\ab} = Z_1(\cay(G, S))$ of the Cayley graph with respect to the generating set $S$. See \cref{fig:relation_module}.

\begin{figure}
\centering
\includegraphics[scale = 0.5]{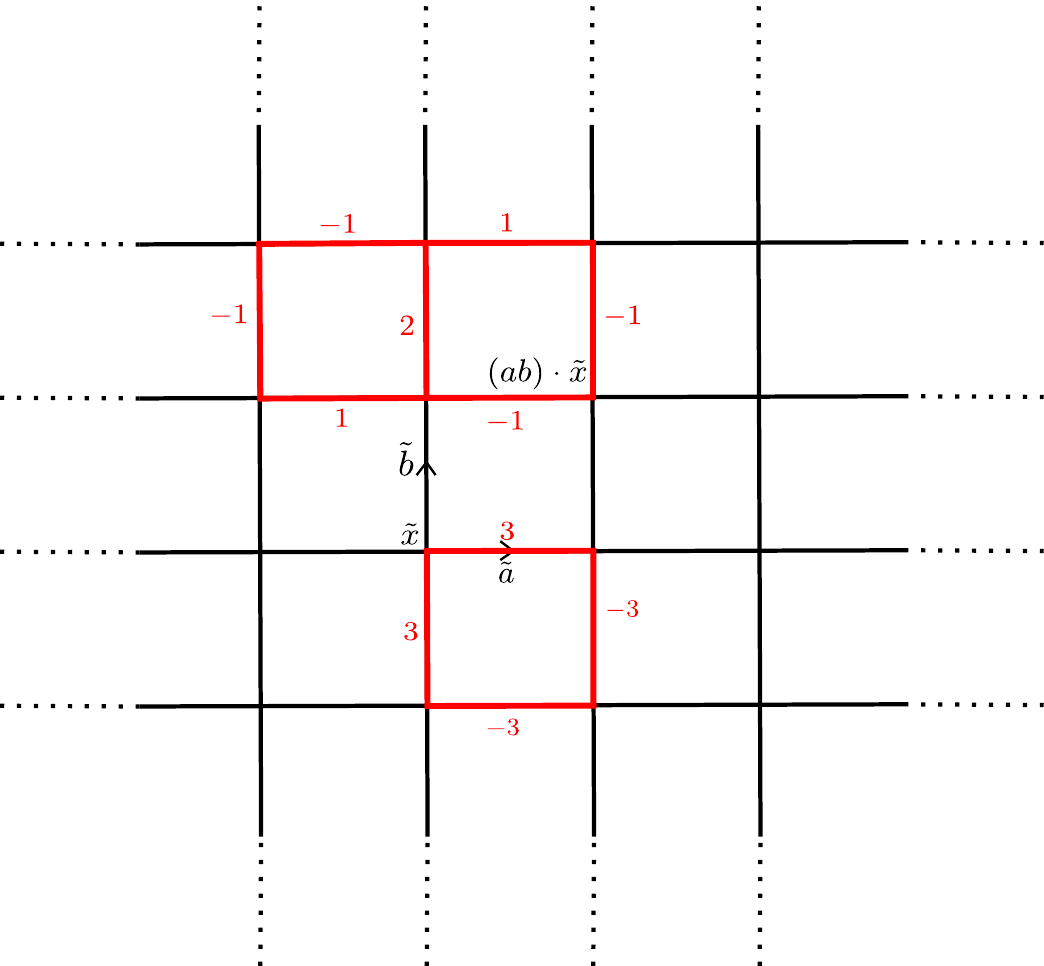}
\caption{The Cayley graph for $G = F(a, b)/\normal{[a, b]}$. The cycle $(3a^{-1} + b - ab^{-1})\cdot [a, b][N, N]$ in $Z_1(\cay(G, \{a, b\})$ is depicted in red.}
\label{fig:relation_module}
\end{figure}

Lyndon computed the relation module of a one-relator group in \cite{Ly50}, a theorem which became known as Lyndon's identity theorem.

\begin{theorem}[Lyndon's Identity Theorem]
\label{lyndon_id}
Let $G = F/\normal{w^n}$ be a one-relator group with $w$ not a proper power and with $n\geqslant 1$. Denoting by $N = \normal{w^n}$, we have
\[
N_{\ab} \isom \Z G/(w-1).
\]
In particular, if $n = 1$ then
\[
N_{\ab} \isom \Z G.
\]
\end{theorem}

The simplest way to prove \cref{lyndon_id} is to show that $\normal{w^n}\triangleleft F$ has the \emph{Cohen--Lyndon property}: there is a set of $\normal{w^n}, \langle w\rangle$ double coset representatives $T\subset F$ such that $\normal{w^n} \isom *_{t\in T}tw^nt^{-1}$. This was first done by Cohen and Lyndon in \cite{CL63}.

The main use for \cref{lyndon_id} is to compute the (co)homology of a one-relator group. Indeed, by \cite[Proposition 5.4]{Br82}, if $G = F(S)/N$, we have an exact sequence
\[
\begin{tikzcd}
0 \arrow[r] & N_{\ab} \arrow[r] & \bigoplus_{s\in S}\Z G \arrow[r] & \Z G \arrow[r] & \Z \arrow[r] & 0.
\end{tikzcd}
\]
So in the case that $G = F/\normal{w^n}$, with $w$ not a proper power and $n=1$, the above gives us a free resolution of $\Z$ as a trivial $\Z G$-module and we may compute homology by applying $-\otimes_{\Z G}M$ or cohomology by applying $\hom_{\Z G}(-, M)$ to the above. If $n>1$, then $N_{\ab}$ is not free (or even projective), but we may extend the exact sequence to a free resolution:
\[
\begin{tikzcd}
\ldots \arrow[r, "w-1"] & \Z G \arrow[r, "\sum_{i=0}^{n-1}w^i"] & \Z G \arrow[r, "w-1"] & \Z G \arrow[r] & \bigoplus_{s\in S} \Z G \arrow[r] & \Z G \arrow[r] & \Z \arrow[r] & 0.
\end{tikzcd}
\]
We have not described three of the above maps. The remaining maps are:
\[
z \mapsto z\cdot\left(\frac{\partial}{\partial s}(w^n)\right)_{s\in S}, \quad (z_s)_{s\in S}\mapsto\sum_{s\in S}z_s\cdot (s-1), \quad z\mapsto \sum_{g\in G}z_g
\]
where $\frac{\partial}{\partial s}$ denotes the \emph{Fox derivative} with respect to $s$.

\begin{corollary}
If $G = F/\normal{w^n}$ is a one-relator group with $w$ not a proper power, then $\cd_{\Q}(G)\leqslant 2$. If $n = 1$, then $\cd_{\Z}(G) \leqslant 2$.
\end{corollary}

The fact that $\cd_{\Q}(G)\leqslant 2$ when $G$ has torsion follows by noting that after applying $-\otimes_{\Z}\Q$ to Lyndon's resultion, the map $w-1$ becomes the zero map and so we obtain a resolution 
\[
\begin{tikzcd}
0 \arrow[r] & \Q G \arrow[r] & \bigoplus_{s\in S} \Q G \arrow[r] & \Q G \arrow[r] & \Q \arrow[r] & 0.
\end{tikzcd}
\]
since $\Q$ is a flat $\Z$-module.

Lyndon's identity theorem allows us to compute the derived functors $\Tor^{RG}_n(-, R)$ and $\Ext_{RG}^n(-, R)$ when $G$ a one-relator group. In order to compute 
\[
\Tor^{RG}_n(-, -), \quad \Ext_{RG}^n(-, -)
\]
one needs to be able to produce resolutions of arbitrary $RG$-modules $M$, not just the trivial one $R$. Lewin--Lewin proved a `two sided simple identity theorem' in \cite{LL78}. As a corollary, they showed that if $G = F(S)/\normal{w}$ is a torsion-free one-relator group, $K$ is a field and $M$ is a (left) $KG$-module, then
\[
\begin{tikzcd}
0 \arrow[r] & KG\otimes_{K}M \arrow[r] & \bigoplus_{s\in S}KG\otimes_{K}M \arrow[r] & KG\otimes_{K}M \arrow[r] & M \arrow[r] & 0.
\end{tikzcd}
\]
is a free resolution of $M$. In particular, $\Tor_i^{KG}(-, -) = \Ext^i_{KG}(-, -) = 0$ for all $i>2$. This resolution was used in \cite{JZL23} to compute the weak dimension of certain $KG$-modules with $G$ a torsion-free one-relator group.

Recall that a ring $R$ has \emph{(left) global dimension} at most $n$ if $\Ext^{n+1}_R(A, B) = 0$ for all (left) $R$-modules $A$ and $B$. \cite[Theorem 9]{LL78} states:

\begin{corollary}
Let $G = F/\normal{w}$ be a torsion-free one-relator group and let $R$ be a ring. Then the (left) global dimension of $RG$ is at most the (left) global dimension of $R$ plus two. In particular, if $R = K$ is a field, the global dimension of $KG$ is at most two.
\end{corollary}

Note that if $R$ is commutative, the opposite ring $RG^{\op}$ is isomorphic to $RG$ and so left and right global dimension of $RG$ coincide.

\begin{remark}
Such a resolution can be constructed in a similar way for arbitrary groups, see \cite[Proposition 2.2]{JZL23}.
\end{remark}

\subsection{Groups with cyclic relation module}

A natural question that arises from \cref{lyndon_id} is whether having a cyclic relation module is limited to one-relator groups.

\begin{qstn}[Harlander]
\label{qstn:Harlander}
If $G = F/N$ is a group with relation module $N_{\ab}$ cyclic, is $N$ normally generated by a single element?
\end{qstn}

This question is a special case of the \emph{relation gap problem} which asks whether a finitely presented group $G = F/N$ can have strictly fewer relation module generators than $N$ has normal generators. Bestvina--Brady showed in \cite{BB97} that there exist groups with infinite relation gap. That is, groups that are not finitely presented but have type $\fp_2(\Z)$, a property equivalent to the relation module being finitely generated. Counterexamples to the relation gap problem could lead to counterexamples to old conjectures on the homotopy type of CW-complexes such as Wall's $D(2)$ conjecture. The following theorem from \cite{Lin24b} makes it unlikely that counterexamples will come from groups with cyclic relation module.

\begin{theorem}
\label{no_relation_gap}
If $G = F/N$ is a right orderable group with $N_{\ab}$ a cyclic relation module, then $N$ is normally generated by a single element. In particular, $G$ is a one-relator group.
\end{theorem}

Recall that a locally indicable group is right orderable by the Burns--Hale theorem \cite{BH72} and a right orderable group is one with a total order on its elements, invariant under multiplication from the right. A key step in the proof of \cref{no_relation_gap} involved showing that normal subgroups of one-relator groups which have right orderable quotients, must be normally generated by a single element. Indeed, if $G = F/N$ and $N_{\ab}$ is cyclic, then there is an element $w\in N$ whose image in $N_{\ab}$ generates the relation module and such that the map $F/\normal{w}\to G$ has perfect kernel. This is explained in detail in \cite{Ha15}.

\begin{qstn}
Are normal perfect subgroups of one-relator groups normally generated by a single element?
\end{qstn}

A positive answer to this question would solve Question \ref{qstn:Harlander}. Note that perfect subgroups of one-relator groups are necessarily infinitely generated. In the torsion-free case this is by local indicability, in the torsion case, the natural quotient map $F/\normal{w^n}\to F/\normal{w}$ onto the torsion-free one-relator group allows us to reduce to considering finitely generated subgroups in the kernel. But the kernel splits as a free product of (infinitely many) copies of $\Z/n\Z$, a group which contains no perfect subgroups. This last fact is due to Fisher--Karrass--Solitar \cite[Theorem 1]{FKS72}.

Another closely related problem is the \emph{relation lifting problem}, asked by Wall in \cite{Wa66}. If $G = F/N$ is a presentation and $r_1, \ldots, r_n\in N_{\ab}$ are $\Z G$-module generators, then the relation lifting problem asks whether there exist elements $w_1, \ldots, w_n\in N$ such that
\[
\normal{w_1, \ldots, w_n} = N, \quad w_1[N, N] = r_1, \ldots, w_n[N, N] = r_n.
\]
The relations $w_1, \ldots, w_n$ are called \emph{lifts} of $r_1, \ldots, r_n$. Using the fact that $\Z G$ is a domain and contains no non-trivial units under the hypotheses of \cref{no_relation_gap}, one can show that every generator of $N_{\ab}$ can be lifted. On the other hand, Dunwoody showed in \cite{Du72} that the following element
\[
(1-a+a^2)b\cdot a^5[N, N]
\]
cannot be lifted to a normal generator of $\normal{a^5} =N\triangleleft F(a, b)$. The proof uses the Magnus property for free groups, \cref{Magnus_property}, to show that there is no lift and the fact that $(1-a+a^2)b$ is a non-trivial unit in $\Z G$ to show that $(1-a+a^2)b\cdot a^5[N, N]$ is a generator.

Generalisations of both the relation lifting and relation gap problems can be considered for relation modules over arbitrary rings $R$. In this case, the $R$-relation module of $G = F/N$ is the left $RG$-module
\[
R\otimes_{\Z}N_{\ab}.
\]
The main theorem of \cite{Lin24b} applies to $R$-relation modules with $R$ an arbitrary domain. In this slightly more general setting, it is no longer true that every generator $r\in R\otimes_{\Z}N_{\ab}$ may be lifted to a normal generator of $N$. However, there will always be some element $w\in N$ and a unit $u\in R^{\times}$ such that $\normal{w} = N$ and $r = u\otimes w[N, N]$.

\begin{conjecture}
Let $G = F/N$ be a torsion-free group such that $\Q\otimes_{\Z}N_{\ab}$ is cyclic as a $\Q G$-module, then $N$ is normally generated by a single element.
\end{conjecture}

\begin{remark}
There exist groups with torsion which are not one-relator groups, but which have a presentation with cyclic $\Q$-relation module. Such an example is given in \cite{Lin24b}.
\end{remark}

\subsection{$L^2$-Betti numbers and the strong Atiyah conjecture for one-relator groups}

Let $G$ be a group and denote by $\ell^2(G)$ the space of square summable functions $G\to \C$. In other words, $\ell^2(G)$ is the space of formal sums
\[
\ell^2(G) = \left\{ \sum_{g\in G} c_g\cdot g \,\Bigm\vert\, c_g\in \C, \sum_{g\in G}\|c_g\|<\infty\right\}.
\]
with square summable support. The \emph{group von Neumann algebra} is the algebra of bounded $G$-equivariant operators from $\ell^2(G)$ to itself
\[
\mathcal{N}(G) := \mathcal{B}(\ell^2(G))^G.
\]
The ring of affiliated operators $\mathcal{U}(G)$ is the Ore localisation of $\mathcal{N}(G)$ with respect to the set of non-zero divisors. Then one way to define the $n^{\text{th}}$ \emph{$L^2$-Betti number} of $G$ is:
\[
b_n^{(2)}=\dim_{\mathcal{U}(G)}\Tor^{\Z G}_n(\mathcal{U}(G), \Z)
\]
where $\dim_{\mathcal{U}(G)}$ is the von Neumann dimension, see L\"uck's book \cite[Definition 8.30]{Lu02} for details. The $L^2$-Betti numbers of one-relator groups were computed by Dicks--Linnell \cite{DL07}.

\begin{theorem}
\label{L2_betti}
Let $G = F/\normal{w^n}$ be a one-relator group with $w$ not a proper power and $n\geqslant 1$. Then
\begin{align*}
b_0^{(2)}(G) =& \begin{cases}
	\frac{1}{n} & \text{ if $G\isom \Z/n\Z$}\\
	0 & \text{ otherwise}
	\end{cases}\\
b_1^{(2)}(G) =&\, 1 - \rk(F) + \frac{1}{n}\\
b_2^{(2)}(G) =&\, 0
\end{align*}
\end{theorem}

The strong Atiyah conjecture predicts that the von Neumann dimension of the kernel of operators $\ell^2(G)^n\to \ell^2(G)^m$ induced by right multiplication by $m\times n$ matrices over $\C G$ should lie in $\frac{1}{\lcm(G)}\Z$, where $\lcm(G)$ denotes the lowest common multiple of the orders of finite subgroups of $G$. Another equivalent formulation of the strong Atiyah conjecture can be stated in terms of $L^2$-Betti numbers, see \cite[Chapter 10]{Lu02}. For torsion-free groups, the strong Atiyah conjecture over a subfield $K\subset \C$ is equivalent to the division closure of $K G$ in $\mathcal{U}(G)$ being a division ring. See \cite[Proposition 1.2]{LS12} for this fact.

Jaikin-Zapirain--L\'{o}pez-\'{A}lvarez confirmed the strong Atiyah conjecture for one-relator groups in \cite{JL20}. In fact, they prove something much stronger.

\begin{theorem}
\label{strong_Atiyah}
If is a $G$ is a locally indicable group, then $G$ satisfies the strong Atiyah conjecture over any subfield $K\subset \C$.
\end{theorem}

Combining \cref{strong_Atiyah} with a result of Kevin Schreve \cite{Sc14}, Jaikin-Zapirain--L\'{o}pez-\'{A}lvarez obtain the following corollary.

\begin{corollary}
If $G$ is a one-relator group, then $G$ satisfies the strong Atiyah conjecture.
\end{corollary}

\subsection{Embedding the group ring of a torsion-free one-relator group into a division ring}

Recall that a \emph{division ring} is a ring in which each non-zero element is a unit. Lewin--Lewin proved in \cite{LL78} that the group ring of a torsion-free one-relator group embeds into a division ring.

\begin{theorem}
\label{LL_embedding}
If $G = F/\normal{w}$ is a torsion-free one-relator group and $k$ is a division ring, then $kG\injects \Di$ for some division ring $\Di$. 
\end{theorem}

\cref{LL_embedding} provided the first proof of Kaplansky's zero divisor conjecture for torsion-free one-relator groups. Kaplansky's zero divisor conjecture is also a consequence of local indicability, proved later by Brodskii and Howie.

The proof of \cref{LL_embedding} uses the Magnus hierarchy and various ring constructions to explicitly and inductively construct the division ring $\Di$. These constructions include Ore localisation, skew polynomial ring constructions, Cohn's universal matrix inverting construction and coproducts of rings. The background and details on these notions are covered in detail in Cohn's book \cite{Co06}. 

One of the main difficulties in Lewin--Lewin's construction was ensuring that at each stage in the hierarchy, the division rings were $(H, G)$-free for all Magnus subgroups $H$ of the one-relator group $G$. Let $k$ be a division ring, $G$ a group and $kG\injects \Di$ an embedding into a division ring. If $N\leqslant H\leqslant G$ are subgroups and $\Di_{N}\leqslant \Di$ denotes the division closure of $kN$ in $\Di$, then $\Di$ is \emph{$(N, H)$-free} if the multiplication map
\[
\Di_N\otimes_{kN}kH \to \Di
\]
given by $d\otimes h \mapsto dh$ is injective. Note that if $T\subset H$ is a right transversal for $N$, then we have $\Di_N\otimes_{kN}kH\isom \bigoplus_{t\in T}\Di_N\otimes t$ as a left $kN$-module.

There is another way to embed a group ring of a torsion-free one-relator group into a division ring which we shall now explain. Firstly, we define two types of division ring embeddings which play a particularly important role. Let $kG\injects \Di$ be an epic division ring embedding. Recall that $kG\injects \Di$ is epic of $\Di$ coincides with the division closure of $kG$. The embedding is \emph{Hughes-free} if it is $(N, H)$-free for every finitely generated subgroup $H\leqslant G$ and every normal subgroup $N\triangleleft H$ with $H/N\isom \Z$. A well-known theorem of Hughes \cite{Hu70} states that a locally indicable group admits at most one Hughes-free embedding into a division ring, up to isomorphism. The embedding is \emph{Linnell} if it is $(N, G)$-free for any subgroup $N\leqslant G$. When the division closure of $\C G$ in the ring of affiliated operators $\mathcal{U}(G)$ is a division ring, this division ring embedding is Linnell. Furthermore, if $G$ is locally indicable, then by Hughes' theorem, its Linnell division ring in $\mathcal{U}(G)$ will also be Hughes-free. In particular, the division ring embedding one obtains for $KG$, with $G$ a torsion-free one-relator group and $K\subset \C$ a subfield, from \cref{strong_Atiyah} is Linnell.

Let $G$ be a right orderable group, let $\leqslant$ be a right $G$-invariant total order and let $k$ be a division ring. Denote by
\[
k\series{G, \leq} = \left\{r = \sum_{g\in G}r_g\cdot g \Bigm\vert \supp(r) \text{ is well-ordered}\right\}
\]
the \emph{space of Mal'cev--Neumann series}. A powerful theorem of Malcev and Neumann \cite{Ma48,Ne49} shows that if $\leq$ is a bi-ordering of $G$, then $k\series{G, \leq}$ actually has the structure of a division ring called the Mal'cev--Neumann completion. In particular, if $G$ is a free group, then it admits a bi-ordering and so it embeds in its Mal'cev--Nuemann completion. This is in fact the base case of the Lewin--Lewin division ring construction. Lewin showed in \cite{Le74} that in this case the embedding of $kG$ into its division closure in its Mal'cev Neumann completion is Hughes-free (in fact, also Linnell).

When $\leq$ is not a bi-ordering, $k\series{G, \leq}$ does not have a natural ring structure. In this case one should look at the \emph{endomorphism ring} of $k\series{G, \leq}$, denoted by $\End(k\series{G, \leq})$. The group ring $kG$ acts naturally on the right, and so $kG$ embeds in $\End(k\series{G, \leq})$. It turns out that the endomorphism ring $\End(k\series{G, \leq})$ is a division ring in the case that $G$ is locally indicable and $\leq$ is a \emph{Conradian} right order. Note that admitting a Conradian right order is equivalent to being locally indicable, a fact which is a consequence of a result of Brodskii \cite{Bro84}, explained in \cite{RR02}. The following is obtained by combining a result of Jaikin-Zapirain--L\'opez-\'Alvarez \cite[Corollary 1.4]{JL20} with a result of Gr\"ater \cite[Theorem 8.1 \& Corollary 8.3]{Gr20}. See also \cite[Theorem 2.5]{JZL23} and the preceding discussion.

\begin{theorem}
\label{division_ring}
Let $G$ be a locally indicable group, $\leq$ a Conradian right order on $G$ and let $K$ be a field of characteristic zero. Then the ring $\End(K\series{G, \leq})$ is a division ring. Moreover, the division closure of $KG$ in $\End(K\series{G, \leq})$ is a Linnell ring.
\end{theorem}

\begin{corollary}
If $G$ is a torsion-free one-relator group and $K$ is a field of characteristic zero, then $KG$ embeds into a Linnell division ring $\Di_{KG}$. Moreover, $\Di_{KG}$ is the unique epic Linnell division $KG$-ring.
\end{corollary}

By a result of Fisher--S\'{a}nchez-Peralta \cite{FSP24}, the Lewin--Lewin division ring construction coincides with the division ring from \cref{division_ring} for $KG$ where $K$ is a field of characteristic zero and $G$ a torsion-free one-relator group.

\subsection{Coherence of the group ring}
\label{sec:ring_coherence}

Recall that a presentation of a (left) $R$-module $M$ is an exact sequence
\[
\begin{tikzcd}
F_1 \arrow[r] & F_0 \arrow[r] & M \arrow[r] & 0
\end{tikzcd}
\]
where $F_0$ and $F_1$ are free (left) $R$-modules.

\begin{definition}
A ring $R$ is \emph{(left) coherent} if all of its finitely generated (left) ideals are finitely presented.
\end{definition}

Coherence of rings was introduced by Bourbaki in the 1960's. In general, it is important to distinguish between left and right modules as there are examples of rings which are left coherent but not right coherent. See, for example, \cite[Example (e), p. 139]{La72}. However, for group rings over a commutative ring, the two notions coincide. One of the main results in \cite{JZL23} is the following.

\begin{theorem}
\label{ring_coherence}
If $G$ is a one-relator group and $K$ is a field of characteristic $0$, then $K[G]$ is coherent.
\end{theorem}

A right $R$-module $M$ is said to have \emph{weak dimension at most $k$} if $\Tor_{k+1}^R(M, N) = 0$ for all left $R$-modules $N$. Say $M$ has weak dimension $k$ if it has weak dimension at most $k$, but not at most $k-1$. The proof of \cref{ring_coherence} makes use of the following criterion for when a finitely generated $R$-module is finitely presented. See \cite[Proposition 3.1]{JZL23}.

\begin{proposition}
\label{ring_coherence_criterion}
Let $R$ be a ring of global dimension at most two and let $R\injects \Di$ be an embedding into a division ring. If $M$ is a finitely generated right $R$-module, then
\[
M \text{ is finitely presented} \iff \dim_{\Di}\Tor_1^R(M, \Di)<\infty.
\]
In particular, if $\Di$ has weak dimension at most one, then $R$ is coherent.
\end{proposition}

To see that the last claim in \cref{ring_coherence_criterion} follows from the first, note that if $I\leqslant R$ is an ideal, then
\[
\Tor_1^{R}(I, \Di) = \Tor_2^R(R/I, \Di)
\]
and so if $\Di$ has weak dimension at most one, then $\Tor_1^R(I, \Di) = 0$.

The idea to prove \cref{ring_coherence} is to use \cref{ring_coherence_criterion} on the group ring of a one-relator group. There are two cases to consider: when $G$ is torsion-free and when $G$ has torsion.

Suppose that $G$ is a torsion-free one-relator group and $K$ is a field. As we saw in \cref{sec:homology}, $K[G]$ has global dimension at most two and so it suffices to show that $\Di$ has weak dimension at most one. Let $M$ be a left $K[G]$-module. From \cref{sec:homology} we have the following resolution
\[
\begin{tikzcd}
0 \arrow[r] & KG\otimes_{K}M \arrow[r] & \bigoplus_{s\in S}KG\otimes_{K}M \arrow[r] & KG\otimes_{K}M \arrow[r] & M \arrow[r] & 0.
\end{tikzcd}
\]
which comes from the standard resolution of $K$ as a trivial $K[G]$-module. By \cref{division_ring}, there exists a division ring embedding $K[G]\injects \Di$, turning $\Di$ into a right $K[G]$-module. Applying $\Di\otimes_{K[G]}-$ we obtain a complex
\[
\begin{tikzcd}
\Di\otimes_{K}M \arrow[r, "d_2"] & \bigoplus_{s\in S}\Di\otimes_{K}M \arrow[r, "d_1"] & \Di\otimes_{K}M \arrow[r, "d_0"] & M \arrow[r] & 0.
\end{tikzcd}
\]
By definition:
\[
\Tor_2^{K[G]}(\Di, M) =\ker(d_2).
\]
The delicate part of the proof of \cref{ring_coherence} is now to show that $\Di\otimes_KM$ is a torsion-free $K[G]$ module, and show that this implies that $\ker(d_2) = 0$, see \cite[Lemma 5.1 \& Theorem 5.3]{JZL23}. In order to carry out this step, one needs to know that the division ring $\Di$ into which $K[G]$ embeds is Hughes-free. Since such a division ring is only known to exist when $K$ has characteristic 0 by \cref{division_ring}, \cref{ring_coherence} requires this hypothesis. However, the proof carries over directly if one were to show that a Hughes-free division ring exists for $K[G]$ with $G$ a torsion-free one-relator group and $K$ an arbitrary field.

If $G$ is a one-relator group with torsion, the ring $K[G]$ does not have global dimension at most two. Dawid Kielak and the first author showed in \cite{KL24} that one-relator groups with torsion have a finite index subgroup that is free-by-cyclic, see \cref{sec:vfibring}. Let $F\rtimes_{\psi} \Z$ be the finite index free-by-cyclic subgroup of $G$. Now $K[F]$ embeds into its Mal'cev--Neumann completion, giving us a division ring embedding $K[F]\injects \Di$. Since $K[F\rtimes_{\psi} \Z] \isom (K[F])[t^{\pm1}, \psi]$, localising $\Di[t^{\pm1}, \overline{\psi}]$, where $\overline{\psi}$ denotes the extension of $\psi\colon K[F]\to K[F]$ to $\Di$, yields a division ring embedding $K[F\rtimes_{\psi}\Z]\injects \Ore\left(\Di\left[t^{\pm1}, \overline{\psi}\right]\right)$. Then general properties of the twisted Laurent polynomial ring and Ore localisation implies that $\Ore\left(\Di\left[t^{\pm1}, \overline{\psi}\right]\right)$ has weak dimension one and so \cref{ring_coherence_criterion} can also be applied in this case. Since the coherence property passes between finite codimensional ideals, the group ring $K[G]$ is also coherent.

We conclude this section by remarking that the proof of group ring coherence does not carry over to $R[G]$ when $R$ is not a field as the global dimension of $R[G]$ is no longer at most two. Even the case $R = \Z$ is open and interesting.

\section{Subgroup properties}
\label{sec:subgroups}

\subsection{Stackings and non-positive immersions}
\label{sec:npi}

Dani Wise introduced the notions of non-positive immersions and negative immersions in \cite{Wi03,Wi04} as tools to understand better the subgroup structure of certain groups with an eye towards solving Baumslag's conjecture on the coherence of one-relator groups.

The independent proofs that torsion-free one-relator groups have presentation complexes with non-positive immersions by Helfer--Wise \cite{HW16} and Louder--Wilton \cite{LW17}, led to a sequence of beautiful developments in understanding the subgroups of one-relator groups which we shall explain in the next few sections.

\begin{definition}
\label{def:npi}
A 2-complex $X$ has \emph{non-positive immersions} if for every immersion $Y\immerses X$ with $Y$ compact and connected, either $\pi_1(Y) = 1$ or $\chi(Y)\leqslant 0$.
\end{definition}

\begin{remark}
Other variations of the non-positive immersions property replace the condition in \cref{npi} that $Y$ be simply connected with the condition that $Y$ be contractible or collapsible. Yet another variation, \emph{weak non-positive immersions}, requires that $\chi(Y)\leqslant 1$ for all immersions $Y\immerses X$ with $Y$ compact and connected.
\end{remark}

When Dani Wise introduced the notion of non-positive immersions in \cite{Wi03}, he conjectured that all one-relator complexes $X = (\Lambda, \lambda)$ with $\deg(\lambda) = 1$ have non-positive immersions. Wise had reduced his conjecture to a conjecture about free groups which he called the $w$-cycles conjecture: if $F$ is a free group, $w\in F$ is not a proper power and $H\leqslant F$ is any finitely generated free group, then
\[
\sum_{WfH}\rk(W^f\cap H)\leqslant \rk(H)
\]
with equality only if $H\leqslant \normal{w}$. This can also be interpreted as a rank 1 version of the (strengthened) Hanna Neumann inequality. In \cite{LW17} Louder--Wilton solved the $w$-cycles conjecture. Their proof uses the idea of a stacking.

\begin{definition}
\label{def:stacking}
Let $\lambda\colon\bbS\immerses \Gamma$ be an immersion of a disjoint union of copies of $S^1$ into a graph $\Gamma$. A \emph{stacking} of $\lambda$ is an embedding $\widetilde{\lambda}\colon \bbS\injects \Gamma\times \R$ such that $\pi\circ\widetilde{\lambda} = \lambda$ where $\pi\colon \Gamma\times \R\to \Gamma$ is the projection map.
\end{definition}

Now fix an immersion $\lambda$ as in \cref{def:stacking} and a stacking $\widetilde{\lambda}$. Denote by $\iota\colon \Gamma\times \R\to \R$ the projection to $\R$. Define the \emph{top} of the stacking to be
\[
\mathcal{T} = \{x\in \bbS \mid \forall y\neq x \quad (\lambda(x) = \lambda(y) \implies \iota(\widetilde{\lambda}(x))>\iota(\widetilde{\lambda}(y)))\}\subset \bbS
\]
whereas the \emph{bottom} is the subset
\[
\mathcal{B} = \{x\in \bbS \mid \forall y\neq x \quad (\lambda(x) = \lambda(y) \implies \iota(\widetilde{\lambda}(x))<\iota(\widetilde{\lambda}(y)))\}\subset \bbS
\]
The stacking is \emph{good} if the top and the bottom intersect each component of $\bbS$ non-trivially. Louder--Wilton then show that good stackings always exist for immersions $\lambda\colon S^1\immerses \Lambda$, where $\deg(\lambda) = 1$. Their argument makes use of the existence of maximal $\Z$-tower liftings of $\lambda$, a fact due to Howie \cite{Ho82}. The topological one-relator hierarchy described in \cref{sec:hierarchy} yields, in fact, an explicit maximal tower lifting of $\lambda$. Alternatively, as Wise points out in \cite{Wi20}, one can use a result of Farrell \cite{Fa76}, the right orderability of torsion-free one-relator groups and Weinbaum's theorem \cite{We72}.

A key insight of Louder--Wilton is that good stackings can be used to compute Euler characteristics. This led to the following result, whch is a combination of \cite[Lemma 12 \& 16]{LW17}. Recall that the \emph{core} of a graph $\Gamma$, denoted by $\core(\Gamma)$, is the union of the images of all immersed cycles $S^1\immerses\Gamma$.

\begin{proposition}
\label{good_stacking}
Let $\Lambda$ be a graph and let $\lambda\colon \bbS\immerses \Lambda$ be an immersion of a disjoint union of cycles $S^1$. Suppose that $\lambda$ admits a good stacking (for example, if $\bbS = S^1$ and $\deg(\lambda) = 1$). Let $\Lambda'\immerses \Lambda$ be a graph immersion and consider the following diagram
\[
\begin{tikzcd}
\core(\mathbb{S}\times_{\Lambda}\Lambda') \arrow[r, "\lambda'"] \arrow[d, "\sigma"] & \Lambda' \arrow[d] \\
\mathbb{S} \arrow[r, "\lambda"]                        & \Lambda           
\end{tikzcd}
\]
arising from the pullback diagram. Then either there is some edge $e\subset \Lambda'$ such that $\lambda'^{-1}(e)$ is a single edge, or we have
\[
-\chi(\im(\lambda'))\geqslant \deg(\sigma).
\]
\end{proposition}

The $w$-cycles conjecture follows immediately from \cref{good_stacking}. We explain in some detail how to obtain the non-positive immersions property from \cref{good_stacking}.

\begin{theorem}
\label{npi}
If $X = (\Lambda, \lambda)$ is a one-relator complex with $\deg(\lambda) = 1$, then $X$ has non-positive immersions.
\end{theorem}

\begin{proof}
Let $Z = (\Gamma, \gamma)$ be a 2-complex and let $Z\immerses X$ be an immersion. If there is a 1-cell $e\subset \Gamma$ such that $\gamma^{-1}(e)$ contains a single 1-cell, then by removing $e$ and the 2-cell whose boundary cycle traverses $e$, we obtain a subcomplex of $Z$ which is homotopy equivalent to $Z$. Thus, we may assume that no such $1$-cell exists. Consider the following commutative diagram
\[
\begin{tikzcd}
\mathbb{S}_{\Gamma} \arrow[rrd, loop->, bend left, "\gamma"] \arrow[rd, hook] \arrow[rdd, loop->, bend right, "\sigma"'] &                                                                &                  \\
                                                                       & S^1\times_{\Lambda}\Gamma \arrow[r, loop->] \arrow[d, loop->] & \Gamma \arrow[d, loop->] \\
                                                                       & S^1 \arrow[r, loop->, "\lambda"]                      & \Lambda         
\end{tikzcd}
\]
where $\gamma\colon\bbS_{\Gamma}\immerses \Gamma$ is the attaching map for the 2-cells in $Z$ and where the fact that $\bbS_{\Gamma}\injects S^1\times_{\Lambda}\Gamma$ is injective follows from the fact that $Z\immerses X$ is an immersion. We may apply \cref{good_stacking} to the immersion $\Gamma\immerses\Lambda$ to conclude that
\[
-\chi(\Gamma)\geqslant-\chi(\gamma(\bbS_{\Gamma})) \geqslant \deg(\sigma) = |\pi_0(\bbS_{\Gamma})|.
\]
Since
\[
\chi(Z) = \chi(\Gamma) - |\pi_0(\bbS_{\Gamma})|
\]
the proof is complete.
\end{proof}

\begin{remark}
The proof actually shows something stronger: if $Z\immerses X$ is an immersion of a compact and connected 2-complex, then either $\chi(Z)\leqslant 0$ or $Z$ collapses through free faces to a point (collapsing non-positive immersions).
\end{remark}

Other proofs of \cref{npi} are also worth mentioning. Helfer--Wise prove \cref{npi} at the same time as Louder--Wilton in \cite{HW16}, using the notion of \emph{(bi)-slim} structures. In \cite{BCGW24} Bamberger--Carrier--Gaster--Wise showed that the existence of a good stacking for an immersion $\lambda\colon \bbS\immerses\Lambda$ is equivalent to the existence of a bi-slim structure on $X = (\Lambda, \lambda)$. Since the existence of a slim structure is strictly more general, the proof in \cite{HW16} applies to a somewhat larger class of 2-complexes. For example, they show that \emph{reducible} 2-complexes without proper powers (in the sense of Howie \cite{Ho82}) have slim structures and so have non-positive immersions. A third proof was provided by Howie--Short in \cite{HS23} using orderability and diagrams. This last proof could be applied in a more general `relative' setting. A similar result was independently obtained by Millard in his thesis \cite{Mi21} using `relative' stackings. Finally, a completely different proof was given in \cite{JZL23}, using homological algebra.

Not many general properties of 2-complexes with non-positive immersions are known. Wise conjectured in \cite{Wi03} that if $X$ has non-positive immersions, then $\pi_1(X)$ is coherent. That is, every finitely generated subgroup of $\pi_1(X)$ is finitely presented. Although this conjecture remains open, in \cref{sec:coherence} we will see that significant progress has been made. Properties that are known to hold for all 2-complexes with non-positive immersions are the following, due to Wise \cite{Wi22} (see \cite[Proposition 2.7]{JZL23}).

\begin{theorem}
\label{npi_properties}
If $X$ is a 2-complex with non-positive immersions and $\pi_1(X)\neq 1$, then:
\begin{itemize}
\item $\pi_1(X)$ is locally indicable.
\item $X$ is aspherical.
\end{itemize}
\end{theorem}

Note that these properties were known to hold for (presentation complexes of) torsion-free one-relator groups, see \cref{sec:hierarchy}.

\subsection{(Uniform) Negative immersions}
\label{sec:unpi}

In \cite{Wi04}, Dani Wise introduced a strengthening of non-positive immersions which he called negative immersions. After Louder--Wilton introduced stackings in \cite{LW17}, they then wrote \cite{LW22} in which they used stackings to characterise precisely when a one-relator complex has negative immersions. However, the definition Louder--Wilton used was slightly different, although morally very similar, to that of Wise. The advantage of the Louder--Wilton definiton is that it is invariant under Nielsen equivalence. Since we shall be covering the results of Louder--Wilton, we shall be working with their definition here.

Before stating the definition of negative immersions, we shall need to explain what it means for a 2-complex to be reducible or irreducible. For this we follow \cite{LW24}. There the notion of \emph{branched immersions} is also used. However, we stick to regular immersions.

Let $X$ be a 2-complex. We say $X$ is \emph{visibly reducible} if one of the following holds:
\begin{enumerate}
\item\label{itm:red_1} $X$ contains a vertex of degree at most one.
\item\label{itm:red_2} $X$ contains a locally separating vertex.
\item\label{itm:red_3} $X$ contains a free face.
\end{enumerate}
Morally, $X$ is visibly reducible if there is an obvious way to simplfy it: remove hanging trees (corresponding to (\ref{itm:red_1})), decompose $X$ as a wedge of two subcomplexes (corresponding to (\ref{itm:red_2})) or collapse a 2-cell through an edge (corresponding to (\ref{itm:red_3})). A combinatorial map of 2-complexes $\phi\colon Y\to X$ is an \emph{essential equivalence} if the following two conditions hold:
\begin{enumerate}
\item  $\phi^{(1)}\colon Y^{(1)}\to X^{(1)}$ is a homotopy equivalence that factors as a sequence of folds (in the sense of Stallings \cite{St83})
\item $\phi$ is a homeomorphism on the interiors of 2-cells.
\end{enumerate}

\begin{definition}
A 2-complex $X$ is \emph{reducible} if there exists an essential equivalence $Y\to X$ such that $Y$ is visibly reducible, \emph{irreducible} otherwise.
\end{definition}

\begin{definition}
A 2-complex $X$ has \emph{negative immersions} if for every immersion $Y\immerses X$ with $Y$ compact and connected, either $\chi(Y)<0$ or $Y$ is reducible.
\end{definition}

One of the main theorems of \cite{LW22} characterises when a one-relator complex has negative immersions in terms of the subgroup structure of the fundamental group. Recall that a group is \emph{$k$-free} if every $k$-generated subgroup is free. 

\begin{theorem}
\label{nim}
If $X$ is a one-relator complex, then $X$ has negative immersions if and only if $\pi_1(X)$ is 2-free.
\end{theorem}

Louder--Wilton later introduced a stronger uniform variant of negative immersions in \cite{LW24}.

\begin{definition}
\label{def:unim}
A 2-complex $X$ has \emph{uniform negative immersions} if there exists some $\epsilon>0$ such that for every immersion $Y\immerses X$ with $Y$ compact, connected and irreducible, we have
\[
\frac{\chi(Y)}{\#\{\text{2-cells in $Y$}\}} \leqslant -\epsilon.
\]
\end{definition}

\begin{remark}
Wise's definition of negative immersions is more akin to \cref{def:unim}, except that reducible is replaced with `no free faces and not a point'. In \cite[Section 3.4]{LW24} it is shown that these two notions are distinct even for one-relator complexes.
\end{remark}

The main advantage of the uniform negative immersions property over the negative immersions property is that it allows for much stronger results about subgroups. This was exploited in \cite{LW24} in which Louder--Wilton showed that 2-free one-relator groups are coherent. We will cover this result and other further corollaries in \cref{sec:coherence}. The following is \cite[Theorem C]{LW24}.

\begin{theorem}
\label{unim}
If $X$ is a one-relator complex with negative immersions, then $X$ has uniform negative immersions.
\end{theorem}

The proof of \cref{unim} is a beautiful application of linear programming techniques. The authors encode certain types of maps from 2-complexes to a 2-complex $X$ in a linear system of equations and inequalities, showing that the maps which maximise ``curvature" are rational vertices in a convex polytope. Since these rational vertices can be realised by some map $Y_{\max}\to X$, the ``curvature" of $Y_{\max}$ provides the $\epsilon$ needed to verify that $X$ has uniform negative immersions. The reader is invited to consult \cite{Wi24} in which Wilton formally defines these curvature invariants for all 2-complexes, explores their properties and proposes several interesting directions for further research.

So far, we have only considered one-relator complexes $X = (\Lambda, \lambda)$ with $\deg(\lambda) = 1$ and so the reader might be wondering what can one say about the case $\deg(\lambda)>1$. In \cite{LW17}, Louder--Wilton show that such one-relator complexes have the property of \emph{not-too-positive immersions}. Stronger properties were obtained in \cite{LW20} and \cite{LW24}. We state \cite[Theorem 5.1]{LW24} below which combines \cref{good_stacking} with a result of Fisher--Karrass--Solitar \cite{FKS72} stating that a one-relator group with torsion is virtually torsion-free. A similar result can be found in \cite{Wi22} (see also \cite{Wi20}).

\begin{theorem}
\label{v_unim}
If $X = (\Lambda, \lambda)$ is a one-relator complex with $\deg(\lambda)>1$, then there is a finite sheeted cover $Y\immerses X$ and an inclusion $\iota\colon Z\injects Y$ inducing an isomorphism on $\pi_1$ such that $Z$ has uniform negative immersions.
\end{theorem}

In order to get around passing to a finite index subgroup, a one-relator group with torsion can be considered as the fundamental group of a one-relator orbicomplex with a single non-trivial cone point in the centre of the 2-cell. Louder--Wilton show in \cite{LW20} that such complexes have uniform negative immersions.

\subsection{Primitivity rank and $w$-subgroups}
\label{sec:primitivity_rank}

Recall that an element $w$ in a free group $F$ is \emph{primitive} if it is part of a free basis for $F$. Equivalently, if $F\isom F'*\langle w\rangle$ for some $F'\leqslant F$. Call $w$ \emph{imprimitive} if $w$ is not primitive.

Primitivity rank was introduced by Puder in \cite{Pu14}. The definition plays a key role in Louder--Wilton's results in \cite{LW22}. 

\begin{definition}
\label{def:prim_rank}
If $F$ is a free group and $w\in F$ is an element, its \emph{primitivity rank} $\pi(w)$ is the following quantity:
\[
\pi(w) = \min{\{\rk(H) \mid w\in H\leqslant F, \, \text{$w$ is imprimitive in $H$}\}}\in \N\cup \{+\infty\},
\]
where $\pi(1) = 0$ and $\pi(w) = +\infty$ if $w$ is primitive.
\end{definition}

The first connection Louder--Wilton discovered between properties of a one-relator group and the primitivity rank of its defining relation is the following.

\begin{theorem}
\label{k-free}
If $G = F/\normal{w}$ is a one-relator group, then $G$ is $k$-free if and only if $k<\pi(w)$. In particular:
\begin{enumerate}
\item $G$ is torsion-free if and only if $\pi(w) \geqslant 2$.
\item The one-relator presentation complex for $G$ has uniform negative immersions if and only if $\pi(w)\geqslant 3$.
\end{enumerate}
\end{theorem}

Several $k$-freeness results had been previously established for cyclically pinched, cyclically conjugated and related one-relator groups, for instance, by Baumslag in \cite{Ba62}, Rosenberger in \cite{Ro82} and Fine--Gaglione--Rosenberger--Spellman in \cite{FGRS95}. These are all encompassed by Louder--Wilton's results.

If $w\in F$ is an element, one might wonder what the subgroups that arise in \cref{def:prim_rank} tell us about the subgroup structure of $F/\normal{w}$. This leads us to the notion of $w$-subgroups.

\begin{definition}
If $F$ is a free group and $w\in F$ is an element, the $w$-subgroups of $F$ are the subgroups $H\leqslant F$ satisfying the following:
\begin{itemize}
\item $w\in H$ and $\pi(w) = \rk(H)$.
\item $H$ is maximal for inclusion amongst subgroups satisfying the above.
\end{itemize}
\end{definition}

The following is due to Puder \cite[Appendix A]{Pu14}. See also \cite[Lemma 6.4]{LW22}.

\begin{lemma}
For any given $w\in F$, there are finitely many $w$-subgroups and there is an algorithm to compute them, and hence to compute $\pi(w)$.
\end{lemma}

The upshot of the above and \cref{k-free} is that there is an algorithm to decide when a one-relator group is $k$-free for some $k\geqslant 1$ and hence when its presentation complex has (uniform) negative immersions.

We briefly sketch the algorithm, due to Puder \cite{Pu14}. Let $\lambda\colon S^1\immerses \Lambda$ be an immersion realising the inclusion $\langle w\rangle \injects F$. Enumerate all graph immersions $\gamma\colon\Gamma\immerses\Lambda$ that $\lambda$ surjectively factors through:
\[
\begin{tikzcd}
S^1\arrow[r, twoheadrightarrow] & \Gamma \arrow[r, loop->, "\gamma"] & \Lambda
\end{tikzcd}
\]
Using Whitehead's algorithm \cite{Wh36}, determine in which graphs $\Gamma$ the cycle $S^1\surjects\Gamma$ represents a primitive element of $\pi_1(\Gamma)$ and throw all of these away. Additionally, throw away all graphs $\Gamma$ such that there is already a graph immersion $\Gamma'\immerses\Lambda$ in our list that $\Gamma\immerses\Lambda$ factors through and such that $\chi(\Gamma')\geqslant \chi(\Gamma)$. Finally, throw away all graphs $\Gamma$ whose Euler characteristic do not attain the maximal value amongst the graphs remaining. The graph immersions $\Gamma\immerses\Lambda$ left over each correspond to $w$-subgroups of $F = \pi_1(\Lambda)$ (under the $\pi_1$-map at the basepoint) and the primitivity rank can be computed as $\pi(w) = 1-\chi(\Gamma)$.

In \cite{LW22} it is shown that if $H\leqslant F$ is a $w$-subgroup, then $H$ is malnormal and the homomorphism
\[
H/\normal{w} \to F/\normal{w}
\]
is injective \cite[Lemma 6.3 \& Theorem 6.17]{LW22}. The image subgroups $H/\normal{w}\leqslant F/\normal{w}$ are then also called \emph{$w$-subgroups} of the one-relator group. \cref{k-free} tells us about subgroups of rank at most $\pi(w) - 1$, it turns out that one can also say something about the subgroups of rank $\pi(w)$, see \cite[Theorem 1.5]{LW22}.

\begin{theorem}
\label{w-subgroup}
If $G = F/\normal{w}$ is a one-relator group, then every subgroup $H\leqslant G$ with $\rk(H) = \pi(w)$ is either free or conjugate into a $w$-subgroup of $G$.
\end{theorem}

If $X = (\Lambda, \lambda)$ is a one-relator complex, there are natural one-relator complexes $Q\immerses X$ immersing into $X$ which represent the $w$-subgroups. These are obtained by taking the unique core graph immersion $\Gamma\immerses \Lambda$ which represents a $w$-subgroup of $\pi_1(\Lambda)$ and attaching a 2-cell along the unique lift of $\lambda$. The following is \cite[Theorem 3.3.16]{Lin22_thesis}.

\begin{theorem}
If $X = (\Lambda, \lambda)$ is a one-relator complex and $Q\immerses X$ represents a $w$-subgroup, then there is a one-relator tower
\[
Q = X_N\immerses\ldots\immerses X_1\immerses X_0 = X.
\]
\end{theorem}

This has the corollary that one can decide whether a given element of a one-relator group lies inside a $w$-subgroup.

When $\pi(w) = 2$, Louder--Wilton show that there is only one $w$-subgroup.

\begin{lemma}
\label{pi=2}
If $w\in F$ is an element with $\pi(w) = 2$, then there is precisely one $w$-subgroup $H\leqslant F$.
\end{lemma}

In light of \cref{w-subgroup} and \cref{pi=2}, we ask the following question.

\begin{qstn}
\label{qstn:w-subgroup}
If $F$ is a free group and $w\in F$ is an element with $\pi(w)\geqslant 3$, is it true that there is a unique $w$-subgroup?
\end{qstn}

In \cite{CH23}, Cashen--Hoffman used a computer to verify various conjectures about one-relator groups with relator length that is not too long. In particular, they answer Question \ref{qstn:w-subgroup} positively for all words $w\in F_4$ of length at most $16$ and with $\pi(w) = 3$. Note that if $\pi(w) = \rk(F)$, then $w$ trivially has a unique $w$-subgroup, namely $F$. The primitivity rank may also have a relation with the stable commutator length, see \cite[Conjecture 1.7]{LW24}. This was also explored computationally by Cashen--Hoffman.

Moldavanski\u{\i} asked the following question in the case $\pi(w) = 1$ \cite[11.63]{Ko18}. We believe it should be true for any primitivity rank.

\begin{qstn}
\label{qstn:normal_w-subgroup}
If $G = F/\normal{w}$ is a one-relator group and $K$ is a $w$-subgroup. Is it true that if $H\leqslant G$ is a subgroup such that $H\cap\normal{K} = 1$, then $H$ is free?
\end{qstn}

If Question \ref{qstn:w-subgroup} turns out to have a negative answer, then Question \ref{qstn:normal_w-subgroup} should be modified appropriately.

If $K$ is a finite field of characteristic $p$, Ernst-West--Puder--Seidel define in \cite{EPS24} the \emph{$p$-primitivity rank} of an element $w\in F$ in terms of ideals in the group ring $K[F]$ of the free group with coefficients in $K$:
\[
\pi_p(w) = \min{\{\rk(I) \mid w- 1 \in I \lneq K[F], \text{ $w-1$ is imprimitive in $I$}\}}.
\]
They proved that $\pi_q(w)\leqslant \pi(w)$ for all $q$ and conjectured that $\pi_q(w) = \pi(w)$. One could also conjecture whether an analog of \cref{k-free} holds for group rings of free groups. Say a ring $R$ is $k$-free if all ideals of rank at most $k$ in $R$ are free.

\begin{qstn}
\label{q:ideals}
If $G = F/\normal{w}$ is a one-relator group and $K$ is a finite field of characteristic $p$, is it true that $K[G]$ is $k$-free for $k < \pi_q(w)$?
\end{qstn}

Question \ref{q:ideals} is true when $\pi_q(w) = 2$ for the following reason. We remarked that $\pi_q(w)\leqslant \pi(w)$ and so $G = F/\normal{w}$ is torsion-free. Hence, if $\pi_q(w) = 2$, then $G$ is locally indicable and so $K[F]$ has no-zero divisors by \cref{LL_embedding}. Since cyclic ideals of $K[G]$ being free is the same thing as $K[G]$ being a domain, the claim follows.

If true, Question \ref{q:ideals} is likely to be very difficult to prove. Some related work was done by Avramidi in \cite{Av22} who proved that the group ring of the fundamental group of a surface of sufficiently high genus is 2-free. A possibly simpler and also interesting variation of Question \ref{q:ideals} could be the question as to whether $K[G]$ is 2-free when $G = F/\normal{w}$ with $\pi(w)\geqslant 3$.

\subsection{Coherence}
\label{sec:coherence}

We saw in \cref{sec:ring_coherence} the coherence property for rings, in this section we cover the analogous notion for groups. 

\begin{definition}
A group $G$ is \emph{coherent} if all finitely generated subgroups of $G$ are finitely presented.
\end{definition}

Although the coherence property for groups had been studied before, the term appears to have first been used for groups by Serre in \cite{Ne74}. The first explicit example of a finitely presented incoherent group was provided by Baumslag--Boone--Neuman in \cite{BBN59}. The example they gave was $F_2\times \bs(1, 2)$.

Some of the first classes of groups shown to be coherent were the classes of cyclically pinched and cyclically conjugated one-relator groups, a result due to Karrass--Solitar \cite{KS70,KS71}. A group is \emph{cyclically pinched} if it has a presentation of the form:
\[
\langle A, B \mid u(A) = w(B)\rangle \isom F(A) \underset{\langle u\rangle = \langle w\rangle}{*}F(B).
\]
That is, if it is the amalgamated free product of two free groups over a cyclic subgroup. The cyclically conjugated one-relator group are defined analogously, but with HNN-extensions. A group is \emph{cyclically conjugated} if it has a presentation of the form:
\[
\left\langle A, t \mid t^{-1}u(A)t = w(A)\right\rangle\isom F(A)*_{\langle u\rangle^t = \langle w\rangle}.
\]
Karrass--Solitar mention in \cite{KS70} that their coherence result solved a conjecture of Baumslag's. Later, Baumslag conjectured that in fact all one-relator groups should be coherent in \cite{Ba74}. 

A push to solve this conjecture by Dani Wise, and later Lars Louder and Henry Wilton, sparked a lot of progress in understanding the subgroup structure of one-relator groups. We will discuss all this progress in detail in this section. First we mention the most general, but weakest, result which was proven by Andrei Jaikin-Zapirain and the first author in \cite{JZL23}.

\begin{theorem}
\label{coherence}
If $G$ is a one-relator group, then $G$ is coherent.
\end{theorem}

In general, how might one prove a finitely generated group $H$ is finitely presented? Starting with a finitely generated free group $G_0$ and a surjection $G_0\surjects H$, we might try and choose elements $r_1, r_2, \ldots\in G_0$ such that $\normal{r_1, r_2, \ldots} = \ker(G_0\surjects H)$. This gives us a diagram of epimorphisms
\begin{align}
\label{sequence}
G_0\surjects G_1\surjects \ldots \surjects G_n\to G_{n+1} \surjects \ldots \surjects H
\end{align}
where each $G_i = G_0/\normal{r_1, \ldots, r_i}$ is finitely presented and so that
\[
\lim_{i\to \infty} G_i = H.
\]
Then $H$ will be finitely presented if and only if for some $k$ the epimorphism $G_k\to H$ is an isomorphism. When $H$ is a subgroup of a group with sufficiently nice properties, then it might be possible to define a `complexity' on the candidate presentation, choose relations $r_i$ in a clever way so that this complexity goes down, eventually reaching something of minimal complexity which yields an actual presentation. This is the main idea in the proof of coherence for instance for three-manifold groups \cite{Sc73} or mapping tori of free groups \cite{FH99}.

The `sufficiently nice' properties we need for coherence results will be the uniform negative immersions and non-positive immersions properties. As such, we need a version of \cref{sequence} for immersions. The following is due to Louder--Wilton \cite{LW24} and makes fundamental use of a generalisation of Scott's Lemma \cite{Sc73_1}, due to Delzant (see \cite{Sw04}).

\begin{lemma}
\label{2-complex_delzant}
Let $X$ be a compact 2-complex and let $H\leqslant \pi_1(X)$ be a finitely generated subgroup. There is a sequence of $\pi_1$-surjective immersions of compact connected 2-complexes:
\[
Z_0\immerses Z_1\immerses \ldots \immerses Z_k\immerses\ldots\immerses X
\]
such that 
\[
\lim_{i\to \infty}\pi_1(Z_i) \isom H. 
\]
Moreover, if $H$ is non-cyclic and freely indecomposable, then after possibly replacing $H$ by a conjugate, each $Z_i$ can be taken to be irreducible.
\end{lemma}

In the next two sections, we shall use \cref{2-complex_delzant} to understand subgroups of $\pi_1(X)$ for $X$ a finite 2-complex with uniform negative immersions or non-positive immersions.

\subsubsection{Strong coherence and uniform negative immersions}
\label{sec:coherence_uni}

If we plug into \cref{2-complex_delzant} a 2-complex with uniform negative immersions and a non-cyclic freely indecomposable subgroup $H\leqslant \pi_1(X)$, then there is an $\epsilon>0$ such that
\[
\frac{\chi(Z_i)}{\#\{\text{2-cells in $Z_i$}\}}\leqslant -\epsilon.
\]
Rearranging the inequality and using the fact that $\chi(Z_i)\geqslant 1 - b_1(Z_i)\geqslant 1 - b_1(Z_0)$, we may obtain a uniform bound on the number of 2-cells:
\[
\#\{\text{2-cells in $Z_i$}\} \leqslant \frac{b_1(Z_0) - 1}{\epsilon}.
\]
Since each $Z_i$ is irreducible, every 0-cell and every 1-cell in $Z_i$ is traversed by the attaching map of some 2-cell, implying a uniform bound on the number of cells in $Z_i$. Hence, all but finitely many of the immersions $Z_i\immerses X$ are the same. Being careful with the direct limit, this implies that $H$ is finitely presented. To obtain coherence of $\pi_1(X)$, one needs to combine this fact with Grushko's theorem. With a more careful analysis, Louder--Wilton establish a much stronger theorem on subgroups of one-relator groups with uniform negative immersions, \cite[Theorems A \& B]{LW24}.

\begin{theorem}
\label{strong_coherence}
Let $G$ be a 2-free one-relator group. Or more generally, let $G = \pi_1(X)$ with $X$ a finite 2-complex with uniform negative immersions. Then $G$ is coherent and the following properties hold:
\begin{enumerate}
\item\label{itm:finitely_many} For any integer $n$, there are only finitely many conjugacy classes of finitely generated one-ended subgroups $H\leqslant G$ with $b_1(H)\leqslant n$.
\item\label{itm:cohopf} Every finitely generated one-ended subgroup $H\leqslant G$ is co-hopfian, in the sense that $H$ is not isomorphic to any proper subgroup of itself.
\item Every finitely generated non-cyclic subgroup $H\leqslant G$ is large, in the sense that it contains a finite index subgroup that surjects a rank two free group.
\end{enumerate}
\end{theorem}

By \cref{v_unim}, \cref{strong_coherence} applies to a finite index subgroup of a one-relator group with torsion. In particular, \cref{itm:finitely_many} from the above hold for one-relator groups with torsion. The proof that one-relator groups with torsion are coherent was obtained previously by Louder--Wilton in \cite{LW20} and, independently, by Wise in \cite{Wi22}. 

We remark that coherence, \cref{itm:finitely_many} and \cref{itm:cohopf} are properties shared by (one-ended) locally quasi-convex hyperbolic groups, see work of Sela \cite[Theorem 4.4]{Se97} and Kapovich--Weidmann \cite[Corollary 1.5]{KW04}. In fact Wilton has conjectured that all finite 2-complexes with uniform negative immersions have locally quasi-convex hyperbolic fundamental group in \cite[Conjecture 12.9]{Wi24}.

\subsubsection{Homological coherence and non-positive immersions}

The same strategy employed in \cref{sec:coherence_uni} does not work for non-positive immersions. Indeed, except for possibly the coherence, none of the conclusions from \cref{strong_coherence} hold for finite 2-complexes with non-positive immersions. In order to get around this, we will need to consider a weaker property than coherence.

A finitely generated group $G = F/N$ has \emph{type $\fp_2(\Z)$} if any one of the following equivalent conditions hold:
\begin{enumerate}
\item Its relation module $N_{\ab}$ is finitely generated as a $\Z G$-module.
\item The augmentation ideal $I_{\Z G}\leqslant \Z G$ is finitely presented.
\item The trivial $\Z G$-module $\Z$ admits a projective resolution $P_*\to \Z$ with $P_i$ finitely generated for $i = 0,1,2$.
\item There is a surjective homomorphism $\phi\colon H\to G$ from a finitely presented group $H$ such that $H_1(\ker(\phi), \Z) = 0$.
\end{enumerate}

\begin{definition}
A group $G$ is said to be \emph{homologically coherent} if all of its finitely generated subgroups have type $\fp_2(\Z)$. 
\end{definition}

These definitions naturally extend to arbitrary rings $R$ and to dimensions other than 2. Note that homological coherence of a group $G$ follows from coherence of the group ring $\Z G$ by the second definition of $\fp_2(\Z)$.

A priori, it seems that homological coherence should be a weaker property than coherence. Indeed, there are well known examples of groups that have type $\fp_2(\Z)$ but are not finitely presented by work of Bestvina--Brady \cite{BB97}. There are also examples which distinguish $\fp_2(\Z)$ from $\fp_2(K)$ where $K$ is any field by Bieri--Strebel \cite{BS80}. However, it turns out that within a very large class of groups the two notions coincide. See \cite[Theorem 4.5]{JZL23}, stated below. 

\begin{theorem}
\label{hom_coherence_upgrade}
Let $\mathcal{CG}$ be the smallest class of groups containing all coherent groups and that is closed under subgroups, finite extensions, amalgamated free products, HNN-extensions and directed unions. Then within $\mathcal{CG}$, the homologically coherent groups coincide with the coherent groups.
\end{theorem}

Just as in the case of group rings from \cref{sec:ring_coherence}, homological coherence of a group can be shown via a criterion which depends on an embedding of the group ring into a division ring. The following is \cite[Corollary 3.4]{JZL23}.

\begin{proposition}
\label{hom_coherence_criterion}
Let $R$ be a commutative ring and let $G$ be finitely generated group with $\cd_R(G)\leqslant 2$. If $R[G]\injects \Di$ is an embedding into a division ring, then
\[
G \text{ has type $\fp_2(R)$} \iff \dim_{\Di}H_2(G, \Di)<\infty.
\]
In particular, if $\Di$ has weak dimension at most one, then $G$ is homologically coherent.
\end{proposition}

Now let us plug into \cref{2-complex_delzant} a 2-complex $X$ with non-positive immersions and a finitely generated non-cyclic and freely indecomposable subgroup $H\leqslant \pi_1(X)$. We have a sequence of immersions
\[
Z_0\immerses Z_1\immerses \ldots \immerses Z_k\immerses\ldots\immerses X
\]
such that the following properties hold:
\begin{enumerate}
\item The homomorphisms $\pi_1(Z_i)\to \pi_1(Z_{i+1})\to H\leqslant \pi_1(X)$ are surjective for all $i$.
\item $\lim_{i\to \infty}\pi_1(Z_i) = H$.
\item $Z_i$ is compact for all $i$.
\item $\chi(Z_i)\leqslant 0$ and $Z_i$ is aspherical for all $i$ (see \cref{npi_properties}).
\end{enumerate}
In particular, letting $F = \pi_1\left(Z_0^{(1)}\right)$ and $N_i = \ker(F\to \pi_1(Z_i))$, we obtain a dual sequence
\[
N_0\injects N_1\injects \ldots \injects N_k\injects \ldots \injects F
\]
with the following dual properties:
\begin{enumerate}
\item The homomorphisms $N_i\to N_{i+1}\to \ker(F\to H)\leqslant F$ are injective for all $i$.
\item $\lim_{i\to \infty}N_i = \bigcup_{i\in \N}N_i = \ker(F\to H) = N$.
\item The relation module $(N_i)_{\ab}$ for $\pi_1(Z_i) = F/N_i$ is a finitely generated (left) $\Z[\pi_1(Z_i)]$-module for all $i$.
\item If $\Di$ is a division ring which is also a right $\Z[H]$-module, then 
\[
\dim_{\Di}(\Di\otimes_{\Z[F/N_i]}(N_i)_{\ab}) - \rk(F) + 1\leqslant 0
\]
for all $i$.
\end{enumerate}
This last inequality comes from the fact that each $Z_i$ is aspherical and so we have the following projective resolution of $\Z$ as a trivial $\Z[\pi_1(Z_i)]$-module:
\[
\begin{tikzcd}
0 \arrow[r] & (N_i)_{\ab} \arrow[r] & \bigoplus_{s\in S}\Z[\pi_1(Z_i)] \arrow[r] & \Z[\pi_1(Z_i)] \arrow[r] & \Z \arrow[r] & 0.
\end{tikzcd}
\]
where here $S\subset F$ is a free basis. See \cref{sec:homology}.

Since $\pi_1(X)$ is locally indicable by \cref{npi_properties}, we may use \cref{division_ring} to obtain a division ring embedding $\Q[H]\injects \Di$. When $X$ is a one-relator complex, we may instead use the Lewin--Lewin division ring from \cref{LL_embedding}. We have
\[
\dim_{\Di}(\Di\otimes_{\Z H}N_{\ab})\leqslant \sup_{i\in \N}\dim_{\Di}\left(\Di\otimes_{\Z[\pi_1(Z_i)]}(N_i)_{\ab}\right)\leqslant \rk(F) - 1.
\]
Since $H_2(H, \Di)$ is a $\Di$-submodule of the finite dimensional $\Di$-module $\Di\otimes_{\Q H}N_{\ab}$, it is itself finite dimensional. Hence, we obtain the following theorem by \cref{hom_coherence_criterion}. This is \cite[Theorem 1.2]{JZL23}.

\begin{theorem}
\label{npi_hom_coherence}
If $X$ is a 2-complex with non-positive immersions, then $\pi_1(X)$ is homologically coherent.
\end{theorem}

The coherence of torsion-free one-relator groups now follows by applying \cref{npi,hom_coherence_upgrade,npi_hom_coherence}.

It remains an open question of Wise as to whether all 2-complexes with non-positive immersions are coherent, see \cite[Conjecture 12.11]{Wi20}. It is also open as to whether the rational group ring of the fundamental group of a 2-complex with non-positive immersions is coherent, see \cite[Conjecture 6]{JZL23}.

\subsubsection{Effective coherence}

A coherent group $G$ is \emph{effectively coherent} if there is an algorithm which, given as input a finite subset $S\subset G$, computes a finite presentation for the subgroup $\langle S\rangle \leqslant G$. For now, the following is wide open.

\begin{problem}
\label{effective_coherence}
Prove that one-relator groups are effectively coherent.
\end{problem}

Cyclically pinched and cyclically conjugated one-relator groups are known to be effectively coherent by work of Kapovich--Myasnikov--Weidmann \cite{KMW05}. They show that graphs of effectively coherent groups with virtually polycyclic edge groups, and some extra conditions, are effectively coherent. Unfortunately, the same techniques do not work when applied to graphs of groups with free edge groups. New ideas will be needed for one-relator groups.

The proof that one-relator groups are coherent is very non-constructive and so unfortunately does not shed light on further structure of finitely generated subgroups of one-relator groups. Thus, in order to approach Problem \ref{effective_coherence}, it will be necessary to understand the subgroup structure of one-relator groups more explicitly. Even the case of one-relator groups $F/\normal{w}$ with $\pi(w)\geqslant 3$ is open and was explicitly proposed by Louder--Wilton in \cite{LW24}.

\subsection{Virtual algebraic fibring and generalisations}
\label{sec:vfibring}

Recall from \cref{sec:brown} that a group $G$ algebraically fibres if there is a finitely generated normal subgroup $N\triangleleft G$ such that $G/N\isom \Z$. We say a group $G$ \emph{virtually algebraically fibres} if it has a finite index subgroup $H\leqslant G$ which algebraically fibred. We saw with Brown's criterion that we may characterise precisely when a one-relator group algebraically fibres. Determining when it is virtually algebraically fibred is significantly more difficult.

Dawid Kielak obtained a powerful criterion in \cite{Ki20} for when a group with the RFRS property virtually algebraically fibres. A group $G$ is \emph{RFRS} if it admits a sequence of finite index subgroups $N_i\triangleleft G$ such that $N_{i+1}\leqslant \ker(N_i\to H_1(N_i, \Q))$ and $\bigcap N_i = 1$. The main class of examples of RFRS groups are right angled Artin groups and all their subgroups. In particular, fundamental groups of special cube complexes are RFRS. Kielak proved that an infinite RFRS group $G$ virtually algebraically fibres if and only if $b_1^{(2)}(G) = 0$.

 An immediate consequence of Kielak's result and Dicks--Linnell's computation of the $L^2$-Betti numbers of one-relator groups is that If $G$ is a virtually RFRS torsion-free two-generator one-relator group, then $G$ is virtually algebraically fibred. Since one-relator groups are coherent, the fibre subgroup will be a finitely presented subgroup. Using a result of Fel'dman \cite[Theorem 2.4]{Fe71}, we may conclude that the fibre subgroup drops in cohomological dimension and so is a free group by a result of Stallings \cite{St68}. These results together imply the following statement.

\begin{theorem}
\label{v_fibring}
If $G$ is a torsion-free two-generator one-relator group that is virtually RFRS, then $G$ is virtually \{finitely generated free\}-by-cyclic.
\end{theorem}

Note that by \cite[Theorem 7.2]{Lu02}, a virtually algebraically fibred group $G$ has $b_1^{(2)}(G) = 0$. Hence, if $G$ is a one-relator group that virtually algebraically fibres, then $G$ must be two-generator and be torsion-free by \cref{L2_betti}.

An interesting variation of the virtual algebraic fibration problem which has no restriction on $b_1^{(2)}$ is the following. Given a group $G$ with $\cd(G) = n$, when does $G$ admit a finite index subgroup which embeds as a subgroup of a group $H$ that algebraically fibres over a group $N$ with $\cd(N) = n-1$? For one-relator groups this amounts to asking when they admit finite index subgroups that embed in \{finitely generated free\}-by-cyclic groups. Under strong geometric hypotheses, Kielak and the first author proved a criterion in \cite{KL24} for when a one-relator group admits a positive answer to this question, solving a problem of Baumslag from \cite{Ba86}.

\begin{theorem}
\label{v_fbc}
Let $G$ be a one-relator group. If $G$ is hyperbolic and virtually special, then $G$ has a finite index subgroup $H$ that is a subgroup of a \{finitely generated free\}-by-cyclic group.
\end{theorem}

The class of \{finitely generated free\}-by-cyclic groups form an extensively studied class of groups with a deep and rich theory, connecting dynamics, geometry and algebra. Such groups lie within the slightly more general class of free-by-cyclic groups, where here the free part is not necessarily finitely generated. In this generality, less is known although most results for \{finitely generated free\}-by-cyclic groups can be upgraded to the more general case. Weakening the assumptions from \cref{v_fbc}, Fisher proved in \cite{Fi24} that any virtually RFRS one-relator group is virtually free-by-cyclic.

\begin{theorem}
\label{v_fbc2}
Let $G$ be a one-relator group. If $G$ is virtually RFRS, then $G$ is virtually free-by-cyclic.
\end{theorem}

Replacing the hypothesis that $G$ be a one-relator group with the hypothesis that $b_2^{(2)}(G) = 0$ and $\cd_{\Q}(G)\leqslant 2$, \cref{v_fbc} and \cref{v_fbc2} remain valid. There are also similar statements for groups of arbitrary finite cohomological dimension in \cite{KL24,Fi24}.

Chong--Wise show in \cite{CW24} that a finitely generated ascending HNN-extension of a free group is a subgroup of an ascending HNN-extension of a finitely generated free group. The slight mismatch between the conclusions of \cref{v_fbc} and \cref{v_fbc2} would be resolved if a similar statement was proven for free-by-cyclic groups (see \cite[Conjecture 1.2]{CW24}):

\begin{conjecture}[Chong--Wise]
If $G$ is a finitely generated free-by-cyclic group, then $G$ is a subgroup of a \{finitely generated free\}-by-cyclic group.
\end{conjecture}

We do not know to what extent \cref{v_fbc2} characterises which one-relator groups are virtually free-by-cyclic. Button showed in \cite[Corollary 8]{Bu19} that there are residually finite one-relator groups which are not virtually free-by-cyclic (not even virtually an ascending HNN-extension of a free group). Wise made the following conjecture in \cite[Conjecture 17.8]{Wi20}.

\begin{conjecture}[Wise]
If $G$ is a hyperbolic one-relator group, then $G$ is virtually free-by-cyclic.
\end{conjecture}

\subsection{Mel'nikov's conjecture, finite index subgroups and infinite index subgroups}

A well-known conjecture in the theory of one-relator groups is due to Mel'nikov, first stated in the 1980 edition of the Kourovka notebook \cite[Problem 7.36]{Ko18}: if $G$ is a residually finite group in which all finite index subgroups are one-relator, then $G$ is either free or the fundamental group of a closed surface. The entry in the Kourovka notebook notes that a counterexample was quickly provided by Churkin: the Baumslag--Solitar groups $\bs(1, n)$ each have all their finite index subgroups isomorphic to $\bs(1, m)$ for various values of $m$. However, the following modification is still open.

\begin{conjecture}[Mel'nikov]
\label{melnikov}
If $G$ is a residually finite group in which all finite index subgroups are one-relator, then $G$ is either free, isomorphic to $\bs(1, n)$ for some $n$ or the fundamental group of a closed surface.
\end{conjecture}

A group satisfying the assumptions of Conjecture \ref{melnikov} is known as a \emph{Mel'nikov group}. Some progress on Conjecture \ref{melnikov} has been made. For instance, Curran studied finite index subgroups of one-relator groups with surface-like presentations in \cite{Cu89}, characterising when they are actually surface groups. Ciobanu--Fine--Rosenberger confirm Conjecture \ref{melnikov} for all cyclically and conjugacy pinched one-relator groups in \cite{CFR13}. Gardam--Kielak--Logan showed in \cite{GKL23} that 2-generator Mel'nikov groups satisfy the conclusion of Conjecture \ref{melnikov}. Jaikin-Zapirain--Morales showed in \cite{JZM23} that Mel'nikov groups generated by at least three elements are 2-free, so in particular they have (uniform) negative immersions by \cref{k-free}. Finally, we mention that the pro-$p$ version of Mel'nikov's original conjecture has been known to hold for a long time by work of Dummit--Labute \cite{DL83}, where the pro-$p$ analogue of surface groups are known as Demu\v{s}hkin groups.

Several variations of Mel'nikov's conjecture have also been posed. For instance, Fine--Rosenberger--Wienke asked \cite[Problem 20.22]{Ko18}  whether a one-relator group in which every infinite index subgroup is free is either free or the fundamental group of a closed surface. This was also confirmed for two-generator one-relator groups in \cite{GKL23}. Wilton later confirmed this variation for all one-relator groups, see \cite[Theorem D]{Wi24_2}.

\begin{theorem}
\label{infinite_index}
If $G$ is an infinite one-relator group that is not free or a surface group, then $G$ contains an infinite index non-free subgroup.
\end{theorem}

When $G$ is a one-relator group generated by at least three elements, Wilton actually shows that there is a finitely generated non-free subgroup of infinite index. When $G$ is also hyperbolic and cubulated (for instance when $G$ is 2-free, see \cref{sec:qch}), this subgroup is also quasi-convex. Whether this stronger property holds also for two-generator one-relator groups is an open question of Wilton's \cite[Question 5.4]{Wi24_2}.

\begin{qstn}[Wilton]
If $G$ is a two-generator one-relator group which is not free and is not isomorphic to $\bs(1, n)$ for any $n$, does $G$ contain a finitely generated non-free subgroup of infinite index?
\end{qstn}

A related and well-known conjecture is Gromov's surface subgroup conjecture which predicts that every one-ended hyperbolic group contains a surface subgroup. Wilton made a significant contribution to this conjecture in \cite{Wi18} when he confirmed the surface subgroup conjecture for graphs of free groups with cyclic edge groups. Many special cases were established previously by several authors, the reader is directed to consult Wilton's article for details. The following is a special case of Wilton's theorem.

\begin{theorem}
\label{surface_subgroup}
If $G$ is a non-free cyclically pinched or cyclically conjugated one-relator group, then either $G$ contains a $\bs(1, n)$ subgroup for some $n$, or $G$ is hyperbolic and contains a quasi-convex surface subgroup.
\end{theorem}

Whether an analogous statement holds for all one-relator groups is wide open, although Calegari--Walker showed that random one-relator groups do contain surface subgroups \cite[Theorem 5.2.4]{CW15}.

\subsection{SQ-universality, virtual Betti numbers and largeness}

A group $G$ is \emph{SQ-universal} if every countable group can be embedded in a quotient of $G$. Peter Neumann made the following conjecture in \cite{Ne73}.

\begin{conjecture}[P. Neumann]
A one-relator group is either solvable or SQ-universal.
\end{conjecture}

Note that a solvable one-relator group is either cyclic or isomorphic to $\bs(1, n)$ for some $n\neq 0$, see work of \v{C}ebotar$'$ \cite[Theorem 3]{Ch71}.

Sacerdote--Schupp \cite{SS74} established a criterion for an HNN-extension to be SQ-universal (which is almost identical to the Minasyan--Osin criterion for acylindrical hyperbolicity from \cite{MO15}) and used it to show that any one-relator group generated by at least three elements is SQ-universal. Currently, the most general result regarding SQ-universality of one-relator groups is due to Button--Kropholler \cite{BK16}. One of their statements is conditional on a conjecture which was resolved by Mutanguha in \cite{Mu21}, so below we state the unconditional statement.

\begin{theorem}
\label{SQ-universal}
If $G$ is a one-relator group, then either $G$ is SQ-universal or $G$ is residually finite and $b_1(H) = 1$ for all finite index subgroups $H\leqslant G$.
\end{theorem}

The \emph{virtual $n^{\text{th}}$ Betti number} of a group $G$ is
\[
\vb_n(G) = \sup{\{b_n(H) \mid H \text{ is a finite index subgroup of $G$}\}}.
\]
By \cref{SQ-universal}, Neumann's conjecture reduces to the case of residually finite one-relator groups $G$ with $\vb_1(G) = 1$.

\begin{conjecture}
A residually finite one-relator group with $\vb_1(G) = 1$ is either infinite cyclic or $\bs(1, n)$ for some $n\neq 0$.
\end{conjecture}

In \cite{BK16} the authors show that any such group that is not SQ-universal must split as a strict ascending HNN-extensions of a finitely generated free group and must also contain some $\bs(1, n)$ for $n\geqslant 2$.

\begin{problem}
Characterise one-relator groups $G$ with $\vb_i(G)<\infty$ for $i = 1, 2$.
\end{problem}

Fruchter--Morales characterise when a cyclically pinched and cyclically conjugated one-relator group $G$ has $\vb_2(G)<\infty$ in \cite{FM24}: either $G$ is free or the fundamental group of a closed surface.

A property related to SQ-universality is that of largeness. A group $G$ is \emph{large} if it admits a finite index subgroup which surjects onto the free group of rank two. Edjvet--Pride pose the problem of determining which torsion-free two-generator one-relator groups are large \cite[Problem 9]{EP84}. Baumslag--Pride had shown previously in \cite{BP78,BP79} that all non-elementary one-relator groups with torsion and all one-relator groups generated by at least three elements are large. In fact, something much stronger is true for a one-relator group $F/\normal{w}$ when $\pi(w)\neq 2$: all finitely generated non-elementary subgroups are large by the results in \cref{sec:coherence_uni}. 

A large group is SQ-universal, but the converse is not necessarily true. For example, $\bs(2, 3)$ is an SQ-universal group that is not large (see \cite{EP84}). It is not clear whether there is a natural conjectural picture for which one-relator groups are large as the property appears far more subtle than SQ-universality. We direct the reader to the work of Jack Button for many partial results, see for instance \cite{Bu10,Bu11}.

\section{Geometric properties}
\label{sec:geometry}

Geometric group theory is the study of groups via their actions by isometries on metric spaces. $\cat(0)$ cube complexes and hyperbolic spaces are two rich classes of spaces that have taken centre stage in geometric group theory since Gromov's famous essay \cite{Gr87}. If a one-relator group $G$ acts in some non-trivial way on one of these spaces, then a wealth of new tools become available to study $G$ and often this leads to significant constraints on its structure. In this section we shall go into depth on what is known about geometric actions of one-relator groups.

Metric spaces will always be assumed to be proper and geodesic and our actions will be geometric, unless stated otherwise. Recall that a \emph{geometric action} is one that is by isometries and is properly discontinuous and cocompact.

\subsection{Hyperbolic one-relator groups}

A group $G$ is \emph{hyperbolic} if it acts geometrically on a hyperbolic metric space. Introduced by Gromov in \cite{Gr87}, their structure is now well-understood enough that their automorphism groups have been completely computed and their isomorphism problem has been solved, two things which are currently far from true for one-relator groups. A simple obstruction for a group to be hyperbolic is the presence of Baumslag--Solitar subgroups $\bs(1, n)$. A famous conjecture of Gersten's predicts that this is the only obstruction for one-relator groups.

\begin{conjecture}[Gersten's conjecture]
\label{gersten_conjecture}
If $G = F/\normal{w}$ is a one-relator group containing no Baumslag--Solitar subgroups $\bs(1, n)$ for any $n\neq 0$, then $G$ is hyperbolic.
\end{conjecture}

In this section we cover recent progress on Conjecture \ref{gersten_conjecture}, mostly based on work from \cite{Lin22,Lin24}.

\begin{remark}
A more general version of Gersten's conjecture was open until recently, when Italiano--Martelli--Migliorini showed in \cite{IMM23} that there exist groups $G$ of type $F$, containing no Baumslag--Solitar subgroups, but that are not hyperbolic. Their examples have cohomological dimension four and arise from a fibration of a hyperbolic 5-manifold over $S^1$. 
\end{remark}

It is still open as to whether there exist finitely presented groups of cohomological dimension two without Baumslag--Solitar subgroups that are not hyperbolic. Groups of cohomological dimension two are quite special: Gersten proved that finitely presented subgroups of finitely presented hyperbolic groups of cohomological dimension two are themselves hyperbolic \cite{Ge96}. Combining \cite[Corollary 7.8]{Ge96} with the coherence of one-relator groups, we have the following.

\begin{theorem}
\label{gersten_subgroup}
If $G$ is a hyperbolic one-relator group, then every finitely generated subgroup of $G$ is hyperbolic.
\end{theorem}

An early geometric result for one-relator groups is known as Newman's spelling theorem \cite{Ne68}, see \cref{Subsec:BBNewman}. A corollary of Newman's result is that a one-relator group with torsion $G = F/\normal{w^n}$ is hyperbolic. Below is a slight strengthening of Newman's result, and a generalisation of the Freiheitssatz, which was obtained by Schupp in \cite{Sc76}. Here when we say a word $u$ mentions a letter $s$, we mean that $s$ or $s^{-1}$ appears in $u$.

\begin{theorem}
\label{Newman}
Let $n\geqslant 1$ be a positive integer, let $F(S)$ be a free group, $w\in F(S)$ a cyclically reduced word that is not a proper power and consider the one-relator group $G = F(S)/\normal{w^n}$. Let $A\subset S$ be a subset of the generators, $v\in F(S)$ a freely reduced word and suppose that $w$ and $v$ mention a letter in $S - A$.

If $v$, as an element of $G$, is contained in the Magnus subgroup $\langle A\rangle\leqslant G$, then $v$ contains a subword $u^{n-1}u_0$ where $u$ is a cyclic conjugate of $w$ or $w^{-1}$ and $u_0$ is a prefix of $u$ which mentions each letter from $S - A$ that appears in $w$.
\end{theorem}

Setting $A = \emptyset$, Newman's theorem shows that the presentation of a one-relator group with torsion is actually a Dehn presentation (after symmetrising) and so the uniform word problem and conjugacy problem can both be solved very quickly in practice. More generally, Newman's result shows that we can also solve membership in $\langle A\rangle$ using Dehn's algorithm. In particular, this implies that Magnus subgroups are actually convex in the sense that the shortest words over the generating set $S$ that represent elements in a Magnus subgroup $\langle A\rangle$, must actually be words over $A$. Schupp's strengthening yields the same consequences when applied to torsion-free one-relator groups whose relator contains appearances of letters that are sufficiently far apart. For instance, it applies to surface relators, relators of the form $a^nb^nc^nd^n$ and many more.

Although in general hyperbolic one-relator group might not have a presentation with the properties from \cref{Newman}, it is known that Magnus subgroups are always at least \emph{quasi-convex}. That is, geodesics in a hyperbolic one-relator group between two elements in a Magnus subgroup always remain uniformly close to elements in the Magnus subgroup. This is proved in \cite{Lin24}, building on work from \cite{Lin22}. Note that for hyperbolic groups, quasi-convexity is equivalent to being undistorted.

\begin{theorem}
\label{quasi-convex}
Magnus subgroups of hyperbolic one-relator groups are quasi-convex.
\end{theorem}

A hyperbolic one-relator group may have many distorted (free) subgroups, see \cite{Ka99} for an explicit example.

The Bestvina--Feighn combination theorem \cite{BF92} is an important tool that can be used to show when a HNN-extension of a hyperbolic group is again hyperbolic. The combination theorem states that if $G = H*_{\psi}$ is a HNN-extension with $H$ hyperbolic, both associated subgroups quasi-convex in $H$ and such that the HNN-extension has a property called \emph{flaring annuli}, then $G$ is hyperbolic. Conversely, Gersten showed in \cite[Theorem 6.4 \& Corollary 6.7]{Ge98} that if $H$ is hyperbolic, both associated subgroups are quasi-convex in $H$ and $G$ is hyperbolic, then the HNN-extension has flaring annuli.

By the above discussion, if $G\isom H*_{\psi}$ is a one-relator splitting, either $H$ is not hyperbolic in which case $G$ is not hyperbolic by \cref{gersten_subgroup}, or $H$ is hyperbolic, both associated subgroups are quasi-convex by \cref{quasi-convex} and so $G$ is hyperbolic if and only if annuli flare in $H*_{\psi}$. As such, if one can understand annuli in one-relator splittings, then in principal one understands when a one-relator group is hyperbolic or not. Unfortunately, flaring annuli turns out to be a tricky condition to check. See the article of Kapovich \cite{Ka99} for an explicit use of the combination theorem to show that a specific one-relator group is hyperbolic.

On the other hand, if $G$ acts acylindrically on its Bass--Serre tree, then $G$ automatically has no flaring annuli. Hence, combining this with \cref{acylindrical} and an inductive argument on one-relator hierarchies, in \cite{Lin22} it was shown that 2-free one-relator groups are hyperbolic. In \cite{Lin24}, a slightly more general statement was established. 

\begin{theorem}
\label{conditional}
Let $G = F/\normal{w}$ be a one-relator group. If two-generator one-relator subgroups of $G$ are hyperbolic, then $G$ is hyperbolic.
\end{theorem}

A more refined statement than \cref{conditional} reducing Gersten's conjecture to two specific families of two-generator one-relator groups was proven in \cite{Lin24}. We describe these families in the next section. Combining Newman's theorem and \cref{conditional} with \cref{k-free}, we have a clean condition for hyperbolicity. 

\begin{corollary}
If $G = F/\normal{w}$ with $\pi(w)\neq 2$, then $G$ is hyperbolic.
\end{corollary}

If $G = F/\normal{w}$ with $\pi(w) = 2$, recall that by \cref{w-subgroup} there is a two-generator subgroup $H\leqslant F$ containing $w$ such that every non-free two-generator subgroup of $G$ is conjugate into $P = H/\normal{w}\leqslant G$, the $w$-subgroup. Thus, in order to apply \cref{conditional} to such a one-relator group, one only needs to check that $P$ is hyperbolic. Louder--Wilton conjectured that in fact $G$ should be hyperbolic relative to $P$ in \cite[Conjecture 1.9]{LW22}, this remains open.

\begin{conjecture}[Louder--Wilton]
If $G = F/\normal{w}$ with $\pi(w) = 2$, then $G$ is hyperbolic relative to its $w$-subgroup $P\leqslant G$.
\end{conjecture}

A well-known problem asks whether hyperbolic groups are residually finite. This is not known for hyperbolic one-relator groups.

\begin{qstn}
\label{qstn:hyperbolic_rf}
Are hyperbolic one-relator groups residually finite?
\end{qstn}

We close this section with some remarks on small cancellation conditions for one-relator groups. Pride showed that one-relator groups $F/\normal{w^n}$ with torsion satisfy the small cancellation condition $C(2n)$ \cite{Pr83}. However, it is possible that they do not satisfy $C(2n+1)$ and may not even satisfy $T(4)$ when $n = 2$. More refined small cancellation conditions were developed by Ivanov--Schupp and applied to characterise hyperbolicity of certain families of one-relator groups, not necessarily with torsion, in \cite{IS98}. Blufstein--Minian introduce a small cancellation condition $T'$ in \cite{BM22} and use it to prove even more one-relator groups are hyperbolic. Cashen--Hoffmann \cite{CH23} implemented several hyperbolicity tests for one-relator groups, including the small cancellation conditions mentioned above, and produced a database of one-relator groups of relator length up to 17, verifying hyperbolicity when possible. They found many one-relator groups whose hyperbolicity could not be confirmed with any known small cancellation conditions. Finally, Ol'shanski\u{i} proved in \cite{Ol92} that one-relator groups are (exponentially) generically $C'(1/6)$ and hence hyperbolic.

\subsubsection{Primitive extension groups}

In order to describe our two families of one-relator groups, we shall need to introduce Christoffel words, so called as they were first studied by Christoffel in \cite{Ch73}.

Let $F = F(a, b)$ be the free group of rank two and let $p/q\in \Q_{\geqslant 0}\cup \{+\infty\}$. Let $\Gamma\subset \R^2$ be the Cayley graph for $\Z^2$ with generating set $\{\overline{a}, \overline{b}\}$, the images of $a, b$ under the abelianisation map $F\to \Z^2$. The generator $\overline{a}$ is the vertex $(1, 0)$ while the generator $\overline{b}$ is the vertex $(0, 1)$. Now let $L\subset \R^2$ be the line passing through the origin with slope $p/q$. Now let $I\to \Gamma$ be the path connecting $(0, 0)$ with the first integral point on $L$, not going above $L$ and such that the region enclosed by $L$ and the image of this path contains no integral points in its interior. Then the \emph{$p/q$-Christoffel word}, which we denote by
\[
\pr_{p/q}(a, b)
\]
is the word over $\{a, b\}$ traced out by this path. See \cref{fig:Christoffel} for an illustration.

\begin{figure}
\centering
\includegraphics[scale = 0.5]{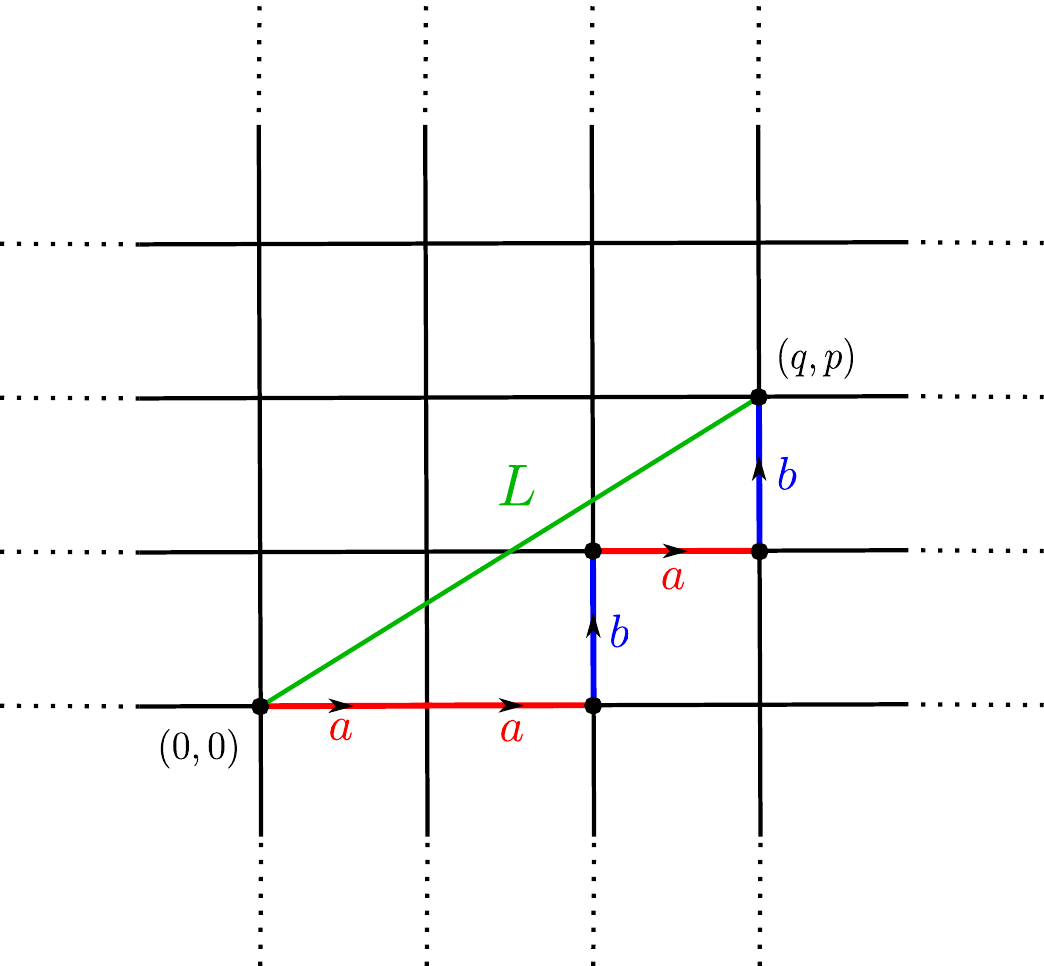}
\caption{The $2/3$-Christoffel word is $a^2bab$.}
\label{fig:Christoffel}
\end{figure}

The notation for the Christoffel words comes from the fact that they are all primitive elements of $F(a, b)$. Nielsen showed that the abelianisation map $F\to \Z^2$ induces a bijection between conjugacy classes of primitive elements in $F$ and primitive elements in $\Z^2$ \cite{Ni17}. Osborne--Zieschang showed in \cite{OZ81} that the Christoffel words and their inverses provide a complete set of conjugacy class representatives for the primitive elements of $F$. This means that if $w\in F$ is a primitive element, then $w$ is conjugate to some $\pr_{p/q}(a, b)^{\pm1}$. Moreover, $p/q$ is precisely the exponent sum of $b$ in $w$ divided by the exponent sum of $a$ in $w$.

The two types of \emph{primitive extension groups} are the two-generator one-relator groups with presentations:
\begin{align}
\label{first_type} E_{p/q}(x, y) &= \langle a, b \mid \pr_{p/q}(x, y)\rangle\\
\label{second_type} F_{p/q}(x, y, z) &= \langle a, b \mid \pr_{p/q}(xy, z)\rangle
\end{align}
where
\begin{align*}
A_{i, j} &= \langle b^{-i}ab^i, b^{-i-1}ab^{i+1}, \ldots, b^{-j}ab^j\rangle \quad \text{ where $i<j\in \N$}\\
x &\in A_{0, k-1} - A_{1, k-1}\\
y &\in A_{1, k} - A_{1, k-1}\\
z &\in A_{1, k-1} - 1
\end{align*}
and $\{\langle x\rangle, \langle y\rangle\}$ is a malnormal family in case (\ref{first_type}) and $\langle z\rangle$ is malnormal in case (\ref{second_type}). There are actually a couple of extra conditions on the words $x, y, z$ in \cite{Lin24}, but these are just to ensure that there is no other way to express the relator as a primitive word over words which violate the malnormality conditions. The upshot of these extra conditions is that when we go down one step in the hierarchy for $E_{p/q}$ or $F_{p/q}$ we reach a 2-free one-relator group, see \cite[Theorem 3.7]{Lin24}.

Note that since $\pr_{p/q}(-, -)$ is a primitive word, this implies that in $E_{p/q}(x, y)$ we have $\langle x, y\rangle\isom \Z$ and in $F_{p/q}(xy, z)$ we have $\langle xy, z\rangle \isom \Z$. 

The motivation behind introducing these groups is the main theorem from \cite{Lin24} which reduces the problem of determining which one-relator groups are hyperbolic to determining which primitive extension groups are hyperbolic.

\begin{theorem}
\label{main_hyperbolic}
A one-relator group $G = F/\normal{w}$ is hyperbolic if and only if all its primitive extension subgroups are hyperbolic. In particular, Gersten's conjecture is true if it is true for primitive extension groups.
\end{theorem}

Let us mention some examples of primitive extension groups.

\begin{itemize}
\item If $p, q>0$ are coprime integers, then we have
\[
\bs(p, \pm q) \isom E_{p/q}(a, b^{-1}a^{\pm1}b).
\]
We can see this as follows. We have
\[
E_{p/q}(a, b^{-1}a^{\pm1}b) \isom \langle a_0, a_1, b \mid b^{-1}a_0b = a_1, \pr_{p/q}(a_0, a_1^{\pm1})\rangle.
\]
We have $F(a_0, a_1)/\normal{\pr_{p/q}(a_0, a_1^{\pm1})} \isom \Z$ and so $E_{p/q}(a, b^{-1}a^{\pm1}b)$ splits as a HNN-extension over $\Z$ where the associated isomorphism is given by $p\Z\to \pm q\Z$.
\item Any one-relator group $G \isom F/\normal{w}$ with non-empty BNS-invariant is a primitive extension group. This is explained more in detail in \cite[Example 5.5]{Lin24}. More precisely, if $G$ has non-empty BNS-invariant, then $G\isom E_{p/1}(a, y)$ for some $p, y$.

Note that such groups have been shown by Mutanguha \cite{Mu21} to be hyperbolic if and only if they contain no Baumslag--Solitar subgroups.
\end{itemize}

The presentations (\ref{first_type}) and (\ref{second_type}) may seem a little mysterious, but there are very explicit graphs of groups decompositions which express their structure better.

The primitive extension groups $E_{p/q}(x, y)$ have the following graph of groups decomposition:
\[
\begin{tikzcd}[sep=2cm]
{A_{0, k-1}} \arrow[rr, "{A_{1, k-1}*\langle x\rangle = A_{1, k-1}*\langle t^p\rangle}"', no head] &  & {A_{1, k-1}*\langle t\rangle} \arrow[rr, "{A_{1, k-1}*\langle t^q\rangle = A_{1, k-1}*\langle y\rangle}"', no head] &  & {A_{1, k}} \arrow[llll, "{A_{0,k-1} = A_{1, k}}"', no head, bend right]
\end{tikzcd}
\]
where the top isomorphism is given by shifting $b^{-i}ab^i\to b^{-i-1}ab^{i+1}$ for each $0\leqslant i<k$ and the lower isomorphisms are given by the identity on $A_{1, k-1}$ and by $x\to t^p$ and $y\to t^q$.

The primitive extension groups $F_{p/q}(xy, z)$ have the following graph of groups decomposition:
\[
\begin{tikzcd}[sep=2cm]
{A_{0, k-1}} \arrow[rr, "{A_{1, k-1}*\langle x\rangle = A_{1, k-1}*\langle x\rangle}"', no head] &  & H \arrow[rr, "{A_{1, k-1}*\langle y\rangle = A_{1, k-1}*\langle y\rangle}"', no head] &  & {A_{1, k}} \arrow[llll, "{A_{0,k-1} = A_{1, k}}"', no head, bend right]
\end{tikzcd}
\]
where $H$ splits as:
\[
\begin{tikzcd}[sep=2cm]
{A_{1, k-1}} \arrow[r, "\langle z\rangle = \langle t^p\rangle", no head] & \langle t\rangle \arrow[r, "\langle t^q\rangle = \langle xy\rangle", no head] & {\langle x, y\rangle}
\end{tikzcd}
\]
Note that we have $x^{-1}zx = yzy^{-1}$ in $H$. Moreover, since $xy$ is primitive in the free group $\langle x, y\rangle$, we see that we have
\[
H \isom \left(A_{1, k-1}\underset{\langle z\rangle = \langle t^p\rangle}{\ast}\langle t\rangle\right)*\Z
\]
In light of \cref{main_hyperbolic} and its proof, understanding the geometry of primitive extension groups appears to be an important goal in order to understand the geometry of all one-relator groups.

\begin{problem}
Characterise non flaring annuli in the sense of \cite{BF92} of primitive extension groups with respect to the above graph of groups structures.
\end{problem}

\subsection{$\cat(0)$ and cubulable one-relator groups}
\label{sec:cat(0)}

A group $G$ is a \emph{$\cat(0)$ group} if it acts geometrically on a $\cat(0)$ space. Roughly speaking, a metric space is $\cat(0)$ if its geodesic triangles are not fatter than Euclidean triangles with the same edge lengths. We refer the reader to \cite{BH99} for a more formal definition and an in depth treatment of the theory. A group is a \emph{$\cat(0)$ group} if it acts geometrically on a $\cat(0)$ space. 

A class of $\cat(0)$ spaces which are easy to construct and have particularly nice combinatorial properties are $\cat(0)$ cube complexes. A \emph{cube complex} is a combinatorial complex built out of Euclidean cubes which are glued to each other isometrically along faces. A group is \emph{cubulable} if it acts properly and cocompactly on a $\cat(0)$ cube complex. A cube complex is \emph{non-positively curved} if the link of every vertex, which is a simplicial complex, is flag. In other words, if the link of each vertex has no missing faces. The universal cover of a non-positively curved cube complex is $\cat(0)$ by Gromov's well-known link condition. The reader is directed to \cite{Wi21} for details.

The question of which one-relator groups are $\cat(0)$ appears to be more subtle than the hyperbolicity question. A $\cat(0)$ group cannot contain \emph{unbalanced Baumslag--Solitar} subgroups $\bs(m, n)$ where $m \neq \pm n$. But a $\cat(0)$ group can contain balanced $\bs(m, \pm m)$ subgroups. Gardam--Woodhouse considered the following family of one-relator groups in \cite{GW19}:
\[
R_{p, q} = \langle x, t \mid t^{-1}x^{-2q}tx^{2p-1}t^{-1}x^{-2q}tx^{2p+1}\rangle
\]
They showed that $R_{1, 1}$ contains as an index two subgroup the \{fg free\}-by-cyclic group that Gersten showed was not $\cat(0)$ in \cite{Ge94}. In particular, $R_{1, 1}$ does not contain any unbalanced Baumslag--Solitar subgroups, but is not $\cat(0)$. The groups $R_{p, q}$ have some further interesting properties we believe are worth noting.

\begin{theorem}
If $p, q>0$ are positive integers, the groups $R_{p, q}$ have the following properties:
\begin{enumerate}
\item $R_{p, q}$ does not contain any unbalanced Baumslag--Solitar subgroups.
\item $R_{p, q}$ is residually finite if and only if $q\mid 2p$.
\item If $p<q$, then $R_{p, q}$ is $\cat(0)$.
\item If $p = q$, then $R_{p, q}$ is virtually \{fg free\}-by-cyclic and is not $\cat(0)$.
\item If $p>q$, then $R_{p, q}$ has Dehn function $\simeq n^{2\log_2(2p/q)}$ and does not act freely on any $\cat(0)$ cube complex.
\end{enumerate}
\end{theorem}

The first and the third fact are due to Gardam who introduced these groups in \cite[Theorem G]{Ga17}. The second fact is due to Button \cite[Corollary 7]{Bu19}. The fourth fact follows from a slight generalisation of Gardam--Woodhouses argument for $R_{1, 1}$. The claimed finite index subgroup is the kernel of the map $R_{p, p} \to \Z/2p\Z$ given by $x\to 1$ and $t\to 0$. Then the statement for when $p>q$ is \cite[Theorem A]{GW19}. Gardam--Wodhouse deduce this last fact by showing that $R_{p, q}$ contains as an index two subgroup a snowflake group, see work of Brady--Bridson \cite{BB00}. Note that admitting a free action on a $\cat(0)$ space is much weaker than admitting a geometric action.

Other examples of $\cat(0)$ one-relator groups with interesting properties are:
\begin{itemize}
\item Hsu--Wise showed in \cite{HW10} that graphs of free groups with cyclic edge groups are $\cat(0)$ (cubulable in fact) if and only if they do not contain unbalanced Baumslag--Solitar subgroups. In particular, all cyclically pinched one-relator groups are $\cat(0)$ and a cyclically conjugated one-relator group $\langle A, t \mid t^{-1}ut = v\rangle$ is $\cat(0)$ if and only if there is no $w\in F(A)$ such that $u$ and $v$  are conjugate to $w^m$, $w^n$ respectively with $m\neq \pm n$.
\item Infinite one-relator groups with non-trivial centre are $\cat(0)$. Murasugi showed in \cite{Mu64} that they are all \{fg free\}-by-cyclic. Pietrowski showed \cite{Pi74} that they all split as graphs of infinite cyclic groups. Since \{fg free\}-by-cyclic groups cannot contain unbalanced Baumslag--Solitar subgroups, Hsu--Wise's result \cite{HW10} implies the claim.
\item Lauer--Wise showed in \cite{LW13} that any finitely generated one-relator group $F/\normal{w^n}$ acts properly on a $\cat(0)$ cube complex if $n\geqslant 2$ and acts cocompactly if $n\geqslant 4$.
\item Hydra groups 
\[
H_k = \langle a, t \mid [[\ldots[a, \underbrace{t], \ldots], t]}_{k}\rangle
\]
introduced by Dison--Riley in \cite{DR13} are $\cat(0)$, biautomatic and \{fg free\}-by-cyclic groups which contain free subgroups with distortion function $\simeq A_k$, where $A_k$ is the $k^{\text{th}}$ Ackermann function, defined recursively by $A_1(n) = 2n, A_{k+1}(n) = A_{k}^{(n)}(1)$.
\item Brady--Crisp \cite{BC07} proved that the one-relator group $\langle a, b \mid aba^2 = b^2 \rangle\isom F_3\rtimes\Z$ is $\cat(-1)$ (and hence is hyperbolic), but its $\cat(0)$ dimension is two, whereas its $\cat(-1)$ dimension is three. 
\end{itemize}

In the next section we shall see that any one-relator group $F/\normal{w}$ with $\pi(w)\geqslant 3$ is virtually cocompactly cubulated. The following conjecture should provide more cubulated one-relator groups.

\begin{conjecture}
Let $G = F/\normal{w}$ be a one-relator group with $\pi(w) = 2$ and such that the $w$-subgroup $P\leqslant G$ is virtually (cocompactly) cubulated. Then $G$ is virtually (cocompactly) cubulated.
\end{conjecture}

\subsection{Virtually compact special one-relator groups}
\label{sec:qch}

A class of non-positively curved cube complexes known as \emph{special cube complexes} have played a particularly important role in several significant developments in geometric group theory. Special cube complexes were introduced by Haglund--Wise in \cite{HW08}. If $X$ is a compact special cube complex, Haglund--Wise showed that $\pi_1(X)$ embeds into a right angled Artin group, $\pi_1(X)$ is residually finite and linear and, if hyperbolic, all its quasi-convex subgroups are separable.

In his monograph, Dani Wise \cite{Wi21} developed criteria for when a group is virtually special in terms of hierarchies.

\begin{definition}
Define the class $\mathcal{VQH}$ of groups with a \emph{virtual quasi-convex hierarchy} as the smallest class of groups containing the trivial group and such that the following hold:
\begin{enumerate}
\item If $A, B\in \mathcal{VQH}$, then $G = A*_CB\in \mathcal{VQH}$ if $C$ is quasi-isometrically embedded in $G$.
\item If $A\in \mathcal{VQH}$, then $G = A*_C\in \mathcal{VQH}$ if $C$ is quasi-isometrically embedded in $G$.
\item If $A\leqslant \mathcal{VQH}$ and $A\leqslant G$ with $[G:A]<\infty$, then $G\in \mathcal{VQH}$.
\end{enumerate}
\end{definition}

One of the main results in \cite{Wi21} determines when a hyperbolic group is virtually compact special.

\begin{theorem}[Wise]
A hyperbolic group $G$ lies in $\mathcal{VQH}$ if and only if it is virtually compact special.
\end{theorem}

Wise used his quasi-convex hierarchy machinery in \cite{Wi21}, combined with Newman's spelling theorem \cite{Ne68}, to prove that one-relator groups with torsion are virtually compact special. The first author then proved in \cite{Lin22} that 2-free one-relator groups are hyperbolic and that all their one-relator hierarchies are quasi-convex. The proof involved a careful analysis of the action of a one-relator group on its Bass--Serre tree, leading to \cref{acylindrical}, proving that certain Magnus subgroups are quasi-convex and using a criterion for quasi-convexity in graphs of hyperbolic groups due to Kapovich \cite{Ka01}.

\begin{theorem}
\label{v_special}
If $G = F/\normal{w}$ is a one-relator group with $\pi(w)\neq 2$, then every Magnus hierarchy for $G$ is a quasi-convex hierarchy. Hence, $G$ has a finite index subgroup that is the fundamental group of a compact special cube complex.
\end{theorem}

\cref{v_special} provides an answer to Question \ref{qstn:hyperbolic_rf} for a large class of one-relator groups.

\begin{corollary}
If $G = F/\normal{w}$ is a one-relator group with $\pi(w)\neq 2$, then $G$ is residually finite, linear and quasi-convex subgroups are separable.
\end{corollary}

We conclude with a conjecture. This should be compared with a conjecture of Wise \cite[Conjecture 14.2]{Wi04} and a conjecture of Wilton \cite[Conjecture 12.9]{Wi24}.

\begin{conjecture}
If $G = F/\normal{w}$ is a cubulated one-relator group, then every finitely generated subgroup of $G$ is undistorted if and only if $\pi(w) \neq 2$ or $\pi(w) = 2$ and the $w$-subgroup $P\leqslant G$ is virtually $F_n\times \Z$ for some $n\geqslant 1$.
\end{conjecture}

\section{Residual properties}
\label{sec:residual}

\subsection{Residually solvable and rationally solvable one-relator groups}

Recall that $G^{(i)} = [G^{(i-1)}, G^{(i-1)}]$ denotes the $i^{\text{th}}$ term of the derived series of $G$ and that the $i^{\text{th}}$ term of the rational derived series is the subgroup
\[
G^{(i)}_{\Q} := \left\{g \bigm\vert g^k\in \left[G_{\Q}^{(i-1)}, G^{(i-1)}_{\Q}\right], k\neq 0\right\}\trianglelefteq G.
\]
A group $G$ is residually solvable if $\bigcap_{i\in \N}G^{(i)} = 1$ and it is residually rationally solvable if $\bigcap_{i\in \N}G^{(i)}_{\Q} = 1$. The transcendental (rational) derived series is the extension of the (rational) derived series to arbitrary ordinals.

A one-relator group is positive if its relator is a positive word over the generators. A classical result due to Baumslag \cite{Ba71} states that a positive one-relator group is residually solvable. In \cite{Lin24_1}, residually rationally solvable one-relator groups were characterised in terms of their relator, generalising Baumslag's result. Additionally, if $G$ is a one-relator group and $H^2(G) = 0$, it was shown that the rational derived series and the usual derived series of $G$ coincide using work of Strebel \cite{Str74}, yielding a characterisation of residual solvability amongst such one-relator groups. We now explain this characterisation and its consequences. 

The usual example of a one-relator group which is not residually (rationally) solvable is the Baumslag--Gersten group
\[
\bg(1, n) = \langle a, t \mid [a^t, a]  = a^{n-1}\rangle.
\]
Since the relator expresses a power of $a$ as a commutator of $a$ and a conjugate of $a$, we see that $\normal{a}$ is rationally perfect in $\bg(1, n)$ for $n\geqslant 2$. In other words, $H_1(\normal{a}, \Q) = 0$. When $n = 2$, we also have $H_1(\normal{a}, \Z) = 0$. Hence, $G_{\Q}^{(1)} = G_{\Q}^{(2)} = \normal{a}$ for each $n\geqslant 2$ and $G^{(1)} = G^{(2)} = \normal{a}$ for $n = 2$. The main theorem in \cite{Lin24_1} shows that this phenomenon is not far off from the general case.

\begin{theorem}
\label{residually_solvable}
If $G = F/\normal{w}$ is a one-relator group, then there is some $r\in F$ and some $n\geqslant 1$ such that $w\in r^n[\normal{r}, \normal{r}]$ and
\[
G_{\Q}^{(\omega+1)} = G_{\Q}^{(\omega)} = \normal{r}_G.
\]
Furthermore, if $H^2(G) = 0$, then $n = 1$ and $G^{(\omega+1)} = G^{(\omega)} = \normal{r}_G$.
\end{theorem}

With the notation of \cref{residually_solvable}, $G$ is residually (rationally) solvable if and only if $r =_G 1$ and so $r$ is conjugate to $w$ or $w^{-1}$ by \cref{Magnus_property}.

\begin{corollary}
\label{residually_solvable_corollary}
If $G = F/\normal{w}$ is a one-relator group, then $G/G_{\Q}^{(\omega)}$ is a residually rationally solvable one-relator group.
\end{corollary}

In particular, \cref{residually_solvable} characterises precisely when the only rationally solvable quotients of a one-relator group $F/\normal{w}$ have a certain derived length:
\begin{itemize}
\item If all rationally solvable quotients of $F/\normal{w}$ have rational derived length 1, then $F/\normal{r}$ would have to be abelian. This means that either $\rk(F) = 1$ or $\rk(F) = 2$ and $r$ is primitive.
\item If all rationally solvable quotients of $F/\normal{w}$ have rational derived length at most 2 (but not all 1), then $F/\normal{r}$ would have to be non-abelian and metabelian. Since the Baumslag--Solitar groups $\bs(1, n)$ with $n\neq 0, 1$ are the only non-abelian metabelian one-relator groups, this means that $F/\normal{r}\isom \bs(1, n)$ for some $n\neq 0, 1$. Moreover, by a result of Brunner \cite[Theorem 2.4]{Br74}, $r$ must lie in the same $\Aut(F)$-orbit of $(b^{-1}aba^{-n})^{\pm1}$.
\item If $F/\normal{w}$ admits rationally solvable quotients of rational derived length 3 but not 2, then $F/\normal{w}$ admits rationally solvable quotients of arbitrary rational derived length.
\end{itemize}

Another corollary obtained in \cite{Lin24_1} is that it is algorithmically decidable whether a one-relator group is residually rationally solvable; one can compute the element $r$ from \cref{residually_solvable} and decide whether it is trivial in $G$.

It is not known whether the derived series of a one-relator group in general has length at most $\omega$, but in \cite{Lin24_1} it is conjectured that this is the case.

\begin{conjecture}
Let $G = F/\normal{w}$ be a one-relator group. There is a word $r\in F$ such that $w\in r[\normal{r}, \normal{r}]$ and
    \[
    G^{(\omega+1)} = G^{(\omega)}  = \normal{r}_G.
    \]
In particular, the maximal residually solvable quotient of $G$ is the one-relator group $F/\normal{r}$.
\end{conjecture}

A related class of groups is that of amenable groups. Like solvable groups, amenable groups also cannot contain non-abelian free subgroups. Unlike solvable groups, amenable groups do not form a variety and so finding a non-residually amenable group appears to be quite difficult. Arzhantseva has asked the following general question, see \cite[Problem 18.6]{Ko18}.

\begin{qstn}[Arzhantseva]
\label{qstn:Arzhantseva}
Is every one-relator group residually amenable?
\end{qstn}

It is not known whether the non residually solvable Baumslag--Gersten group $\bg(1, 2)$ is residually amenable. If the answer to Question \ref{qstn:Arzhantseva} is negative, then it is likely to be negative for $\bg(1, 2)$.

\subsection{Residually nilpotent and parafree one-relator groups}

Recall that $G_i = [G_{i-1}, G]$ denotes the $i^{\text{th}}$ term of the lower central series of $G$ and that $G$ is residually nilpotent if $\bigcap_{i\in \N}G_i = 1$. In contrast with the (rational) derived series, it is known that the lower central series of a one-relator group can have length $\omega^2$ by work of Mikhailov \cite{Mi16}. Thus, a solution to the following problem is likely to be more subtle than the case of residually solvable one-relator groups.

\begin{problem}
\label{nilp_q}
Characterise when a one-relator group is residually nilpotent.
\end{problem}

Azarov proved in \cite[Theorem 1]{Az98} that a cyclically pinched one-relator group 
\[
G = F_1*_{\langle u\rangle = \langle v\rangle} F_2
\]
is residually nilpotent if and only if either $u$ or $v$ is not a proper power in the free group $F_1$ or $F_2$ respectively. Outside of this result, only sporadic families are known to be residually nilpotent.

A group $G$ is \emph{parafree} if it is residually nilpotent and if there is a free group $F$ such that
\[
G/G_i \isom F/F_i \quad\quad \text{for all $i\in \N$}
\]
Baumslag provided the first non-free examples of such groups in \cite{Ba64}. The example he provided was the one-relator group $\langle a, b, c \mid a^5b^3c^2\rangle$. Baumslag then initiated the systematic study of parafree groups in a series of two papers \cite{Ba67b,Ba69b}. In \cite{Ba69b}, Baumslag proved that a parafree group must be 2-free. The reader is directed to Baumslag's survey \cite{Ba05} for further background on parafree groups and many one-relator examples. Using \cref{v_special}, the first author showed in \cite{Lin22} that parafree one-relator groups are hyperbolic and virtually special.

A major open problem in group theory is Remeslennikov's conjecture: if $G$ is a residually finite group with the same finite quotients as a free group, is $G$ necessarily free? A positive answer is known within several classes of groups, but no systematic approach has been taken for the class of one-relator groups. Recently, Jaikin-Zapirain \cite{JZ23} has shown that a residually finite group with the same finite quotients as a free group must be parafree. In light of this, understanding when a one-relator group is parafree becomes very interesting. 

Magnus showed in \cite{Ma39} that an $n$-generator group with the same lower central series quotients as a free group of rank $n$ is actually a free group of rank $n$. Combining this with a result of Stallings \cite{St65}, we see that a one-relator group $G$ that is residually nilpotent is parafree if and only if $H^2(G) = 0$. Hence, solving \cref{nilp_q} would also yield a characterisation of parafree one-relator groups.

In particular, combining Azarov's result with these remarks, a cyclically pinched one-relator group $G = F_1*_{u = v}F_2$ is parafree if and only if $u$ or $v$ is not a proper power and $H_1(G, \Z) = \Z^{b_1(F_1) + b_1(F_2) - 1}$. This was greatly generalised by Jaikin-Zapirain--Morales in \cite{JZM24} who characterised when a graph of parafree groups with cyclic edge groups is also parafree.

\subsection{Residual finiteness and profinite completions}

If $G$ is a group, the \emph{profinite completion} of $G$, denoted by $\hat{G}$, is the inverse limit of the inverse system of finite quotients $G/N$ of $G$. The group $G$ is residually finite precisely when the canonical map $G\to \hat{G}$ is injective. A classical result of Dixon--Formanek--Poland--Ribes \cite{DFPR82} states that the profinite completions of two finitely generated groups $G$ and $H$ are isomorphic (as topological groups) precisely when the set of finite quotients of $G$ and $H$ coincide. 

A group $G$ is \emph{profinitely rigid} (within a class of groups) if for any residually finite group $H$ (within that class of groups), an isomorphism $\hat{H}\isom \hat{G}$ implies that $H\isom G$. The profinite genus of a group $G$ is the set of residually finite groups $H$ such that $\hat{H}\isom \hat{G}$. With this language, Remeslenikov's conjecture can be stated as predicting that free groups are profinitely rigid.

The study of profinite invariants and profinite rigidity of groups is an active area of research. However, within the class of one-relator groups, very little is known. Moldavanski\u{\i}--Sibyakova \cite{MS95} showed that the Baumslag--Solitar groups $\bs(1, n)$ are profinitely rigid amongst one-relator groups and Jaikin-Zapirain--Morales showed in \cite{JZM23} that closed surface groups are profinitely rigid amongst one-relator groups. Experimental work was also carried out in \cite{LL94,BCH04} where it was shown that certain families of parafree one-relator groups had different sets of finite quotients.

\begin{problem}
\label{profinite_problem}
Find more examples of residually finite one-relator groups that are profinitely rigid amongst one-relator groups.
\end{problem}

Baumslag posed the question of distinguishing cyclically pinched one-relator groups \cite[Problem 1]{Ba74} by their finite quotients after having shown that they are all residually finite in \cite{Ba63}. This appears to still be open and is likely to be a source of examples for Problem \ref{profinite_problem}.

\begin{qstn}
\label{residf_q}
Are there residually finite one-relator groups with the same finite quotients that are not isomorphic? Are there finitely many residually finite one-relator groups with the same finite quotients?
\end{qstn}

Baumslag showed in \cite{Ba74_2} that even within the class of virtually cyclic groups, there exist (two-relator) groups which are not profinitely rigid so one cannot be too optimistic regarding Question \ref{residf_q}.

By \cref{v_special}, a one-relator group $G = F/\normal{w}$ is residually finite if $\pi(w)\neq 2$. This would be a natural family of one-relator groups to attempt to answer Question \ref{residf_q} given how much extra structure is known to hold in these cases. In particular, since parafree one-relator groups $F/\normal{w}$ are 2-free and hence have $\pi(w)\geqslant 3$, it suffices to look at these one-relator groups for Remeslenikov's conjecture. Wilton showed in \cite{Wi18} that any cyclically pinched or cyclically conjugated one-relator group $F/\normal{w}$ with $\pi(w)\geqslant 3$ does not have the same finite quotients as a free group.

The first examples of pairs of finitely presented residually finite groups $H, G$ and a homomorphism $H\to G$ that is not an isomorphism, but induces an isomorphism on profinite completions $\hat{H}\to \hat{G}$ were discovered by Bridson--Grunewald in \cite{BG04}, after finitely generated ones had been discovered by Platonov--Tavgen' \cite{PT90}. Such examples should not exist amongst one-relator groups.

\begin{problem}
Prove that if $G$ is a residually finite one-relator group and $H\leqslant G$ is a finitely generated subgroup, then the induced map $\hat{H}\to \hat{G}$ is an isomorphism only if $H = G$.
\end{problem}

One could also wonder whether it is worth studying profinite completions of non-residually finite groups. However, Baumslag--Miller--Troeger showed in \cite{BMT07} that, given any one-relator group $G$, one can construct another one-relator group $H$ that is not residually finite but has profinite completion isomorphic to $G$. In fact, the isomorphism they construct is induced by a surjection $H\to G$. The first example of such a surjection was the abelianisation map $\bg(1, 2)\to \Z$, proven by Baumslag in \cite{Ba69}.

\section{Algorithmic properties}
\label{sec:algorithms}

\subsection{The word problem}

Although the word problem for all one-relator groups was solved almost a century ago, the complexity of the word problem has not seen much improvement from  Magnus' original solution. One way of measuring the complexity of the word problem of a group $G$ is via its Dehn function $\delta_G$. An upper bound on the Dehn function of one-relator groups was obtained by Bernasconi in \cite{Be94}.

\begin{theorem}
\label{global_bound}
If $G = F/\normal{w}$ is a one-relator group, then
\[
\delta_G(n) \precsim A_{2|w|}(n),
\]
where $A_k$ denotes the $k^{\text{th}}$ Ackermann function, defined recursively by $A_1(n) = 2n, A_{k+1}(n) = A_{k}^{(n)}(1)$.
\end{theorem}

Note that $A_2(n) = 2^n$ and $A_3(n) = \exp_2^{(n)}(1)$ already grows faster than any finite tower of exponentials. Gersten states in \cite{Ge93} that it is an open question as to whether for each $k$, there exists a one-relator group $G$ with $\delta_G(n)\succsim A_k(n)$. We conjecture that no such one-relator groups exist for all $k$ and they perhaps do not even exist for $k = 3$.

A candidate one-relator group for having the worst possible Dehn function is the Baumslag--Gersten group 
\[
\bg(1, 2) = \left\langle a, t \mid a^{a^t} = a^2\right\rangle.
\]
which splits as a HNN-extension of the Baumslag--Solitar group $\bs(1, 2)$ where one generator is conjugated to the other. Platonov determined the precise Dehn function of the Baumslag--Gersten group \cite{Pl04} after Gersten established lower bounds in \cite{Ge92} and Bernasconi established upper bounds in \cite{Be94}.

\begin{theorem}
If $G = \bg(1, 2)$ is the Baumslag--Gersten group, then
\[
\delta_{G}(n) \sim \exp^{(\log_2(n))}(1).
\]
\end{theorem}

One can attempt to find one-relator groups with worse Dehn functions by performing iterated HNN-extensions to $\bg(1, 2)$, conjugating one generator to the other, and hoping this yields a growth in Dehn function comparable to that of $\bg(1, 2)$ from $\bs(1, 2)$. Interestingly, Bernasconi also showed that $A_3(n)$ is an upper bound for the Dehn function of such groups. Other one-relator groups with interesting Dehn functions were already discussed in \cref{sec:cat(0)}.

Another way of measuring the complexity of the word problem is via the time function of an algorithm to solve it. The Dehn function of a group provides an upper bound for the time function of a non-deterministic Turing machine solving its word problem. This roughly means that we can verify that an input word is trivial in time bounded above by the Dehn function. In general it is much more difficult to prove lower bounds on the time complexity of the word problem than on the Dehn function. Indeed, the following question, which appears as \cite[Question (OR3)]{BMS02}, is open.

\begin{qstn}[A. Myasnikov]
\label{polynomial_question}
Is the word problem in every one-relator group decidable in polynomial time?
\end{qstn}

Myasnikov--Ushakov--Won \cite{MUW11} proved the following upper bounds for the complexity of the word problem of the Baumslag--Gersten group, lending credence to a positive answer to Question \ref{polynomial_question}. Their proof used new integer compression data structures known as power circuits.

\begin{theorem}
\label{polynomial_time}
There is an algorithm which solves the word problem for the Baumslag--Gersten group $\bg(1, 2)$ in polynomial time.
\end{theorem}

Diekert--Laun--Ushakov \cite{DLU12} refined the complexity estimates to obtain a cubic time solution to the word problem in $\bg(1, 2)$ and Laun extended this to apply to any $\bg(1, n)$ for $n\geqslant 2$ \cite[Theorem 3.1]{La14}. For $\bg(2, 3)$ the situation is more complicated and no time bounds are known (other than those offered by Magnus' algorithm), but would be interesting (see \cite[Problem 1.8]{MUW11}).

One obstacle to understanding the complexity of the word problem in one-relator groups is that there is no general method to obtain a simple finite set of rewriting rules over the generators to solve the word problem. In \cite{MUW11}, the authors find a particularly nice rewriting system solving the word problem for $\bg(1, 2)$ using the Magnus--Moldavanski\u{\i} hierarchy for $\bg(1, 2)$ which expresses it as a double HNN-extension of $\Z$. In general, the Magnus hierarchy is more intricate. It would be of considerable interest to develop rewriting systems with nice properties for all one-relator groups, see \cref{Prob:FCRS-for-OR}.

\subsection{The conjugacy, membership and normal root problems}

Outside of a few special cases, the conjugacy problem for one-relator groups has remained largely untouched. Excluding hyperbolic one-relator groups, it is known to be decidable in cyclically pinched and conjugated one-relator groups (which includes Baumslag--Solitar groups) by Lipschutz \cite{Lipschutz1966} and Larsen \cite{Lar77} respectively, the Baumslag--Gersten group by Beese \cite{Be12} and for one-relator groups $G$ with non-empty BNS-invariant $\Sigma^1(G)$ by work of Logan \cite{Lo23}.

One particular special case of the conjugacy problem which should be more approachable in general is to determine when an element is conjugate into a Magnus subgroup.

\begin{problem}
If $G$ is a one-relator and $A\leqslant G$ is a Magnus subgroup, find an algorithm which decides whether an element $g\in G$ is conjugate into $A$.
\end{problem}

Another problem whose solution should be useful for understanding the conjugacy problem better would be an effective version of Collins' \cref{Collins_intersections_2} on intersections of conjugates of Magnus subgroups. As remarked previously, \cref{Collins_intersections_1} already has an effective version proved by Howie.

\begin{problem}
If $G = F(S)/\normal{w}$ is a one-relator group and $A, B\leqslant G$ are Magnus subgroups generated by subsets of $S$, determine the double cosets $AgB\neq AB$ such that $A^g\cap B\neq 1$.
\end{problem}

The membership problem for finitely generated subgroups of one-relator groups is another problem which is wide open.

\begin{qstn}
Is the subgroup membership problem decidable for one-relator groups?
\end{qstn}

It was shown by Gray \cite[Theorem B]{Gra20} that there are one-relator groups with undecidable submonoid membership problem. As such, one cannot hope to go too far with regards to membership problems.

\begin{problem}[Lyndon]
\label{Lyndon_Q}
If $G = F(S)/\normal{w}$ is a one-relator group and $A, B\leqslant G$ are Magnus subgroups generated by subsets of $S$, find an algorithm which decides membership in a double coset $AgB$.
\end{problem}

Lyndon posed this problem for $g =1$ \cite[Problem 3.6]{Ly62}, which Newman solved for one-relator groups with torsion in \cite[Corollary 3]{Ne68}. The extended word problem for a group $G$ with presentation $\langle S \mid R\rangle$ is the problem of determining membership in subgroups generated by subsets of the generators $S$. Lyndon's motivation for this problem came from trying to solve the extended membership problem in groups with staggered presentations, following Magnus' ideas for the extended word problem for one-relator groups.

The \emph{normal root problem} was first studied by Magnus in \cite{Ma30}. If $F$ is a free group and $w\in F$, then an element $r\in F$ is a \emph{normal root} of $w$ if $w\in \normal{r}$.

\begin{problem}[Normal Root Problem]
Let $F$ be a free group and let $w\in F$. Determine the normal roots of $w$.
\end{problem}

In \cite{Ma32}, Magnus determined the normal roots of the commutator $[a, b]$. Some further work on the root problem was carried out by Steinberg in \cite{Ste86} who characterised normal roots of $a^kb^l$ when $k, l$ are primes. Remeslennikov conjectured that if $r$ is a (cyclically reduced) normal root of $w$ and $|r|>|w|$, then $w$ is conjugate to $[r, f]$ or $[r^{-1}, f]$ for some $f\in F$, see \cite[Problem F16]{BMS02}. However, this was then disproved by McCool in \cite{Mc01}. One can also ask a more specific version of the Normal Root Problem.

\begin{qstn}
If $F$ is a free group, does there exist an infinite sequence of words $w_0, w_1, \ldots\in F$ such that $\normal{w_i}\lneq \normal{w_{i+1}}$ for all $i\in \N$?
\end{qstn}

This question is of particular interest because an answer determines whether or not one-relator groups are closed in the space of marked groups.

\subsection{The isomorphism problem, Nielsen generating sets and $T$-systems}

Despite almost a century of progress since Magnus solved the word problem for one-relator groups, the conjugacy and isomorphism problem remain very far from being solved in general.

Currently the most general result in this direction is that the isomorphism problem for one-relator groups $F/\normal{w}$ with $\pi(w)\neq 2$ is decidable amongst one-relator groups. This is because they are hyperbolic and one can appeal to deep work of Sela \cite{Se95} (in the torsion-free case) and Dahmani--Groves \cite{DG11} (in the torsion case).

\begin{theorem}
If $G = F/\normal{w}$ is a one-relator group with $\pi(w)\neq 2$, then the problem of deciding whether another one-relator group is isomorphic to $G$ is decidable.
\end{theorem}

One way to understand the isomorphisms between one-relator groups is by understanding their generating sets.

\begin{definition}
Two sets $X_1,X_2\subset G = F/N$ are \emph{Nielsen equivalent} if there is an automorphism $\psi\in \Aut(F)$ such that $\psi(X_1) = X_2$. They are in the same \emph{$T$-system} if there is an automorphism $\phi\in \Aut(G)$ such that $\phi(X_1)$ is Nielsen equivalent to $X_2$.
\end{definition}

Let $F$ be a free group with free generating set $S\subset F$. Given two normal subgroups $N_1, N_2\triangleleft F$ and an isomorphism $\theta\colon F/N_1\to F/N_2$, if the two generating sets $S_2$ and $\theta(S_1)$ lie in the same $T$-system, where here $S_i$ denotes the image of $S$ in $F/N_i$, then there is an automorphism $\psi\in \Aut(F)$ such that $\psi(N_1) = N_2$. 

When $F/N_1$ and $F/N_2$ are one-relator presentations, we may say something much stronger. If $N_1 = \normal{w}$ and $N_2 = \normal{r}$, the Magnus property for free groups (see \cref{Magnus_property}) implies that if $S_2$ and $\theta(S_1)$ lie in the $T$-system, then there is an automorphism $\psi\in \Aut(F)$ such that $\psi(w) = r^{\pm1}$. Hence, one-relator groups $G = F/\normal{w}$ with a single $T$-system of (one-relator) generating set admit a particularly simple solution to the isomorphism problem amongst one-relator groups: given another one-relator presentation $H = F/\normal{r}$, to determine whether $G\isom H$, we only need to check whether $w$ and $r^{\pm1}$ lie in the same $\Aut(F)$-orbit which can be done in polynomial time using Whitehead's algorithm \cite{RVW07} (see \cite{Wh36} for Whitehead's original algorithm).

The isomorphism problem for two-generator one-relator groups with torsion was solved first by Pride \cite{Pr77} who showed that if $G = F/\normal{w^n}$ is a two-generator one-relator group with $n\geqslant 2$, then $G$ has only one $T$-system of generating pair, see \cref{Thm:Pride-isomorphism-theorem}. Note that $G = \langle a, b \mid a^5\rangle$ has more than one Nielsen class of generating set since $\{a^2, b\}$ also generates $G$. However, the automorphism given by $b\to b$ and $a\to a^2$ shows that $\{a, b\}$ and $\{a^2, b\}$ lie in the same $T$-system.

Although now it is known that the isomorphism problem is decidable for all one-relator groups with torsion, one still might wonder whether Pride's $T$-system result can be extended to obtain a more intrinsic solution to the word problem. This question was explicitly raised by Pride in \cite{Pr75}.

\begin{qstn}[Pride]
\label{Nielsen_class_torsion}
If $G = F/\normal{w^n}$ is a one-relator group $n\geqslant 2$, is it true that $G$ has one $T$-system of generating set of minimal cardinality?
\end{qstn}

Making explicit the connection between the work of Louder--Wilton on two-complexes $X$ with uniform negative immersions and Nielsen equivalence classes of generating sets for $\pi_1(X)$ would be a worthwhile task. In particular, their work should imply at least that one-relator complexes $X$ with negative immersions have fundamental group with finitely many Nielsen classes of generating sets of any given finite cardinality. Furthermore, we suspect that Pride's Question \ref{Nielsen_class_torsion} can be answered positively also for this setting.

\begin{qstn}
\label{Nielsen_class}
If $X$ be a one-relator complex with negative immersion, is it true that $G = \pi_1(X)$ has one $T$-system of generating set of minimal cardinality?
\end{qstn}

\cref{Nielsen_class} is known to have a positive answer for closed surfaces: if $G$ is the fundamental group of a closed surface, Lars Louder showed in \cite{Lo15} that $G$ has precisely one Nielsen class of generating set of any finite cardinality, improving on previous work of Zieschang \cite{Zi70}.

See also \cite[Problem (OR13)]{BMS02} for a more general phrasing of Question \ref{Nielsen_class}. 

During the 70's, a significant number of contributions were made towards understanding the generating sets of one-relator groups. We summarise below some of the main results in this direction. The reader is directed to \cref{Subsec:iso-problem} for more details.
\begin{itemize}
\item The first known examples of one-relator groups with more than one Nielsen class of one-relator generating sets were the torus knot groups $\langle a, b\mid a^m = b^n\rangle$. This was proved by McCool--Pietrowski \cite{MP73}.
\item Delzant \cite{De91} showed that torsion-free two-generator hyperbolic groups have finitely many Nielsen classes of generating pairs (a result announced by Gromov \cite{Gr87}), although even $C'(1/6)$ two-generator one-relator groups can gave more than one Nielsen generating pair by work of Pride \cite{Pr76}.
\item If $G\isom \bs(1, n)$ for any $n\neq 0$, Brunner showed that $G$ has only one $T$-system of generating pair \cite[Theorem 2.4]{Br74} and thus every two-generator presentation of $G$ is a one-relator presentation. In the same article, Brunner also showed that $\bs(2, 3)$ has infinitely many $T$-systems of generating pairs which are not associated with one-relator presentations and at least two $T$-systems of generating pairs which are associated with one-relator presentations.
\item Brunner provided an example in \cite{Br76} of a two-generator one-relator group with infinitely many Nielsen equivalence classes of generating pairs associated with one-relator presentations. The group was the Baumslag--Gersten group and the one-relator presentations associated with the generating pairs were:
\[
\bg(1, 2) \isom \left\langle a, t \mid t^{-1}a^{-2^n}tat^{-1}a^{2^n}t = a^2\right\rangle
\]
where $n$ is any natural number. It can be checked that each relator has minimal length in its $\Aut(F(a, t))$-orbit and so the associated generating sets 
\[
\{t^{-1}a^{-n}tat^{-1}a^{n}t, t\}\subset F(a, t)/\normal{a^{a^t}a^{-2}}
\]
are also in distinct $T$-systems.
\item Non-cyclic one-relator groups with centre may have many $T$-systems of generating pairs (they are all generated by two elements). However, Pietrowski \cite{Pi74} showed that their isomorphism problem is decidable amongst all one-relator groups. Indeed, Baumslag--Taylor proved that it is decidable whether a given one-relator group has centre in \cite{BT68}. Pietrowski then showed that any non-cyclic one-relator group with centre splits as a graph of infinite cyclic groups in an essentially unique way \cite{Pi74} which then yields a solution to the isomorphism problem using Tietze transformations.
\item Rosenberger solved the isomorphism problem for cyclically pinched one-relator groups in \cite{Ro94} by an in depth analysis of the Nielsen classes of generating sets. Some partial results were also obtained for cyclically conjugated one-relator groups by Fine--R\"ohl--Rosenberger \cite{FRR93}.
\item A generic one-relator group has precisely one Nielsen class of (minimal size) generating set by work of Kapovich--Schupp \cite{KS05}.
\end{itemize}
Although it appears that the isomorphism problem might be easily solvable amongst one-relator groups $G = F/\normal{w}$ with $\pi(w) \neq 2$, the above results for two-generator one-relator groups show that the general case presents several additional difficulties we do not yet know how to overcome.

\bibliographystyle{amsalpha}
\bibliography{joint_bibliography}

\end{document}